\documentclass[11pt,reqno]{amsart}
\usepackage{times}
\usepackage{amsmath,amsfonts,amstext,amssymb,amsbsy,amsopn,amsthm,eucal}
\usepackage{txfonts}
\usepackage{dsfont}
\usepackage{graphicx}   
\usepackage{hyperref}
\usepackage{color}
\usepackage{cancel}
\usepackage{verbatim}

\numberwithin{equation}{section}

\hypersetup{
  pdftitle={Energy identity and quantization for stationary harmonic maps},
  pdfauthor={Aaron Naber and Daniele Valtorta},
  pdfsubject={},
  pdfkeywords={},
  pdfpagelayout=SinglePage,
  pdfpagemode=UseOutlines,
  colorlinks,
  linkcolor=[rgb]{0,0,0.7},
  urlcolor=[rgb]{0,0,0.4},
  citecolor=[rgb]{0.4,0.1,0}
}

\newcommand{\tr}{\text{tr}}

\newcommand{\Vol}{\text{Vol}}
\newcommand{\diam}{\text{diam}}

\newcommand{\inj}{\text{inj}}
\newcommand{\Sing}{\text{Sing}}

\newcommand{\dC}{\mathds{C}}

\newcommand{\dN}{\mathds{N}}

\newcommand{\dR}{\mathds{R}}

\newcommand{\cA}{\mathcal{A}}
\newcommand{\cB}{\mathcal{B}}

\newcommand{\cE}{\mathcal{E}}
\newcommand{\cEE}{\mathcal{E}_{\alpha,x,r}}
\newcommand{\cEn}{e_{\alpha,x,r}}
\newcommand{\cFn}{f_{\alpha,x,r}}
\newcommand{\cF}{\mathcal{F}}
\newcommand{\cG}{\mathcal{G}}
\newcommand{\cH}{\mathcal{H}}

\newcommand{\cL}{\mathcal{L}}

\newcommand{\cM}{\mathcal{M}}
\newcommand{\cN}{\mathcal{N}}

\newcommand{\cP}{\mathcal{P}}

\newcommand{\cS}{\mathcal{S}}

\newcommand{\ff}{\mathfrak{f}}

\newcommand{\fr}{\mathfrak{r}}
\newcommand{\ft}{\mathfrak{t}}

\newcommand{\norm}[1]{\left\|#1\right\|}
\newcommand{\ps}[2]{\left\langle#1,#2\right\rangle}
\newcommand{\ton}[1]{\left(#1\right)}
\newcommand{\qua}[1]{\left[#1\right]}
\newcommand{\cur}[1]{\left\{#1\right\}}
\newcommand{\abs}[1]{\left|#1\right|}

\newcommand{\B}[2]{B_{#1}\ton{#2}}

\newcommand{\supp}[1]{\operatorname{supp}\ton{#1}}

\newcommand{\N}{\mathbb{N}}

\newcommand{\R}{\mathbb{R}}

\renewcommand{\paragraph}[1]{\ \newline \ \textbf{#1\ }}

\newcommand{\IF}{\Theta}
\newcommand{\ellperp}{\ell^\perp}

\newcommand{\dvx}{\operatorname{d}\mathfrak v(x)}
\newcommand{\dVy}{\operatorname{dV}(y)}

\newcommand{\hvt}{\hat \vartheta}

\newcommand{\rf}{\mathfrak{r} }
\newcommand{\LL}{\mathcal L}

\newcommand{\TLm}{\vartheta_{\LL}}

\newcommand{\II}{\operatorname{A}}
\newcommand{\Gr}{\operatorname{Gr}}
\newcommand{\id}{\operatorname{id}}
\newcommand{\Lip}{\operatorname{Lip}}
\newcommand{\T}{\mathcal T}
\newcommand{\dive}{\operatorname{div}}

\newcommand{\err}{\mathfrak E}
\newcommand{\ers}{\mathfrak F}

\newtheorem{theorem}{Theorem}[section]

\newtheorem{proposition}[theorem]{Proposition}
\newtheorem{lemma}[theorem]{Lemma}

\newtheorem{corollary}[theorem]{Corollary}
\theoremstyle{definition}
\newtheorem{definition}[theorem]{Definition}
\theoremstyle{remark}
\newtheorem{remark}[theorem]{Remark}
\theoremstyle{remark}
\newtheorem{example}[theorem]{Example}
\theoremstyle{remark}

\theoremstyle{remark}

\theoremstyle{remark}

\begin{document}

\title{Energy Identity for Stationary Harmonic Maps}


\author{Aaron Naber and Daniele Valtorta}\thanks{}

\date{\today}

\begin{abstract}


In this paper we consider sequences $u_j:B_2\subseteq M\to N$ of stationary harmonic maps between smooth Riemannian manifolds with uniformly bounded energy $E[u_j]\equiv \int |\nabla u_j|^2\leq \Lambda$ .  After passing to a subsequence it is known one can limit $u_j\to u:B_1\to N$ with the associated defect measure $|\nabla u_j|^2 dv_g \to |\nabla u|^2dv_g+\nu$, where $\nu = e(x)\, H^{m-2}_S$ is an $m-2$ rectifiable measure \cite{lin_stat}.  For a.e. $x\in S=\operatorname{supp}(\nu)$ one can produce a finite number of bubble maps $b_j:S^2\to N$ by blowing up the sequence $u_j$ near $x$.

We prove the energy identity in this paper.  Namely, we have at a.e. $x\in S$ that $e(x)=\sum_j E[b_j]$ for a complete set of such bubbles.  That is, the energy density of the defect measure $\nu$ is precisely the sum of the energies of the bubbling maps.

The analysis requires several new ideas.  The energy near a blow up region can be naturally decomposed $|\nabla u|^2 = |\pi_L u|^2+ |\pi_{L^\perp} u|^2= |\pi_L u|^2+ \langle\nabla u, n^\perp\rangle^2+ \langle\nabla u, \alpha\rangle^2$ into the $L^{m-2}$-symmetric and $L^\perp$-bubble directions, with the $2$-dimensional $L^\perp$-bubble directions themselves broken down into radial and angular components.  Each of the three components of the energy requires a new point.
In the symmetric $L$-energy directions we are able to provide a new log improved energy bound, and for the angular $\alpha$-energy we produce a new superconvexity estimate.  The most delicate new piece of analysis is needed for the radial $n$-energy component, where {\it apriori} the errors for the energy identity strictly larger than what is allowed.  To handle this we will approximate the $m-2$ blow up set by a submanifold $\mathcal T$ which solves an appropriate Euler-Lagrange equation.  This can be done so that these highest order errors are canceled by the equation for $\mathcal T$, at least up to controllable terms.
\end{abstract}

\maketitle

\tableofcontents

\section{Introduction}

\subsection{Background on Stationary Harmonic Maps}

Consider a finite energy stationary harmonic map $u:B_2\subseteq M\to N$, where $M$ and $N$ are smooth manifolds with $N$ compact.  The regularity of harmonic maps from surfaces $M=M^2$ is understood in great detail \cite{SacksUhlenbeck}, \cite{lamm_identity}, \cite{Jost2D}, \cite{ParkerBubbleTree},\cite{breiner_TOC_metric_target}, \cite{Breiner_CAT}, \cite{WWZ}, \cite{LinWang}, \cite{DingTian}, \cite{JLZ}.  In particular that such maps are smooth was proved in the minimizing case by Morrey \cite{morrey}, the stationary case by Schoen \cite{schoen_analytical_aspects}, and finally in the weakly harmonic case by Helein \cite{helein_1}.  A concise history of these and other results relating to harmonic maps can be found in \cite{HelWood_harmonic}.  There is some care here as the smoothness is only uniform away from a finite number of points, which is the essence of the bubbling phenomena being studied in this paper.\\

In higher dimensions singularities begin to appear for harmonic maps.  The worst singularities appear for weakly harmonic maps, which can be everywhere discontinuous according to \cite{Riviere_discontinuous}.  However, partial regularity can be proved in the stationary and minimizing cases. 
The regularity of stationary harmonic maps into sphere targets was first proved by Evans \cite{Evans_Stationary}.  
The regularity theory of harmonic maps in higher dimensions began in earnest with the work of Schoen-Uhlenbeck \cite{ScUh_RegHarm}, who were able to provide a stratification $\cS^0(u)\subseteq \cdots \cS^{m-1}(u)\subseteq \Sing(u)$ of the singular set in the sense of Federer, and prove the dimensional estimate $\dim \cS^k(u)\leq k$.  
Schoen and Uhlenbeck were also able to provide an $\epsilon$-regularity theorem for minimizing harmonic maps, which allowed them to show such maps are smooth away from a codimension $3$ set.  It was not until the work of Bethuel \cite{beth} that an $\epsilon$-regularity for stationary harmonic maps was available, which allowed him to show that stationary harmonic maps are smooth away from a set of $m-2$ measure zero, where $m$ is the dimension of the domain.  
The smoothness estimates on these mappings is once again only uniform away from a set of potentially positive $m-2$ measure.  Conjecturally, stationary harmonic maps are smooth away from a codimension three set.\\

Further structure theory on the singular sets $\cS^k(u)$ was first produced by Simon \cite{Simon_RegMin}, who introduced his infinite dimensional Lojasiewisz inequality and was then able to prove that for a minimizing $u$ into an analytic $N$, the top stratum $\Sing(u)=\cS^{m-3}(u)$ of the singular set is $m-3$ rectifiable.  These results were refined by the authors \cite{NV_RH}, who proved that for a stationary $u$ into a smooth $N$, every stratum $\cS^k(u)$ is $k$-rectifiable.\\

The analysis of sequences of stationary harmonic maps $u_j:B_2\subseteq M\to N$ in higher dimensions enjoyed a lot of progress with the work of Lin \cite{lin_stat}.  Among other things, Lin studied the defect measure $\nu$ defined as the non-absolutely continuous part of the measure limit $|\nabla u_j|^2 dv_g\to |\nabla u|^2 dv_g+\nu$.  It was proved \cite{lin_stat} that $\nu \equiv e(x) \cH^{m-2}_S$ is $m-2$ rectifiable, and further the sub-energy identity was shown.  That is, for a.e. $x\in S$ it was shown that for bubbles $b_k(x)$ associated to $x$ we have

\begin{align}\label{e:intro:energy_inequality}
	e(x)\geq \sum E[b_k]\, . 
\end{align}

The next major progress was accomplished by Lin and Riviere \cite{LR_W21}, who proved two primary results.  First, they showed that if the sequence of stationary harmonic maps $u_j$ satisfy a uniform $W^{2,1}$-estimate $\fint |\nabla^2 u_j|\leq A$, then the energy identity holds.  In the case that the target $N$ is a sphere, they were able to use concentration compactness arguments in order to prove the $W^{2,1}$ estimate directly.  Based on all of this, they conjectured such a $W^{2,1}$ estimate should also hold.  For $2$ dimensional base manifolds, then $W^{2,1}$ conjecture has been proved in \cite{LammSharp}, building on the techniques developed in \cite{LauRiv}. 
Independently, the energy identity was proven for $2$ dimensional base $M$ and general targets $N$ by Jost \cite{Jost2D} and Parker \cite{ParkerBubbleTree}. Lamm extended this result to approximate harmonic maps in \cite{lamm_identity}, and recently Breiner and Lakzian extended the results to metric targets in \cite{breiner_TOC_metric_target}.  It is worth mentioning that the energy identity does not hold for Sacks-Uhlenbeck generalized harmonic maps, see \cite{LiWangCounterexample}, making this a rather delicate phenomenon.\\

There are very related conjectures for stationary Yang-Mills.  In \cite{NV_YM} the authors proved both the $W^{2,1}$ conjecture and the energy identity conjecture for limits of stationary Yang-Mills connections.  The two results, the $W^{2,1}$ conjecture and the energy identity, were proved simultaneously by an underlying series of estimates that were stronger than either.  Rather unfortunately, the methods of \cite{NV_YM} break down for stationary harmonic maps and an entirely new set of tools need to be developed in this paper.  Indeed, we will see that each of the three different components of the energy tensor\footnote{$|\nabla u|^2 = |\nabla_L u|^2+|\ps{\nabla u}{n^\perp}|^2+|\ps{\nabla u}{\alpha^\perp}|^2$, see Section \ref{s:energy_decomposition}} will require a distinct and new idea.  See Sections \ref{s:outline_toymodel} and \ref{s:outline_general} for a more complete outline of the proof ideas of this paper. 

As the results of \cite{NV_YM} are closely related to those of this paper, it is worth a brief discussion of the methods of \cite{NV_YM}, and in particular what works and what breaks down in the context of stationary harmonic maps.  There were three main ingredients in the proof in \cite{NV_YM}.  The first was an effective annulus-bubble decomposition, which itself has its roots in \cite{NV_RH},\cite{JN}.  These decompositions allow one to break the mapping $u$ into various regions with special structure and analyze each independently.  The decompositions are covering arguments based on the monotone quantity, and all still hold in this paper essentially verbatim, see Section \ref{ss:broad_outline:quant_annulusbubble}.  

The next two new ingredients from \cite{NV_YM} were analytical, and both fail  completely in the harmonic map context.  The first was the introduction of an $\epsilon$-gauge, which was a linear replacement for a Coulomb gauge in singular regions where they may not exist.  The second was a new log-superconvexity estimate, which proved appropriately sharp estimates on these $\epsilon$-gauges.  A key property for the $\epsilon$-gauges in the Yang-Mills context was that, at least in the annular regions where most of the analysis takes place, the gauges were indeed a basis and an almost orthonormal one.  The analogous statement actually fails in the harmonic map context.  The corresponding linear $\epsilon$-gauges in the harmonic map context decay to zero at a logarithmic rate.  
Similarly, the log-superconvexity for the full energy cannot be generalized to the harmonic map context.  Both of these failures can be attributed to the codimension $2$ nature of the singularities instead of a codimension $4$ behavior.  
Morally, though clearly vastly oversimplifying, one can understand this as a consequence that the Green's function in dimension 2 is only a log, while it is a polynomial in higher dimensions.  Thus the methods of this paper are quite different, and indeed we do not prove the $W^{2,1}$ conjecture at all, we prove directly the energy identity.\\

\subsection{Main Results: The Energy Identity} 

We will break our main results down into several groups.  Let us begin by discussing the ineffective results, namely the energy identity itself.  In a subsequent subsection we will discuss the quantitative energy identity.  This takes a bit more work to even state with precision, however it is also likely to be more useful in future applications.  More to the point, this is really the result one necessarily proves in the paper, the ineffective result is simply a consequence.\\

Our setup is that we are considering a sequence of bounded energy stationary harmonic maps
\begin{align}\label{e:stationary_bounded_energy}
	u_j: B_2(p)\subseteq M\to N\text{ with }E[u_j]\equiv \int_{B_2}|\nabla u_j|^2\leq \Lambda\, ,
\end{align}
between smooth Riemannian manifolds with $N$ compact.  After passing to a subsequence we have that
\begin{align}\label{e:defect_measure}
	&u_j\rightharpoonup u:B_2(p)\subseteq M\to N \text{ with }E[u]\equiv \int_{B_2}|\nabla u|^2\leq \Lambda\, ,\notag\\
	&|\nabla u_j|^2 dv_g\to |\nabla u|^2dv_g+\nu \, ,
\end{align}
where the first convergence is weak in $H^1_{loc}$ while the second is in the sense of measures.  Note that the limit $u$ will be weakly harmonic, but need not be stationary harmonic (see \cite{dinglili}).  We have by \cite{lin_stat} that the defect measure $\nu$ is $m-2$ rectifiable and so can be written
\begin{align}\label{e:defect_measure_density}
	\nu \equiv e(x)\, \cH^{m-2}_S\, ,
\end{align}
as Hausdorff integration over an $m-2$ set $S$ with respect to a measurable function $e:S\to \dR^+$.\\

To state the energy identity carefully we need to discuss bubbling.  In short, a bubble is a blow up of our sequence of mappings which is nontrivial but maximally symmetric.  Consequently each bubble can be associated with a harmonic map $S^2\to N$, hence the term bubbling.  Precisely:

\begin{definition}[Blow Ups and Bubbling]
We define the following:
\begin{enumerate}
	\item We say $u_x:\dR^m\to N$ is a blow up of the sequence $u_j$ at $x$ if there exists $x_j\to x$ and $r_j\to 0$ such that $u_{x_j,r_j}\rightharpoonup u_x$, where  $u_{x_j,r_j}:B_{r_j^{-1}}(0)\subseteq T_{x_j}M\to N$ are defined by $u_{x_j,r_j}(y) = u_j(\exp_{x_j}(r_j y))$ .
	\item We say $b_x:\dR^2\to N$ is a bubble at $x$ if there exists an $m-2$ plane $L\subseteq \dR^m$ and a blow up $u_x:\dR^m\to N$ at $x$ such that $u_x(y,y^\perp)=b_x(y^\perp)$ , where we have written $\dR^m=L\times L^\perp$ .
\end{enumerate}
\end{definition}
\begin{remark}
	Note that $u_{x,r}$ is a harmonic map on $B_{r^{-1}}(0)\subseteq T_xM$ with respect to the pullback geometry under the exponential map.
\end{remark}
\begin{remark}
	It follows from removable singularity theorems that $b$ is necessarily a smooth harmonic map.  Due to the conformal invariance in dimension two, we can also view $b:S^2\to N$ as a harmonic map from the 2-sphere.
\end{remark}

The main Theorem of this paper is then the following:

\begin{theorem}[Energy Identity]\label{t:energy_identity}
Let $u_j:B_2(p)\subseteq M\to N$ be stationary harmonic maps between smooth Riemannian manifolds with $N$ compact and $\int_{B_2}|\nabla u_j|^2\leq \Lambda$ .  Let $u_j\rightharpoonup u$ with $|\nabla u_j|^2dv_g\to |\nabla u|^2dv_g+\nu$ the $m-2$ rectifiable defect measure $\nu = e(x)\,\cH^{m-2}_S$.  Then for a.e. $x\in S$ there exists bubbles $b_1,\ldots, b_K:S^2\to N$ at $x$ with $K\leq K(M,N,\Lambda)$ such that
\begin{align}
	e(x) = \sum_{j=1}^K E[b_j] = \sum_{j=1}^K \int_{S^2}|\nabla b_j|^2\, .
\end{align}
\end{theorem}
\begin{remark}
	The target space $N$ needs only $C^2$ regularity, see the quantitative energy identity of the next subsection and Remark \ref{rm:eps_reg_C2} . 
\end{remark}
\vspace{.5cm}

\subsection{Main Results: Quantitative Energy Identity}

The proof of Theorem \ref{t:energy_identity} necessarily must be done through much more effective means.  The quantitative energy identity will be applied not to a sequence of harmonic maps, but is instead a decomposition which holds for each individually fixed stationary harmonic map $u:B_2\to N$ .  In short, the quantitative energy identity will tell us that most of the time, in a precise sense, when we restrict $u$ to 2-dimensional subspaces $u:B_2\cap L^\perp\to N$ then $u$ will look very close to a 2 dimensional harmonic map on this subspace.  In the case where we apply the quantitative energy identity to a sequence $u_j$ and take the limit, we will recover the classical energy identity of Theorem \ref{t:energy_identity}. To begin, let us quantify our conditions on the manifolds $M$ and $N$, which are $C^2$ and satisfy :

\begin{align}\label{e:manifold_bounds}
&|\sec_{B_2(p)}|\leq K^2_M,\,\, \inj(B_2(p))\geq K^{-1}_M\, ,\notag\\
&|\sec_N|\leq K^2_N,\,\, \inj(N)\geq K^{-1}_N,\,\,\diam(N)\leq K_N \, ,\notag\\
&\dim(M)=m,\, \dim(N)=n\, .\\ \notag
\end{align}

In order to state the Theorem properly we must introduce a host of structure, beginning with the notion of quantitative symmetry.  We will give a general definition, in order to put the results into a broader context, though our main interest will be on $(m-2,\epsilon)$-symmetry in this paper.  Recall from the previous subsection that if $u:B_2(p)\subseteq M\to N$ is a mapping and $B_r(x)\subseteq B_2$, then we define the blow up mapping $u_{x,r}:B_1(0)\subseteq T_xM\to N$ by $u_{x,r}(y)\equiv u(\exp_x(ry))$.\\

\begin{definition}[Quantitative Symmetry]
	Let $u:B_2\subseteq M\to N$ be a stationary harmonic map.  We say that $u$ is $(k,\epsilon)$-symmetric on $B_r(x)\subseteq B_2$ if $K_M^2<\epsilon\cdot r^{-2}$ and there exists a $k$-plane $L^k\subseteq T_xM$ such that
\begin{align}
	\int_{}\;\ton{\abs{\pi_L \nabla u_{x,r}}^2+\langle\nabla u_{x,r},y^\perp\rangle^2 }\,\rho(y)\,dy\leq \epsilon\, ,
\end{align}
where $y^\perp\in L^\perp$ is the perpendicular radial vector and $\pi_L:\dR^n\to L$ is the projection.
\end{definition} 
\begin{remark}
	One can think of $\rho(y)$ as the characteristic function $\chi_1(y)$ on $B_1(0)$, so that the above becomes $\int_{B_1}\;\abs{\pi_L \nabla u_{x,r}}^2+\langle\nabla u_{x,r},y^\perp\rangle^2 \,dy\leq \epsilon$ .  It will be convenient to choose $\rho\approx e^{-\frac{1}{2}|y|^2}\chi_R$ to be a smooth cutoff on $B_R(0)$, where $R=R(m)$.  See Section \ref{ss:outline:heat_mollified_energy} for more details.
\end{remark}
\begin{remark}
	Note that $|y^\perp| = d(y,L)$, so that the $L^\perp$ radial direction of the energy is being weighted more strongly than $L$-directions.  This phenomena is what allows for bubbling in the first place.
\end{remark}
\begin{remark}
	We say $B_r(x)$ is $(k,\epsilon)$-symmetric with respect to $L$ if we want to emphasize the $k$-plane.  We may also write $L^k_x$ to emphasize that this is an affine $k$-plane through the point $x$.
\end{remark}
\begin{remark}
 Note that this definition is slightly stronger than the ones in \cite{NV_RH},\cite{ChNa1},\cite{ChNaVa}. In those papers the quantitative symmetry is defined via a notion of strong $L^2$ closeness.  As we will see in Section \ref{ss:prelim:cone_splitting}, the definition in this paper is more suited for quantitative analysis.\\
\end{remark}

The quantitative stratification $\cS^k_\epsilon$ is then defined as in \cite{NV_RH},\cite{ChNa1},\cite{ChNaVa} 
\begin{align}
	\cS^k_\epsilon \equiv \{x\in B_1: \forall\; 0<r<1 \text{ we have that $u$ is not }(k+1,\epsilon)\text{-symmetric on } \B r x\}\, .
\end{align}
as those points for which no ball is $(k+1,\epsilon)$-symmetric.  Note that the classical stratification may be recovered as $\cS^k=\bigcup_{\epsilon}\cS^k_\epsilon$ .  See Section \ref{ss:prelim:quant_strat} for additional refinements of the quantitative stratification which will be used.  The main result of \cite{NV_RH} is that $\cS^k_\epsilon$ is $k$-rectifiable with finite Minkowski and Hausdorff volume estimates
\begin{align}
	\Vol(B_r\cS^k_\epsilon )\leq C_\epsilon(n,m,K_M,K_N,\epsilon)r^{n-k}\implies \cH^k(\cS^k_\epsilon)<C_\epsilon\, .
\end{align}
See Section \ref{ss:prelim:quant_strat} for a review\footnote{Note that this result holds even with the stronger definition of quantitative symmetry of the present paper as the key cone-splitting theorems also hold in this case (see Section \ref{ss:prelim:cone_splitting}).}. \\

The main work of our quantitative energy identity will be on balls which are $(m-2,\epsilon)$-symmetric, which by the quantitative stratification of \cite{NV_RH} are most balls.  Given that $u$ is $(m-2,\epsilon)$ symmetric on $B_{2r}(x)$ wrt $L^{m-2}$, it is a classical result that the energy measure $|\nabla u|^2 dv_g$ must be close (as a measure) to a $m-2$ Dirac measure $e_0 \cH^{m-2}_L$ for some constant $e_0>0$, where $\cH^{m-2}_L$ is the $m-2$ Hausdorff measure on the plane $L$.  See Theorem \ref{t:prelim:spacial_gradient} and the remarks after for a precise statement and quantitative proof.\\

The quantitative energy identity will discuss the smaller scale behavior of $u$ on an $(m-2,\epsilon)$-symmetric ball.  We will show that for most $2$-planes $L^\perp_y \equiv y+L^\perp$ with $y\in B_r(x)$ that $u:L^\perp_y\cap B_r(x)\to N$ is very close to a $2$-dimensional harmonic map.  We need to make this precise:

\begin{definition}[$\epsilon$-Almost Harmonic Map] Let $v:B_r(0^2)\cap \dR^2\to N$ be a smooth mapping, we say it is $\epsilon$-harmonic if $\exists$ $\Big\{B_{r_j}(x_j)\Big\}_1^K\subseteq B_r$ along with nontrivial harmonic maps $b_j:\dR^2\to N$ such that
\begin{enumerate}
\item $|\nabla b_j|\leq 1$ with $\int_{\dR^2\setminus B_{\epsilon^{-1}}}|\nabla b_j|^2 < \epsilon$.
\item If $r_j\leq r_k$ then either $B_{r_j}(x_j)\cap B_{r_k}(x_k)=\emptyset$ or $B_{r_j}(x_j)\subseteq B_{r_k}(x_k)$
\item If we denote $\tilde b_j\equiv b_{j}\big(\epsilon r_j^{-1}(x-x_{j})\big)$ then it holds that
\begin{align}
&r_k \big|\nabla(v-\tilde b_k)\big| < \epsilon \text{ on each }	B_{r_k}(x_k)\setminus \bigcup_{r_j\leq r_k} B_{r_j}(x_j)\, ,\notag\\
&\text{with }r^2\fint_{B_r\setminus \bigcup B_{r_j}} |\nabla v|^2 < \epsilon\, .
\end{align}
\end{enumerate}
\end{definition}

Our main theorem is now the following, which tells us that if $u$ is $(m-2,\delta)$-symmetric on a ball $B_r(x)$ wrt some $L$, then on most $L^\perp$ slices the functions looks like a $2$-dimensional harmonic map in a strong sense: \\

\begin{theorem}[Quantitative Energy Identity]\label{t:main_quant_energy_id}
Let $u:B_{10R}(p)\subseteq M\to N$ be a stationary harmonic map with $R^2\int_{B_{10R}} |\nabla u|^2\leq \Lambda$ such that \eqref{e:manifold_bounds} is satisfied with $R=R(m)$.  Let $0<\epsilon<\epsilon(m,K_N,\Lambda)$ and assume $u$ is $(m-2,\delta)$-symmetric on $B_2(p)$ wrt $L^{m-2}$ with $\delta\leq \delta(n,m,K_N,\epsilon,\Lambda)$.  Then there exists $K(m,K_N,\Lambda)<\infty $ and $\cG_\epsilon\subseteq L\cap B_1$ such that
\begin{enumerate}
	\item $\Big|L\cap \big(B_1\setminus \cG_\epsilon\big)\Big|<\epsilon$ ,
	\item For each $y\in \cG_\epsilon$ we have $\int_{L^\perp_y\cap B_1}\abs{\pi_L^\perp \nabla u}^2<C(n)\Lambda$ and $\int_{L^\perp_y\cap B_1}\abs{\pi_L \nabla u}^2<\epsilon$ ,
	\item For each $y\in \cG_\epsilon$ we have that $u:L^\perp_y\cap B_1\to N$ is an $\epsilon$-harmonic map with at most $K$ bubbles.
	\item For each $y\in \cG_\epsilon$ we have that $\big|\,\fint_{B_1(y)}|\nabla u|^2 - \omega_{m-2}\int_{y+L^\perp}|\pi^\perp_{L}\nabla u|^2 \,\big| < \epsilon$
\end{enumerate}
\end{theorem}
\begin{remark}
	Note that the number of bubbles on each slice of $\cG_\epsilon$ is independent of $\epsilon>0$, and depends only on the geometry of $N$ and the energy of $u$.
\end{remark}

The classical Energy Identity follows from the above in a few lines by combining this with the quantitative stratification of \cite{NV_RH}.  This is done in Section \ref{ss:energy_identity_proof}.\\

Note that in the technical estimates of this paper we will only work on balls $B_r(p)\subseteq M$ for which $r<<\inj(M)$, and as such we will often implicitly be working on exponential or harmonic coordinate charts of $M$.  Further we will then proceed with most proofs under the assumption that $M=\dR^n$ is flat, with the understanding that the errors arising from the curvature of $M$ are of a lower order nature and tend to just add technical headaches which distract from the more fundamental errors.

\vspace{.5cm}

\section{Preliminaries}\label{s:prelim}

Before beginning an in depth discussion of the proof and paper, it seems a convenient time to introduce some preliminaries.  Some of this section is meant as review of other literature, and some of this section is meant as an introduction to some of the technical structure which will play an important role.  The reader may skip much of this and return to it as needed.\\

\subsection{Weakly and Stationary Harmonic Maps}

In the case of a real valued mapping $u:(M,g)\to \dR$ there is no confusion in what is meant for $u$ to be a harmonic map $\Delta u = 0$.  Such a mapping will automatically be smooth, so distributional solutions are the same as smooth solutions.  More relevant, whether $u$ is viewed as being a solution of $\Delta u = 0$, a critical point of the energy functional $\int_M |\nabla u|^2$, or a minimizer of the energy functional are all equivalent.\\

In the case of a mapping between Riemannian manifolds $u:(M,g)\to (N,h)$ all these various notions are distinct.  In particular, we can consider the energy functional
\begin{align}\label{e:prelim:energy_functional}
	E[u]\equiv \int_M|\nabla u|^2\, .
\end{align}
Though not necessary, it is often convenient to view $N\subseteq \dR^N$ as being isometrically embedded into Euclidean space, so that $u:M\to \dR^N$ becomes a mapping into Euclidean space with values in the subset $N$.  In this case if $u$ is a critical point of the energy functional \eqref{e:prelim:energy_functional} then its the Euler-Lagrange is given by
\begin{align}\label{e:prelim:weakly_harmonic}
	\Delta u (y)= \II_{u(y)}(\nabla u(y),\nabla u(y))\, ,
\end{align}
where $\Delta$ is the linear Laplacian on $M$ and $\II_{u(y)}$ is the second fundamental form on $N\subseteq \dR^N$ at $u(y)$.  We call $u\in H^1(M;\dR^N)$ a weakly harmonic map if it is a weak solution of the Euler Lagrange \eqref{e:prelim:weakly_harmonic}.\\

In this nonlinear case it turns out that being a weak solution of \eqref{e:prelim:weakly_harmonic} is {\it not }equivalent to being a critical point of the energy functional.  The Euler-Lagrange arises from so called target variations, while if $u$ is not smooth one is potentially missing domain variations.  Namely, let $\phi_t:M\to M$ be a family of diffeomorphisms generated by a compactly supported smooth vector field $\xi^j$, and let $u_t\equiv u\circ \phi_t$ be the associated family of variations of $u$.  Then if $\frac{d}{dt}E[u_t]=0$ this would give us the equality
\begin{gather}\label{e:stationary_equation}
 \int \underbrace{\ton{\abs{\nabla u(y)}^2\delta_{ij}-2\langle\nabla_i u(y), \nabla_j u(y)\rangle }}_{\equiv\, S_{ij}(y)}\partial^i \xi^j(y) =0 \, .
\end{gather}
In other words, the stress energy tensor $S_{ij}$, which is a symmetric $m\times m$ matrix, is divergence free in the distributional sense.  We call $u$ a stationary harmonic mapping if in addition to being weakly harmonic it satisfies \eqref{e:stationary_equation} for all compactly supported smooth vector fields.  The stationary equation is in fact the source for much of the important structure in the nonlinear context.\\

Note that if we assume that $u$ is smooth, the stationary equation is equivalent to the pointwise equality
\begin{gather}
 \ps{\Delta u}{\nabla u}=0\, ,
\end{gather}
which is implied by $\Delta u=\II(\nabla u,\nabla u)$. 

\vspace{.3cm}

Minimizers are of course stationary, but the other implication fails in the nonlinear context. Let us just remark that thanks to Luckhaus' lemma in \cite{luck}, minimizers enjoy strong $W^{1,2}$ compactness. In particular, a sequence of minimizers with bounded energy always has a $W^{1,2}$ strongly convergent subsequence whose limit is again a minimizer. This is not true for stationary maps, as bubbling can occur (see example \ref{ex:bubble1}). Moreover, it is worth mentioning that weak $W^{1,2}$ limits of stationary maps are in general not stationary, see \cite{dinglili}.

\subsection{\texorpdfstring{$\epsilon$}{epsilon}-Regularity for Harmonic Maps}\label{ss:prelim:eps_reg}

An important result in the study of stationary harmonic maps is the $\epsilon$-regularity theorem, see for example \cite{beth} and a generalization in \cite{RS}. Here we briefly recall it for the reader's convenience.
\begin{theorem}[$\epsilon$-Regularity]\label{t:eps_reg}
 Let $u:M\to N$ be a stationary harmonic map. Then there exists a constant $\epsilon_0=\epsilon_0(m,n,K_M,K_N)$ as in \eqref{e:manifold_bounds} such that
 \begin{gather}
  r^{2-m}\int_{\B {2r} p} \abs{\nabla u}^2\leq \epsilon_0 \qquad \implies \qquad u\in C^\infty(\B {r}{p})\, .
 \end{gather}
Moreover, elliptic estimates hold, and in particular for all $z\in \B r p $:
\begin{gather}
 r^2\abs{\nabla u(z)}^2+r^4\abs{\nabla^2 u(z)}^2\leq C(m)r^{2-m}\int_{\B {2r} p} \abs{\nabla u}^2\, .
\end{gather}
\end{theorem}
\begin{remark}\label{rm:eps_reg_C2}
	The $\epsilon$-regularity of Bethuel requires at least $C^5$ regularity of the target space $N$. With a different method based on antisymmetric potentials developed by Rivi{\`e}re \cite{Riviere_Cons}, Rivi{\`e}re and Struwe \cite{RS} extended the $\epsilon$-regularity of Bethuel to targets with only $C^2$ regularity.
\end{remark}
\vspace{.3cm}


Let us remark on a technical implication of the above regularity that will be useful when we prove superconvexity results for the angular energy:  

\begin{lemma}\label{l:nabla_Delta_u_epsilon_reg}
If $u:M\to N\subset \R^N$ be a stationary harmonic map with $r^{2-m}\int_{\B {2r} 0} \abs{\nabla u}^2\leq \epsilon\leq \epsilon_0$. Let $w$ be any unit tangent vector, then we have on $\B r 0$ the pointwise inequality:
\begin{gather}\label{eq_uw2}
 r^2\abs{\ps{\nabla_w u}{\nabla_w \Delta u}}\leq c(n,K_N) \epsilon \abs{\nabla_w u}^2\, .
\end{gather}
\end{lemma}
\begin{proof}
Let us denote $u_w \equiv \nabla_w u$, and recall the Euler Lagrange equation $\Delta u = A(\nabla u,\nabla u)$, where $A$ is the second fundamental form of the isometric embedding $N\subseteq \dR^N$, which coincides with the Hessian of the nearest point projection $\Pi:\R^N\to N$. We will keep using $A$ to denote the Hessian $\nabla^2 \Pi$, which is a quadratic form $A(x):TR^N\times TR^N\to TR^N$ for all $x$ close enough to $N$.  Differentiating we obtain
\begin{align}
	\ps{\nabla_w u}{\nabla_w \Delta u} =\ps{u_w}{\nabla A[u_w,\nabla u,\nabla u]}+2\ps{u_w}{A[\nabla u_w,\nabla u]}\, .
\end{align}
Using Theorem \ref{t:eps_reg} the correct estimate holds on the first term, and so we need only deal with the second term above.  Observe that $\ps{u_w}{A[u_w,\nabla u]} = 0$ as $A$ applied to two tangent vectors to $N$ is orthogonal to $N$.  Thus we can write
\begin{align}
	&\ps{u_w}{A[\nabla u_w,\nabla u]} = \nabla\ps{u_w}{A[u_w,\nabla u]}-\ps{\nabla u_w}{A[u_w,\nabla u]}-\ps{u_w}{A[u_w,\Delta u]}\, .
\end{align}
By standard arguments, see for example \cite[lemma 3.2 p 60]{moser_appr_harm}, we have that $\ps{\nabla u_w}{A[u_w,\nabla u]}=\ps{(\nabla u_w)^\perp}{A[u_w,\nabla u]}=\ps{u_w}{A[(\nabla u_w)^\perp,\nabla u]}=\ps{u_w}{A[\nabla u_w,\nabla u]}$, and so we obtain:
\begin{align}
	& 2\ps{u_w}{A[\nabla u_w,\nabla u]} = -\ps{u_w}{A[u_w,\Delta u]}\, .
\end{align}
Plugging in the estimates of  Theorem \ref{t:eps_reg} finishes the proof. 
\end{proof}

\vspace{.3cm}

\subsection{Heat Mollified Energy of Nonlinear Harmonic Maps}\label{ss:outline:heat_mollified_energy}

Let us begin our discussion with the energy of a nonlinear harmonic map.  Besides being a useful review, our precise choices of energy in this paper will be a theme as they are more than just convenient, some of the more precise estimates will depend on these choices in a delicate manner. \\

Beginning at the beginning, arguably the most useful aspect of a nonlinear stationary harmonic map $u:B_2\to N$ is that its normalized energy
\begin{align}
	\theta(x,r)\equiv r^{2-m}\int_{B_r(x)}|\nabla u|^2\, ,
\end{align}
is monotone nondecreasing in $r$ for each $x\in B_2$ .  Some of the common consequences of this is the existence and $0$-symmetry of tangent maps.  This is the starting point of a stratification theory, see \cite{ScUh_RegHarm},\cite{NV_RH}. \\

 It can be convenient at times to mollify the above, so that the mollified energy is a smooth function of both $x$ and $r$.  More precisely, for any choice of smooth function $\rho:\dR^+\to \dR^+$ with $\rho\geq 0$ and $\dot\rho\leq 0$ we can define the mollified energy
 \begin{align}\label{e:outline:mollified_energy}
 	\vartheta(x,r)\equiv r^{2-m}\int \rho\ton{\frac{|y-x|^2}{2r^2}}|\nabla u|^2 \equiv r^{2}\int \rho_r(y-x)\,|\nabla u|^2(y)\, ,
 \end{align}
 which is a monotone nondecreasing function of $r$ for each $x$ satisfying
 \begin{align}\label{e:outline:mollified_energy_monotonicity}
 	r\frac{d}{dr}\vartheta(x,r) = -2\int\dot\rho_r(y-x)\langle\nabla u,y-x\rangle^2 \geq 0\, .
 \end{align}
The above is proven using the stationary equation \eqref{e:stationary_equation} together with the radial vector field $\xi=\rho_r(y-x) (y-x)$. \\

It will turn out that some choices of $\rho$ will be better than others, and ideally we would like to pick the Gaussian $\rho(t) = (2\pi )^{-m/2}e^{-t}$ .  For this choice our mollifier $\rho(y-x)$ should be viewed as the heat kernel, so that $\vartheta(x,\sqrt{t})/t$ solves the heat flow as a function of space and time.  This particular choice of mollifier will play a special role in our analysis and makes for a better behaved energy functional.  As we are working locally, we will cut off $\rho$ after a finite number of scales.  In particular, we will choose $\rho$ so that
 
 \begin{definition}[Cutoff Heat Kernel Mollified Energy]\label{d:heat_mollifier}
 Our mollified energy functional $\vartheta(x,r)$ from \eqref{e:outline:mollified_energy} will be taken with respect to a smooth function $\rho:[0,\infty)\to [0,\infty)$ such that
 \begin{enumerate}
 	\item $\rho(t) = c_m e^{-t}$ for $t\in [0,R]$ 
 	\item $\text{supp }\rho \subseteq [0,\,R+2]$
 	\item $0\leq \ddot\rho(t) \leq \,c_m\, e^{-R}$ for $t\in[R,R+2]$.
 \end{enumerate}
 Notice that with this properties we have $\rho(t)\leq -2\dot \rho(t)$ for all $t\in [0;\infty)$.

 As done above, we will also set $\rho_r, \dot \rho_r:\R^m\to [0,\infty)$ as
 \begin{gather}
  \rho_r(y)= r^{-m} \rho \ton{\frac{\abs{y}^2}{2r^2} }\, , \qquad \dot \rho_r(y)= r^{-m}\dot \rho\ton{\frac{\abs{y}^2}{2r^2} }\, .
 \end{gather}

 \end{definition}
 \begin{remark}
 	We will choose $R=R(m,\Lambda)$ at later stages of the proof.
 \end{remark}
 \begin{remark}
 	The coefficient $c_m=c(m,R)\approx (2\pi)^{-m/2}$ is uniquely chosen so that $\int_{\dR^m} \rho\Big(\frac{|x-y|^2}{2}\Big)=1$ is a probability measure.
 \end{remark}
 \begin{remark}
 	We see that $\theta(x,r)$ and $\vartheta(x,r)$ are uniformly equivalent in that there exists $C(m,R)$ such that $C^{-1}\theta(x,r)\leq \vartheta(x,r)\leq C\,\theta(x,Rr)$ .
 \end{remark}

 The choice of $\rho$ is going to be important because it will allow us to exchange $\rho$ with $\dot \rho$ up to a $e^{-R/2}$-small error. In particular, we have 
  \begin{lemma}[Basic Properties of the Heat Kernel Mollifier]\label{l:rho_basic_properties}
We have the following uniform bounds
:
\begin{enumerate}
	\item For all $R$ and $t\in [0,\infty)$ we have
	  \begin{gather}\label{e:trho_doubleradius}
   -\dot \rho(t)\geq 0\, , \qquad -\dot\rho(t)-t\dot \rho(t)+t^2\ddot\rho(t)\leq C \min\cur{\rho(t/1.1),-\dot \rho(t/1.1)}\, .
  \end{gather}
 	\item Using the convexity of $\rho$ we have for all $R$ and $t\in [0,\infty)$ that:
 	 \begin{gather}\label{e:rho_leq_dot_rho}
  \rho(t) \leq -C(m)\dot \rho(t) \, ,
 \end{gather}
 	\item The almost Gaussian nature of $\rho$ gives for all $R$ and $t\in [0,\infty)$ that:
 \begin{gather}\label{e:rho_primitive_difference}
  \abs{\rho(t)+\dot \rho(t)}\leq 2c_m e^{-R}\chi_{[R,R+2]}\leq 20e^{-R/2}\rho(t/2)\, .
 \end{gather}
 	\item If $R\geq 20$ and $r\leq s\leq 10 r$, then we have that
  \begin{gather}
  \B {r}{x}\subset \B {s}{x'}\qquad \Longrightarrow \qquad \rho_r(y-x)\leq C(m)\rho_{s}(y-x')\, .
 \end{gather}
\end{enumerate}
\end{lemma}
\begin{proof}
 All the statements in this lemma are straightforward from the definition and standard computations.
\end{proof}

The last lemma of this section is a standard comparison, analogous to the one available with the standard definition of $\theta$.
\begin{lemma}
 Let $u:\B {10 R}{p}\to N$ be a stationary harmonic map. If $\B {r}x\subset \B {r'} {x'}\subset  \B {10R}{p}$, then
 \begin{gather}
  \vartheta(x,r)\leq C(m) \vartheta(x',r')\, .
 \end{gather}
\end{lemma}
\begin{proof}
 The proof is straightforward from integral comparison and monotonicity.
\end{proof}

We can then recognize that the heat mollified energy is roughly equivalent to the standard energy but with better space-time behavior.  For much of the analysis any choice of smooth mollifier $\rho$ will do, however in Section \ref{s:radial_energy} the specific choice of the heat kernel will eliminate certain higher order errors which are otherwise difficult to deal with.\\

\vspace{.3cm}

\subsection{Energy and Partial Energy Functionals}\label{ss:prelim:partial_energy}

Recall from Section \ref{ss:outline:heat_mollified_energy} the definition of the heat mollified energy
 \begin{align}\label{e:prelim:mollified_energy}
 	\vartheta(x,r)\equiv r^{2-m}\int \rho\ton{\frac{|x-y|^2}{2r^2} }|\nabla u|^2 \equiv r^{2}\int \rho_r(x-y)\,|\nabla u|^2(y)\, ,
 \end{align}
 where $\rho(t)\approx c_m e^{-t}$ is the heat kernel mollifer which is cutoff at scale $R>0$ as in Definition \ref{d:heat_mollifier}.  For a stationary harmonic map the heat mollified energy $\vartheta(x,r)$ is monotone with $r\dot\vartheta(x,r)\geq 0$, a consequence of the stationary equation.  We will imprecisely refer to $r\dot\vartheta(x,r)\approx |\vartheta(x,2r)-\vartheta(x,r)|$ as the pinching of the monotone quantity $\vartheta$ at the point $x$ and at scale $r$.\\
 
 It will be important throughout the article to consider similar definitions that focus only on some components of the energy.  These will of course no longer be monotone, though will be obviously bounded by the monotone quantity.  More importantly, in practice these partial energies will be controlled by the pinching of the monotone quantity.  Much of this subsection will be spent making this statement effectively precise.\\
 
To begin let us define our collection of partial energies:

\begin{definition}[Partial Energy Functions]\label{d:partial_energies}
	Given $u:B_{10R}\to N$ with $B_{r}(x)\subseteq B_2$ and $L^k\subseteq \dR^m$ a $k$-dimensional plane we define
\begin{align}
	\vartheta(x,r;L) &\equiv r^2\int \rho_r(y-x)\abs{\pi_L \nabla u(y)}^2\,dy = r^2\int \rho_r(y-x)|\pi_L\nabla u|^2\, ,\notag\\
	\vartheta(x,r;n_{L^\perp}) &\equiv \int \rho_r(y-x)\langle\nabla u,\pi_{L^\perp}(y-x)\rangle^2\,dy = \int \rho_r(y-x)|\pi_{L^\perp}(y-x)|^2\langle\nabla u,n_{L^\perp}\rangle^2\, ,\notag\\
	\vartheta(x,r;\alpha_{L^\perp}) &\equiv  \int \rho_r(y-x)|\pi_{L^\perp}(y-x)|^2\langle\nabla u,\alpha_{L^\perp}\rangle^2\,dy\equiv  \int \rho_r(y-x)|\pi_{L^\perp}(y-x)|^2\abs{ \pi_{\alpha_{L^\perp}} \nabla u}^2\, ,\notag\\
	\vartheta(x,r;L^\perp) &\equiv\vartheta(x,r;n_{L^\perp})+\vartheta(x,r;\alpha_{L^\perp})= \int \rho_r(y-x)|\pi_{L^\perp}(y-x)|^2\abs{\pi_L^\perp \nabla u(y)}^2\, ,
\end{align} 
where $n_{L^\perp}=\frac{\pi_{L^\perp}(y-x)}{|\pi_{L^\perp}(y-x)|}$ is the unit radial direction corresponding to the affine $2$-plane $x+L^\perp$ and $\alpha_{L^\perp}$ represents the corresponding $m-k-1$ spherical directions. 
\end{definition}
\begin{remark}
	Note that $\abs{\pi_L^\perp \nabla u(y)}^2 = \langle\nabla u,n_{L^\perp}\rangle^2+\langle\nabla u,\alpha_{L^\perp}\rangle^2$.  For most of this paper we will consider $k=m-2$, and so $\alpha_{L^\perp}$ is the single angular direction complementing $L^\perp$ and $n_{L^\perp}$ at every point away from $L$.
\end{remark}

\begin{remark}
Observe our notational convention that energies with respect to a perpendicular direction $\vartheta(x,r;L^\perp)$ are weighted with a decaying factor $|\pi_{L^\perp}(y-x)|^2$ near $L$.  As the context will always emphasize if we are looking in the $L$ or $L^\perp$ directions, this should not cause confusion.
\end{remark}

Observe that if $B_s(y)\subset B_r(x)$ with $s\geq r/10$ then we have the relatively straightforward estimate 
\begin{align}\label{e:prelim:partial_energy:basic_bounds}
	\vartheta(y,s;L)\leq C(m)\vartheta(x,r;L)\, , \text{ for }s\geq r/10\, . 
\end{align}
In the case $L=\dR^m$ one can drop the assumption that $s\geq r/10$, but of course this is essentially due to monotonicity. \\

\vspace{.3cm}

\subsection{Quantitative Stratification and \texorpdfstring{$(k,\epsilon)$}{(k,epsilon)}-Symmetry}\label{ss:prelim:quant_strat}

Let us carefully define the notion of quantitative symmetry.  We will base our definitions on the heat mollified energies of Section \ref{ss:outline:heat_mollified_energy} and Section \ref{ss:prelim:partial_energy}:\\

\begin{definition}[$(k,\epsilon)$-Symmetric Mappings]\label{d:quant_symmetries}
	Let $u:B_{10R r}\to N$ be a mapping, then we say that $u$ is $(k,\epsilon)$-symmetric on the ball $B_r(x)\subseteq B_2$ if $r^2 K_M^2<\delta$ and there exists a $k$-plane $L=L^k$ such that
\begin{align}
	\vartheta(x,r; L)+\vartheta_{}(x,r; n_{L^\perp}) \equiv r^2\int_{} \rho_r(x-y)\abs{\pi_L \nabla u(y)}^2+\int_{} \rho_r(x-y)\langle \nabla u,\pi_{L^\perp}(y-x)\rangle^2 < \epsilon\, .
\end{align}
\end{definition}

One should interpret the above as saying that in an $H^1$ sense the mapping is close to being invariant in the $L$-directions, and in a weak $H^1$ sense the mapping is close to being radially invariant.  The weak aspect is because the energy in the $L^\perp$ radial directions $\langle \nabla u,\pi_{L^\perp}(y-x)\rangle^2 = |\pi_{L^\perp}(y-x)|^2\langle\nabla u, n_{L^\perp}\rangle^2$ is naturally weighted so that control degenerates near $L$ itself.  As observed previously, this definition is slightly stronger than the ones in \cite{ChNa1},\cite{ChNaVa},\cite{CJN}.  We can then define the quantitative stratification of our solution, see \cite{ChNa1},\cite{ChNaVa},\cite{CJN}:\\

\begin{definition}[Quantitative Stratification of Harmonic Maps]
Let $u:B_{10R}\to N$ be a stationary harmonic map.  Then we define the quantitative stratifications
\begin{align}
	&\cS^k_{\epsilon}(u)\equiv \Big\{x\in B_1:\forall\,\, 0< s<1 \text{ we have that $u$ is not $(k+1,\epsilon)-$symmetric on $B_s(x)$}\Big\}\, .
\end{align}
\end{definition}

\begin{remark}
Note that for $\epsilon_1 < \epsilon_2$ we have $\cS^k_{\epsilon_2}(u)\subseteq \cS^k_{\epsilon_1}(u)$.
\end{remark}

\vspace{.2cm}

A reasonable question is to ask the relationship of the quantitative stratifications to the classical stratification.  It is a nice and useful exercise to check that one has the following identities:
\begin{align}
\cS^k(u) = \bigcup_{\epsilon>0} \cS^k_\epsilon\, ,
\end{align}
and therefore one can recover the classical stratification directly as a countable union of quantitative stratifications.\\  

In this way we see how we have decomposed the classical stratification into what are more manageable pieces.  Though most of the results in this paper are about further control for the top $m-2$ stratum, let us mention the main result on the quantitative stratification in general:

\begin{theorem}[Naber-Valtorta \cite{NV_RH}][Quantitative Stratification for Stationary Maps]\label{t:main_quant_strat}
	Let $u:B_{10R}(0)\to N$ be a stationary harmonic map satisfying \eqref{e:manifold_bounds} with $R^2\int_{B_{10R}}|\nabla u|^2 \leq \Lambda$.  Then
	\begin{enumerate}
	\item $\cS^k_\epsilon(u)$ is $k$-rectifiable.  Further, for $k$-a.e. $x\in \cS^k_\epsilon(u)$ $\exists$ a $k$-plane $L^k_x\subseteq \dR^m$ such that {\it every} tangent cone at $x$ is $k$-symmetric with respect to $L_x$.
	\item There exists $C(n,K_M,K_N,\Lambda,\epsilon)>0$ such that we have the effective Minkowski estimates\newline $\Vol(B_r(\cS^k_{\epsilon}))\leq C\, r^{n-k}	$.  In particular, we have the finite volume estimate $\cH^k(\cS^k_\epsilon)\leq C$.
	\end{enumerate}
\end{theorem}
\begin{remark}
	The results stated in \cite{NV_RH} involved a slightly weaker notion of symmetry, however the results hold verbatim for the stronger notion used here using the cone splitting of Section \ref{ss:prelim:cone_splitting}.
\end{remark}
\vspace{.3cm}

\subsection{Cone Splitting and Quantitative Cone Splitting}\label{ss:prelim:cone_splitting}

The notion of cone splitting \cite{ChNa1} is a technique for producing higher order symmetries from lower order symmetries.  The starting point for symmetries is the monotonicity formula \eqref{e:outline:mollified_energy_monotonicity}.  An application of that formula tells us that if $r\dot\vartheta(x,r)=0$, which is essentially equivalent to the pinching condition $\vartheta(x,2r)-\vartheta(x,r/2)=0$, then $u$ is radially invariant at $x$ on $B_{Rr}(x)$.  That is, $u$ is $0$-symmetric at $x$.  An elementary, but still enlightening, argument tells us that if we have $k+1$ linearly independent points $x_0,\ldots,x_k$ such that $r\dot\vartheta(x_j,r)=0$ then $u$ is also invariant by the plane of symmetry given by $L=\text{span}\{x_j-x_0\}$ .  That is $u$ is $k$-symmetric at $x_0$.  In this way we see that multiple $0$-symmetries add to give a $k$-symmetry.  See \cite{ChNa1} for a careful argument of this.\\

 We wish to state a (sharp) quantitative version of this principle, which essentially comes from \cite{ChNa1}, \cite{ChNaVa}, \cite{NV_RH}.  To do so let first discuss the notion of effective linear independence.\\

\subsubsection{Effective Linear Independence}

\begin{definition}
 Given $\{x_i\}_0^k\in B_r \subseteq \R^n$, we say that $\cur{x_i}_{i=0}^k$ is $\alpha-$linearly independent at scale $r$ if
 \begin{gather}
  x_{i+1}\not \in \B{\alpha r }{x_0+L_i}\, ,
 \end{gather}
 where $L_i\equiv \operatorname{span}\cur{x_1-x_0,\cdots,x_{i-1}-x_0}$ .
\end{definition}
 This implies that the vectors $x_i-x_0$ are linearly independent in a quantitative way. In particular, by a Gramm-Schmidt argument we obtain immediately that
\begin{lemma}\label{lemma_effspa}
 If $\cur{x_i}_{i=0}^k$ is $\alpha$-linearly independent at scale $r$ with $$x_0+L=x_0+\operatorname{span}\cur{x_1-x_0,\cdots,x_k-x_0}\, $$ then for all $q\in (x_0+L)\cap B_{2r}(x_0)$ there exists a \textit{unique} set $\cur{q_i}_{i=1}^k$ such that
\begin{gather}\label{e:effectively_span}
 q=x_0+\sum_{i=1}^k q_i (x_i-x_0)\, , \quad \abs{q_i}\leq C(m,\alpha)\frac{|q-x_0|}{r}\, .
\end{gather}
\end{lemma}

\vspace{.3cm}

\subsubsection{Quantitative Cone Splitting}

Let us apply all of this to state and prove the quantitative cone splitting:\\

\begin{theorem}[Quantitative Cone Splitting]\label{t:prelim:cone_splitting}
Let $u:B_{10R}(p)\to N$ be a stationary harmonic map with $R^2\fint_{B_{10R}}|\nabla u|^2\leq \Lambda$ .  Let $\{x_j\}_{j=0}^k\subset B_{r/20}(x)\subset \B 1 p$ be $\alpha$-linearly independent, then we can estimate
\begin{align}\label{e:pinching_controls_energy_standard}
	\vartheta(x_0,r;L)+\vartheta(x_0,r;n_{L^\perp})=\int\rho_r(y-x_0)\Big(r^2\abs{\pi_L \nabla u(y)}^2+|\pi_{L^\perp}(y-x_0)|^2\langle\nabla u,n_{L^\perp}\rangle^2\Big)\leq C(m,\alpha)\sum r\dot\vartheta(x_j,1.2 r)\, ,
\end{align}
where $L=\text{span}\{x_j-x_0\}$ .
\end{theorem}
\begin{proof}
By scaling, we can assume wlog that $r=1$. Let $\ell$ be any unit vector in $L$. Since $\cur{x_j-x_0}_{j=1}^k$ span the linear subspace $L$, by \eqref{e:effectively_span}:
 \begin{gather}
  \ps{\nabla u(y)}{\ell}^2\leq C(m,\alpha) \sum_{j=1}^k \ps{\nabla u(y)}{x_j-x_0}^2\leq C(m,\alpha) \sum_{j=0}^k \ps{\nabla u(y)}{y-x_j}^2\, .
 \end{gather}
By Lemma \ref{l:rho_basic_properties}, for all $j$ we have $\rho_1(y-x_0)\leq -C(m)\dot \rho_{1.1}(y-x_j)$, and so
\begin{align}\label{e:pinching_controls_L_energy_basic}
	\vartheta(x_0,1;L)=\int\rho_1(y-x_0)\abs{\pi_L \nabla u(y)}^2\leq C(m,\alpha)\sum \dot\vartheta(x_j,1.1)\, .
\end{align}

To get the radial energy bound, we can first observe that
\begin{gather}
 \abs{\pi_{L^\perp}(y-x_0)}^2\ps{\nabla u(y)}{n_{L^\perp}}^2\leq 2\ps{\nabla u(y)}{y-x_0}^2+2\ps{\nabla u(y)}{\pi_L(y-x_0)}^2
\end{gather}
and so
\begin{align}
 \vartheta\ton{x_0,1;n_{L^\perp}}= &\int \rho_1(y-x_0)\abs{\pi_{L^\perp} (y-x_0)}^2\ps{\nabla u(y)}{n_{L^\perp}}^2\notag\\
 \leq & 2\int \rho_1(y-x_0)\ps{\nabla u(y)}{y-x_0}^2 + 2\int \rho_1(y-x_0)\ps{\nabla u(y)}{\pi_L(y-x_0)}^2\notag\\
 \leq &  \dot \vartheta(x_0,1) + 2\int \rho_1(y-x_0)\ps{\nabla u(y)}{\pi_L(y-x_0)}^2\, .
\end{align}
By \eqref{e:trho_doubleradius} and a rescaled version of \eqref{e:pinching_controls_L_energy_basic} we can bound 
\begin{gather}
 \int \rho_1(y-x_0) \ps{\nabla u(y)}{\pi_L(y-x_0)}^2\leq \int \rho_1(y-x_0) \abs{y-x_0}^2\abs{\pi_L \nabla u(y)}^2
 \leq C(m,\alpha)\sum \dot\vartheta(x_j, 1.2)\, .
\end{gather}
This concludes the proof.

\end{proof}

It is a fascinating point that the above can be improved in the case when $k=m-2$.  The above controls the energy in the $L$ directions and the radial energy in the $L^\perp$-directions.  In the $k=m-2$ case one can capture the full $L^\perp$-energy, albeit with the usual quadratic decay weight $|\pi_{L^\perp}(y-x)|^2$ near $L$.  Precisely:\\

\begin{theorem}[Top Stratum Quantitative Cone Splitting]\label{t:prelim:cone_splitting_m-2}
Let $u:B_{10R}(p)\to N$ be a stationary harmonic map with $R^2\fint_{B_{10R}}|\nabla u|^2\leq \Lambda$ .  Let $\{x_j\}_0^{m-2}\subset B_{r/20}(x)\subset \B 1 p$ be $\alpha$-linearly independent, then we can estimate
\begin{align}
	\vartheta(x_0,r;L)+\vartheta(x_0,r;L^\perp)=\int\rho_1(x_0-y)\Big(r^2\abs{\pi_L \nabla u(y)}^2+|\pi_{L^\perp}(y-x_0)|^2\abs{\pi_L^\perp \nabla u(y)}^2\Big)\leq C(m,\alpha)\sum r \dot\vartheta(x_j,1.4 r)\, ,\notag
\end{align}	
where $L^{m-2}=\text{span}\{x_j-x_0\}$ .
\end{theorem}
\begin{remark}
	The primary corollary of this result is that if $u$ is $(m-2,\delta)$-symmetric on $B_{2r}(p)$, then it holds that the energy measure $|\nabla u|^2\,dy$ is very close to being supported on $p+L$.  We will see in the next subsection that it must be close to the Hausdorff measure $\vartheta(p,r)\,\cH^{m-2}_L$.
\end{remark}
\begin{proof}

	As before, we can assume wlog that $r=1$. Given \eqref{e:pinching_controls_energy_standard}, we just need to prove that
 \begin{gather}
  \int\rho_1(y-x_0)|\pi_{L^\perp}(y-x_0)|^2\langle\nabla u,\alpha_{L^\perp}\rangle^2\leq C(m,\alpha)\sum \dot\vartheta(x_j,1.4)\, .
 \end{gather}
To prove this let us define $P(t)$ to be the only compactly supported primitive of $-\rho(t)$ in $[0,\infty)$, that is $\dot P(t) \equiv -\rho(t)$.  Consider the stationary equation \eqref{e:stationary_equation} applied to the vector field
\begin{gather}
 \xi(y)= P\ton{\frac{\abs{y-x_0}^2}{2}} \pi_{L^\perp}(y)\, .
\end{gather}
The stationary equation gives
\begin{align}
&\int P\ton{\frac{\abs{y-x_0}^2}{2}}  \abs{\pi_L \nabla u(y)}^2 \notag \\
 =&\int \rho_1(y-x_0) \abs{\pi_{L^\perp}(y-x_0)}^2\qua{\abs{\pi_L \nabla u(y)}^2 -\ps{\nabla u}{n_{L^\perp}}^2+\ps{\nabla u}{\alpha_{L^\perp}}^2}\notag \\
 -2 &\int \rho_1(y-x_0)\ps{\nabla u}{\pi_L(y-x_0)}\ps{\nabla u}{\pi_{L^\perp}(y-x_0)}\, .
\end{align}
By re-arranging the terms, and using the fact that $P(t)\leq C \rho(t)$ along with the estimates in \eqref{e:trho_doubleradius}, we get
\begin{align}
 &\int \rho_1(y-x_0)\abs{\pi_{L^\perp}(y-x_0)}^2 \ps{\nabla u}{\alpha_{L^\perp}}^2 \notag \\
 \leq C(m)&\int \rho_{1.1}(y-x_0) \qua{\abs{\pi_L \nabla u(y)}^2 +\abs{\pi_{L^\perp}(y-x_0)}^2\ps{\nabla u}{n_{L^\perp}}^2}\, .
\end{align}
This and the rescaled version of \eqref{e:pinching_controls_energy_standard} complete the proof.
\end{proof}

\vspace{.3cm}

\subsubsection{Quantitative Cone Splitting on Annular Regions}

The notion on an annular region is introduced and discussed in Section \ref{ss:broad_outline:annular_region}.  Let us apply the previous results to an annular region:

\begin{theorem}[Quantitative Cone Splitting on Annular Regions]\label{t:prelim:cone_splitting_annular}
Let $u:B_{10R}(p)\to N$ be a stationary harmonic map with $R^2\fint_{B_{10R}}|\nabla u|^2\leq \Lambda$, and let $\cA=B_2\setminus \overline{B_{\rf_x}(\T)}$ be a $\delta$-annular region.  For each $x\in \T$ with $B_r(x)\subseteq B_2$ and $r\geq \rf_x$, there exists $\bar x $ with $\abs{x-\bar x}\leq C(m)\sqrt \delta r$ and an $m-2$ dimensional plane $\bar L=L_{x,r}$ such that
\begin{align}
	\vartheta(\bar x,r;L)+\vartheta(\bar x,r;L^\perp)&\leq C(m)\fint_{\T\cap B_{r/2}(x)} r\dot\vartheta(z,3r/2)\,dv_\T(z)
\end{align}	
Moreover, $d_{\Gr}(L,L_{\cA})\leq C(m) \sqrt \delta$.
\end{theorem}
\begin{remark}
	Combining this with Theorem \ref{t:prelim:spacial_gradient} in the next subsection we can conclude that if $x\in \T$ and $r\geq \rf_x$ then on $B_{r}(p)$ then the energy measure $|\nabla u|^2\,dy$ is $C(m)\delta$ close to the Hausdorff measure $\vartheta(p,r)\,\cH^{m-2}_{p+L}$.
\end{remark}
\begin{remark}\label{rm:vartheta_comparison}
	A small but useful consequence of the previous remark is that $ c(m)\leq \frac{\vartheta(x,10r)}{\vartheta(x,r)}\leq C(m)\, , \;\forall x\in B_{r/10}(\T) \text{ with }\rf_x/10\leq r\leq 1\, .$ 
\end{remark}
\begin{remark}
	We will improve this in Theorem \ref{t:cone_splitting_annular} in order to take $x=\bar x$ at the cost of an additional error term which is on the same order.
\end{remark}

\begin{proof}

	Suppose for simplicity that $x=0$ and $r=1$, and consider $m-1$ points $\cur{x_i}_{i=0}^{m-2}\subset \T\cap \B {1/20} 0$ that are $1/100$-effectively linearly independent, with $x_0=0$.  Note that these exist by $(a1)$ in Definition \ref{d:annular_region}.  Set
	\begin{gather}
	 \bar x_i = \fint_{\B {1/1000}{x_i}}x dv_\T(x)\, .
	\end{gather}
	It is easy to check that $\bar x_i$ are still $1/1000$-linearly independent if $\delta$ is sufficiently small. By Jensen's inequality we can estimate 
	\begin{align}
	 \dot \vartheta(\bar x_i, 1.4)=&C(m) \int -\dot\rho_{1.4}(y-\bar x_i)\ps{\nabla u(y)}{y-\fint_{\B {1/1000}{x_i}}z dv_\T(z)}^2\,dy\notag\\
	 \leq &C(m) \int -\dot\rho_{1.4}(y-\bar x_i)\fint_{\B {1/1000}{x_i}}\ps{\nabla u(y)}{y-z}^2dv_\T (z)\, dy\notag\\
	 \leq &C(m) \fint_{\B {1/1000}{x_i}}\int -\dot\rho_{3/2}(y-z)\ps{\nabla u(y)}{y-z}^2 dy\, dv_\T (z)\leq C(m) \fint_{\B {1/1000}{x_i}}\dot \vartheta(z,3/2)dv_\T (z)\, ,
	\end{align}
where we have used that for $z\in B_{1/1000}(x_i)$ we can bound $-\dot\rho_{1.4}(y-\bar x_i)\leq -\dot\rho_{3/2}(y-z)$

	By the Ahlfors regularity of \eqref{e:annular_region:ahlfors_regularity}, we have that $c(m)\leq dv_\T\ton{\B {1/1000}{x_i}\cap \T}\leq dv_\T\ton{\B {1/2}{0}\cap \T}\leq C(m)$, and thus we can bound 
	\begin{gather}
	 \dot \vartheta(\bar x_i, 1.4)\leq C(m)\fint_{\T\cap B_{1/2}(0)} \dot\vartheta(z,3/2)\,dv_\T\, .
	\end{gather}
	The proof now follows from the cone splitting Theorem \ref{t:prelim:cone_splitting_m-2}, where $\bar L$ is the span of $\cur{\bar x_i-\bar x}_{i=1}^{m-2}$.
\end{proof}

\vspace{.3cm}

\subsection{Spatial Gradients of \texorpdfstring{$\vartheta$}{energy}}

An advantage of the mollified energy $\vartheta(x,r)$ is that it is smooth in both the $x$ and $r$ variables.  More importantly, we can compute and estimate its derivatives.  The end formulas which are useful are a bit more than just direct computations, they also require the use of the stationary equation.  Our main estimate of interest is the following:\\

\begin{theorem}[Spacial Gradient of $\vartheta$]\label{t:prelim:spacial_gradient}
	Let $u:B_{10R}(p)\to N$ be a stationary harmonic map,
	and let $L^k\subseteq \dR^m$ be a subspace.  Then if $\B r x\subset \B 1 p$ we can estimate
\begin{align}
	r^2|\nabla_L\vartheta|^2(x,r)&\leq C(m)r\dot\vartheta(x,r)\vartheta(x,2r;L)\notag
\end{align}
\end{theorem}
\begin{remark}
	Combining this with the top stratum Cone Splitting of Theorem \ref{t:prelim:cone_splitting_m-2}, we can conclude that if $u$ is $(m-2,\delta)$-symmetric on $B_{2r}(p)$ then the energy measure $|\nabla u|^2\,dy$ is very close to the Hausdorff measure $\vartheta(p,r)\,\cH^{m-2}_{p+L}$ on $B_r(p)$.
\end{remark}
\begin{proof}
 Let $\ell$ be any fixed vector on $L$, and consider that
 \begin{gather}
  r\nabla_\ell \vartheta(x,r)=r^2 \int \dot \rho_r (y-x) \ps{\frac{x-y}{r}}{\ell} \abs{\nabla u}^2\, .
 \end{gather}
The stationary equation \eqref{e:stationary_equation} applied to the vector field $\xi^j(y) = \rho_r(y-x)\ell^j$ tells us that
\begin{gather}\label{e:spacial_gradient_vartheta}
 r^2 \int \dot \rho_r (y-x) \ps{\frac{x-y}{r}}{\ell} \abs{\nabla u}^2=2r^2 \int \dot \rho_r (y-x) \ps{\frac{x-y}{r}}{\nabla u}\ps{\nabla u}{\ell}\, .
\end{gather}
 By Cauchy-Schwartz, we get
 \begin{gather}
  \abs{r\nabla_\ell \vartheta(x,r)}^2\leq 4\underbrace{\ton{r^2 \int -\dot \rho_r (y-x) \abs{\ps{\frac{y-x}{r}}{\nabla u}}^2}}_{=r\dot \vartheta(x,r)}\ton{r^2 \int -\dot \rho_r (y-x) \ps{\nabla u}{\ell}^2}\stackrel{\eqref{e:trho_doubleradius}}{\leq }C(m) r\dot \vartheta(x,r)\vartheta(x,2r,L)\, .
 \end{gather}

\end{proof}

An argument in the same spirit as the above also gives us an estimate on the hessian of $\vartheta(x,r)$:

\begin{theorem}[Spacial Hessian of $\vartheta$]\label{t:prelim:spacial_hessian}
	Let $u:B_{10R}(p)\to N$ be a stationary harmonic map, and let $L^k\subseteq \dR^m$ be a subspace.  Then if $B_r(x)\subseteq B_1(p)$ we can estimate the hessian in the $L$ directions:
\begin{align}
	r^2|\nabla^2_L\vartheta|^2(x,r)\leq C(m)r\dot\vartheta(x,r)\vartheta(x,2r;L)+C(m)\vartheta(x,2r;L)^2\, .
\end{align}
\end{theorem}
\begin{proof}
Using \eqref{e:spacial_gradient_vartheta} we can compute that for any two vectors $v,w\in \dR^m$ we have the equality
\begin{align}\label{e:spacial_hessian_vartheta}
	r^2 \nabla^2 \vartheta(x,r)[v,w]= 2r^2 \int \dot \rho_r(y-x) \ps{\nabla u}{w}\ps{\nabla u}{v}+2r^2 \int \ddot \rho_r(y-x) \ps{\nabla u}{\frac{y-x}{r}}\ps{\frac{y-x}{r}}{w}\ps{\nabla u}{v}\, .
\end{align}
The estimate now follows as in Theorem \ref{t:prelim:spacial_gradient}.
\end{proof}

\vspace{.3cm}

\subsection{Restricted Energy Functionals}\label{ss:restricted_energy}

We introduce one more piece of technical structure.  It will be important at moments to alter our heat kernel mollifer $\rho$ in several manners.
For example, we will need to cut it off in a small neighborhood of an affine plane $L\subseteq \dR^m$, and this will produce for us $\hat\rho$.
In particular, for radial and angular energy estimates we do not want to look inside the bubble region itself as the errors produced will become full energy errors.
Additionally, we will alter $\rho$ near $L$ to produce the mollifier $\tilde\rho$.   In this case we will cut off the $L^\perp$-derivative of $\rho$, and thus we will see that $\tilde\rho$ is $C^1$ close to $\rho$.
This kernel will be used to produce precisely chosen vector fields for the stationary equation and will likewise have the property of making sure various errors do not occur over the bubble region itself.\\

Let us work toward the precise definitions.

\subsubsection{\bf $L$-Heat Mollifiers $\hat\rho$ and $\tilde\rho$}  Let $\hat\psi:\dR^+\to \dR^+$ be a fixed smooth cutoff with $\hat\psi(t)=0$ for $t\leq 1$ and $\hat\psi(t)=1$ for $t\geq 2$.  Let $\hat\psi_R(t) = \hat\psi\ton{e^{2R}t}$ be the corresponding cutoff for $t\leq e^{-2R}$ .  Now we define: 

\begin{definition}[$L$-Heat Mollifier]\label{d:restricted_energy_functionals}
	Let us define $\hat\rho(\,\cdot\,;L), \tilde\rho(\,\cdot\,;L):\dR^m\to \dR^+$  by the formulas
\begin{align}\label{e:prelim:L_mollifier}
	\hat\rho_r(y;L) &\equiv \,\rho_r(y)\cdot \hat\psi_{R}\ton{\frac{|\pi_{L^\perp} (y)|^2}{2r^2} }\, ,\notag\\
	\pi_{L^\perp}\nabla\tilde\rho_r(y;L)&\equiv -\frac{\pi_{L^\perp}(y)}{r^2}\,\hat\rho_r(y;L)\, .
\end{align}
where $\tilde\rho$ is uniquely defined under the constraint that it is compactly supported.
\end{definition}

\begin{remark}
It is worth observing, as it is useful in estimates, that an explicit definition of $\tilde \rho$ is given by
 \begin{gather}
  \tilde \rho_r(y;L)=r^{-m}\int_{\abs{\pi_{L^\perp}(y)}}^\infty \hat \psi_R\ton{\frac{s^2}{2r^2}}\rho \ton{\frac{s^2+\abs{\pi_L(y)}^2}{2r^2}}\frac{sds}{r^2}
 \end{gather}
\end{remark}
\vspace{.3cm}

Note that $\hat\rho$ is $\approx e^{-R}$-close to $\rho$ as a measure, but the two have $L^\infty$ distance $1$.  On the other hand, $\tilde\rho$ is $\approx e^{-R}$-close to $\rho$ in $C^1$. More precisely, we have
\begin{lemma}\label{l:rho_tilde_estimates}
With the definitions above, we have that:
\begin{align}\label{e:prelim:L_Mollifier2}
	|r\nabla^{(y)} \hat\rho_r(y;L)|&\leq C(m) \rho_{2r}(y)\, ,\\
	\nabla_{L^\perp}\tilde\rho_r(y;L) &= 0 \text{ for }y\in B_{re^{-R}}(L)\, .
\end{align}
Moreover, if we define
\begin{gather}
 	\tilde e_r(y;L)\equiv \rho_r(y)-\tilde\rho_r(y;L)\geq 0\, ,
\end{gather}
then this function is small, precisely:
\begin{gather}
    |\tilde e_r(y;L)|+\abs{r\nabla \tilde e_r(y;L)}\leq C(m)e^{-R/2}\rho_{2r}(y)\, .
\end{gather}
\end{lemma}
\begin{proof}
 The first two estimates are immediate from the definition. Notice that
 \begin{gather}
  \tilde e_r(y;L)=r^{-m}\int_{\abs{\pi_{L^\perp}(y)}}^{2re^{-R}} \qua{1-\hat \psi_R\ton{\frac{s^2}{2r^2}}}\rho \ton{\frac{s^2+\abs{\pi_L(y)}^2}{2r^2}}\frac{sds}{r^2}\, .
 \end{gather}
 This function is $0$ for $\abs{\pi_{L^\perp}(y)}\geq 2re^{-R}$, supported in $\B {rR}{0}$ and $\norm{\tilde e_r(y;L)}_\infty\leq C(m)r^{-m}e^{-2R}$. Thus it is clear that
 \begin{gather}
  \abs{\tilde e_r(y;L)}\leq C(m) e^{-R} \rho_{2r}(y)\, .
 \end{gather}
 As for the gradient, we observe that
 \begin{align}
  r\pi_{L^\perp} \nabla \tilde e_r(y;L)=&r^{-m}\qua{1-\hat \psi_R\ton{\frac{\abs{\pi_{L^\perp}(y)}^2}{2r^2}}}\rho \ton{\frac{\abs y^2}{2r^2}}\frac{\pi_{L^\perp}(y)}{r}\, ,\\
  \abs{r\pi_{L^\perp} \nabla \tilde e_r(y;L)}\leq& 2 e^{-R}\rho_r(y) \leq 2 e^{-R}\rho_{2r}(y) \, ;\\
  r \pi_L \nabla \tilde e_r(y;L)=&r^{-m}\int_{\abs{\pi_{L^\perp}(y)}}^{2re^{-R}} \qua{1-\hat \psi_R\ton{\frac{s^2}{2r^2}}}\rho \ton{\frac{s^2+\abs{\pi_L(y)}^2}{2r^2}}\frac{\pi_L(y)}{r}\frac{sds}{r^2}\, ,\\
  \abs{r \pi_L \nabla \tilde e_r(y;L)}\leq &C r^{-m} R \abs{\tilde e_r(y;L)}\leq C(m) Re^{-R}\rho_{2r}(y)\leq C(m) e^{-R/2} \rho_{2r}(y)\, .
 \end{align}
\end{proof}

\vspace{.3cm}

\subsubsection{\bf Restricted Energy Functionals}

In the spirit similar to Definition \ref{d:partial_energies} we can then define the restricted energy functions:\\

\begin{definition}[Restricted Energy Functions]\label{d:restricted_energies}
	Given $u:B_{10R}\to N$ with $B_{r}(x)\subseteq B_R$ and $L^k\subseteq \dR^m$ a $k$-dimensional plane we define
\begin{align}
	\hat\vartheta(x,r) &\equiv r^2\int \hat\rho_r(x-y;L)|\nabla u|^2\,dy\, ,\notag\\
	\hat\vartheta(x,r;n_{L^\perp}) &\equiv \int \hat\rho_r(x-y;L)\langle\nabla u,\pi_{L^\perp}(y-x)\rangle^2\,dy = \int \hat\rho_r(x-y)|\pi_{L^\perp}(y-x)|^2\langle\nabla u,n_{L^\perp}\rangle^2\,dy\, ,\notag\\
	\hat\vartheta(x,r;\alpha_{L^\perp}) &\equiv  \int \hat\rho_r(x-y;L)\,|\pi_{L^\perp}(y-x)|^2\langle\nabla u,\alpha_{L^\perp}\rangle^2\,dy\, .
\end{align} 
\end{definition}

\vspace{.5cm}

\section{Proof of the Quantitative Energy Identity and Energy Identity Modulo the Small Annular Energy Theorem}\label{s:broad_outline}

The essence of the Energy Identity are the regions between the bubbles, which we will precisely define in the next subsections and refer to as annular regions.  The subenergy identity of \eqref{e:intro:energy_inequality}:
\begin{align}
	e(x)\geq \sum E[b_j]\, ,
\end{align}
comes from an understanding that each bubble $b_j$ at a point $x$ must necessarily contribute its energy to the defect measure.  Therefore to get equality in the above, we need to understand that there is no other source of energy.  This is to say that the regions between bubbles, the annular regions, do not contribute energy to the defect measure.\\

The proof of the Quantitative Energy Identity of Theorem \ref{t:main_quant_energy_id} may be naturally split into two pieces.  One piece takes the form of a quantitative annulus/bubble decomposition in Theorem \ref{t:quant_annulusbubble} as in \cite{NV_YM},\cite{JN}, \cite{CJN}.
This allows us to very precisely split the harmonic map $u$ into (quantitative) bubble regions and annular regions, and then focus our attention on the individual parts of the decomposition in a very effective fashion.
We will work toward introducing the annular and bubble regions in this Section, and state the quantitative annulus/bubble decomposition in Theorem \ref{t:quant_annulusbubble}.
Though quite technically involved, the ideas for this decomposition follow closely the ideas of \cite{NV_YM},\cite{JN} .  \\

The second piece of the proof, which is where the new ideas in this paper take place, is in the analysis of the annular regions and in particular Theorem \ref{t:outline:annular_regions_energy}.  This result will tell us that annular regions, which in this paper are large regions with locally small energy, have globally small energy.  This will be the key to the energy identity and is what will stop energy from piling up away from the bubble regions.\\

  We will spend the next subsections working up to the precise definitions and statements of these main components of the proof.  We will see how to put all these pieces together in Section \ref{ss:quantitative_energy_identity_proof} in order to prove the Quantitative Energy Identity of Theorem \ref{t:main_quant_energy_id}.  In the next section we will begin the process of outlining the proof of the small annular energy of Theorem \ref{t:outline:annular_regions_energy}, which will be the main focus of the remainder of this paper .\\

\subsection{Annular Regions}\label{ss:broad_outline:annular_region}

Let us first refer the reader to Section \ref{ss:outline:heat_mollified_energy}, where the energy functions which play a role in this paper are introduced and discussed.  We begin our discussions in this Section with a precise definition of annular regions, which are large regions with locally small energy.
In Section \ref{ss:broad_outline:quant_annulusbubble} we discuss the Quantitative Annulus/Bubble Decomposition, where we see that up to sets of small codimension two content, annular regions do indeed make up those regions between bubbles.
These ideas and higher dimensional versions were first used in \cite{NV_RH},\cite{NV_YM}, \cite{JN},\cite{CJN} to study the stratification of singular sets.  Our precise definition is as follows:\\

\begin{definition}[$\delta$-Annular Regions]\label{d:annular_region}
	Let $\T\subseteq B_{4}(p)$ be an $m-2$ submanifold with $\rf_z:\T\to [0,\infty)$ such that $\rf_x<\delta$, $|\Lip \rf_x|<\delta$ and $\rf_x \abs{\nabla^2 \rf_x}< \delta$.  We call
	\begin{gather}
	 \cA=B_{1}(p)\setminus \overline{B_{\rf_x}(\T)}
	 \equiv \B 1 p \setminus \bigcup_{x\in \T} \overline{\B {\rf_x}{x}}\,
	\end{gather}
	a $\delta$-annular region if $K_M^2<\delta $ and if $\exists$ an $m-2$ plane $L_{\cA}$ such that
\begin{enumerate}
	\item[(a1)] $\T=\text{Graph}\Big\{\ft_\T :L_\cA\to L_\cA^\perp\Big\}$ with $|\ft_\T|+|\nabla\ft_\T|+\fr_x|\nabla^2\ft_\T|< \delta$
	\item[(a2)] For each $x\in \T$ and $\rf_x\leq r \leq 20$ we have that $r\dot \vartheta(x,r)<\delta$ .
	\item[(a3)]\label{i:annular_nontrivialness} For each $x\in \T$ we have $\vartheta(x,\rf_x)>\epsilon_0(m,K_N)$
\end{enumerate}
\end{definition}
\begin{remark}\label{rm:annular:regularity}
Note that conditions $(a1)$ and $(a2)$ imply that for each $x\in \T$ and $r\geq \rf_x$ we have that $B_r(x)$ is $(m-2,C(m)\delta)$-symmetric, see Theorem \ref{t:prelim:cone_splitting_m-2}.  By Theorem \ref{t:prelim:cone_splitting_annular} and Theorem \ref{t:eps_reg} we have that if $\delta\leq \delta(n,K_N,\epsilon)$ then $u$ is smooth on $\cA_\epsilon\equiv B_{\epsilon^{-1}}(p)\setminus B_{\epsilon\,\rf_x}(\T)$ with the pointwise estimate $d(x,\T)^2|\nabla u|^2(x) < C(m,\epsilon)\delta$.
\end{remark}
\begin{remark}
We can likewise define an annular region $\cA\subseteq B_r(p)$	by rescaling the above definition..
\end{remark}

\begin{remark}
By convention, $\overline{\B 0 x}=\cur{x}$.
\end{remark}

\begin{remark}
Notice that by definition of $\rho$ we can replace $(a2)$ with the essentially equivalent assumption that $\big|\vartheta(x,2r)-\vartheta(x,r)\big|<\delta$ . 
\end{remark}

\begin{remark}
The first condition implies that $\pi_L:\T\to L_\cA$ is bilipschitz and that $\T$ has the local second fundamental form bound $\rf_x\big|\II_\T(x)\big|<C(m)\delta$.
\end{remark}

\begin{remark}
	By a standard mollification procedure, there would be no loss in assuming additionally that  $\rf_x^{k}|\nabla^{(k+1)}\ft_\T|<C(m,k)\delta$ for all $k\in\dN$.
\end{remark}
\begin{remark}
The constant $\epsilon_0(n,K_N)$ in the third condition is chosen as in the $\epsilon$-regularity of Theorem \ref{t:eps_reg}, and guarantees nontrivialness of the annular region.\\
\end{remark}

One can define a natural projection map onto $\T$ by
 \begin{gather}
  \pi^{\T}:B_2\to \T\, , \qquad \pi_{L_{\cA}}\ton{\pi^{\T}(x)}=\pi_{L_{\cA}}(x)\, .
 \end{gather}
In other words, $\pi^{\T}(x)$ is the only element of $\T$ with the same $L_{\cA}$ projection as $x$.

\begin{remark}\label{r:annular_region:extension}
	One can naturally extend $\rf_x:\T\to \dR^+$ to a $2\delta$-lipschitz function on $B_2$ by $\rf_x = \rf\ton{\pi^{\T}(x)}$.\\
\end{remark}

Condition $(a1)$ is much stronger in comparison to annular regions in \cite{NV_RH},\cite{NV_MS},\cite{JN},\cite{CJN}.  This additional control can be obtained fundamentally because we are only dealing with top dimensional annular regions in this paper.  On the other hand Condition $(a2)$ is much weaker than in the definition of annular regions in \cite{NV_RH},\cite{JN}.   There one assumes a stronger and more global pinching condition of the form $|\vartheta_{}(x,4r)-\vartheta_{}(x,\rf_x)|\approx \int_{\rf_x}^{2r} |\vartheta_{}(x,2s)-\vartheta(x,s)|\frac{ds}{s} <\delta$ .  Our weakening of this condition is in analogue to \cite{NV_YM} and is designed to insure that maximal annular regions connect all the space between different bubbles, at least up to a controllably small set.  The effect of this is explained more in Section \ref{ss:broad_outline:quant_annulusbubble}.  \\

Let us discuss several basic examples in order to get a feel for the definition:

\begin{example}\label{ex:bubble1}
 As an easy example of an annular region in $\R^2$, consider the standard bubble example. In particular, let $\lambda>0$ and consider the homotopy $u_\lambda(x)=\lambda^{-1}x$ from $\R^2$ to itself. Via the stereographic projection, these maps induce a family of smooth harmonic maps $u_\lambda:\R^2\to S^2$ with
 \begin{gather}
  \int_{S^2} \abs{\nabla u_\lambda}^2 = 8\pi\,
 \end{gather}
 independent of $\lambda$, and so it is easily seen that in the sense of measures that $  \abs{\nabla u_\lambda}^2 d\mu \to 8\pi \delta_0\, .$  For $\lambda<\lambda(\delta)$ we then have that $\cA\equiv B_1(0)\setminus B_{r_\lambda}(0)$ is a $\delta$-annular region with $\T=\{0\}$ and $r_\lambda\to 0$ as $\lambda\to 0$. $\qed$\\
\end{example}

\begin{example}\label{ex:bubble2}
	One can extend the trivial example to higher dimensions by defining
\begin{gather}
 U_\lambda:\R^m=\dR^{m-2}\times \dR^2\to S^2\, , \qquad U_\lambda(x)=u_\lambda\ton{\pi_{\R^2}(x)}\, .
\end{gather}
In this case we obtain that $\cA= \B {1} 0 \setminus \B {\rf_x}{\T}$ is a $\delta$-annular region provided that $\rf_x\geq r_\lambda$ with $\T=\dR^{m-2}$. $\qed$\\
\end{example}

\begin{example}\label{ex:bubble3}
As a slightly less trivial example let $u_\lambda:\dC P^1\to \dC P^1$ be the harmonic maps $u_\lambda([z_0,z_1]) = [z_0z_1,\lambda z_0^2 +z_0z_1+\lambda z_1^2]$ which develop poles at the north and south pole.  We can consider the cone mapping 
\begin{gather}
 U_\lambda:\R^3\to S^2\, , \qquad U_\lambda(x)=u_{\lambda}\ton{x/|x|\,}\, ,
\end{gather}
which can be shown to be stationary harmonic maps.  This defines a $\delta$-annular region $\cA=B_1\setminus B_{\rf_x}(\T)$ where $\T=\{0^2\}\times\dR$ is the z-axis and $\rf_z=r_\lambda \,|z|$ .  Note there is a single singular point at the origin. $\qed$\\
\end{example}

\subsubsection{Existence and Properties of Annular Regions}  Let us first remark on some properties of annular regions that follow almost immediately from the definition.  One of the most important is the Ahlfor's regularity of the $m-2$ Hausdorff measure $\cH^{m-2}_\T$ on $\T$.  Namely, it follows from Condition (a1) that
\begin{align}\label{e:annular_region:ahlfors_regularity}
	\big(1-C(m)\delta\big)\omega_{m-2} s^{m-2}\leq dv_\T[B_s(x)]\equiv \cH^{m-2}_\T[B_s(x)]\leq \big(1+C(m)\delta\big)\omega_{m-2} s^{m-2} \;\;\forall\;\; x\in\T\text{ with }B_s(x)\subseteq B_1(p)\, .
\end{align}

It should be pointed out that for annular regions which are not top dimensional this result is highly nontrivial and can be viewed as one of the main achievements of \cite{NV_RH}.  For the top dimensional stratum it follows easily since $\T$ is bilipschitz to $L_\cA$ .\\

The second takeaway from the conditions of an annular region is the energy measure almost vanishes locally on $\cA$.  More precisely and strongly, it is a consequence of Conditions (a1) and (a2) together with Theorem \ref{t:prelim:cone_splitting_annular} that for every  $x\in \T$ with $\rf_x\leq s\leq r$ we have that the energy measure
\begin{align}\label{e:outline:weakly_flat}
	|\nabla u|^2 dy \approx \vartheta(x,s) \, \cH^{m-2}_{x+L_\cA} \text{ on }B_s(x)\, ,
\end{align}
is $\delta$-close to the $m-2$ Haudorff measure on the affine plane $x+L_\cA$.  In particular, by the $\epsilon$-regularity of theorem \ref{t:eps_reg} the energy measure is smoothly small away from the affine plane $x+L_\cA$.  Condition $(a3)$ tells us that $\vartheta(x,s)>\epsilon_0(n,K_N)$, and so the energy measure is not trivially close to the zero measure.\\

What is left to discuss is that annular regions actually exist.  Let us recall from Definition \ref{d:quant_symmetries} the notion of $(k,\epsilon)$-symmetries, and recall from Definition \ref{d:heat_mollifier} the cutoff scale $R>0$ which is ever present.  An argument essentially verbatim from \cite{NV_YM},\cite{JN} tells us that if $u$ is $(m-2,\delta')$-symmetric on $B_2(p)$ then a maximal $\delta$-annular region exists on $B_1(p)$.  Precisely:

\begin{theorem}[Existence of Annular Regions]\label{t:annular_existence}
	Let $u:B_{10R}(p)\to N$ be a stationary harmonic map with $R^2\fint_{B_{10R}}|\nabla u|^2\leq \Lambda$ .  For each $0<\delta'<\delta$ there exists $\delta''(m,\Lambda,K_N,R,\delta,\delta')$ such that if $u$ is $(m-2,\delta'')$ symmetric on $B_2$, then there exists a $\delta$-annular region $\cA=B_1\setminus B_{\rf_x}(\T)$.  Further, we can assume the following hold:
\begin{enumerate}
	\item For each $x\in \T$ with $\rf_x<r<1$ we have $\vartheta(x,r;L_\cA)\equiv r^2\int \rho_r(y-x)\langle\nabla u, L_\cA\rangle^2<\delta'$ .
	\item If $\cN \cM=\big\{x\in \T:\vartheta(x,10\rf_x)-\vartheta(x,\rf_x/10)<\delta\big\}$ is the nonmaximal set,  then $\cH^{m-2}_\T[\cN \cM]<\delta'$ .
\end{enumerate}
\end{theorem}
\begin{remark}\label{r:maximal_annular}
	Conditions $(1)$ and $(2)$ are a form of maximal condition.  They are not necessary for the structural theorems for annular regions, but they play a role in the proof of the quantitative annulus/bubble decomposition.  Specifically condition (2) tells us we have not stopped the construction of our annular region too early.  Combined with condition (1) and Theorem \ref{t:bubble_existence} we see that when we do stop on $B_{\rf_x}(x)$ we will begin to see a bubble region form.  See Theorem \ref{t:quant_annulusbubble} for more.
\end{remark}
\begin{remark}
	The proof of Theorem \ref{t:annular_existence} follows almost verbatim the constructions of \cite{NV_YM}, see Theorem 5.4 and Section 5.2 of \cite{NV_YM}, which are only covering arguments relying on the monotone quantity and not specific to any situation. 
\end{remark}

\vspace{.3cm}

\subsubsection{Annular Regions have Small Energy}

Let us now state what is the most important new technical theorem of this paper.  Recall that (a2) and \eqref{e:outline:weakly_flat} give us that the energy on $\cA$ is small on each scale.  However, the total energy on $\cA$ may apriori still be large.  The main result in this paper toward the proof of the Energy Identity is to show this cannot happen.  Namely: \\

\begin{theorem}[Annular Regions have Small Energy]\label{t:outline:annular_regions_energy}
	Let $u:B_{10R}(p)\to N$ be a stationary harmonic map with $R^2\fint_{B_{10R}}|\nabla u|^2\leq \Lambda$ , and let $\cA=B_2\setminus \overline{B_{\rf_x}(\T)}$ be a $\delta$-annular region.  Then for $R\geq R(m)$ and each $\epsilon>0$ if $\delta\leq \delta(m,K_N,R,\Lambda,\epsilon)$ then
\begin{align}
	\int_{\cA_\epsilon\cap B_{1}} |\nabla u|^2  < \epsilon\, ,
\end{align}
where $\cA_\epsilon \equiv B_2\setminus B_{\epsilon\, \rf_x}(\T)\supseteq \cA$ .
\end{theorem}
\begin{remark}\label{rm:outline:annular_regions_energy}
	Note that if $B_{4r}(x)\subseteq B_2(p)$ then $\cA$ restricts to an annular region on $B_{2r}(x)$.  Consequently, by covering $B_{2-\epsilon}$ by small balls and applying the above, if we take $\delta\leq \delta(m,K_N,R,\Lambda,\epsilon)$ sufficiently small then we also have the estimate $\int_{\cA_\epsilon\cap B_{2-\epsilon}} |\nabla u|^2  < \epsilon$ .
\end{remark}
\vspace{.2cm}

The vast majority of this paper is about proving the above.   Sections \ref{s:outline_toymodel} and \ref{s:outline_general} will focus on outlining the proof of Theorem \ref{t:outline:annular_regions_energy} above, and the most of the remainder of the paper will focus on making the outline rigorous.  The relevant application of the above is the following, which tells us that the energy does not drop from the top scale to the bottom on the annular region:\\

\begin{corollary}\label{c:outline:annular_regions_energy}
		Let $u:B_{10R}(p)\to N$ be a stationary harmonic map with $R^2\fint_{B_{10R}}|\nabla u|^2\leq \Lambda$ , and let $\cA=B_2\setminus \overline{B_{\rf_x}(\T)}$ be a $\delta$-annular region.  Then for $R\geq R(m)$ and each $\epsilon>0$ if $\delta\leq \delta(m,K_N,R,\Lambda,\epsilon)$ then
\begin{align}
	\int_{\T\cap B_{1}} \,\big|\vartheta(x,1)-\vartheta(x,\rf_x)\big|\leq \epsilon\, .
\end{align}
\end{corollary}
\begin{proof}
Let us begin with a quick and imprecise sketch, just to see how the small annular energy of Theorem \ref{t:outline:annular_regions_energy} implies the nearly constant energy density of the current corollary.  The rough idea of the sketch is better exemplified using the standard monotone quantity $\theta(x,r)\equiv r^{2-n}\int_{B_r}|\nabla u|^2$, as opposed to the mollified quantity $\vartheta(x,r)$.  To begin, the $L$-gradient estimates on $\theta(x,1)$ imply that this quantity is almost constant on $\T\cap \B 1 0$, so that
\begin{gather}
 \fint_{\T\cap \B 1 0}\theta(x,1)\approx \theta(0,1)=\int_{\B 1 0}\abs{\nabla u}^2\, .
\end{gather}
On the other hand, roughly we have that
\begin{gather}
 \fint_{\T\cap \B 1 0} \theta(x,\rf_x)\approx \int_{\B {\rf_x}{\T\cap \B 1 0}}\abs{\nabla u}^2\approx \int_{\B {\rf_x}{\T}\cap \B 1 0}\abs{\nabla u}^2\, .
\end{gather}
Comparing the integrals, we have
\begin{gather}
 \fint_{\T\cap \B 1 0}\qua{\theta(x,1)-\theta(x,\rf_x)} \approx \int_{\B 1 0}\abs{\nabla u}^2-\int_{\B {\rf_x}{\T}\cap \B 1 0}\abs{\nabla u}^2\approx \int_{\cA}\abs{\nabla u}^2\leq \epsilon'\, ,
\end{gather}
where the last inequality is due to Theorem \ref{t:outline:annular_regions_energy}.  In order to make this moral accurate, one needs to be a bit careful about the presence of the mollifier $\rho$ in the definition of $\vartheta$, and to exploit the properties of $\T$ and $\rf_x$.

Now more precisely.  We introduce a parameter $\chi=\chi(m,K_N,R,\Lambda,\epsilon')>>1$, which will be chosen later. Note first that if $\delta\leq \delta(\chi,\epsilon')$ then it follows that $\big|\vartheta(x,1)-\vartheta(x,\chi^{-1})\big|<\epsilon'$ for all $x\in \T$ by $(a2)$. In a similar fashion, we obtain also $\big|\vartheta(x,\chi\rf_x)-\vartheta(x,\rf_x)\big|<\epsilon'$ for all $x\in \T$.

By definition of $\vartheta$, and the fact that $\T$ is a Lipschitz graph with Lipschitz constant $\leq \delta$, we have that for $\delta$ sufficiently small
\begin{gather}
 \fint_{\T\cap \B 1 0} \vartheta(x,\chi^{-1})\leq c_m'(1+\epsilon')\int_{\B {1+10R\chi^{-1}}{0}\cap \B {10R\chi^{-1}}{\T}} \abs{\nabla u}^2\leq c_m'(1+\epsilon')\int_{\B {1+\epsilon'}{0}} \abs{\nabla u}^2
\end{gather}
where $c_m'=\int_{\R^{m-2}}\rho(\abs{y}^2/2)$. On the other hand, using the fact that $\T$ is a Lipschitz graph with Lipschitz constant $\leq \delta$ and $\rf_x$ as well has Lipschitz constant $\leq \delta$, we can estimate for $\delta\leq \delta(\epsilon')$:
\begin{gather}
 \fint_{\T\cap \B 1 0} \vartheta(x,\chi \rf_x)\geq c_m'\int_{\B {1-\epsilon'}0} \psi(y) \abs{\nabla u(y)}^2\, ,
\end{gather}
where $\psi(y)\geq \min\cur{(1-\epsilon')\ton{1-\frac{d(y,\T)}{\chi \rf_y}};0}$. Choosing $\chi=\chi(\epsilon')$ sufficiently large, we get
\begin{gather}
 \fint_{\T\cap \B 1 0} \vartheta(x,\chi \rf_x)\geq c_m'(1-2\epsilon')\int_{\B{\rf_x}{\T}\cap \B {1-\epsilon'}0} \abs{\nabla u(y)}^2\, .
\end{gather}
Thus we can conclude that
\begin{gather}
 (c_m')^{-1}\fint_{\T \cap \B 1 0} \vartheta(x,1)-\vartheta(x,\rf_x)\leq \\
 \leq 2\epsilon' + (1-2\epsilon') \int_{\B{1-\epsilon'}{0}\setminus \B {\rf_x}{\T}}\abs{\nabla u(y)}^2 +2\epsilon'\int_{\B{1+\epsilon'}{0}}\abs{\nabla u}^2 + 2\int_{\B {1+\epsilon'}{0}\setminus \B {1-\epsilon'}{0}}\abs{\nabla u}^2\leq\\
 \leq 2\epsilon' + C\epsilon' \Lambda + \int_{\B{1-\epsilon'}{0}\setminus \B {\rf_x}{\T}}\abs{\nabla u(y)}^2\, .
\end{gather}
By Theorem \ref{t:outline:annular_regions_energy}, we have $\int_{\B{1-\epsilon'}{0}\setminus \B {\rf_x}{\T}}\abs{\nabla u(y)}^2\leq \epsilon'$, and the conclusion follows by picking $\epsilon'=\epsilon'(\epsilon,m,\Lambda)$ sufficiently small.

\end{proof}

We will use the above in Section \ref{ss:quantitative_energy_identity_proof} as part of the proof of the Quantitative Energy Identity of Theorem \ref{t:main_quant_energy_id}.  \\

\vspace{.3cm}

\subsection{Bubble Regions}\label{ss:broad_outline:bubble_region}

In the Quantitative Energy Identity of Theorem \ref{t:main_quant_energy_id} the main claim is that if $u$ is $(m-2,\delta)$-symmetric on a ball $B_2(0)$ wrt $L^{m-2}$, then on most 2-dimensional slices $\ell+L^\perp$ we have that $u$ looks like a $2$-dimensional harmonic map in some strong sense.  One of the ingredients in this Theorem is the Quantitative Bubble/Annulus Decomposition, which breaks $u$ into annular regions and bubble regions.  As we have seen, the annular regions represent small energy regions between the bubbles.  The bubble regions should then represent those regions where $u$ looks like a nontrivial 2-dimensional harmonic map on some fixed scale.  Let us now make a quantitative version of this precise as follows:

\begin{definition}[$\delta$-Bubble Regions]\label{d:bubble_region}
	Let $p\in M$ with $L^{m-2}\subseteq T_pM$ a subspace, and consider disjoint balls $\{B_{r_j}(x_j)\}\subseteq p+L^\perp$.  We call $\cB=B_r(p)\setminus \bigcup \overline{B_{r_j}(x_j+L)}$ a $\delta$-bubble region if $r^2K_M^2<\delta$ and $\exists$ harmonic map $b:p+L^\perp\to N$ such that
\begin{enumerate}
\item[(b1)] $r^2\abs{\pi_L\nabla u}^2 <\delta$ on $\cB$.
\item[(b2)] $|b-u|^2+r^2|\nabla b-\nabla u|^2<\delta$ on $\cB_R\equiv B_{Rr}(p)\setminus \bigcup \overline{B_{R^{-1}r_j}(x_j+L)}$ .
\item[(b3)] $\int_{\cB\cap L^\perp}|\nabla b|^2 >\epsilon(n,K_N)$ with $\int_{\cB^c\cap L^\perp}|\nabla b|^2< \sqrt\delta$ .
\item[(b4)] $r_j\geq r_0(K_M,K_N,\Lambda,\delta) r$ .
\end{enumerate} 
\end{definition}
\begin{remark}
	The constant $r_0(K_M,K_N,\Lambda,\delta)$ can be written explicitly as a power of $\delta$, however it is much easier to produce it ineffectively as part of the existence of Theorem \ref{t:bubble_existence}.
\end{remark}
\begin{remark}
	The constant $\epsilon(n,K_N)$ is the $\epsilon$-regularity constant from Theorem \ref{t:eps_reg}.  The constant $R$ is as always the one from Definition \ref{d:heat_mollifier}.
\end{remark}
\begin{remark}
	The constant $\sqrt\delta$ in $(b3)$ could be replaced with $C(m,R,\Lambda)\delta$, but instead of attempting to identify the right constant it is easier to absorb it into a power of $\delta$.
\end{remark}

Our main result is an existence statement, that tells us if we are on a ball with sufficiently small $L$-energy and boundary energy, then there exists a bubble region on this ball.  Precisely:\\

\begin{theorem}[Bubble Region Existence]\label{t:bubble_existence}
	Let $u:B_{10R}(p)\to N$ be a stationary harmonic map with $R^2\fint_{B_{10R}}|\nabla u|^2\leq \Lambda$ .  For each $0<\delta<\delta(m,R,\Lambda)$ there exists $\delta'(m,\Lambda,K_N,R,\delta)>0$ such that if $\vartheta(p,1;L)=r^2\int \rho_r(y-p)|\pi_L\nabla u|^2 <\delta'$ with $10^{-2}\delta \leq \dot\vartheta(p,1)\leq 10^{2}\delta$, then for some $r_0(K_M,K_N,\Lambda,\delta)>0$ there exists a $\delta$-bubble region $\cB\subseteq B_1(p)$.
\end{theorem}
\begin{proof}
	The proof may be done by contradiction, as $\delta'\to 0$ one arrives at a two dimensional harmonic map and can read the result off.  See Section 4 of \cite{NV_YM} for a nearly identical statement and proof.\\
\end{proof}

Our main structure theorem tells us that the energy of $u$ on the bubble region may be computed from the energy of $b$ plus energies at smaller scales:

\begin{theorem}[Bubble Structure Theorem]\label{t:bubble_structure}
	Let $u:B_{10R}(p)\to N$ be a stationary harmonic map with $R^2\fint_{B_{10R}}|\nabla u|^2\leq \Lambda$ with $\cB\subseteq B_1(p)$ a $\delta$-bubble region wrt $L$ and $\vartheta(p,1;L)=r^2\int \rho_r(y-p)|\pi_L\nabla u|^2 <\delta'$.  Then for each $\epsilon>0$ if $\delta<\delta(m,K_N,\Lambda,\epsilon)$ and $\delta'<\delta'(m,K_N,\Lambda,\epsilon)$ then following hold
\begin{enumerate}
	\item For $y\in L$ with $\cB_y\equiv \cB\cap (y+L^\perp)$ and $y_j=(x_j+L)\cap(y+L^\perp)$ we have that
\begin{align}
	\big|\vartheta(p,1)-\omega_{m-2}\int_{\cB_y}|\nabla u|^2-\sum \vartheta(y_j,r_j)\big| < \epsilon\, .
\end{align}
	\item For $y\in L$ with $\cB_y\equiv \cB\cap (y+L^\perp)$, there exists $\{y_j\}_1^N\in \cB_y$ with $N\leq N(n,K_N,\Lambda)$ and $s_j>0$ with $\vartheta(y_j,s_j)>\epsilon_0(n,K_N)$ so that if $S\geq S(n,K_N,\Lambda,\epsilon)$ then $\int_{\cB_y\setminus \bigcup B_{Ss_j}(y_j)}|\nabla u|^2 <\epsilon$ .
\end{enumerate}
\end{theorem}
\begin{remark}
	The first estimate tells us that the energy of $u$ on $B_1$ may be computed as the energy of the $2$-dimensional bubble slice plus energies at smaller scales.
\end{remark}
\begin{remark}
As $\delta\to 0$ a sequence of $\delta$-bubble regions may itself break into several distinct bubbles.  The second estimate bounds the number of such bubbles which may appear and tells one the scales to find them. 
\end{remark}
\begin{proof}
	The proof is once again a contradiction argument, where again as $\delta'\to 0$ we arrive at a $2$-dimensional harmonic map picture, where the results are classically proved.  See Section 4 of \cite{NV_YM} for a nearly identical statement and proof.
\end{proof}
\vspace{.5cm}

\subsection{Quantitative Annulus/Bubble Decomposition for Symmetric Balls}\label{ss:broad_outline:quant_annulusbubble}

The Quantitative Annulus/Bubble Decomposition will effectively decompose a ball into $\delta$-annular regions, plus $\delta$-bubble regions, plus a set of small content.  There are two primary distinctions between the Quantitative Annulus/Bubble Decomposition introduced in this subsection and those of \cite{JN},\cite{CJN}, which are closely related.
The first is the notion of $\delta$-annular region, which as discussed in Section \ref{ss:broad_outline:annular_region} is less restrictive than that of \cite{JN},\cite{CJN}.  In particular the $\delta$-annular regions in this paper are {\it apriori} assumed to only have locally small energy, not globally small energy.
The second is that the Quantitative Annulus/Bubble Decomposition in this subsection takes place on a ball for which $u$ is already $(m-2,\delta')$-symmetric.  The combination of these two points, the weaker annular region and the initial symmetry, allows one to improve the decompositions of \cite{JN},\cite{CJN} in an important manner.  Namely, the content of balls in the decomposition will be independent of the background parameter $\delta$.  This will be a key point in proving the Quantitative Energy Identity.  
This form of Quantitive Annulus/Bubble Decomposition originated in \cite{NV_YM}, and its proof follows the verbatim strategy.  Let us now give the precise Theorem:\\

\begin{theorem}[Quantitative Annular/Bubble Decomposition for Symmetric Balls]\label{t:quant_annulusbubble}
	Let $u:B_{10R}(p)\to N$ be a stationary harmonic map with $R^2\fint_{B_{10R}}|\nabla u|^2\leq \Lambda$ .  For each $0<\delta<\delta(m,R,\Lambda)$ there exists $\delta'(m,\Lambda,K_N,R,\delta)>0$ such that if $u$ is $(m-2,\delta')$-symmetric on $B_1(p)$ wrt $L$, then we can write
\begin{align}
	B_1(p)\subseteq \cA_0\cup\bigcup \cA_a\cup \bigcup \cB_b \cup \bigcup B_{r_c}(x_c)\, ,
\end{align}
such that
\begin{enumerate}
	\item $\cA_0\subseteq B_2(p)$ and $\cA_a\subseteq B_{2r_a}(x_a)$ are $\delta$-annular regions wrt $L$ .
	\item $\cB_b\subseteq B_{2r_b}(x_b)$ are $\delta$-bubble regions wrt $L$ .
	\item $\sum r_a^{m-2}+\sum r_b^{m-2}\leq N(m,\Lambda)$ with $\sum r_c^{m-2}<\delta$ .
	\item $\forall$ $\cA_a$ let $\cN\cM_a \equiv \{x\in \T_a\cap B_{r_a} : B_{2r_{x}}(x)\cap \cA_a$ intersects an $r_a$ or $r_c-$ ball$\}$, then $\cH^{m-2}[\cN\cM_a]\leq \delta r_a^{m-2}$.
\end{enumerate}
\end{theorem}
\begin{remark}
	The decomposition begins with an annular region $\cA_0$ because the initial ball is $(m-2,\delta')$-symmetric.
\end{remark}
\begin{remark}
	The maximality condition $(4)$ uses Theorem \ref{t:annular_existence}.2, which together with Theorem \ref{t:bubble_existence} guarantees in the covering argument that most balls $B_{r_x}(x)$ can be recovered with a bubble region.
\end{remark}
\begin{remark}
	The proof is verbatim from \cite{NV_YM}, and is based on a covering argument which only exploits the previously discussed properties of annular and bubble regions.
\end{remark}

\vspace{.3cm}

\subsection{Proof of the Quantitative Energy Identity}\label{ss:quantitative_energy_identity_proof}

In this subsection we prove the Quantitative Energy Identity Theorem \ref{t:main_quant_energy_id} using the Quantitative Annulus/Bubble Decomposition of Theorem \ref{t:quant_annulusbubble} and the Small Annular Energy of Theorem \ref{t:outline:annular_regions_energy} .  Recall that we have not yet proved Theorem \ref{t:outline:annular_regions_energy}, which will be the focus of the remainder of this paper.\\

Let us begin with a ball $B_2(p)$ which is $(m-2,\delta')$ symmetric with respect to $L$, where $\delta'\leq\delta''(n,m,K_N,\delta,R,\Lambda)$ will be chosen later.  We apply Theorem \ref{t:quant_annulusbubble} in order to build the covering

\begin{align}
	B_1(p)\subseteq \cA_0\cup\bigcup \cA_a\cup \bigcup \cB_b \cup \bigcup B_{r_c}(x_c)\, ,
\end{align}
where $\cA_a\subseteq B_{2r_a}(x_a)$ are $\delta$-annular regions,  $\cB_b\subseteq B_{2r_b}(x_b)$ are $\delta$-bubble regions, and we have the estimates  $\sum r_a^{m-2}+\sum r_b^{m-2}\leq N(m,\Lambda)$ with $\sum r_c^{m-2}<\delta$ .  \\

For $\epsilon'>0$ to be chosen later, if we let $\delta<\delta(m,K_N,R,\Lambda,\epsilon')$ then for each annular region we have by Theorem \ref{t:outline:annular_regions_energy} that
\begin{align}
	r_a^{2-m}\int_{L\cap B_{r_a}}\Big(\int_{(\ell+L^\perp)\cap B_{r_a}}|\nabla u|^2\Big)\,d\ell = r^{2-m}_a\int_{\cA_a\cap B_{r_a}} |\nabla u|^2 < (\epsilon')^2 \, .
\end{align}

Let $\ell_a\equiv \pi_L(x_a)\in L$ be the projections of the ball centers to $L$. For each $\ell\in L$ let us also denote $\ft_\ell\in \T$ as the unique point such that $\pi_L(\ft_\ell)=\ell$, i.e. $\ft_\ell=\pi^\T(\ell)$.  
Let us define
\begin{align}\label{e:quant_energy_identity_proof:1}
	\cG_a \equiv \big\{\ell\in L\cap B_{r_a}(\ell_a):& \big|\vartheta(\ft_\ell,r_{\ft_\ell})-\vartheta(\ft_\ell,r_a)\big|+\int_{(\ell+L^\perp)\cap B_{r_a}}|\nabla u|^2<\epsilon'\notag\\
	&\text{ and }B_{2r_{\ft_\ell}}(\ft_\ell)\cap \cA_a \text{ intersects only bubble regions}\,  \big\}\, .
\end{align}

Using the above, Corollary \ref{c:outline:annular_regions_energy} and Theorem \ref{t:quant_annulusbubble}.4 we have the estimate
\begin{align}
\Vol(B_{r_a}\setminus\cG_a) \leq C(m)\epsilon' r_a^{m-2}	\, .
\end{align}

Let us now define $\cG\subseteq L\cap B_1$ by
\begin{align}
	\cG \equiv \bigcup_a \cG_a \setminus \bigcup_c \pi_L (B_{r_c}(x_c))\, .
\end{align}
It follows that 
\begin{align}
	\Vol(B_1\setminus \cG) \leq C(m)\epsilon'\sum_a r_a^{m-2}+\sum r_c^{m-2}\leq N(m,\Lambda)\epsilon'+\delta\leq \epsilon\, ,  
\end{align}
where in the last inequality we have chosen $\epsilon'\leq \epsilon'(m,\Lambda,\epsilon)$ and $\delta\leq \delta(\epsilon)$ . \\

Consider now $\ell\in \cG$ , and observe first that $\ell+L^\perp$ only intersects annular and bubble regions, no $r_c$-balls.  Let us first focus on computing $\vartheta(\ell,1)$ .  Observe by $(a1)$ and Theorem \ref{t:prelim:spacial_gradient} that we have the estimate $\big|\vartheta(\ell,1)-\vartheta(\ft_\ell,1) \big| < C(m)\sqrt{\delta}$.  Observe also that by the definition of $\cG$, and in particular \eqref{e:quant_energy_identity_proof:1}, we have $\big|\vartheta(\ft_\ell,1)-\vartheta(\ft_\ell,r_{\ft_\ell}) \big| < \epsilon'$ .  Now $B_{2r_{\ft_\ell}}(\ft_\ell)\cap \cA_0$ intersects only bubble regions, so let us choose one of these bubble regions 
\begin{gather}
\cB^0_{}=\B{2r^0_{}}{x^0_{}}\setminus \bigcup_{j=1}^{ K ^0} \overline{\B{r^0_{j}}{x^0_{j}}}\, ,
 \end{gather}
and let $y^0_{j}=(x^0_{j}+L)\cap (\ell+L^\perp)$ .  It follows from Theorem \ref{t:bubble_structure} that 
\begin{align}
	\big|\vartheta(\ft,\ft_\ell)-\omega_{m-2}\int_{\cB^0\cap(\ell+L^\perp)}|\nabla u|^2-\sum \vartheta(y^0_{j},r^0_{j})\,\big| < \epsilon'\, .
\end{align}  

Combining this with the first estimates we get that 
\begin{align}
	\big|\vartheta(\ell,1)-\omega_{m-2}\int_{\cB^0\cap(\ell+L^\perp)}|\nabla u|^2-\sum \vartheta(y^0_{j},r^0_{j})\,\big| < 2\cdot \epsilon'\, ,
\end{align}  
where we have chosen $C(m)\sqrt{\delta}<\epsilon'$ .  We have so far written
\begin{align}
	(\ell+L^\perp)\cap B_1 \subseteq \cA_0\cup \cB^0\cup\bigcup_1^{K^0} \overline{ B_{r^0_{j}}(y^0_{j})} \cap B_1\, .
\end{align}

Observe by (b2) and (b3) of Definition \ref{d:bubble_region} that $\int_{\cB^0\cap(\ell+L^\perp)}|\nabla u|^2>\epsilon_0(m,K_N)$, and in particular $\vartheta(\ell,1)>\epsilon_0(m,K_N)-2\cdot C(m)\sqrt{\delta}$.  Consider now the balls $B_{r^0_{j}}(y^0_{j})$, they too must intersect either annular regions or bubble regions as $\ell+L^\perp$ only intersects annular and bubble regions.  If the intersection is with an annular region, then the next region to be intersected must be a bubble region as in the previous paragraphs.  Thus we can repeat the arguments from the previous paragraphs in order to recover each $B_{r^0_{j}}(y^0_{j})$ and get the estimate
\begin{align}
	\big|\vartheta(\ell,1)-\omega_{m-2}\int_{\cB^0\cap(\ell+L^\perp)}|\nabla u|^2-\omega_{m-2}\sum_{b=1}^{K^0}\int_{\cB^1_b\cap(\ell+L^\perp)}|\nabla u|^2-\sum_1^{K^1} \vartheta(y^1_{bj},r^1_{bj})\,\big| < 2(1+K^0)\cdot \epsilon'\, ,
\end{align}   
where $B_{r^1_{bj}}(y^1_{bj})$ are balls in the bubble regions $\big\{\cB^1_b\big\}_1^{K^0}$ .  Note there is one new bubble region for each $\big\{ B_{r^0_j}(y^0_j)\big\}$ and at most one new annular region for each such ball.  We have now formed the covering
\begin{align}
	(\ell+L^\perp)\cap B_1 &\subseteq \cA_0\cup \bigcup_1^{K^0_a}\cA^1_a \cup \cB^0\cup\bigcup_1^{K^0} \cB^1_b\cup\bigcup_1^{K^1} B_{r^1_{bj}}(y^1_{bj}) \cap B_1\, ,\notag\\
		&=\bigcup^{1+K^0_a}_{j,a}\cA^j_a\cup \bigcup^{1+K^0}_{j,b}\cB^j_b\cup \bigcup_1^{K^1} B_{r^1_{bj}}(y^1_{bj}) \cap B_1\, ,
\end{align}
where $K^0_a\leq K^0$ as recall once again that there is at most one annular region in each $B_{r^0_{j}}(y^0_{j})$ before the next bubble region appears.  Note that by \eqref{e:quant_energy_identity_proof:1} we also have the energy estimate
\begin{align}
	\int_{\bigcup^{1+K^0_a}_{j,a}\cA^j_a\cap (\ell+L^\perp)} |\nabla u|^2 \leq (1+K^0_a)\epsilon'\leq (1+K^0)\epsilon'\, .
\end{align}

Now let us continue to repeat the arguments of the reprevious paragraphs, so that after $k$-iterations we have arrived at the estimate
\begin{align}
\big|\vartheta(\ell,1)-\omega_{m-2}\sum_{j,b}\int_{\cB^j_b\cap(\ell+L^\perp)}|\nabla u|^2-\sum_1^{ K^{k+1}} \vartheta(y^1_{bj},r^1_{bj})\,\big| < 2\cdot (1+ K ^0+\cdots+ K ^k)\epsilon'\, ,
\end{align}
along with the covering
\begin{align}
	(\ell+L^\perp)\cap B_1 &\subseteq \bigcup^{1+K^0_a+\cdots+K^k_a}_{j,a}\cA^j_a\cup \bigcup^{1+K^0+\cdots+K^k}_{j,b}\cB^j_b\cup \bigcup_1^{K^{k+1}} B_{r^1_{bj}}(y^1_{bj}) \cap B_1\, .
\end{align}
As before we can apply \eqref{e:quant_energy_identity_proof:1} in order to estimate
\begin{align}
	\int_{\bigcup^{}\cA^j_a\cap (\ell+L^\perp)} |\nabla u|^2 \leq (1+K^0+\cdots+K^k)\epsilon'\, .
\end{align}

Note that each bubble contributes $\epsilon_0(m,K_N)$ energy, and so this implies the lower bound
\begin{align}
	\vartheta(\ell,1)\geq (1+ K ^0+\cdots+ K ^k)\,\big(\epsilon_0(m,K_N)-\epsilon'\,\big)\, .
\end{align}
An important consequence of this, because we have the bound $\vartheta(\ell,1)\leq C(m)\Lambda$, is that
\begin{align}
	1+ K ^0+\cdots+ K ^k \leq C(m,K_N)\Lambda\, .
\end{align}

In particular, this covering process must stop after a finite number of iterations with a uniformly bounded number of bubbles $ K =1+\sum  K^j\leq C(m,K_N)\Lambda$.  Thus if we clump all the bubbles together and use these bounds we have the final estimate
\begin{align}
\big|\vartheta(\ell,1)-\omega_{m-2}\sum_{1}^K \int_{\cB^j_b\cap(\ell+L^\perp)}|\nabla u|^2\,\big| < C(m,K_N,\Lambda)\epsilon'\, ,
\end{align}
together with the covering and estimate
\begin{align}\label{e:quant_energy_identity_proof:2}
	(\ell+L^\perp)\cap B_1 &\subseteq \bigcup^{K}\cA^j_a\cup \bigcup^{K}\cB^j_b\cap B_1\, ,\notag\\
	\int_{\bigcup^{}\cA^j_a\cap (\ell+L^\perp)} |\nabla u|^2&=\int_{(\ell+L^\perp)\setminus \bigcup \cB^j_b} |\nabla u|^2\leq C(m,K_N,\Lambda)\epsilon'\, .
\end{align}
If we take $\epsilon'\leq\epsilon'(m,N,\epsilon,\Lambda)$ then this tells us that $u$ is an $\epsilon$-harmonic map on $(\ell+L^\perp)\cap B_1$ with the energy identity of Theorem \ref{t:main_quant_energy_id}.4 .  In particular we have proved Theorem \ref{t:main_quant_energy_id}. $\qed$\\

\subsection{Proof of Energy Identity}\label{ss:energy_identity_proof}

Let us now use the Quantitative Energy Identity in order to prove the classical Energy Identity of Theorem \ref{t:energy_identity}.  Indeed it will be a little more direct to use \eqref{e:quant_energy_identity_proof:2} as our starting point.  The setup is that $u_j:B_2(p)\subseteq M\to N$ are stationary harmonic maps between smooth Riemannian manifolds with $N$ compact and $\int_{B_2}|\nabla u_j|^2\leq \Lambda$ .  Let $u_j\rightharpoonup u$ with $|\nabla u_j|^2dv_g\to |\nabla u|^2dv_g+\nu$ the $m-2$ rectifiable defect measure $\nu = e(x)\,\cH^{m-2}_S$.   \\

Let us now define $E_{\epsilon,\delta}\subseteq \supp\nu$ as those points $x\in E_{\epsilon,\delta}$ such that there exists an $m-2$ plane $L_x$ with points $x_j\to x$ such that
\begin{enumerate}
	\item $\exists$ $\delta$-bubbles $\cB_{j,a}\subseteq B_{r_{j,a}}(x_{j,a})$ with $x_{j,a}\in x_j+L^\perp_x$ \, ,
	\item If $\cB_j\equiv \bigcup \cB_{j,a}\cap (x_j+L_x^\perp)$ then $\big| \omega_{m-2}\int_{\cB_j}|\nabla u_j|^2 -e(x)\big| <\epsilon $ .
\end{enumerate}

Note that since each bubble region $\cB_{j,a}$ contributes $\epsilon_0(n,K_N)$ of energy, there are at most $N(n,K_N)\Lambda$ such bubbles for each $j$.  Our goal is to show that each set $E_{\epsilon,\delta}$ has full measure in $\supp\nu$.  If this is the case then observe that $E_0\equiv \bigcap E_{\epsilon,\delta}$ also has full measure in $\supp\nu$, and this will be our desired set.  Indeed for each $x\in E_0$ we can find a sequence $x_j\to x$ such that $\big| \omega_{m-2}\int_{\cB_j}|\nabla u_j|^2 -e(x)\big|\to 0$ .  By Theorem \ref{t:bubble_structure}.2 for each $\cB_{j,a}$ we can find at most $N(K_N,\Lambda)$ points $\{y_{j,a,b}\}\in \cB_{j,a}$ and scales $s_{j,a,b}>0$ with  $\vartheta(y_{j,a,b},s_{j,a,b})>\epsilon_0(n,K_N)$ and such that for each $\eta>0$ if $S\geq S(m,n,K_N,\Lambda,\eta)$ then $\big|\int_{\cB_{j,a}\setminus \bigcup B_{Ss_{j,a}}(y_{j,a})} |\nabla u|^2\big|< \eta$ .  In particular, if we define the bubble maps $b_{a,b}\equiv \lim_j s_{j,a,b}^{-1}\,u:\dR^{m-2}\times \dR^2\to N$ then we see that
\begin{align}
	e(x) = \sum \omega_{m-2}\int_{\dR^2} |\nabla b_{a,b}|^2\, ,
\end{align}
as claimed.\\

Thus what remains is to show that $E_{\epsilon,\delta}\subseteq \supp\nu$ all have full measure.  Let us begin by remarking that as a consequence of the $m-2$ rectifiable nature of $\nu$ we have for a.e. $x\in \supp\nu$ that $\nu$ has a unique tangent measure of the form $e(x)\cH^{m-2}_{L_x}$ for some $m-2$ plane $L_x$ .  In particular, it holds for such points that for all $\delta'>0$ if $r<r(x,\delta')$ and $j$ is sufficiently large then
\begin{align}
	\big|\vartheta(x,r)-e(x)\big|&<\delta'\notag\\
	B_r(x) \text{ is }(m-2,\delta')&-\text{symmetric}\, .
\end{align}
If we choose $\delta'$ sufficiently small then we can apply the Quantitative Energy Identity of Theorem \ref{t:main_quant_energy_id}, and the estimate \eqref{e:quant_energy_identity_proof:2}, for $\epsilon,\delta>0$ to all such balls.  Consider the sets $\cG_{\epsilon,j}\subseteq L_x\cap B_r(x)$ coming from this Theorem as well as their Hausdorff limits $\cG_{\epsilon,j}\to \cG_\epsilon$ .  It is immediate that $\cG_\epsilon\subseteq E_{\epsilon,\delta}$, and in particular its clear that $E_{\epsilon,\delta}$ has $\epsilon$-full measure in $\supp\nu$.  However, as $r>0$ was arbitrarily small, it follows that every $x\in \supp\nu$ with unique tangent measure is a density point of $E_{\epsilon,\delta}$ and hence $E_{\epsilon,\delta}$ itself has full measure. $\qed$\\
\vspace{.5cm}

\section{Proof Annular Regions have Small Energy of Theorem \ref{t:outline:annular_regions_energy} for Toy Model}\label{s:outline_toymodel}

We now return our attention to Theorem \ref{t:outline:annular_regions_energy}, which states that annular regions have small energy.  Our goal in this section is to outline the proof of the result in a vastly simplified context where we assume that for our annular region $\cA\subseteq B_2\setminus B_{\rf_x}(\T)$ the bubble center manifold $\T=L_\cA=L$ is a fixed plane and our bubble radius function $\rf_x=\rf_0$ is a constant.  Though the proof we will outline in this subsection requires a couple new ideas already, it is worth emphasizing that generalizing the proof of this subsection to the general case requires more than just technical work.  The errors which will appear for the general case will be of strictly larger order than those which appear in the toy model, and more to the point are of strictly larger order than are allowed.  Dealing with these will require additional new ideas on the construction itself.  However, it seems helpful in explaining these new points to first present the big picture in this simplified context, so that we can understand the methodology and where the problems arise.  Additionally, the notation we will introduce in this Section will not go to waste.  From a technical perspective we will see how to view the general case, when $\T$ is not a plane and $\rf_x$ is not a constant, as a family of toy models depending on points and scale.\\

Thus let us introduce again the setup for this Section.  We will deal with an annular region $\cA = B_2(0^m)\setminus B_{\rf_x}(\T)$ under the assumption that $\T=L$ is a fixed plane and $\rf_x=\rf_0$ is a constant.  Note that one can always build such annular regions, we simply cannot typically build {\it maximal} annular regions under this assumption, in particular condition (2) of Theorem \ref{t:annular_existence} will typically fail.  In order to prove Theorem \ref{t:outline:annular_regions_energy} in this context we will decompose the energy $|\nabla u|^2$ into its three base components.  Namely, let us first write
\begin{align}\label{e:outline:toy_model:energy_decomposition}
	\int_{\cA\cap B_1}|\nabla u|^2 =\int_{\cA\cap B_1}\abs{\pi_L\nabla u}^2+\int_{\cA\cap B_1}\langle\nabla u,n_{L^\perp}\rangle^2+\int_{\cA\cap B_1}\langle\nabla u,\alpha_{L^\perp}\rangle^2\, ,
\end{align}

where 
$\langle\nabla u,n_{L^\perp}\rangle^2 = \langle\nabla u,\frac{y^\perp}{|y^\perp|}\rangle^2$  is the $L^\perp$-radial energy of $u$, and $\langle\nabla u,\alpha_{L^\perp}\rangle^2$ is the remaining $L^\perp$-angular energy of $u$ .  Each of these three components will require a distinct technique to deal with.  Let us discuss these one at a time and attempt to introduce notation in manner which will be leading for the general case.\\

\subsection{The \texorpdfstring{$L$}{L}-Energy and the Log Improvement for Toy Model}\label{ss:outline_toymodel:L_energy}

Recall from Definition \ref{d:partial_energies} the heat mollified $L$-energy
\begin{align}\label{e:d:L_energy}
	\vartheta(x,r;L)\equiv r^{2}\int \rho_r(y-x)\,\abs{\pi_L\nabla u}^2\, .
\end{align}
Note that we can estimate
\begin{align}
	\int_{\cA\cap B_1}\abs{\pi_L\nabla u}^2\leq C(m)\vartheta(0,1;L)\, ,
\end{align}
and thus estimating $\vartheta(0,1;L)$ 
is sufficient for controlling the $L$-energy term in \eqref{e:outline:toy_model:energy_decomposition}.   This turns out to be quite straightforward.  Indeed, using the quantitative cone splitting of Theorem \ref{t:prelim:cone_splitting_annular} we have the estimate
\begin{align}\label{e:outline:toy_model:L_energy}
	\vartheta(x,r;L)\leq C(m)\fint_{\T\cap B_r(x)}|\vartheta_{}(y,2r)-\vartheta(y,r)|<C(m)\,\delta\, .
\end{align}

In particular, applying \eqref{e:outline:toy_model:L_energy} to $\vartheta(0,1;L)$ produces the desired control on the $L$-energy in \eqref{e:outline:toy_model:energy_decomposition}.  However it is worth observing that we can use \eqref{e:outline:toy_model:L_energy} to acquire a more refined estimate on annular regions, and this plays an implicit role in future analysis. Roughly speaking, we will obtain a log-like improvement on the $L$-energy of the form
\begin{gather}
 \int_{\B r L \cap \B 1 0 } \abs{\pi_L\nabla u}^2 \leq C(m)\frac{\int_{\B 2 0}\abs{\nabla u}^2}{\abs{\ln(r)}}\, .
\end{gather}

To state it more precisely let us first consider a family of smooth cutoffs $\psi_\T:\T\times \dR^+\to \dR$ with
\begin{gather}
 \psi_\T(x,r)=\phi\ton{\abs x^2} \eta(r)\, ,
\end{gather}
where $\phi:[0,2)\to [0,1]$ is a simple smooth cutoff with
\begin{gather}
 \phi([0,1])=1\, , \qquad \supp \phi \subset [0,2]\, , \qquad \dot \phi \leq0\quad \text{and}\quad \abs{\dot \phi}+\abs{\ddot \phi}\leq C\, .
\end{gather}
For convenience, we also assume that $\dddot \phi(t) \leq 0$ if $t\geq 3/2$ with $\phi(t)\geq 10^{-1}$ if $t\leq 3/2$.  We also set $\eta:[0,\infty)\to [0,1]$ such that
\begin{gather}
 \eta([\rf_0,1])=1\, , \qquad \supp \eta \subseteq [\rf_0/2,2]\, , \qquad r\abs{ \dot \eta}+r^2\abs{ \ddot \eta}\leq C \quad \text{and}\quad \dot \eta (r) \geq 0  \ \ \text{for}  \ r\leq 1\, .
\end{gather}

\begin{remark}
 In the general case, this cutoff function will be more complicated. For future comparison, notice that this $\psi_\T(x,r)$ trivially satisfies the following scale-invariant estimate
\begin{align}\label{e:outline:toy_model:psi_T}
	 \begin{cases}
 	\psi_\T(x,r) = 1 &\text{ if } |x|\leq 1\text{ and }\fr_0\leq |r|\leq 1\, ,\\
 	\psi_\T(x,r) = 0 &\text{ if } |x|\geq 2\, ,\,r<\rf_0/2\, ,\text{ or }r\geq 2\\
 	|\nabla^2_L\psi_\T|+r^2|\ddot\psi_\T|\leq C(m) &\text{ everywhere }\, .
 \end{cases}
\end{align}
\end{remark}

Let us remark on the following interesting and useful construction point:\\

{\bf Claim: } We have that
\begin{gather}\label{eq_Delta_L_psi_toy}
 \Delta_L^{(x)}\psi_\T(x,r)\geq -C(m) \psi_\T(x,r)\, .
\end{gather}

To prove the claim let us compute $\Delta_L^{(x)}\psi_\T(x,r)=2\eta(r)\ton{(m-2)\dot \phi(\abs x^2)+\abs x^2 \ddot \phi(\abs x^2)}$ , and then our goal will be to show that $(m-2)\dot \phi(t)+t\ddot \phi(t)\geq -C(m)\phi(t)$.  Let us define the function
\begin{gather}
 F(t)=(m-1)(m-2)\phi(t)+(m-2)\dot \phi(t)+t\ddot \phi(t)\, ,
\end{gather}
then our claim is implied by showing that $F\geq 0$. Note first that $F(t)$ is clearly smooth everywhere and nonnegative for $t\leq 1$ and $t\geq 2$.  Let $\bar t$ be the absolute minimum of $F$.  If $\bar t\leq 3/2$ then by definition we have $\phi(t)\geq 10^{-1}$, and so the result is clear.  Otherwise, we have at least have that $\phi(\bar t), \ddot\phi(\bar t)\geq 0$ with $\dddot \phi (\bar t)\leq 0$.  Plugging this into the critical point formula we get
\begin{gather}
 0=\dot F(\bar t )=(m-1)(m-2)\dot \phi(\bar t)+(m-1)\ddot \phi(\bar t)+\bar t\underbrace{\dddot \phi(\bar t)}_{\leq 0}\, , \qquad \Longrightarrow \qquad (m-2)\dot \phi(\bar t) + \ddot \phi(\bar t)\geq 0\, .
\end{gather}
Thus we have
\begin{gather}
 F(\bar t)=(m-1)(m-2)\underbrace{\phi(\bar t)}_{\geq 0}+(t-1)\underbrace{\ddot \phi(\bar t)}_{\geq 0} + (m-2)\dot \phi(\bar t) + \ddot \phi(\bar t)\geq 0\, ,
\end{gather}
as claimed.$\qed$\\

Having constructed our cutoff function, we are now in a position to define the $L$-energy at scale $r$ over $\T$ by 
\begin{align}
	\vartheta(\T,r;L)\equiv \int_{\T} \psi_\T(x,r)\,\vartheta(x,r;L) = r^{2}\int_{\T}\psi_\T(x,r)\int \rho_r(x-y)\,\abs{\pi_L\nabla u}^2\, ,
\end{align}
where recall $\rho(t)$ is defined as in Definition \ref{d:heat_mollifier}.  Let us first see the following: \\

{\bf Claim: } We have the estimate $\vartheta(\T,r;L)\leq C(m)\int_{\T}\,|\vartheta(x,8r)-\vartheta(x,r)|\, .$ \\
The claim follows from integrating \eqref{e:outline:toy_model:L_energy}. $\qed$\\

By summing the above over all scales and using that $\vartheta(x,r;L)\leq C(m)\,\vartheta(x,2r;L)$ as in \eqref{e:prelim:partial_energy:basic_bounds} we get for all $\rf_0\leq r\leq 1$ the Dini estimate
\begin{align}\label{e:outline_toymodel:Dini_L}
	\int_r^2 \vartheta(\T,s;L)\frac{ds}{s} &\leq C(m)\sum_{r\leq s_a=2^{-a}\leq 2}\vartheta(\T,s_a;L) \notag\\
	&\leq C(m)\int_{\T}\sum_{r\leq s_a=2^{-a}\leq 2}|\vartheta_{}(x,8s_a)-\vartheta_{}(x,s_a)|\notag\\
	&\leq C(m)\int_{\T}|\vartheta(x,16)-\vartheta(x,r)|\, ,\notag\\
	&\leq C(m)\Lambda\, ,
\end{align}
where in the last line we have used that $|\vartheta(x,16)-\vartheta(x,r)|\leq \vartheta(x,16)\leq C(m)\Lambda$ is bounded by the energy.\\

To finish the estimate we need one additional step, which relies strongly on the fact that $\T=L$ is a set of codimension 2. \\

{\bf Claim: } For all $\fr_0\leq r\leq s\leq 1$:
\begin{align}\label{eq_toy_L_en_rs}
	r^2\int_\T \psi_\T(x,r)\rho_r(y-x)\leq c(m)s^2\int_\T \psi_\T(x,s)\rho_s(y-x)\, . \\\notag
\end{align}
Before proving the claim, let's use it to estimate
\begin{align}
	\vartheta(\T,r;L)&=\int_\T \psi_\T(x,r)\vartheta(x,r;L)= r^{2}\int_\T \psi_\T(x,r)\int \rho_r(y-x)\,\abs{\pi_L\nabla u(y)}^2\, \notag\\
	&\leq c(m)s^{2}\int_\T \psi_\T(x,s)\int \rho_s(y-x)\,\abs{\pi_L\nabla u(y)}^2\, \notag\\
	&=c(m)\vartheta(\T,s;L)\, .
\end{align}

If we plug this into our Dini estimate \eqref{e:outline_toymodel:Dini_L} we get for all $\rf_0\leq r\leq 1$ the log improvement:
\begin{align}
	|\ln r|\, \vartheta(\T,r;L) = \vartheta(\T,r;L) \int_{r}^1 \frac{1}{s} \leq C(m)\int_{r}^1 \frac{\vartheta(\T,s;L)}{s} \leq C(m)\Lambda\, ,
\end{align}
equivalently
\begin{align}
	\vartheta(\T,r;L) \leq C(m)\frac{\Lambda}{|\ln r|}\, .
\end{align}

We close this section with the proof of \eqref{eq_toy_L_en_rs}.
\begin{proof}[Proof of \eqref{eq_toy_L_en_rs}]
 Fix $y\in \R^m$ and $r\leq 1$, and define
 \begin{gather}
  f_r(y)=r^2\int_\T \psi_\T(x,r)\rho_r(y-x)\, .
 \end{gather}
 We will prove the almost monotonicity formula:
\begin{gather}
 r\frac{d}{dr} f_r(y)\geq -C(m)r^2 f_r(y)\geq -C(m)r f_r(y)\,  \qquad \Longrightarrow \qquad \text{for } 0<r<s\leq 1:\, \quad e^{Cr} f_r(y) \leq e^{Cs} f_s(y)\, .
\end{gather}
Then this follows from
\begin{align}
 r\frac{d}{dr}& \ton{r^2\int_\T \psi_\T(x,r)\rho_r(y-x)}=r\frac{d}{dr} \ton{r^{2-m}\int_\T \psi_\T(x,r)\rho\ton{\frac{\abs{y-x}^2}{2r^2}}}=r^2\int_\T \underbrace{r\dot \psi_\T(x,r)}_{\geq 0}\rho_r(y-x) + \notag\\
 &+(2-m) r^2\int_\T \psi_\T(x,r)\rho_r(y-x)-\int_\T r^{-m} \psi_\T(x,r) \dot\rho \ton{\frac{\abs{y-x}^2}{2r^2}}\abs{y-x}^2 \notag\\
 &\geq (2-m) r^2\int_\T \psi_\T(x,r)\rho_r(y-x)-\int_\T r^{-m} \psi_\T(x,r) \dot \rho \ton{\frac{\abs{y-x}^2}{2r^2}}\abs{\pi_L(y-x)}^2\notag\\
 &=(2-m) r^2\int_\T \psi_\T(x,r)\rho_r(y-x)-\int_\T r^{2-m} \psi_\T(x,r) \ps{\pi_L \nabla^{(x)}\rho \ton{\frac{\abs{y-x}^2}{2r^2}}}{\pi_L(x-y)}\notag\\
 &=\int_\T r^{2-m} \ps{\pi_L \nabla^{(x)}\psi_\T(x,r)}{\pi_L(x-y)} \rho \ton{\frac{\abs{y-x}^2}{2r^2}}\, .
\end{align}
We set $\varrho(t)=\int_t^\infty\rho(s) \geq 0$ to be the only primitive of $-\rho$ with compact support.  Using that $\rho(t)\leq -2\dot\rho(t)$ as in Definition \ref{d:heat_mollifier} we then have the global pointwise bound $\varrho(t)\leq 2\rho(t)$.  Integrating by parts, we have 
\begin{align}
 r\frac{d}{dr} \ton{r^2\int_L \psi_\T(x,r)\rho_r(y-x)}&\geq -\int_\T r^{2-m} \ps{\pi_L \nabla^{(x)}\psi_\T(x,r)}{\pi_L(x-y)} \dot\varrho \ton{\frac{\abs{y-x}^2}{2r^2}}\notag\\
 &= - r^2\int_\T r^{2-m} \ps{\pi_L \nabla^{(x)}\psi_\T(x,r)}{\pi_L \nabla^{(x)} \varrho \ton{\frac{\abs{y-x}^2}{2r^2}} } \notag \\
 &=r^2\int_\T r^{2-m} \Delta_L\psi_\T(x,r)\varrho \ton{\frac{\abs{y-x}^2}{2r^2}}\, ,\notag\\
 &\geq -C(m)r^2\int_\T r^{2-m} \psi_\T(x,r)\varrho \ton{\frac{\abs{y-x}^2}{2r^2}}\, ,\notag\\
 &\geq -C(m)r^2 \int_\T r^{2-m} \psi_\T(x,r)\rho \ton{\frac{\abs{y-x}^2}{2r^2}}\notag\\
 &=-C(m)r^2\cdot r^2\int_L \psi_\T(x,r)\rho_r(y-x)\, ,
\end{align}
where we have used our previous estimate $\Delta_L\psi_\T\geq -C(m)\psi_\T$ in \eqref{eq_Delta_L_psi_toy}.  This finishes the proof.

\end{proof}

\vspace{5mm}

\vspace{.5cm}

\subsection{The Angular Energy and Superconvexity for the Toy Model}\label{ss:outline_toymodel:angular_energy} Let us now move ourselves to controlling the angular energy in our toy model \eqref{e:outline:toy_model:energy_decomposition}. 
A key estimate in proving the energy identity in the Yang-Mills context in \cite{NV_YM} was the superconvexity estimates for $\epsilon$-gauges.  
This estimate could be used to imply full energy bounds on annular regions.  Corresponding superconvexity estimates for the full energy in the nonlinear harmonic map case fail, however we are able to prove a related superconvexity for the purely angular energy.  
In the two dimensional case this  estimate can essentially be found in \cite{Jost2D,ParkerBubbleTree,LinWang}.  In the higher dimensional toy model presented here the result will be a consequence of the uniformly subharmonic estimate in Lemma \ref{l:outline:toy_model:angular_energy:superconvexity} for the purely angular energy.  
A version of this uniform subharmonicity will be proven in the general context in Section \ref{s:angular_energy}, though will be a bit more subtle to state.  We will essentially view the general context as a family of toy models depending on point and scale, each of which satisfy the required uniform subharmonic estimate.\\

Let us begin by recalling from Section \ref{ss:restricted_energy} the restricted angular energy functional
\begin{align}
\hat\vartheta(x,r;\alpha_{L^\perp}) &\equiv  \int \hat\rho_r(y-x;L)\,|\pi_{L^\perp}(y-x)|^2\langle\nabla u(y),\alpha_{L^\perp}\rangle^2\, ,
\end{align}
where recall $\hat\rho(x-y;L)$ is the restricted heat kernel cutoff near $L$.  Note that for the restricted energy we no longer have that $\hat\vartheta(0,1;\alpha_{L^\perp})$ bounds the angular energy from \eqref{e:outline:toy_model:energy_decomposition}.  This is both because the mollifier $\hat \rho$ vanishes near $L$, and additionally we have that the weight $|\pi_{L^\perp}(y-x)|^2$ decays near $L$.  In particular the control on the angular energy is only effective when $|\pi_{L^\perp}(y-x)|^2\approx r^2$.\\

As with the $L$-energy let us define the angular energy over $\T$ as
\begin{align}
	\hat\vartheta(\T,r;\alpha_{L^\perp}) \equiv \int_\T \psi_\T(x,r) \hat\vartheta(x,r;\alpha_{L^\perp})\, .
\end{align}

Observe that we can then bound the angular energy in \eqref{e:outline:toy_model:energy_decomposition} through the formula

\begin{align}\label{e:outline:toy_model:angular_energy_bound}
	\int_{\cA\cap B_1}\langle\nabla u,\alpha_{L^\perp}\rangle^2\leq \hat\vartheta(\cA;\alpha_{L^\perp}) \equiv \int \hat\vartheta(\T,r;\alpha_{L^\perp})\frac{dr}{r}\, .
\end{align}

Note that as $\hat\vartheta(x,r;\alpha_{L^\perp})\leq C(m)\delta$ we have the estimate
\begin{align}\label{e:outline:toy_model:annular_energy_slice}
	\hat\vartheta(\T,r;\alpha_{L^\perp}) \leq C(m)\delta\, .
\end{align}
However we clearly cannot bound \eqref{e:outline:toy_model:angular_energy_bound} with this estimate, as integrating it over all scales produces a $|\ln\rf_0|$ error.\\

\subsubsection{The $S^1$ Averaged Angular Energy.}

To control $\hat\vartheta(\cA;\alpha_{L^\perp})$ we will view the angular energy in a different fashion.  In writing the angular and radial energies in \eqref{e:outline:toy_model:energy_decomposition} we have already implicitly written $\dR^m=L\times L^\perp$ with $L^\perp$ in polar coordinates. Note that this writing requires an affine plane, not just a subspace, as the origin of our polar coordinates requires a center.  As $\T=L$ this center will be independent of the point $x\in \T$ we are working at.  In particular we have for a point $y\in \dR^m$ that we can naturally decompose
\begin{align}
	y=(y_L,y_{L^\perp})=(y_L,s_{\perp},\alpha_{\perp})\, ,
\end{align}
where $s_\perp = |y_{L^\perp}|=|\pi_{L^\perp}(y)|$ and $\alpha_\perp\in S^1$ are the polar coordinate representation of $y_{L^\perp}\in L^\perp$.  We can define our $S^1$ averaged angular energy functional $\cE_\alpha:\dR^m\to \dR^+$ by

\begin{align}\label{e:outline:toy_model:angular_energy_averaged}
	\cE_\alpha(y) =\cE_\alpha(y_L,s_{\perp},\alpha_{L^\perp})\equiv s_\perp^2\fint_{S^1} \langle \nabla u(y_L,s_\perp,\alpha_\perp), \vec\alpha_\perp\rangle^2 d\alpha_\perp\, ,
\end{align}
where we are sloppily identifying $\alpha_\perp$ as both the angular coordinate and the unit angular vector field in $L^\perp$. When confusion might arise we have written $\vec\alpha_\perp$ to denote the unit vector field.  Notice that $\cE_\alpha(y_L,s_{\perp},\alpha_{L^\perp})$ is clearly independent of $\alpha_{L^\perp}$.  In particular we have taken the angular energy $\langle\nabla u, \alpha_\perp\rangle^2$ and averaged it over each circle to produce a function on $B_2$ which is invariant under rotations around $L$.
This function relates to the annular energy estimate in \eqref{e:outline:toy_model:energy_decomposition} by
\begin{align}\label{e:outline:toy_model:annular_energy}
	\hat\vartheta(x,r;\alpha_{L^\perp})&= \int \hat\rho_r(y-x)\cE_\alpha(y)\, , \text{ for }x\in\T\, .
\end{align}

The main technical tool is to show that $\cE_\alpha $ satisfies a uniform subharmonic equation on the annular region $\cA$. 
Namely, let us define the conformal Laplacian
\begin{align}\label{e:outline:toy_model:conformal_laplacian}
	\bar\Delta_L \equiv \ton{\frac{\partial}{\partial\ln s_\perp}}^2+\ton{\frac{\partial}{\partial\alpha_\perp}}^2+s_\perp^2\Delta_L\, .
\end{align}
Notice that the standard Laplacian in cylindrical coordinates around $L$ takes the form
\begin{gather}
 \Delta f = \frac 1 {s_\perp^2} \ton{\frac{\partial}{\partial\ln s_\perp}}^2 f+ \frac 1 {s_\perp^2}\ton{\frac{\partial}{\partial\alpha_\perp}}^2 f + \Delta_{L} f\, .
\end{gather}
Our main claim is the following:

\begin{lemma}[Toy Model Uniform Subharmonicity]\label{l:outline:toy_model:angular_energy:superconvexity}
For $y\in B_2\setminus B_{e^{-R}\rf_0}(L)$ we have the pointwise inequality $\bar\Delta_L \cE_\alpha \geq \ton{2-C(m,R)\sqrt{\delta}}\cE_\alpha \geq \cE_\alpha$ .
\end{lemma}
\begin{proof}
First of all, we notice by Theorem \ref{t:prelim:cone_splitting_annular} that for all $x\in \B 1 0\cap L$ and $r\in [\rf_0,4]$, we have
\begin{gather}
 \vartheta(x,r;L)+\vartheta(x,r;L^\perp)\leq C(m)\delta\, .
\end{gather}
Fix $y\in \B 2 0$ and set $\bar d=d(y,L)/10$, so that $B_{2\bar d}(y)\subseteq \cA$ satisfies
\begin{gather}\label{eq_eps_reg_w}
 \bar d^{2-m}\int_{\B {2\bar d} y} \abs{\nabla u}^2 \leq C(m,R) \delta\leq \epsilon_0\, .
\end{gather}
In particular $\B {\bar d}{y}$ is an $\epsilon$-regularity ball, see the remarks after Definition \ref{d:annular_region}.  Now we turn to the main estimate.

The key idea is not too dissimilar from Bochner's inequality, which for flat domain states that $\frac 1 2 \Delta \abs{\nabla u}^2=\abs{\nabla^2 u }^2+ \ps{\nabla u}{\nabla \Delta u}$. Here we exploit a similar computation, where we will see that the term corresponding to $\ps{\nabla u}{\nabla \Delta u}$ is negligible (in the $\epsilon$-regularity region), and the term corresponding to $\abs{\nabla^2 u }^2$ will be bounded from below by the angular energy. In particular, if we restrict ourselves to \textit{angular energy} instead of full-blown energy, and consider it integrated over circles as in the definition of $\cE_\alpha$, we can exploit the Poincar\' e inequality:
\begin{gather}
 \int_{S^1} \abs{\nabla \nabla_{\vec \alpha_\perp} u}^2 \geq \int_{S^1} \abs{\nabla_{\vec\alpha_\perp} \nabla_{\vec \alpha_\perp} u}^2 \geq \int_{S^1}\abs{\nabla_{\vec \alpha_\perp} u}^2\, .
\end{gather}

In detail, we can estimate:
\begin{align}\label{eq_switch_polar_coordinates}
 \frac 1 2 \frac{\partial}{\partial \ln s_\perp} \cE_\alpha(y)=&\cE_\alpha(y) + s_\perp^2\fint_{S^1}\nabla_{\vec \alpha_\perp}  u  \cdot \ton{s_\perp \nabla_{\vec s_\perp}\nabla_{\vec \alpha_\perp}  u}=s_\perp^2\fint_{S^1}\nabla_{\vec \alpha_\perp}  u  \cdot \nabla_{\vec \alpha_\perp}  \ton{s_\perp \nabla_{\vec s_\perp}u}\, ,
\end{align}
where we used the fact that $ \nabla_{\vec \alpha_\perp } (s_\perp \vec s_\perp)= \vec \alpha_\perp$, while $\nabla_{\vec s_\perp}\vec \alpha_\perp=0$.

Arguing in a similar way, we have:
\begin{align}
 \frac 1 2 \bar \Delta_L \cE_\alpha(y) &=\frac{\partial}{\partial \ln s_\perp} \ton{\cE_\alpha(y) + s_\perp^2\fint_{S^1}\nabla_{\vec \alpha_\perp}  u  \cdot \ton{s_\perp \nabla_{\vec s_\perp}\nabla_{\vec \alpha_\perp}  u}  }\notag\\
 &+ s_\perp^4\fint_{S^1}\abs{\pi_L\nabla \nabla_{\vec \alpha_\perp} u}^2 + s_\perp^4 \fint_{S^1} \nabla_{\vec\alpha_\perp} u \nabla_{\vec\alpha_\perp}\qua{\Delta_L u}\notag\\
 &\stackrel{\eqref{eq_switch_polar_coordinates}}{=}\frac{\partial}{\partial \ln s_\perp} \ton{s_\perp^2\fint_{S^1}\nabla_{\vec \alpha_\perp}  u  \cdot \nabla_{\vec \alpha_\perp}  \ton{s_\perp \nabla_{\vec s_\perp}u}  }\notag\\
 &+ s_\perp^4\fint_{S^1}\abs{\pi_L\nabla \nabla_{\vec \alpha_\perp} u}^2 + s_\perp^4 \fint_{S^1} \nabla_{\vec\alpha_\perp} u \nabla_{\vec\alpha_\perp}\qua{\Delta_L u}\notag\\
 &= s_\perp^2\fint_{S^1}\abs{\nabla_{\vec \alpha_\perp} (s_\perp \nabla_{\vec s_\perp} u)}^2 + s_\perp^2 \fint_{S^1} \nabla_{\vec\alpha_\perp} u \nabla_{\vec\alpha_\perp}\qua{\ton{s_\perp \nabla_{\vec s_\perp}}^2 u}\notag\\
 &+ s_\perp^4\fint_{S^1}\abs{\pi_L\nabla \nabla_{\vec \alpha_\perp} u}^2 + s_\perp^4 \fint_{S^1} \nabla_{\vec\alpha_\perp} u \nabla_{\vec\alpha_\perp}\qua{\Delta_L u}\notag\\
 &\geq s_\perp^4 \fint_{S^1} \nabla_{\vec\alpha_\perp} u \nabla_{\vec\alpha_\perp}\qua{\ton{s_\perp^{-2}\ton{s_\perp \nabla_{\vec s_\perp}}^2 +\Delta_L} u}= \\
 &= s_\perp ^4 \fint_{S^1} \nabla_{\vec \alpha_\perp} u \nabla_{\vec \alpha_\perp} (\Delta u)- s_\perp^2\fint_{S^1}\nabla_{\vec \alpha_\perp} u \ \nabla_{\vec \alpha_\perp} \qua{\nabla_{\vec \alpha_\perp}\nabla_{\vec \alpha_\perp} u }\notag\\
 &= s_\perp ^4 \fint_{S^1} \nabla_{\vec \alpha_\perp} u \nabla_{\vec \alpha_\perp} (\Delta u)+ s_\perp^2\fint_{S^1}\abs{\nabla_{\vec \alpha_\perp}\nabla_{\vec \alpha_\perp} u }^2\, .
\end{align}
By \eqref{eq_eps_reg_w}, we can apply Lemma \ref{l:nabla_Delta_u_epsilon_reg} to the first term. Using Poincar\`e inequality in the last term, we conclude that
\begin{align}
 \bar \Delta_L \cE_\alpha \geq \ton{2-C(m,R)\sqrt \delta}\cE_\alpha\, ,
\end{align}
as desired.

\end{proof}

\subsubsection{Superconvexity of $\hat\vartheta(\T,r;\alpha_\perp)$}

Let us remark that in the toy model we can use Lemma \ref{l:outline:toy_model:angular_energy:superconvexity} to quickly prove the annular energy bound
\begin{align}
	\int_{\cA\cap B_1}\langle\nabla u,\alpha_{L^\perp}\rangle^2 &\leq \int\int_\T \psi_\T(x,r)\cE_\alpha(x,r)\frac{dr}{r}\notag\\
	&\leq \int\int_\T \psi_\T(x,r)\bar \Delta_L\cE_\alpha(x,r)\frac{dr}{r}\notag\\
	&=\int\int_\T \bar\Delta_L \psi_\T(x,r)\,\cE_\alpha(x,r)\frac{dr}{r}\notag\\
	&\leq C(m)\delta\int\int_\T |\bar\Delta_L \psi_\T|(x,r)\,\frac{dr}{r}\notag\\
	&\leq C(m)\delta\int \int_\T r dr + C(m)\delta\fint_1^2\,dr+C(m)\delta\fint_{\rf_0/2}^{\rf_0}\,dr\notag\\
	&\leq C(m)\delta\, ,
\end{align}
where the three terms above correspond to the $L$ and $r$ cutoff regions of $\psi_\T$ .\\

The strategy will work slightly differently in the general case, where the subharmonic equation takes a more subtle form.  Let us quickly outline another approach to the above computation, which is a bit more involved but will be the correct generalization later.  Namely, one can use Lemma \ref{l:outline:toy_model:angular_energy:superconvexity} to prove a superconvexity estimate for $\hat\vartheta(\T,r;\alpha_\perp)$:

\begin{proposition}[Super Convexity in the Toy Model]\label{p:outline:toy_model:superconvexity}   We have the estimate
\begin{align}
	\ton{r\frac{\partial}{\partial r}}^2\hat\vartheta(\T,r;\alpha_{L^\perp}) \geq \hat\vartheta(\T,r;\alpha_{L^\perp})-\epsilon(r)\, ,
\end{align}
where $\epsilon(r)\leq C(m)\delta\int_\T \Big(r^2\big|\ddot\psi_\T\big|+r\big|\dot\psi_\T\big|+r\big|\nabla_L\psi_\T\big|+r^2\big|\nabla^2_L\psi_\T\big|\,\Big)$ satisfies $\int\epsilon(r)\frac{dr}{r}\leq C(m)\delta$ .
\end{proposition}

The proof is a straightforward if demanding computation based on the combination of Lemma \ref{l:outline:toy_model:angular_energy:superconvexity} and \eqref{e:outline:toy_model:annular_energy}.  We will do the computation explicitly in Section \ref{s:angular_energy} when dealing with the general case.   In the toy model this superconvexity may be used to give polynomial decay/growth estimates on $\hat\vartheta(\T,r;\alpha_{L^\perp})$.  However such estimates will necessarily fail in the general case, so let us just see how to use Proposition \ref{p:outline:toy_model:superconvexity} to directly bound the angular energy bound in the annular region:
 \begin{align}
 	\vartheta(\cA;\alpha_{L^\perp})&=
 	\int \hat\vartheta(\T,r;\alpha_{L^\perp})\frac{dr}{r}\notag\\
 	&\leq \int \ton{r\frac{\partial}{\partial r}}^2\hat\vartheta(\T,r;\alpha_{L^\perp})\frac{dr}{r}+\int\epsilon(r)\frac{dr}{r}\notag\\
 	&=\int\epsilon(r)\frac{dr}{r}\notag\\
 	&\leq C(m)\delta\, ,
 \end{align}
which proves our required bound on the angular energy.


\subsection{Radial Energy and Stationary Estimates for Toy Model}

Let us now discuss the bound for the radial energy in \eqref{e:outline:toy_model:energy_decomposition}.  The goal will be to control the radial energy in terms of the angular and $L$-energies, and as they have already been controlled we will be done.  It is worth pointing out that dealing with this situation in the general setup will require an additional new idea, as the errors involved will become not just more complicated but of strictly larger order.  Handling this will involve approximating our bubble center manifold $\T$ by smooth submanifolds $\T_r$ that solve equations and cancel these highest order errors.   We will outline that in the next Section.\\

As with the angular energy we will focus on the restricted radial energy, which only measures the radial energy away from a neighborhood of $L$ on each scale.  We refer the reader to Section \ref{ss:restricted_energy} for the definition of the restricted energies, where we define the restricted radial energy
\begin{align}
	\hat\vartheta(x,r;n_{L^\perp})&\equiv \int \hat\rho_r(y-x;L)\big\langle\nabla u(y),\pi_{L^\perp}(y-x)\big\rangle^2\, ,\notag\\
	&=\int\hat\rho_r(y-x;L)\big|\pi_{L^\perp}(y-x)\,\big|^2 \langle\nabla u(y),n_{L^\perp}\rangle^2\, ,
\end{align} 
where $\hat\rho_r$ is the heat kernel mollifier cutoff in the $e^{-R}$ neighborhood of $L$ as in \eqref{e:prelim:L_mollifier}.  We can integrate this over $\T$ with our bubble weight $\psi_\T$ from \eqref{e:outline:toy_model:psi_T} to define the radial energy on $\T$ :

\begin{align}\label{e:outline:toy_model:radial_energy}
	\hat\vartheta(\T,r;n_{L^\perp})\equiv \int_{\T}\psi_\T(x,r)\, \hat\vartheta(x,r;n_{L^\perp}) = \int_{\T}\psi_\T(x,r)\int\hat\rho_r(x-y;L)\langle\nabla u(y),\pi_{L^\perp}(y-x)\rangle^2\, .
\end{align}

The $\hat\vartheta(\T,r;n_{L^\perp})$ energy functional represents the radial energy at scale $r$ along $\T$.  Control of $\hat\vartheta(\T,r;n_{L^\perp})$ can be related to the annular radial energy \eqref{e:outline:toy_model:energy_decomposition} through the estimate
\begin{align}
	\int_{\cA\cap B_1}\langle\nabla u(y),n_{L^\perp}\rangle^2 \leq \hat\vartheta(\cA;n_{L^\perp})\equiv \int_{}\,\hat\vartheta(\T,s;n_{L^\perp})\frac{ds}{s}\, .
\end{align}

In order to estimate \eqref{e:outline:toy_model:radial_energy} we want to apply the stationary equation to a vector field which will produce \eqref{e:outline:toy_model:radial_energy} as at least one of its terms.  Precisely let us consider the stationary equation applied to
\begin{align}\label{e:outline:toy_model:radial_energy:stationary}
	\xi_r(x) \equiv \int_{\T}\psi_\T(z,r)\tilde\rho_r(x-z;L)\,\pi_{L^\perp}(x-z) \approx \pi_{L^\perp}\big(x-z\big)\,\frac{e^{-{\abs{\pi_{L^\perp}\ton{\frac{x-z}{r}}}^2}}}{2\pi}\, ,
\end{align}
where recall from \eqref{e:prelim:L_Mollifier2} that the $L$-heat mollifier $\tilde\rho(y;L)\approx \rho(y)$ is $C^1$ close to the heat kernel $\rho$, but has its $L^\perp$-gradient cutoff near $L$.  Note that $\xi(x)$ looks approximately like the $L^\perp$ radial vector field, multiplied by the $L^\perp$-Gaussian, with equality if we take $\psi_{\T}(z,r)\equiv 1 $ and $R=\infty$ in Definition \ref{d:heat_mollifier} of $\rho$.\\

If we plug \eqref{e:outline:toy_model:radial_energy:stationary} into the stationary equation \eqref{e:stationary_equation}, we get 
\begin{align}
	&\hat\vartheta(\T,r;n_{L^\perp}) +2\int_{\T}\psi_\T(x,r)\tilde\rho_r(y-x;L)\abs{\pi_L\nabla u(y)}^2\notag\\
	\leq &\hat\vartheta(\T,r;\alpha_{L^\perp})+ \underbrace{\hat \vartheta(\T,r,L)}_{\geq 0}\, +2 \int_{\T}\psi_\T(x,r)\int \ps{\pi_L \nabla^{(y)}\tilde\rho_r(y-x;L)}{\nabla u(y)}\ps{\nabla u(y)}{(y-x)^\perp}
\end{align}
By noticing that $\pi_L\nabla^{(y)}\tilde\rho_r(y-x;L)=-\pi_L\nabla^{(x)}\tilde\rho_r(y-x;L)$, and integrating by parts wrt $x$, we obtain 
\begin{align}
 &\abs{\int_{\T}\psi_\T(x,r)\int \ps{\pi_L \nabla^{(y)}\tilde\rho_r(y-x;L)}{\nabla u(y)}\ps{\nabla u(y)}{(y-x)^\perp}}\notag \\
 =&\abs{\int_{\T}\int \tilde\rho_r(y-x;L)\ps{\pi_L \nabla^{(x)}\psi_\T(x,r)}{\nabla u(y)}\ps{\nabla u(y)}{(y-x)^\perp}}\notag\\
 \leq & C(m)\delta \int_\T r\abs{\nabla \psi_\T(x,r)}\, .
\end{align}
As we have $\tilde \rho \approx \rho$ (see Lemma \ref{l:radial_energy:stationary_vector_field} for the details), we obtain (roughly)
\begin{align}
 	\hat\vartheta(\T,r;n_{L^\perp}) +&2\vartheta(\T,r,L)\leq \hat\vartheta(\T,r;\alpha_{L^\perp})+\hat \vartheta(\T,r,L)+ C(m) \delta \int_\T r \abs{\nabla \psi_\T(x,r)}\, ,\notag\\
 \implies 	\hat\vartheta(\T,r;n_{L^\perp}) +&2\vartheta(\T,r,L)\leq \hat\vartheta(\T,r;\alpha_{L^\perp})+\hat \vartheta(\T,r,L)+ \epsilon(r)\, ,
\end{align}
where $\epsilon(r)\equiv C(m) \delta \int_\T r \abs{\nabla \psi_\T(x,r)}$ satisfies the estimate $\int \epsilon(r)\frac{dr}{r}\leq C(m)\delta$ .  In particular, by integrating the above wrt $\frac{dr}{r}$ and using the previously derived estimates on the angular and $L$-energies we arrive at our estimate
\begin{align}
	\int_{\cA\cap B_1}\langle\nabla u(y),n_{L^\perp}\rangle^2 &\leq\int_{}\,\hat\vartheta(\T,r;n_{L^\perp})\frac{dr}{r}\leq \int_{}\,\hat\vartheta(\T,r;\alpha_{L^\perp})\frac{dr}{r}+\int_{}\,\hat\vartheta(\T,r;L)\frac{dr}{r} +\int\epsilon(r)\frac{dr}{r}\, ,\notag\\
	&\leq C(m)\delta\, ,
\end{align}
which finishes the estimate for toy model \eqref{e:outline:toy_model:energy_decomposition} .

\vspace{.5cm}

\section{Outlining the General Case of Theorem \ref{t:outline:annular_regions_energy}}\label{s:outline_general}

The proof of Theorem \ref{t:outline:annular_regions_energy} for a general annular region $\cA=B_2\setminus \overline{B_{\rf_x}(\T)}$ will follow in broad strokes as the outline presented for the toy model in the last Section.  The energy will be decomposed into three pieces.  A log improved $L$-energy bound will follow by a careful application of summing the monotone quantity over $\T$.  The angular energy will be controlled by proving a superconvexity estimate on the annular region $\cA$, which itself will follow from a family of uniform subharmonic estimates.  Finally, the radial energy bound will follow by a carefully chosen stationary estimate that bounds it based on the angular energy.\\

However there are new complications that will arise.  
As one can imagine there is a variety of additional technical footwork necessary for the general case, not the least of which is that some of the involved errors are a bit circular, one needs to be quite careful in the application to tie them together correctly.  For the angular energy the subharmonicity will take a new form, as we will treat everything as a family of toy models.  For the radial energy there will be new errors due to the bending of $\T$ that are of strictly higher order than what is allowed.  To deal with the new errors we will build scale-wise approximations $\T_r$ of $\T$ that will {\it best} approximate the energy measure in some precise sense.
In particular the $\T_r$ will smoothly solve some resulting Euler-Lagrange equations.  We will define our energy decompositions in terms of $\T_r$, and the equations solved by $\T_r$ will correspond and hence cancel these highest order error terms that appear, at least up to lower order pieces. \\

In what remains of this outline we will describe the construction of these best approximating planes and submanifolds and see how to use them to correctly define the various energy pieces on an annular region.  For each component of the energy we will state and outline our main structural theorems, and then see how to use them to prove Theorem \ref{t:outline:annular_regions_energy}.  The remainder of the paper will then focus on proving these structural results for the best approximations $\T_r$ and the associated energies.

\vspace{.3cm}

\subsection{Local Best Planes and Approximating Submanifolds}

Recall that an annular region $\cA\subseteq B_2(p)$ comes equipped with a center submanifold $\T\subseteq B_2$ and a plane of symmetry $L_\cA$ such that for each $x\in \T$ the affine plane $x+L_\cA$ well approximates $\T$ on all scales.
Our next steps are to choose for each $x\in B_2$ with $3r\geq \rf_x\vee d(x,\T)=\max\{\rf_x,d(x,\T)\}$ \footnote{Recall from Remark \ref{r:annular_region:extension} that $\rf_x$ can be naturally extended to all of $B_2$} a {\it best} approximating plane $\LL_{x,r}$.
More accurately we have that $\LL_{x,r}$ will be the plane minimizing $\int \rho_r(y-x) \abs{\pi_{L}\nabla u(y)}^2$ among all $m-2$ planes $L$.  On the ball $B_{10r}(x)$ we will have that $u$ looks nearly invariant in the $\LL_{x,r}$ directions, and the energy $\vartheta(y,r)$ of $u$ looks like a Gaussian in the $\LL^\perp_{x,r}$ directions.
The goal will be to choose the submanifold $\T_r$ so that each $z\in \T_r$ is the center of this Gaussian, which is to say that for each $z\in\T_r$ we want $\vartheta(z,r)$ to obtain a maximum in the $\LL_{z,r}^\perp$ directions.  More slowly and precisely:\\

\subsubsection{\bf Local Best Planes}\label{sss:outline:general:local_best_planes}  Let $x\in B_2$ with $3r\geq \rf_x\vee d(x,\T)=\max\{\rf_x,d(x,\T)\}$, then we can define the $L$-energy of $u$ as
\begin{gather}\label{e:outline:L_energy}
\vartheta_\cL(x,r) \equiv \min_{L^{m-2}} \vartheta(x,r;L) = \min_{L^{m-2}} r^{2}\int_{} \rho_r(y-x)\abs{\pi_L\nabla u}^2\, .
\end{gather}
We can use this to define the best plane of symmetry as
\begin{gather}\label{e:outline:best_plane}
 \LL_{x,r}\equiv \arg\min_{L^{m-2}}\vartheta(x,r;L)=\arg\min_{L^{m-2}}\cur{L \to r^{2}\int_{} \rho_r(y-x)\abs{\pi_L\nabla u}^2}\, .
\end{gather}
It follows from the definition of an annular region that
\begin{align}
	\vartheta_\cL(x,r)=r^{2}\int_{} \rho_r(y-x)\abs{\pi_{x,r} \nabla u(y)}^2\leq C(m)\,\delta\, \, , \text{ where } \pi_{x,r}\equiv \pi_{\LL_{x,r}}\, .
\end{align}

The plane $\LL_{x,r}$ is uniquely defined, in fact we will see that it is quantitatively unique.  We will be particularly interested in the orthogonal projection map
\begin{align}\label{e:outline:best_plane_projection}
	\pi_{x,r}\equiv \pi_{\LL_{x,r}}:\dR^m\to \dR^m\, 
\end{align}
and its estimates. 

Our main estimates in Section \ref{s:best_planes} are the following:

\begin{theorem}[Best Plane Existence and Regularity, compare with Theorem \ref{t:best_plane:best_plane}]\label{t:outline:best_plane}
	Let $u:B_{10R}(p)\to N$ be a stationary harmonic map with $R^2\fint_{B_{10R}}|\nabla u|^2\leq \Lambda$ , and let $\cA=B_2\setminus \overline{B_{\rf_z}(\T)}$ be a $\delta$-annular region.  Let $x\in B_2$ with $3r\geq \rf_x\vee d(x,\T)$, then the following hold:
\begin{enumerate}
	\item The best plane $\LL_{x,r}$ as in \eqref{e:outline:best_plane} exists and is unique.  In fact for any $L^{m-2}\subseteq \dR^m$ we have
\begin{align}
	\vartheta(x,r;L)\geq \vartheta(x,r ;\cL_{x,r})+C(m,\epsilon_0)\,d_{\Gr}(L,\LL_{x,r})^2\, ,
\end{align}
where $\epsilon_0(N)$ is the $\epsilon$-regularity from theorem \ref{t:eps_reg} and $d_{\Gr}$ is the Grassmannian distance.
\item In the domain of definition the projection $\pi_{x,r}$ from \eqref{e:outline:best_plane_projection} is smooth and satisfies
\begin{align}
	  r\abs{ \frac{\partial}{\partial r}\pi_{x,r}}+r\abs{ \nabla \pi_{x,r}}+  r^2\abs{ \nabla^2 \pi_{x,r}}\leq C(m) \sqrt{\frac{\vartheta_\cL(x,2r)}{\vartheta(x,r)}}\leq C(m,\epsilon_0) \sqrt{\vartheta_\cL(x,2r)}\, .
\end{align}
\end{enumerate}
\end{theorem}

Here by $\abs{O}$ we mean the $L^\infty$ norm for the operator $O$, though any other equivalent norm like the Hilbert-Schmidt norm works equally well.

\vspace{.3cm}

\subsubsection{\bf Local Approximating Submanifolds $\T_r$} \label{sss:outline:general:local_approx_submanifold}

Recall that on each ball $B_r(z)$ with $z\in \T$ and $r\geq \rf_z$ we have that the energy measure $|\nabla u|^2dy$ is $\delta$-close to the $m-2$ Hausdorff measure $\vartheta(z,r)\cH^{m-2}_{L_\cA}$ , see Section \ref{ss:prelim:cone_splitting} for a quantitative review of this.
The goal now is to find a smooth submanifold $\T_r$ which best approximates the energy measure in some precise best sense.
The care in all of this is in the delicate nature of the estimates which will be involved.\\

Let us begin with a broad discussion where we compare to similar ``best'' approximations built in \cite{NV_RH}, \cite{ENV_QR}, \cite{N_Reif}.  In these references an approximating submanifold $T_r$ is built by gluing together best approximating affine subspaces $\cL_{z,r}$ along some covering of $\T$.  
Such a $T_r$ is even better than $\sqrt\delta$-close to the energy measure, as at each point $z\in \T$ it is essentially $\sqrt{\fint_{\T\cap B_r(z)}r\dot\vartheta}\leq \sqrt\delta$ close to the energy measure on $B_r(z)$.  
This degree of accuracy is square summable in scales, i.e. $\int_{\rf_z}^1 r\dot\vartheta(z,r)\frac{dr}{r}\leq \Lambda$, and this square summable error bound becomes the underlying control needed for the rectifiable behavior of the singular sets of harmonic maps\footnote{For those following very closely, in \cite{NV_RH} the estimate was $\int_{\rf_z}^1 r\dot\vartheta(z,r)\frac{dr}{r}\leq \delta$, not just $\Lambda$.  This is because of the more restrictive definition of an annular region used in that paper, which was used to close the rectifiable loop which appears.  
In the context of the Energy Identity we need the weaker conditions in Definition \ref{d:annular_region}, and thus apriori only have a bound on the Dini sum. }.  \\

This error control will not be sufficient in the context of the Energy Identity.  Instead of building an approximating submanifold by gluing together best approximating subspaces, we will build our approximating $\T_r$ in a more precise fashion.
Recall that for each $x\in B_r(\T)$ with $3r\geq d(x,\T)\vee \rf_x$ there exists a best $m-2$ approximating subspace $\LL_{x,r}$ .  The mollified energy $\vartheta(x,r)$ looks roughly like a two dimensional Gaussian which is constant in the directions of $\LL_{x,r}$, and Gaussian in the $\LL_{x,r}^\perp$-directions with a center which is close to $\T$.  
We want to pick $\T_r$ so that each point $z\in \T_r$ is the center of this Gaussian, which is to say that $z$ is the maximum of $\vartheta(x,r)$ on the affine plane $z+\LL_{z,r}$.  An implicit function theorem shows us such a submanifold will exist and be unique, and will share all the regularity properties one would expect from the submanifolds $T_r$ built by gluing together best subspaces.  
This maximal condition will also enforce an Euler-Lagrange equation, which will be the real key point in later estimates.  Summing up the basics of the construction:\\

\begin{theorem}[Approximating Submanifold]
Let $u:B_{10R}(p)\to N$ be a stationary harmonic map with $R^2\fint_{B_{10R}}|\nabla u|^2\leq \Lambda$ , and let $\cA=B_2\setminus \overline{B_{\rf_z}(\T)}$ be a $\delta$-annular region.  For each $r>0$ there exists a submanifold $\T_r\subseteq B_{3/2}$ such that the following hold:
\begin{enumerate}
	\item We can globally write $\T_r=\text{Graph}\{\ft_r:L_\cA\cap B_{3/2}\to L_\cA^\perp\}$ with $|\ft_r|+|\nabla \ft_r|+r|\nabla^2 \ft_r|\leq C(m)\sqrt\delta$
	\item If $x\in \T_r$ with $\rf_x\geq 2r$ then $\T_r\cap B_r(x)=\T\cap B_r(x)$ ,
	\item If $x\in \T_r$ with $\rf_x\leq r$ then the following equivalent conditions hold
\begin{align}
	 \pi_{\LL_{x,r}^\perp} \nabla \vartheta(x,r)=0 \qquad \Longleftrightarrow \qquad \forall v\in \LL_{x,r}^\perp: \ \ \int -\dot\rho_r \ton{y-x} \ps{\nabla u(y)}{y-x} \ps{\nabla u(y)}{v}dy=0\, .
\end{align} 
\end{enumerate}
\end{theorem}
\begin{remark}
	Theorem \ref{t:approximating_submanifold} states a more technically complete version of the above Theorem.
\end{remark}
\begin{remark}
	In the last equation the equality of the two conditions is not immediate, but it follows from the stationary equation applied to $u$ (see Section \ref{ss:T_r_Euler_Lagange} for the details).
	Note that $|\rho+\dot\rho|(t)<C(m)e^{-R/2}\rho(2t)$, as pointed out in Lemma \ref{l:rho_basic_properties}.
\end{remark}

\vspace{.3cm}

\subsection{\texorpdfstring{$L$}{L}-Energy on Annular Regions}\label{ss:outline:general:L_energy}

In the remaining subsections we will revisit the techniques introduced for the toy model of Section \ref{s:outline_toymodel}.  We will state precisely the correct decomposition of the energy we will use for the general case, exploiting the constructions of the last subsections, and state the main results controlling these pieces.  We will say a few words about the method of the proof, but then refer to the relevant Sections for the details.\\

We have defined $\vartheta_\cL(x,r) = \vartheta(x,r;\LL_{x,r})$ in \eqref{e:outline:L_energy}, which is the $L$-energy at the point $x\in B_2$ at scale $r>0$.  This best plane of symmetry may vary from point to point and scale to scale.  Control of $\vartheta_\cL(x,r)$ over a single point and scale may be obtained through the quantitative cone splitting.  Namely, if $B_r(x)\cap \T\neq \emptyset$ then Theorem \ref{t:prelim:cone_splitting_annular} implies the estimate:

\begin{align}\label{e:outline:general:L_pinching}
	\vartheta_\LL(x,r)&\leq C(m)\fint_{\T\cap B_{2r}(x)} r\dot \vartheta(z,2r)<C(m)\delta\, .
\end{align}

We wish to integrate $\vartheta_\LL(x,r)$ over $\T$ in order to define the $L$-energy of $\T$ at scale $r>0$.  However,  we need to be careful to only integrate on those $r\geq \rf_x$ which are in the scale of the annulus.  Let us introduce the relevant cutoff functions for this purpose:\\

{\bf The Cutoff $\psi_\T(x,r)$ : } We need to extend the cutoff of \eqref{e:outline:toy_model:psi_T} to a general annular region.  The following is shown in Section \ref{ss:bubble_cutoff}, and the construction is fairly explicit:

\begin{lemma}[$\T$-Cutoff]\label{l:outline_general:T_cutoff}
Let $u:B_{10R}(p)\to N$ be a stationary harmonic map with $R^2\fint_{B_{10R}}|\nabla u|^2\leq \Lambda$ , and let $\cA=B_2\setminus \overline{B_{\rf_z}(\T)}$ be a $\delta$-annular region.  Then $\exists$ a smooth function $\psi_\T:B_2\times \dR^+\to \dR^+$ such that
\begin{enumerate}
	\item $\psi_\T(x,r)=1$ if $\rf_x\leq r$ ,
	\item $\psi_\T(x,r)=0$ if $\rf_x>2r$ ,
	\item $\abs{r\frac{\partial}{\partial r} \psi} + r|\nabla \psi_\T|+r^2|\nabla^2 \psi_\T| < C(m)\,$
	\item $\int\int_{\T_r}r|\nabla \psi_\T|+r^2|\nabla^2 \psi_\T|\frac{dr}{r} < C(m)$
\end{enumerate}
\end{lemma}
\begin{remark}
	Though $\psi_\T(y,r)$ is defined on $B_2$ we will in practice restrict it to either $\T$ or $\T_r$.\\
\end{remark}

{\bf The $L$-Energy $\vartheta_\LL(\T_r)$ along $\T$: } Having defined a cutoff for the bubble center manifold $\T_r$ we can define

\begin{align}\label{e:outline:general:L_energy_T}
	\vartheta_\LL(\T_r)\equiv \int_{\T_r} \psi_\T(y,r)\vartheta_\LL(y,r)\, .
\end{align}

Note that as $\psi_\T(x,r)$ vanishes for those $x\in \T_r$ with $\rf_x>2r$ we are only integrating on the portion of $\T_r$ relevant for scale $r>0$.  In Lemma \ref{l:energy_decomposition:scale_r_properties} of Section \ref{s:energy_decomposition} we will study some of the fundamental properties of $\vartheta_\LL(\T_r)$ and prove:\\

\begin{lemma}[Properties of $\vartheta(\T_r\,;\cL)$]\label{l:outline:L_energy:basic_properties}
Let $u:B_{10R}(p)\to N$ be a stationary harmonic map with $R^2\fint_{B_{10R}}|\nabla u|^2\leq \Lambda$ , and let $\cA=B_2\setminus \overline{B_{\rf_z}(\T)}$ be a $\delta$-annular region.	 Then the following hold:
\begin{enumerate}
	\item For any $s\leq r\leq 1$ we have the almost monotonicity estimate 
\begin{align}\label{e:outline:general:L_energy_almost_mon}
	\vartheta_\LL(\T_s)\leq C(m)\,\vartheta_\LL(\T_r) \, .
\end{align}
\item For any $0<r<2$ we have the pinching estimate 
\begin{align}\label{e:outline:general:L_pinching_T}
	\vartheta_\LL(\T_r)\leq C(m,R)\int_\T |\vartheta(x,2r)-\vartheta(x,r)|\, .
\end{align}
\end{enumerate}
\end{lemma}
\begin{remark}
	Note in the second estimate the integral is over $\T$.  This is a technical convenience for estimates, the integral may equally well be taken over $\T_r$.
\end{remark}
\vspace{.3cm}

Note that the first estimate above has some mild subtlety to it.  The almost monotonicity of Lemma \ref{l:outline:L_energy:basic_properties}.(1) is not equation based but instead relies heavily on the $m-2$ dimensional nature of the bubble center manifold $\T$, and one should also recall that the best planes at scale $s$ are not those from scale $r$.\\

Integrating \eqref{e:outline:general:L_pinching_T} over scale and using \eqref{e:outline:general:L_energy_almost_mon} we obtain the log estimate 
\begin{align}
	|\ln r|\vartheta_\LL(\T_r) = \int_{r}^1 \frac{\vartheta_\LL(\T_r)}{s} &\stackrel{\eqref{e:outline:general:L_energy_almost_mon}}{\leq} C(m)\int_{r}^1 \frac{\vartheta_\LL(\T_s)}{s}\notag\\
&\stackrel{\eqref{e:outline:general:L_energy_almost_mon}}{\leq} C(m)\sum_{r\leq s_a=2^{-a}\leq 2}\vartheta_\LL(\T_{s_a}) \notag\\
	&\stackrel{\eqref{e:outline:general:L_pinching_T}}{\leq} C(m)\int_{\T}\sum_{\rf_x\leq s_a=2^{-a}\leq 2}|\vartheta_{}(x,2s_a)-\vartheta_{}(x,s_a)|\notag\\
	&\leq C(m)\int_{\T}|\vartheta(x,4)-\vartheta(x,\rf_x)|\, ,\notag\\
	&\leq C(m)\Lambda\, .
\end{align}

In particular, up to the proofs of Lemma \ref{l:outline_general:T_cutoff} and Lemma \ref{l:outline:L_energy:basic_properties}, we have the log decay estimates:\\

\begin{theorem}[Log Decay of $L$-Energy]\label{t:outline:L-energy_estimate}
Let $u:B_{10R}(p)\to N$ be a stationary harmonic map with \newline $R^2\fint_{B_{10R}}|\nabla u|^2\leq \Lambda$ , and let $\cA=B_2\setminus \overline{B_{\rf_z}(\T)}$ be a $\delta$-annular region.	 Then the following hold:
\begin{align}
	\vartheta(\T_r\,;\LL) &\leq C(m,R)\min\cur{\delta, \frac{\Lambda}{|\ln r|}}
\end{align} 
\end{theorem}
\begin{remark}
	It will be a consequence of the Energy Identity, and more precisely Theorem \ref{t:radial_energy:radial_energy}, that this estimate can be improved to $\vartheta(\T_r\,;\LL) \leq  C \frac{\epsilon}{|\ln r|}$
\end{remark}
\vspace{.3cm}

\vspace{.3cm}

\subsection{Angular Energy on Annular Regions}\label{ss:outline:general:angular_energy}

Let us now move on to analyzing the angular energy in an annular region.  As with the $L$-energy we now have that  the definition of the {\it angular component} of the energy is itself a function of point and scale.  A choice of angle depends not only on a best plane $\cL_{x,r}$ but also a center for this plane, which is to say a best affine plane.  For points $x\in \T_r$ the natural choice is $x$ itself.\\

In order to generalize the angular energy of Section \ref{ss:outline_toymodel:angular_energy}, let us choose a scale $r>0$ along with a point $x\in \T_r$.  Associated with this choice is a best affine plane $x+\cL_{x,r}$ from which we can construct an isometry $\dR^m\equiv \LL_{x,r}\oplus \LL_{x,r}^\perp$ under the affine assignment
\begin{align}
	(y_{x,r},y_{x,r}^\perp)\in \LL_{x,r}\oplus \LL_{x,r}^\perp\longrightarrow x+y_{x,r}+y_{x,r}^\perp\in \dR^m\, .
\end{align}
We can write $\cL^\perp_{x,r}$ in polar coordinates $(s_{x,r}^\perp,\alpha^\perp_{x,r})$, where $s_{x,r}^\perp=|\pi^\perp_{x,r}(y-x)|$.
Note then that a point $y\in \dR^m$ with $s_{x,r}^\perp=0$ must live on the affine plane $x+\cL_{x,r}$.
We can then define for $x\in \T_r$ the angular energy functional $\cEE(\cdot):B_2 \to \dR^+$ by 
\begin{align}
	\cE_\alpha(x,r,y) =&\cEE(y)=|s_{x,r}^\perp|^2\fint_{S^1} \langle \nabla u(y), \alpha^\perp_{x,r}\rangle^2d\alpha_{x,r}^\perp\, ,
\end{align}
where we have sloppily also denoted by $\alpha^\perp_{x,r}$ the unit angular vector field away from $x+\LL_{x,r}$.
By construction, the function $\cE_\alpha(x,r,\cdot)$ is rotationally invariant wrt the affine plane $x+\cL^\perp_{x,r}$.  \\

Our uniformly subharmonic estimate on $\cE_\alpha(x,r,y)$ will be written as an inequality on the $y$ coordinates for each $x\in \T_r$.  To state this, recall as in \eqref{e:outline:toy_model:conformal_laplacian} the definition of conformal Laplacian, which depends on an affine subspace $x+L^{m-2}\subseteq \dR^m$
\begin{align}\label{e:outline:conformal_laplacian}
	\bar\Delta_L \equiv \ton{\frac{\partial}{\partial\ln s_\perp}}^2+s_\perp^2\ton{\frac{\partial}{\partial\alpha_\perp}}^2+s_\perp^2\Delta_L\, .
\end{align}

Our first result is the following:\\

\begin{theorem}[Uniform Subharmonicity]\label{t:outline:uniform_subharmonic}
Let $u:B_{10R}(p)\to N$ be a stationary harmonic map with $R^2\fint_{B_{10R}}|\nabla u|^2\leq \Lambda$ , and let $\cA=B_2\setminus \overline{B_{\rf_z}(\T)}$ be a $\delta$-annular region.	 For each $x\in \T_r$ with $\rf_x\leq r$ and each $y\in B_{R}(x)$ such that $s^\perp_{x,r}=|y^\perp_{x,r}|=|\pi^\perp_{x,r}(y-x)|\in [e^{-R},R]r$ we have that
\begin{align}
	\bar\Delta_{x,r}\cEE=\bar\Delta_{\cL_{x,r}}\cEE \geq \big(2-C(m,R)\delta\,\big)\,\cEE\geq \cEE\, .
\end{align}
\end{theorem}
\vspace{.3cm}

The proof of the above will essentially reduce to the toy model case, as for each $x,r$ fixed we are treating $\cE_\alpha$ as a toy model in the $y$ variable.\\

\subsubsection{Superconvexity for $\hat\vartheta_\alpha(\T_r)$}

As in the toy model, our main application will be to prove angular energy estimates on the annular region.  In the general case the computations will be quite involved at moments, but we will outline the main points to take away here.  To make the estimates precise let us define for $x\in \T_r$ the reduced angular energy $\hat\vartheta_\alpha(x,r)$ as
\begin{align}
	\hat\vartheta_\alpha(x,r) =  \hat\vartheta(x,r;\alpha_{\cL_{x,r}^\perp}) \equiv \int \hat\rho(y-x;\cL_{x,r})|\pi_{x,r}^\perp(y-x)|^2 \langle \nabla u,\alpha^\perp_{x,r}\rangle^2\, ,
\end{align}
where $\alpha^\perp_{x,r}$ is the unit angular direction with respect to $x+\cL_{x,r}$, and $\pi_{x,r}=\pi_{\cL_{x,r}}$ is the projection to the best plane.  We can rewrite this as
\begin{align}
	\hat\vartheta_\alpha(x,r) =  \int \hat\rho(y-x;\cL_{x,r})\cEE(y)\, .
\end{align}
Note that the cutoff heat kernel $\hat\rho$ has support in the region controlled by Theorem \ref{t:outline:uniform_subharmonic}.  We have that $\hat\vartheta_\alpha(x,r)$ roughly measures the angular energy of $u$ on $B_r(x)$.  We can then define the scale $r$ angular energy over $\T$ by
\begin{align}
	\hat\vartheta_\alpha(\T_r) \equiv \int_{\T_r}\psi_\T(x,r)\,\hat\vartheta_\alpha(x,r) = \int_{\T_r}\psi_\T(x,r)\int \hat\rho_r(y-x; \LL_{x,r}) \cEE(y) \,dy\,dx \, .
\end{align}

The angular energy on the annular region can then be understood as
\begin{align}
	\hat\vartheta_\alpha(\cA )\equiv \int \hat\vartheta_\alpha(\T_r)\frac{dr}{r}\, .
\end{align}

We can give a poor estimate on $\hat\vartheta_\alpha(\cA)$ by applying the Quantitative Cone Splitting of Theorem \ref{t:prelim:cone_splitting_m-2}.  Namely, after some manipulation we arrive at the estimate
\begin{align}
	\hat\vartheta_\alpha(\cA)\leq C(m,R)\int_{\T}\big|\vartheta(x,2)-\vartheta(x,\rf_x)\big|\leq C(m,R)\Lambda \, .
\end{align}

However, our goal is to make our control over $\hat\vartheta_\alpha(\cA )$ small.
To accomplish this we will use the uniform subharmonicity of Theorem \ref{t:outline:uniform_subharmonic} in order to prove a subconvexity estimate on $\hat\vartheta_\alpha(\T_r)$ in line with that of the toy model.
As the submanifolds, best planes, etc... are all changing from scale to scale we inherently arrive at new error terms.
The end result of Section \ref{s:angular_energy} will be the following:

\begin{proposition}[Superconvexity for $\hat\vartheta_\alpha(\T_r)$]\label{p:outline:superconvexity} Let $u:B_{10R}(p)\to N$ be a stationary harmonic map with $R^2\fint_{B_{10R}}|\nabla u|^2\leq \Lambda$ , and let $\cA=B_2\setminus \overline{B_{\rf_z}(\T)}$ be a $\delta$-annular region.	 Then we have the estimate
\begin{align}
	\ton{r\frac{\partial}{\partial r}}^2\hat\vartheta_\alpha(\T_r) \geq \hat\vartheta_\alpha(\T_r)-\epsilon(r)\, ,
\end{align}
where 
\begin{align}
\epsilon(r)\leq C(m,R,K_N)\Bigg(\delta\int_\T \Big(r^2\big|\ddot\psi_\T\big|+r\big|\dot\psi_\T\big|+r\big|\nabla_L\psi_\T\big|+r^2\big|\nabla^2_L\psi_\T\big|\,\Big)+\sqrt{\delta}\int_{\T_r}r\dot \vartheta(x,r)	\Bigg)
\end{align}
satisfies $\int\epsilon(r)\frac{dr}{r}\leq C(m,R,K_N,\Lambda)\sqrt\delta$ .
\end{proposition}
\vspace{.3cm}

We can now apply the above directly in order to estimate the angular energy on an annular region, namely we can estimate

\begin{align}
	\vartheta_\alpha(\cA)&=
	\int \hat\vartheta(\T_r)\frac{dr}{r}\notag\\
	&\leq \int \Big(r\frac{\partial}{\partial r}\Big)^2\hat\vartheta_\alpha(\T_r)\frac{dr}{r}+\int\epsilon(r)\frac{dr}{r}\notag\\
	&=\int\epsilon(r)\frac{dr}{r}\notag\\
	&\leq C(m,R,K_N,\Lambda)\sqrt\delta\, .
\end{align}

Thus once we have proved the superconvexity of Proposition \ref{p:outline:superconvexity}, we have then proved the following, which is the main Theorem of Section \ref{s:angular_energy}:

\begin{theorem}[Angular Energy on Annular Regions]\label{t:outline_general:angular_energy}
Let $u:B_{10R}(p)\to N$ be a stationary harmonic map with $R^2\fint_{B_{10R}}|\nabla u|^2\leq \Lambda$ , and let $\cA=B_2\setminus \overline{B_{\rf_z}(\T)}$ be a $\delta$-annular region.	 Then
\begin{align}
	\hat\vartheta_\alpha(\cA) \leq C(m,R,K_N,\Lambda)\,\sqrt\delta\, .
\end{align} 	
\end{theorem}

\vspace{.3cm}

\subsection{Radial Energy on Annular Regions}\label{ss:outline:general:radial_energy}

Let us now discuss the radial energy on an annular region.  As with the toy model, the strategy will be to use a stationary equation to bound the radial energy by the angular energy.  It is in this step that the careful definition of $\T_r$ is required, as some of the new errors which appear are of strictly higher order and cannot be freely estimated away.\\

Let us define the radial energy at a point $x\in \T_r$ and scale $r>0$.  As with the angular energy we rely on the restricted energies, as for a fixed scale we should not attempt to measure the radial energy too close to the bubble region.  In analogy with the previous subsections we define
\begin{align}
	\hat\vartheta_n(x,r)&\equiv \int \hat\rho_r(y-x;\LL_{x,r})\big\langle\nabla u,\pi^\perp_{x,r}(y-x)\big\rangle^2\, ,
\end{align} 
where $\hat\rho_r$ is the heat kernel mollifier cutoff in the $e^{-R}$ neighborhood of $L$ as in \eqref{e:prelim:L_mollifier}.  We can integrate this over $\T_r$ with our bubble center weight $\psi_\T$ from Lemma \ref{l:outline_general:T_cutoff} to define the radial energy on $\T_r$ :
\begin{align}\label{e:outline:general:radial_energy}
	\hat\vartheta_n(\T_r)\equiv \int_{\T_r}\psi_\T(x,r)\, \hat\vartheta_n(x,r) = \int_{\T_r}\psi_\T(x,r)\int\hat\rho_r(y-x;\LL_{x,r})\langle\nabla u,\pi^\perp_{x,r}(y-x)\rangle^2\, .
\end{align}

In Section \ref{s:radial_energy} our goal will be to control the radial energy over the annular region, given as usual by
\begin{align}
	\hat\vartheta_n(\cA)\equiv \int \hat\vartheta_n(\T_r) \frac{dr}{r}\, .
\end{align}

This control will be obtained by studying the stationary equation induced by the vector field
\begin{align}
	  \xi(y) = \int_{\T_r }  \psi_{\T}(x,r) \tilde\rho_{r}(y-x;\LL_{x,r})\pi^\perp_{x,r}(y-x)\, ,
\end{align}
where recall from Section \ref{ss:restricted_energy} that $\tilde\rho_r\approx \rho$ is defined by the property that its perpendicular gradient $\nabla_{\LL_{x,r}^\perp}\tilde\rho_r=0$ vanishes in a small neighborhood of the bubble region.  One should interpret the above as the appropriate approximation of the radial vector field emanating from $\T_r$.  It is worth emphasizing that this choice of vector field is unusually delicate, other reasonable approximations of the radial vector field from $\T_r$, e.g. those obtained by pasting together local best approximations, will have higher order terms which fail to estimate correctly.\\

Our main estimate in Section \ref{s:radial_energy} will be Proposition \ref{p:radial_energy:radial_energy}, which will tell us for $\delta\leq \delta(m,K_N,R,\Lambda,\epsilon)$ that we can estimate:

\begin{proposition}[Stationary Estimate of the Radial Energy]
	Let $u:B_{10R}(p)\to N$ be a stationary harmonic map with $R^2\fint_{B_{10R}}|\nabla u|^2\leq \Lambda$ , and let $\cA=B_2\setminus \overline{B_{\rf_z}(\T)}$ be a $\delta$-annular region.   Then we can estimate
\begin{align}
\hat\vartheta_n(\T_r) +2\vartheta_\cL(\T_r)	= r\,\dot\vartheta_\cL(\T_r)+\hat\vartheta_\alpha(\T_r)+\epsilon(r)\, ,
\end{align}
where $\epsilon(r)$ satisfies $\int \epsilon(r)\frac{dr}{r}\leq \epsilon$ when $\delta\leq \delta(m,n,K_N,\Lambda,R,\epsilon)$.
\end{proposition}

It will take the majority of Section \ref{s:radial_energy} to prove the above.  Both the Euler-Lagrange for $\T_r$ and the Euler-Lagrange for the best plane $\cL_{x,r}$ are used to remove the worst error terms which appear and provide the smallness on Dini error.
  Integrating this over $r$ and using our estimates on $\vartheta_\cL(\T_r)$ and $\vartheta_\alpha(\T_r)$ we obtain our desired estimate on the radial energy:\\

\begin{theorem}[Radial Energy Bound]\label{t:outline:general:radial_energy}
Let $u:B_{10R}(p)\to N$ be a stationary harmonic map with \newline $R^2\fint_{B_{10R}}|\nabla u|^2\leq \Lambda$ , and let $\cA=B_2\setminus \overline{B_{\rf_z}(\T)}$ be a $\delta$-annular region.  Then for each $\epsilon>0$ if $R\geq R(m,\epsilon)$ and $\delta\leq \delta(m,K_N,R,\Lambda,\epsilon)$ we have that
\begin{align}
\hat\vartheta_n(\cA) \leq  \epsilon\, .
\end{align}
\end{theorem}

\vspace{.5cm}

\section{Local Best Planes \texorpdfstring{$\cL_{x,r}$}{Lxr}}\label{s:best_planes}

This Section is dedicated to studying and analyzing our local best planes $\cL_{x,r}$.  Recall that the setup is that $\cA=B_2\setminus \overline{B_{\rf_x}(\T)}$ is an annular region.  For each $x\in B_2$ and $3r\geq \rf_x\vee d(x,\T)=\max\{\rf_x,d(x,\T)\}$ we have that the energy in the $L_\cA$ direction is $\delta$-small:
\begin{align}
	\vartheta(x,r;L_\cA) \equiv r^2\int \rho_r(y-x)\abs{\pi_{\LL_\cA}\nabla u(y)}^2 \leq C(m)\delta\, .
\end{align}

The reasonable question to study is what the minimal $L$-energy of $u$ is on $B_r(x)$, which might be much smaller than $\delta$.  We defined this in Section \ref{sss:outline:general:local_best_planes} as

\begin{gather}\label{e:best_planes:L_energy}
\vartheta_\cL(x,r) \equiv \min_{L^{m-2}} \vartheta(x,r;L) = \min_{L^{m-2}} r^{2}\int_{} \rho_r(y-x)\abs{\pi_L\nabla u}^2\, ,
\end{gather}
and used this to make careful sense of a best plane of symmetry as
\begin{gather}\label{e:best_plane:best_plane}
 \LL_{x,r}\equiv \arg\min_{L^{m-2}}\vartheta(x,r;L)=\arg\min_{L^{m-2}}\cur{L \to r^{2}\int_{} \rho_r(y-x)\abs{\pi_L\nabla u}^2}\, .
\end{gather}
It follows from the definition of an annular region that
\begin{align}
	\vartheta_\cL(x,r)=r^{2}\int_{} \rho_r(y-x)\abs{\pi_{x,r}\nabla u(y)}^2\leq C(m)\,\delta\, .
\end{align}

Finally recall from Section \ref{sss:outline:general:local_best_planes} that we are particularly interested in the projection map
\begin{align}\label{e:best_plane:best_plane_projection}
	\pi_{x,r}\equiv \pi_{\LL_{x,r}}:B_2\times \dR^m\to \dR^m\, ,
\end{align}
to the best plane at $x\in B_2$.  The goal of this Section is to prove the following:

\begin{theorem}[Best Plane Existence and Regularity]\label{t:best_plane:best_plane}
	Let $u:B_{10R}(p)\to N$ be a stationary harmonic map with $R^2\fint_{B_{10R}}|\nabla u|^2\leq \Lambda$ , and let $\cA=B_2\setminus \overline{B_{\rf_z}(\T)}$ be a $\delta$-annular region.  Let $x\in B_2$ with $3r\geq \rf_x\vee d(x,\T)$, then the following hold:
\begin{enumerate}
	\item The best plane $\LL_{x,r}$ as in \eqref{e:outline:best_plane} exists and is unique, in fact for any $L^{m-2}\subseteq \dR^m$ we have
\begin{align}
	\vartheta(x,r;L)\geq \vartheta(x,r;\cL_{x,r})+c(m)\vartheta(x,r)\,d_{\Gr}(L,\LL_{x,r})^2\geq \vartheta(x,r;\cL_{x,r})+c(m)\epsilon_0\,d_{\Gr}(L,\LL_{x,r})^2\, ,
\end{align}
where $\epsilon_0$ is the $\epsilon$-regularity constant from Theorem \ref{t:eps_reg} and $d_{\Gr}$ is the Grassmannian distance.
\item \label{i:best_plane_bounds_on_projections} In the domain of definition the projection $\pi_{x,r}$ from \eqref{e:outline:best_plane_projection} is smooth and satisfies
\begin{align}
	  r\abs{ \frac{\partial}{\partial r}\pi_{x,r}}+r^2\abs{ \frac{\partial}{\partial r}\nabla \pi_{x,r}}+r\abs{ \nabla \pi_{x,r}}+  r^2\abs{ \nabla^2 \pi_{x,r}}\leq C(m) \sqrt{\frac{\vartheta_\cL(x,2r)}{\vartheta(x,r)}}\leq C(m,\epsilon_0) \sqrt{\vartheta_\cL(x,2r)}\, ,
\end{align}
where $\abs{O}$ is the operator norm of $O$, or any other equivalent norm.
\end{enumerate}
\end{theorem}
\begin{remark}\label{r:d_LA_LL_small}
	As a corollary of (1) above we have that $  d_G(\LL_{x,r},L_{\cA})\leq C\sqrt{\frac{\delta}{\epsilon_0}}\leq C(m,\epsilon_0)\sqrt{\delta}\, $ .
\end{remark}
\begin{remark}\label{r:d_LLr_LL2r_small}
 Another easy corollary is that if $x\in B_2 $ with $3r\geq \rf_x\vee d(x,\T)$:
 \begin{gather}
  d_{\Gr}(\LL_{x,r},\LL_{x,2r})^2\leq C(m) \frac{\vartheta_\LL(x,2r)}{\vartheta(x,r)}\, .
 \end{gather}
 In particular, these all imply that
 \begin{gather}
  \vartheta_\LL(x,r)=\vartheta(x,r;\LL_{x,r})\leq C(m) \vartheta(x,r;\LL_{x,2r})\leq C(m) \vartheta_\LL(x,2r)\, ,\notag\\
  \vartheta(x,2r;\LL_{x,r})\leq C(m) \vartheta(x,2;\LL_{x,2r})=C(m) \vartheta_\LL(x,2r)\, .\label{e:vartheta_L_r}
 \end{gather}

\end{remark}

The remaining subsections are dedicated to setting up the necessary structure for the Theorem, and then proving it piece by piece.

\vspace{.3cm}

\subsection{Heat Mollified Energy Tensor}

A useful tool in our discussions will be to consider not just the heat mollified energy $\vartheta(x,r) \equiv r^2 \int \rho_r(y-x)|\nabla u|^2$, but the heat mollified energy tensor
\begin{align}
	Q(x,r)[v,w]\equiv r^2 \int \rho_r(y-x)\langle \nabla u,v \rangle\langle \nabla u,w \rangle\, .
\end{align}

We will typically view $Q(x,r)\in \dR^{m\times m}$ as a symmetric nonnegative $m\times m$ matrix.  Note the identity
\begin{align}
	\operatorname{tr} Q(x,r) = \vartheta(x,r)\, .
\end{align}

At each point $x\in B_2$ and scale $3r\geq \rf_x\vee d(x,\T)$ it is reasonable to look at the eigenvalues $Q(x,r)$, which we will denote by
\begin{align}
	\lambda_1(x,r)\geq \cdots \geq \lambda_m(x,r)\, ,
\end{align}
together with the associated eigenvectors $e_j(x,r)$ .  Our goal for this subsection is the following

\begin{lemma}
	Let $u:B_{10R}(p)\to N$ be a stationary harmonic map with $R^2\fint_{B_{10R}}|\nabla u|^2\leq \Lambda$ , and let $\cA=B_2\setminus \overline{B_{\rf_z}(\T)}$ be a $\delta$-annular region.  Let $x\in B_2$ with $3r\geq \rf_x\vee d(x,\T)$, then the following hold:
\begin{enumerate}
	\item The eigenvalues $\lambda_1\geq \cdots \geq \lambda_m$ of $Q(x,r)$ satisfy 
\begin{align}\label{eq_lambda_sep}
	 \frac{1}2 +C(m) \sqrt{\frac{\delta}{\vartheta(x,r)}}\geq \frac{\lambda_1(x,r)}{\vartheta(x,r)}\geq \frac{\lambda_2(x,r)}{\vartheta(x,r)}\geq \frac 1 2 - C(m)\sqrt{\frac{\delta}{\vartheta(x,r)}}\, , \ \ \text{ with }\lambda_3(x,r)\leq C(m)\delta\, .
\end{align}
\item The unique best subspace $\cL_{x,r}$ is characterized by $\LL_{x,r}= \text{span}\{e_3,\ldots,e_m\}$
\end{enumerate}
\end{lemma}

\begin{proof}

Note that $(2)$ follows from the usual Rayleigh characterization of the eigenvalues.  To prove the $\lambda_3$ estimate from $(1)$ is therefore really an $L$-gradient bound on the region.  This itself is a consequence of the definition of annular region and effective cone splitting. By Definition \ref{d:annular_region}, and since $x\in B_2$ with $3r\geq \rf_x\vee d(x,\T)$, there exists $\cur{x_i}_{i=0}^{m-2}\subset \T \cap \B{r/20}{x}$ that are $\frac{r}{100}$-linearly independent and such that $r\dot\vartheta(x_i,2r)\leq \delta$. Let $L$ be the linear subspace spanned by $\cur{x_i-x_0}_{i=1}^{m-2}$, by Theorem \ref{t:prelim:cone_splitting} we obtain
\begin{gather}
 \sum_{k=3}^m \lambda_k(x,r)=\vartheta(x,r;\LL)\leq \vartheta(x,r;L)\leq C(m) \vartheta(x_0,2r;L)\leq C(m)\delta\, .
\end{gather}

To finish the proof we need to prove the gap bound on $\lambda_1$, $\lambda_2$ in $(1)$.  Let us begin with the following claim:\\

\textbf{Claim:} in order to conclude the eigenvalue separation, we claim that for all unit vectors $e\in L^\perp$ we have
\begin{gather}\label{e:claim_e_separation}
 \abs{r^2\int \rho_r(y-x) \ps{\nabla u}{e}^2 -\frac 1 2 \vartheta(x,r)} \leq C(m) \sqrt{\vartheta(x,r)\delta}
\end{gather}

First let us show that this claim is enough to conclude the proof. Indeed, set
\begin{gather}
 e_1=\mu_1 e + \mu_2 \ell\, ,\qquad \mu_1^2+\mu_2^2=1\, ,
\end{gather}
where $\ell\in L$ is a unit vector. Then
\begin{align}
 \lambda_1(x,r)&= r^2 \int \rho_r(y-x) \abs{\mu_1 \ps{\nabla u}{e} + \mu_2 \ps{\nabla u}{\ell}}^2\notag\\
 &=r^2 \int \rho_r(y-x) \qua{\mu_1^2 \ps{\nabla u}{e}^2 + \mu_2^2 \ps{\nabla u}{\ell}^2 +2\mu_1\mu_2 \ps{\nabla u}{e}\ps{\nabla u}{\ell}}\notag\\
 &\leq r^2 \int \rho_r(y-x) (\mu_1^2+\mu_2^2) \ps{\nabla u}{e}^2 + r^2 \int \rho_r(y-x) (\mu_1^2+\mu_2^2) \ps{\nabla u}{\ell}^2\notag\\
 & \leq \frac 1 2 \vartheta(x,r) +C(m)\sqrt{\vartheta(x,r)\delta} + C(m)\delta\leq \frac 1 2 \vartheta(x,r) +C(m)\sqrt{\vartheta(x,r)\delta}\, .
\end{align}
Since
\begin{gather}
 \lambda_2(x,r)= \tr Q -\lambda_1(x,r)-\sum_{j=3}^m \lambda_j(x,r)\geq \vartheta(x,r)-\lambda_1(x,r)-C(m)\delta\, ,
\end{gather}
we obtain \eqref{eq_lambda_sep}.

\textbf{Proof of the claim.} We conclude the proof by showing that \eqref{e:claim_e_separation} holds. Consider the stationary equation \eqref{e:stationary_equation} applied to the vector field
\begin{gather}
 \xi^j = r^2\rho_r(y-x) \ps{y-x_0}{e}e^j\, .
\end{gather}
This gives
\begin{align}
 &\vartheta(x,r)-2r^2\int \rho_r(y-x) \ps{\nabla u}{e}^2 \notag\\
 =&-\int \dot \rho_r(y-x)\abs{\nabla u}^2 \ps{y-x}{e}\ps{y-x_0}{e} +2 \int \dot \rho_r(y-x)\ps{\nabla u}{y-x}\ps{\nabla u}{e}\ps{y-x_0}{e}\, ,
 \end{align}
 and thus:
 \begin{align}
 &\abs{\vartheta(x,r)-2\int \rho_r(y-x) \ps{\nabla u}{e}^2}\notag\\
 \leq& C \sqrt{-\int \dot \rho_r(y-x)\abs{\nabla u}^2 \abs{y-x}^2}\sqrt{-\int \dot \rho_r(y-x)\abs{\nabla u}^2 d(y,x_0+L)^2}\, .
\end{align}

By \eqref{e:trho_doubleradius} and using Theorem \ref{t:prelim:cone_splitting_m-2} to choose $x_0$ we can bound:
\begin{gather}
 -\int \dot \rho_r(y-x)\abs{\nabla u}^2 d(y,x_0+L)^2\leq C(m)\delta\, .
\end{gather}
Moreover, by \eqref{e:trho_doubleradius},
\begin{gather}
 -\int \dot \rho_r(y-x)\abs{\nabla u}^2 \abs{y-x}^2 \leq\vartheta(x,2r)\, ,
\end{gather}
which together with Remark \ref{rm:vartheta_comparison} concludes the proof.
\end{proof}
\vspace{.3cm}

The ideas presented in the proof of this lemma are the main ones necessary to obtain the results of Theorem \ref{t:best_plane:best_plane}. We split the remaining details of the proof of Theorem \ref{t:best_plane:best_plane} into the next two subsections.\\
\subsection{Quantitative Uniqueness and Proof of Theorem \ref{t:best_plane:best_plane}.(1)}

Before moving to the proof, let us briefly recall that the Grassmannian $\Gr(m,k)$ is the set of all $k$-dimensional subspaces in $\R^m$, and this can be naturally identified with
 \begin{itemize}
  \item the set of $m-k$ dimensional subspaces of $\R^m$, through the orthogonal complementation
  \item the set of orthogonal projections $ P(m,k)=\cur{\pi:\R^m\to \R^m\, , \ \ \pi=\pi^T\, , \ \ \pi^2=1\, , \ \ \operatorname{rank}(\pi)=k}\, .$
 \end{itemize}
$\Gr(m,k)$ set can be naturally endowed with a smooth manifold structure, and up to a constant the Riemannian distance $d_{\Gr}(\cdot,\cdot)$ on $\Gr(m,k)$ can be controlled by
\begin{gather}
 C(m)^{-1} d_H\ton{\overline{V\cap \B 1 0},\overline{W\cap \B 1 0}}\leq d_G(V,W)\leq C(m) d_H\ton{\overline{V\cap \B 1 0},\overline{W\cap \B 1 0}}\, ,
\end{gather}
where $d_H$ is the Hausdorff distance on sets.  Equivalently we can bound
\begin{gather}\label{eq_epsproj}
 C(m)^{-1} \abs{\pi_V-\pi_W}\leq d_G(V,W)\leq C(m) \abs{\pi_V-\pi_W}\, ,
\end{gather}
where the norm on projections is the operator norm, or any other norm equivalent to it with an equivalence constant bounded by $C(m)$.

Our main application of our analysis on $Q(x,r)$ in the last subsection is the following, which in particular implies Theorem \ref{t:best_plane:best_plane}.(1):

\begin{lemma}
		Let $u:B_{10R}(p)\to N$ be a stationary harmonic map with $R^2\fint_{B_{10R}}|\nabla u|^2\leq \Lambda$ , and let $\cA=B_2\setminus \overline{B_{\rf_z}(\T)}$ be a $\delta$-annular region.  Let $x\in B_2$ with $3r\geq \rf_x\vee d(x,\T)$, then
\begin{align}\label{e:distance_L_LL}
	\vartheta(x,r;L)\geq \vartheta(x,r;\LL_{x,r})+C(m)\vartheta(x,r)\,d_{\Gr}(L,\LL_{x,r})^2\, .
\end{align}
\end{lemma}
\begin{proof}
Due to the eigenvalue separation \eqref{eq_lambda_sep}, we are only concerned with the case that $L$ is close to $\LL_{x,r}$, and in particular the case when $L$ is a graph over $\LL_{x,r}$.  Thus let $L$ be the graph of the linear function $F:\LL_{x,r}\to \LL_{x,r}^\perp$, with $F(e_\alpha(x,r))=\mu_\alpha e_{1}(x,r) +\eta_{\alpha} e_{2}(x,r)$ for $\alpha=3,\cdots,m$. By standard estimates,
 \begin{gather}
  C(m)^{-1}d_{\Gr}(L,\LL_{x,r}) \leq \ton{\sum_{\alpha} \mu_\alpha^2 + \eta_\alpha^2}^{1/2}\leq C(m)d_{\Gr}(L,\LL_{x,r})\, .
 \end{gather}
 
 The eigenvalue conditions on the tensor $Q(x,r)$ give
 \begin{align}
  \vartheta(x,r;L) &= \sum_{\alpha=3}^{m} (1-\mu_\alpha^2-\eta_{\alpha}^2) \lambda_\alpha(x,r) + \sum_{\alpha=3}^{m} \ton{\mu_\alpha^2 \lambda_{1}(x,r)+ \eta_\alpha^2\lambda_2(x,r)}\\
  &=\theta(x,r;\LL_{x,r})  + \sum_{\alpha=3}^m \qua{\mu_\alpha^2 (\lambda_1(x,r)-\lambda_{\alpha}(x,r))+\eta_\alpha^2 (\lambda_2(x,r)-\lambda_{\alpha}(x,r))}\\
  & \stackrel{\eqref{eq_lambda_sep}}{\geq} \theta(x,r;\LL_{x,r})  + \ton{\frac 1 2 \vartheta(x,r)-C(m)\sqrt{\vartheta(x,r)\delta} -C(m)\delta}\sum_{\alpha=3}^m \ton{\mu_\alpha^2 +\mu_\alpha^2 }\, .
 \end{align}
Since $\vartheta(x,r)\geq C(m) \epsilon_0$, we obtain the result as long as $\delta<<\epsilon_0$.
\end{proof}

As an immediate corollary, we obtain
\begin{corollary}
 Let $u:B_{10R}(p)\to N$ be a stationary harmonic map with $R^2\fint_{B_{10R}}|\nabla u|^2\leq \Lambda$ , and let $\cA=B_2\setminus \overline{B_{\rf_z}(\T)}$ be a $\delta$-annular region.  Let $x,z\in B_2$ with $3r\geq \rf_x\vee \rf_z\vee d(x,\T)$. Moreover assume that $r\leq 10s$ and $\B s z \subset \B r x$. Then
 \begin{gather}\label{e:distance_LL_LL_sr}
  d_{\Gr}(\LL_{x,r},\LL_{z,s})^2\leq C(m)\frac{\vartheta_{\LL}(x,r)}{\vartheta(x,r)}\, .
 \end{gather}
 In particular,
 \begin{gather}\label{e:L_energy_r-2r}
  \vartheta(x,2r;\LL_{x,r})\leq C(m) \vartheta(x,2r;\LL_{x,2r})= C(m) \vartheta_{\LL}(x,2r)\, .
 \end{gather}

\end{corollary}
\begin{proof}
 Since $\B s z \subset \B r x$ with $s\geq r/10$, and by the definition of $\LL_{z,s}$, we have
 \begin{gather}
  \vartheta_\LL(z,s)= s^2 \int \rho_s (y-z) \abs{\pi_{z,s}\nabla u(y)}^2\leq s^2 \int \rho_s (y-z) \abs{\pi_{x,r}\nabla u(y)}^2\leq C(m) \vartheta_\LL(x,r)\, .
 \end{gather}
\eqref{e:distance_LL_LL_sr} follows from \eqref{e:distance_L_LL}, and \eqref{e:L_energy_r-2r} is an easy consequence of this and Remark \ref{rm:vartheta_comparison}.
\end{proof}

\vspace{.3cm}

\subsection{Estimates on the Best Plane Projection Map and Proof of Theorem \ref{t:best_plane:best_plane}.(2)}

The following summarizes our basic regularity properties of the best plane projection map $\pi_{x,r}$, and in particular proves Theorem \ref{t:best_plane:best_plane}.(2) .

\begin{lemma}
 Let $u:B_{10R}(p)\to N$ be a stationary harmonic map with $R^2\fint_{B_{10R}}|\nabla u|^2\leq \Lambda$ , and let $\cA=B_2\setminus \overline{B_{\rf_z}(\T)}$ be a $\delta$-annular region.  Let $x\in B_2$ with $3r\geq \rf_x\vee d(x,\T)$, then $\pi_{x,r}=\pi_{\LL_{x,r}}$ is a smooth function on its domain of definition with
 \begin{gather}\label{eq_LL_basic}
  \abs{r \frac{\partial}{\partial r}\pi_{x,r}}_{\infty}+\abs{r^2 \frac{\partial^2}{\partial r^2}\pi_{x,r}}_{\infty}+\abs{r \nabla \pi_{x,r}}_{\infty}+  \abs{r^2 \nabla^2 \pi_{x,r}}_{\infty}\leq C(m) \sqrt{\frac{\vartheta(x,2r;\LL_{x,2r})}{\vartheta(x,r)}}\leq C(m,\epsilon_0) \sqrt{\TLm(x,2r)}\, .
 \end{gather}
\end{lemma}
\begin{proof}
 By the eigenvalue separation property \eqref{eq_lambda_sep}, the best plane $\LL_{x,r}$ and so the projection onto it is well-defined, and it depends smoothly on $x$ and $r$. To obtain quantitative estimates, we will focus on the identity
 \begin{gather}\label{e:Q_best_plane}
  Q(x,r)[\pi_{x,r}(v),\pi_{x,r}^\perp(w)]=0
 \end{gather}
 valid for all vectors $v,w$.  We will focus ourselves on proving the estimates for the radial derivative $ r\abs{ \frac{\partial}{\partial r}\pi_{x,r}}_{\infty}$, the other estimates are the verbatim strategy.

 \paragraph{Proof of the radial derivative estimates.} We observe that by taking the radial derivative of \eqref{e:Q_best_plane} we obtain for all norm one vectors $v,w$:
 \begin{align}\label{eq_LL_basic2}
  0=r\frac{\partial}{\partial r} \qua{Q(x,r)[\pi_{x,r}(v),\pi_{x,r}^\perp(w)]}&= \underbrace{(2-m)\qua{Q(x,r)[\pi_{x,r}(v),\pi_{x,r}^\perp(w)]}}_{=0}\\
  &-r^2 \int \dot\rho_r(y-x)\frac{\abs{y-x}^2}{r^2}\langle \nabla u,\pi_{x,r}(v) \rangle\langle \nabla u,\pi_{x,r}^\perp(w) \rangle \\
  &+Q(x,r) \qua{r\dot \pi_{x,r}(v),\pi_{x,r}^\perp(w)} +Q(x,r) \qua{\pi_{x,r}(v),-r\dot \pi_{x,r}(w)}\, ,
 \end{align}
where in the last line we used the fact that $\pi_{x,r}+\pi_{x,r}^\perp = \operatorname{id}$ for all $x,r$, and so $\dot \pi_{x,r}^\perp=-\dot \pi_{x,r}\, $.\\

Using Cauchy-Schwartz, \eqref{e:trho_doubleradius} and \eqref{e:L_energy_r-2r}, we can bound
\begin{gather}
 \abs{r^2 \int \dot\rho_r(y-x)\frac{\abs{y-x}^2}{r^2}\langle \nabla u,\pi_{x,r}(v) \rangle\langle \nabla u,\pi_{x,r}^\perp(w) \rangle }\leq C(m)\sqrt{\vartheta_{\LL}(x,2r)} \sqrt{\vartheta(x,2r)}\, .
\end{gather}

On the other hand, $\dot \pi_{x,r}$ is a symmetric operator such that
\begin{gather}
 \dot \pi_{x,r} :\LL_{x,r}\to \LL_{x,r}^\perp\, , \qquad \dot \pi_{x,r} :\LL_{x,r}^\perp\to \LL_{x,r}\, .
\end{gather}
Let $v\in \LL_{x,r}$, $w\in \LL_{x,r}^\perp$ be unit vectors satisfying
\begin{gather}
 \abs{\dot \pi_{x,r}(v)}= \abs{\dot \pi_{x,r}}\, , \qquad w= \frac{\dot \pi_{x,r}(v)}{\abs{\dot \pi_{x,r}(v)}}\, .
\end{gather}
This is possible if $\abs{\dot \pi_{x,r}}$ is the operator norm of $\dot \pi_{x,r}$. Other norms can be considered, up to adding a suitable multiplicative constant. Then we obtain:
\begin{gather}
 Q(x,r) \qua{r\dot \pi_{x,r}(v),\pi_{x,r}^\perp(w)} +Q(x,r) \qua{\pi_{x,r}(v),-r\dot \pi_{x,r}(w)}\geq \ton{\lambda_2(x,r)-\lambda_3(x,r)} \abs{r\dot \pi_{x,r}}_\infty\, .
\end{gather}
By the eigenvalue separation \eqref{eq_lambda_sep}, we can plug all of this into \eqref{eq_LL_basic2} to get
\begin{gather}
 \abs{r \frac{\partial}{\partial r}\pi_{x,r}}_{\infty}\leq C(m) \sqrt{\vartheta_\LL(x,2r)}\frac{\sqrt{\vartheta(x,2r)}}{\vartheta(x,r)}\, ,
\end{gather}
and the conclusion follows from Remark \ref{rm:vartheta_comparison}.

\end{proof}

\vspace{.5cm}

\subsection{Improved cone-splitting}

With the tools now at our disposal, we can improve slightly Theorem \ref{t:prelim:cone_splitting_annular} by replacing a plane with the best plane. In order to do so, we need this lemma on pinching and distances.

\begin{lemma}
 Let $u:B_{10R}(p)\to N$ be a stationary harmonic map with $R^2\fint_{B_{10R}}|\nabla u|^2\leq \Lambda$ , and let $\cA=B_2\setminus \overline{B_{\rf_z}(\T)}$ be a $\delta$-annular region.  Let $3r\geq \rf_x\vee d(x,\T)$ and $\abs{x-x'}\leq r/10$, then
 \begin{gather}\label{e:distance_x_x'}
  \frac{\abs{\pi_{x,r}^\perp (x'-x)}^2}{r^2} \leq \frac{C(m)}{\vartheta(x,r)}\ton{r\dot\vartheta(x,3r/2)+r\dot \vartheta(x',3r/2)}\, .
 \end{gather}
\end{lemma}
\begin{proof}
 Set for convenience
 \begin{gather}
  v= x'-x\, .
 \end{gather}
  By definition of $\LL_{x,r}$ and by the eigenvalue separation of \eqref{eq_lambda_sep}:
 \begin{gather}
  r^2\int \rho_r(y-x) \ps{\nabla u}{v}^2=r^2\int \rho_r(y-x) \qua{\ps{\nabla u}{\pi_{x,r}(v)}^2+\ps{\nabla u}{\pi_{x,r}^\perp (v)}^2}\geq \frac 1 3 \vartheta(x,r) \abs{\pi_{x,r}^\perp (v)}^2\, .
 \end{gather}
 On the other hand:
 \begin{align}
  \frac 1 2 r^2\int \rho_r(y-x) \ps{\nabla u}{v}^2\leq &r^2\int \rho_r(y-x) \ps{\nabla u}{y-x}^2+r^2\int \rho_r(y-x) \ps{\nabla u}{y-x'}^2\notag \\
  \leq & C(m) r^2 \ton{r\dot\vartheta(x,3r/2)+r\dot \vartheta(x',3r/2)}\, .
 \end{align}
 This concludes the proof.
\end{proof}

Now we are ready to state an improved version of Theorem \ref{t:prelim:cone_splitting_annular}.
\begin{theorem}[Quantitative Cone Splitting on Annular Regions]\label{t:cone_splitting_annular}
Let $u:B_{10R}(p)\to N$ be a stationary harmonic map with $R^2\fint_{B_{10R}}|\nabla u|^2\leq \Lambda$, and let $\cA=B_2\setminus \overline{B_{\rf_x}(\T)}$ be a $\delta$-annular region.  For each $x\in \T$ with $B_r(x)\subseteq B_2$ and $3r>\rf_x$ we have
\begin{align}
	\vartheta(x,r;\LL_{x,r})+\vartheta(x,r;\LL_{x,r}^\perp)&\leq C(m)\ton{r\dot\vartheta(x,2r)+\fint_{\T\cap B_r(x)} r\dot\vartheta(z,2r)}\, .
\end{align}
Similarly, for each $x\in \T_r$ with $B_r(x)\subseteq B_2$ and $3r>\rf_x$ we have
\begin{align}\label{e:cone_splitting_annular_integrated}
	\vartheta(x,r;\LL_{x,r})+\vartheta(x,r;\LL_{x,r}^\perp)&\leq C(m)\ton{r\dot\vartheta(x,2r)+\fint_{\T_r\cap B_r(x)} r\dot\vartheta(z,2r)}\, .
\end{align}
\end{theorem}
\begin{remark}\label{rm:cone_splitting_annular_V}
Consequently if  $d_{\Gr} (V,\LL_{x,r})^2\leq C(m) \vartheta(x,r,\LL_{x,r})/\vartheta(x,r)$, then we also have the estimate $	\vartheta(x,r;V)+\vartheta(x,r;V^\perp)\leq C(m)\ton{r\dot\vartheta(x,2r)+\fint_{\T_r\cap B_r(x)} r\dot\vartheta(z,2r)}$ under the above setup. 
\end{remark}

\begin{proof}
	The proof of the two statements for $\T$ and $\T_r$ is identical, so we prove only the first one.

	According to Theorem \ref{t:prelim:cone_splitting_annular}, there exists $\bar x$ with $\abs{\bar x -x}\leq C(m)\sqrt \delta r$ and $\bar L$ such that 
	\begin{gather}
	 \vartheta(\bar x,1.1 r;\bar L)+\vartheta(\bar x,1.1 r;\bar L^\perp)\leq C(m)\fint_{\T\cap B_{r}(x)} r\dot\vartheta(z,2r)\, , \\
	 r\dot \vartheta(\bar x, 1.1r)\leq C(m)\fint_{\T\cap B_r(x)} r\dot\vartheta(z,2r)\, .
	\end{gather}
    Notice that by definition of $\LL_{x,r}$ and simple comparison:
    \begin{gather}
    \vartheta(x,r;\LL_{x,r})\leq C(m)\vartheta(\bar x,1.1r;\bar L)\, .
    \end{gather}
    Thus we turn out attention to the perpendicular component of the energy. In order to do so, consider that by \eqref{e:distance_L_LL}
    \begin{gather}
     d_{\Gr}(\bar L,\LL_{x,r})^2\leq \frac{C(m)}{\vartheta(x,r)}\vartheta(x,r;\bar L)\leq \frac{C(m)}{\vartheta(x,r)} \fint_{\T\cap B_{r}(x)} r\dot\vartheta(z,2r)\, ,
    \end{gather}
    and similarly by \eqref{e:distance_x_x'}
    \begin{gather}
    \abs{\pi_{x,r}^\perp (x-\bar x)}^2 \leq \frac{C(m)}{\vartheta(x,r)} \qua{r\dot \vartheta(x,2r) + \fint_{\T\cap B_{r}(x)} r\dot\vartheta(z,2r)}\, .
    \end{gather}
    Considering that
    \begin{gather}
    \abs{\pi_{x,r}^\perp (\nabla u(y))}^2 \leq 2\abs{\pi_{\bar L^\perp}(\nabla u(y))}^2 + C\abs{\nabla u(y)}^2 d_{\Gr}(\bar L,\LL_{x,r})^2\notag \\
    \abs{\pi_{x,r}^\perp (y-x)}^2 \leq 2\abs{\pi_{x,r}^\perp(y-\bar x)}^2 + 2\abs{\pi_{x,r}^\perp (x-\bar x)}^2\notag \\
    \abs{\pi_{x,r}^\perp (y-\bar x)}^2 \leq 2\abs{\pi_{\bar L^\perp}(y-\bar x)}^2 + C\abs{y-\bar x}^2 d_{\Gr}(\bar L,\LL_{x,r})^2\, ,
    \end{gather}
    we can estimate
    \begin{align}
    \vartheta(x,r;\LL^\perp_{x,r}) = &\int \rho_r(y-x) \abs{\pi_{x,r}^\perp \nabla u}^2 \abs{\pi_{x,r}^\perp(y-x)}^2\notag\\
    \leq C(m) &\int \rho_r(y-x) \abs{\pi_{\bar L^\perp } \nabla u}^2 \abs{\pi_{\bar L ^\perp} (y-\bar x)}^2 +C(m)\frac{\vartheta(x,2r)}{\vartheta(x,r)} \qua{r\dot \vartheta(x,2r) + \fint_{\T\cap B_{r}(x)} r\dot\vartheta(z,2r)}\notag \\
    \leq C(m) & \vartheta(x,1.1r;\bar L^\perp) + C(m) \qua{r\dot \vartheta(x,2r) + \fint_{\T\cap B_{r}(x)} r\dot\vartheta(z,2r)}\, .
    \end{align}
    This concludes the proof.
\end{proof}

\vspace{.5cm}

\section{Best Approximating Submanifold \texorpdfstring{$\T_r$}{Tr}}\label{s:approximating_submanifold}

Let $\cA\subseteq B_2$ be an annular region with submanifold $\T\subseteq B_2$ and radius function $\rf_x:\T\to \dR$ .  In Section \ref{sss:outline:general:local_approx_submanifold} we outlined the construction of scale $r>0$ smooth approximations $\T_r$ to $\T$ which were {\it best} approximations in some precise manner.  These approximations play a crucial role in our analysis, as their Euler-Lagrange equations will cancel out highest order, and otherwise non-controllable, errors which occur in the proof of Theorem \ref{t:outline:annular_regions_energy}.\\

Recall briefly the moral description of $\T_r$.  For $x\in \T$ with $r\geq\rf_x$ we have that the energy measure $|\nabla u|^2dv_g$ looks roughly like the $m-2$ Hausdorff measure on the plane $\cL_{x,r}$.  As a consequence, the energy function $\vartheta(x,r)$ looks roughly like a $2$-dimensional Gaussian in $B_r(x)$ invariant wrt $\LL_{x,r}$.  We will construct $\T_r$ so that each $x\in \T_r$ lives at the center of this Gaussian.  More precisely, for each $x\in \T_r$ we will have that $\vartheta(y,r)$ obtains a maximum at $x$ on the affine best plane $x+\cL^\perp_{x,r}$.  The goal of this Section is to prove the following:\\

\begin{theorem}[Approximating Submanifold]\label{t:approximating_submanifold}
Let $u:B_{10R}(p)\to N$ be a stationary harmonic map with $R^2\fint_{B_{10R}}|\nabla u|^2\leq \Lambda$ , and let $\cA=B_2\setminus \overline{B_{\rf_z}(\T)}$ be a $\delta$-annular region.  For each $r>0$ there exists a submanifold $\T_r\subseteq B_{3/2}$ such that the following hold:
\begin{enumerate}
	\item\label{i:Tr_graph} We can globally write $\T_r=\text{Graph}\{\ft_r:L_\cA\cap B_{3/2}\to L_\cA^\perp\}$ with
	\begin{gather}
	\abs{\ft_r}+\abs{\nabla \ft_r}+r\abs{\nabla^2 \ft_r}\leq C(m)\sqrt\delta
	\end{gather}
	\item \label{i:approximating_submanifold_space_gradient_hessian} For $x\in \T_r$ with $2r\geq \rf_x$ we can locally write $\T_r\cap B_r(x)=\text{Graph}\{\ft_{x,r}:\LL_{x,r}\cap B_r(x)\to \LL_{x,r}^\perp\}$
	\begin{gather}
	\abs{\nabla \ft_{x,r}}+r\abs{\nabla^2 \ft_{x,r}}+r^2\abs{\nabla^3 \ft_{x,r}}\leq C(m,\epsilon_0)\sqrt{\vartheta_{\cL}(x,2r)}\, ,
	\end{gather}
	\item\label{i:ddrT} For $x\in \T_r$ with $2r\geq \rf_x$,
	\begin{gather}
	\abs{\frac{d}{dr}\T_r}(x)\leq C(m,\epsilon_0)\,\ton{\sqrt{r\dot \vartheta(x,2r)}+ \sqrt {\vartheta_\LL (x,2r)}}\leq C(m,\epsilon_0)\,\sqrt\delta\, ,
	\end{gather}
	\item \label{i:ddr_piTr} For $x\in \T_r$ with $2r\geq \rf_x$
	\begin{gather}
	\abs{r\frac{d}{dr} \nabla \pi_{\T_r}}(x)\leq C(m) \abs{r\frac{d}{dr} \nabla \ft_{x,r}}\leq C(m,\epsilon_0) \ton{\sqrt{r\dot \vartheta(x,2r)}+ \sqrt {\vartheta_\LL (x,2r)}}\leq C(m,\epsilon_0)\sqrt \delta\, ,
	\end{gather}
	\item\label{i:T_r=T}  If $x\in \T_r$ with $\rf_x\geq 3r$ then $\T_r\cap B_r(x)=\T\cap B_r(x)$ ,
	\item\label{i:EL_Tr} If $x\in \T_r$ with $2r\geq \rf_x$, then the following equivalent conditions hold:
	\begin{align}
	 \pi_{x,r}^\perp \nabla \vartheta(x,r)=0 \qquad \Longleftrightarrow \qquad \forall v\in \LL_{x,r}^\perp: \ \ \int -\dot\rho_r \ton{y-x} \ps{\nabla u(y)}{y-x} \ps{\nabla u(y)}{v}dy=0\, .
	\end{align}
	\item\label{i:approximating_submanifold_better_LL_comparison} For each  $x\in \T_r$ with $2r\geq \rf_x$ we have that
	\begin{gather}
	|\pi_{\T_r}-\pi_{x,r}|\leq C(m,\epsilon_0)\sqrt{r\dot\vartheta(x,2r)} \sqrt{\vartheta_\cL(x,2r)}+ C(m,\epsilon_0) e^{-R/2}\sqrt{\vartheta_\cL(x,2r)}
	\end{gather}
	\item\label{i:integral_dot:vartheta_Tr_vs_T} for all $r$ and $r'\in [r/10,10r]$, we have
	\begin{gather}\label{e:integral_dot_vartheta_Tr_vs_T}
	 \int_{\T_r\cap \B {3} p} r\dot \vartheta(x,r')\leq C(m) \int_{\T \cap \B {4}p} r\dot \vartheta(x,1.1r')\, .
	\end{gather}
\end{enumerate}
\end{theorem}
\begin{remark}\label{r:T_r_II_estimates}
	$(2)$ implies the second fundamental form bound
	\begin{gather}
	 r\abs{\II_{\T_r}(z)}+r^2\abs{\nabla \II_{\T_r}(z)}\leq C(m,\epsilon_0) \sqrt{\vartheta_\cL(x,2r)}\, .
	\end{gather}
\end{remark}
\begin{remark}
	In $(3)$ we have that $\frac{d}{dr}\T_r\in T^\perp\T_r$ is the well defined normal vector.
\end{remark}
\begin{remark}
	In $(6)$ the equality of the two equations is not immediate, it requires an application of the stationary equation \eqref{e:stationary_equation}.
\end{remark}
\begin{remark}
	Observe by (7) that for each $\epsilon>0$ if $\delta<\delta(m,\epsilon)$ and $R\geq R(m,\epsilon_0,\epsilon)$ then we have the estimate $|\pi_{\T_r}-\pi_{x,r}|\leq \epsilon\sqrt{\vartheta_\cL(x,r)}$ , where $\pi_{\T_r}$ is the project to the tangent space at $\T_r$ .
\end{remark}

The remainder of this Section will be spent proving the above Theorem piece by piece.

\vspace{.3cm}

\subsection{Local Construction of \texorpdfstring{$\T_r$}{Tr} \texorpdfstring{near $\T\cap \{3r\geq \rf_x\}$}{where there's pinching} }

This subsection is dedicated to the local construction, namely we are interested in building $\T_r$ near $x\in \T$ with $3r\geq\rf_x$, so that the Euler-Lagrange of Theorem \ref{t:approximating_submanifold}.5 will be solved.  The proof will be by an implicit function theorem, and will be designed to both provide the estimates of Theorem \ref{t:approximating_submanifold}.2 and show uniqueness of the construction.  In the next subsection we will paste together this local construction with $\T$ on the region $3r\leq \rf_x$ in order to make the construction global.\\

\begin{lemma}\label{l:approximating_submanifold:local_construction}
Let $u:B_{10R}(p)\to N$ be a stationary harmonic map with $R^2\fint_{B_{10R}}|\nabla u|^2\leq \Lambda$ , and let $\cA=B_2\setminus \overline{B_{\rf_z}(\T)}$ be a $\delta$-annular region.	Let $x\in \T$ with $3r>\rf_x$ and $d_{\Gr}(L,\LL_\cA)<c(m)<<1$, then $\exists !$ $\ft_{L,r}:B_r(x)\cap L\to L^\perp$ with:
\begin{enumerate}
	\item\label{i:approximating_submanifold:local_construction:i1} For each $z\in \tilde \T_r\equiv \ft_{L,r}(B_r(x)\cap L)$ we have $ \pi_{z,r}^\perp \nabla \vartheta(z,r)=0$.
	\item $|\nabla \ft_{L,r}|+r|\nabla^2 \ft_{L,r}|+r^2|\nabla^3 \ft_{L,r}|\leq C(m,\epsilon_0)\sqrt{\vartheta(x,2r;L)}$
	\item $|\frac{d}{dr} \ft_{L,r}|+\abs{r\frac{d}{dr} \nabla \ft_{L,r}}\leq C(m,\epsilon_0)\ton{\sqrt{\vartheta_{\LL}(x,2r)}+\sqrt{r\dot \vartheta(x,2r)}}$.
	\item $|\frac{d}{dr} \ft_{L,r}|\leq C(m,\epsilon_0)\sqrt{\vartheta(x,2r;L^\perp)}$ 
   \item\label{i:ftr} $|\ft_{\T}-\ft_{L_{\hat \T_r}}|<C(m,\epsilon_0)\sqrt\delta\, r$  on $B_r(x)\cap {\hat \T_r}$.
   \item\label{i:pinching_comparison_T_Tr} For all $z\in \tilde \T_r $, let $\hat z$ be the only $\hat z\in \T$ with $\pi_L(\hat z)=\pi_L(z)$. Then for all $r'\in [r/10,10r]$:
   \begin{gather}\label{e:pinching_comparison_T_Tr} r'\dot \vartheta(z,r')\leq C(m) r'\dot \vartheta(\hat z,1.1r')\, .
   \end{gather}

\end{enumerate}
\end{lemma}

\begin{remark}
 Notice that in point \eqref{i:approximating_submanifold:local_construction:i1}, the projection $\pi_{z,r}^\perp=\pi_{\LL_{z,r}^\perp}$ is not the projection onto $L^\perp$, but the projection onto the pointwise best plane $\LL_{z,r}^\perp$.
\end{remark}

\begin{remark}
By applying the above to $L=L_\cA$ we will eventually obtain the global graphing function of	 Theorem \ref{t:approximating_submanifold}.1, and by applying the above to $L=\cL_{x,r}$ we will obtain the estimates of Theorem \ref{t:approximating_submanifold}.2 .
\end{remark}
\begin{remark}
By applying (3) to $L=\cL_{x,r}$ we obtain Theorem \ref{t:approximating_submanifold}.3 .
\end{remark}
\begin{remark}
Recall $\ft_\T:L_\cA\to L_\cA^\perp$ in \eqref{i:ftr} is the graphing function for $\T$, thus $d_H(\tilde \T_r,\T\cap B_r)<C(m)\sqrt{\delta} r$ .
\end{remark}

\begin{proof}[Outline of the proof]
 The proof is based on an implicit function theorem. Fix $x$, $r$ and $L$. We want to make sure that for all $z\in \T_r\cap \B {3r/2}{x}$ we have
 \begin{gather}\label{e:characterization_Tr}
  \pi_{z,r}^\perp\nabla \vartheta(z,r)=0\, .
 \end{gather}
\paragraph{Setting the stage for the implicit function theorem.} One inconvenience in this equation is that the range of $\pi_{z,r}$ changes as a function of $z$. This change is small since by \eqref{e:distance_LL_LL_sr}, for all $z\in \B {2r}x \cap \B {r/10}{\T}$ we have
\begin{gather}
 d_{\Gr}(\LL_{z,r},\LL_{x,r})\leq C\sqrt \delta<<1\, .
\end{gather}
However, in order to apply the implicit function theorem, we need to focus on the map
\begin{gather}
 \IF:\B {2r}{x} \to L^\perp\, , \qquad \IF(z)=r\pi_{L^\perp}\qua{\pi_{z,r}^\perp\nabla \vartheta(z,r)}\, .
\end{gather}
This map has the advantage of having a fixed range independent of $z$: the $2$-dimensional subspace $L^\perp$. Moreover, $d_{\Gr}(L,\LL_{z,r})<<1$ for all $z\in \B {2r}{x}\cap \B {r/10}{\T}$, so that
\begin{gather}
 \pi_{L^\perp}: \LL_{z,r}^\perp \to L^\perp
\end{gather}
is a linear bijection with bi-Lipschitz constant very close to $1$, say bounded by $8/7$ if $d_{\Gr}(L,\LL_\cA)$ and $\delta$ are chosen sufficiently small. This is enough to ensure that
\begin{gather}
 \IF(z)=0 \qquad \Longleftrightarrow \qquad \pi_{z,r}^\perp\nabla \vartheta(z,r)=0\, .
\end{gather}

\paragraph{Implicit function claims.} We claim that
\begin{enumerate}
 \item for $z\in \B {2r}{x}\cap \T$, $\abs{\IF(z)}\leq C\sqrt \delta <<1$
 \item for all $z\in \B {2r}{x} \cap \B {r/C(m)}{\T}$, the $2\times 2$ matrix $r\pi_{L^\perp} \nabla \IF(z)$ is bounded from above by $-\frac 1 {10}\vartheta(z,r)\id_{L^\perp}$
\end{enumerate}

It will follow from the above by the implicit function theorem that there exists $\ft_{L,r}:L\cap \B {3r/2}{0}\to L^\perp$ such that for all $p\in \B {3/2}{x}\cap L$:
\begin{gather}\label{e:implicit_function_identity}
 \Theta(p,\ft_{L,r}(p))=0\, .
\end{gather}

The estimates on $\ft_{L,r}$ will be by-products of the implicit function theorem.  The rest of this section is dedicated to proving the two claims above, as well as the ensuing estimates, and it will be divided into smaller subsections.
\end{proof}

\subsubsection{Proof of the implicit function claims.} The first claim is straight forward. Indeed, for all $z\in \T\cap \B {2r} x$  we have that $r\dot \vartheta(z,r)\leq \delta$, and so by Lemma \ref{t:prelim:spacial_gradient}:
\begin{gather}\label{e:implicit_function_claim_1}
 \abs{\IF(z)}\leq r\abs{\nabla \vartheta(z,r)}\leq C(m)\sqrt{\vartheta(z,2r)}\sqrt{r\dot \vartheta(z,2r)}\leq C(m)\sqrt{\vartheta(z,2r)}\sqrt{\delta}\, .
\end{gather}

Now we turn our attention to the uniform upper bound on $r\pi_{L^\perp}\nabla \IF(z)$ for $z$ close enough to $\T$. Let $\ellperp$ be any unit vector in $L^\perp$, then
\begin{align}
 r\pi_{L^\perp}\nabla \IF(z)[\ellperp,\ellperp] &=r^2 \ps{\nabla \ton{\ps{\pi_{z,r}^\perp \nabla \vartheta(z,r)}{\ellperp}}}{\ellperp}\\
 &= r^2 \nabla^2(\vartheta(z,r))\qua{\ellperp,\pi_{z,r}^\perp(\ellperp)}- r \nabla_{\ellperp} \pi_{z,r} \qua{\ps{r\nabla \vartheta(z,r)}{\ellperp}}\, .
\end{align}
The second piece can be estimated to be small. Indeed, by Lemma \ref{t:prelim:spacial_gradient} we have the rough bound $\abs{r\nabla \vartheta(z,r)}\leq C(m)\vartheta(z,2r)$, and \eqref{eq_LL_basic} ensures that as long as $z\in \B{1/2}{\T}$ we have $\abs{r \nabla \pi_{z,r}}\leq C(m)\sqrt{\delta/\vartheta(z,2r)}$. Thus we can bound
\begin{gather}
 \abs{r \nabla_{\ellperp} \pi_{z,r} \qua{\ps{r\nabla \vartheta(z,r)}{\ellperp}}}\leq C(m) \sqrt{\vartheta(z,2r)}\sqrt{\delta}\, .
\end{gather}
As for the first piece, by \eqref{e:spacial_hessian_vartheta} we know that
\begin{align}\label{e:nabla2_vartheta_proof}
 &r^2\nabla^2 (\vartheta(z,r))[\ellperp,\pi_{z,r}^\perp \ellperp]=\notag\\
 =&\underbrace{2r^2 \int \dot \rho_r(y-z) \ps{\nabla u}{\ellperp}\ps{\nabla u}{\pi_{z,r}^\perp \ellperp}}_{=A}\ \ \underbrace{+2r^2 \int \ddot \rho_r(y-z) \ps{\nabla u}{\frac{y-z}{r}}\ps{\frac{y-z}{r}}{\pi_{z,r}^\perp\ellperp}\ps{\nabla u}{\ellperp}}_{=B}\, .
\end{align}
In order to estimate $B$, consider $\pi^{\T}(z)=\tilde z\in \T$ to be the only point in $\T$ with $\pi_{L_\cA}(\tilde z)=\pi_{L_{\cA}}(z)$, so that $\abs{z-\tilde z} \leq 2 d(z,\T)$. Then we have
\begin{gather}
 2r^2 \int \ddot \rho_r(y-z) \ps{\nabla u}{\frac{y-z}{r}}\ps{\frac{y-z}{r}}{\pi_{z,r}^\perp\ellperp}\ps{\nabla u}{\ellperp}\stackrel{\eqref{e:trho_doubleradius}}{\leq}C(m)\sqrt{r^2 \int \rho_{2r}(y-z) \ps{\nabla u}{\frac{y-z}{r}}^2} \sqrt{\vartheta(z,2r)}
 \end{gather}
 \begin{align}
 r^2 \int \rho_{2r}(y-z) \ps{\nabla u}{\frac{y-z}{r}}^2&\leq 2r^2 \int \rho_{2r}(y-z) \ps{\nabla u}{\frac{y-\tilde z}{r}}^2 + 8 \int \rho_{2r}(y-z) \abs{\nabla u}^2 d(z,\T)^2\notag\\
 &\leq 2r^2 \int -\dot \rho_{3r}(y-\tilde z) \ps{\nabla u}{\frac{y-\tilde z}{r}}^2 + C(m) \vartheta(z,2r) \frac{d(z,\T)^2}{r^2}\notag\\
 &\leq C(m) \ton{\delta + \vartheta(z,2r) \frac{d(z,\T)^2}{r^2}}\notag \\
 &\stackrel{Remark\, \ref{rm:vartheta_comparison}}{\leq}C(m) \ton{\delta + \vartheta(z,r) \frac{d(z,\T)^2}{r^2}}
\end{align}

On the other hand, $A$ in \eqref{e:nabla2_vartheta_proof} can be estimated by
\begin{align}
 &2r^2 \int \dot \rho_r(y-z) \ps{\nabla u}{\ellperp}\ps{\nabla u}{\pi_{z,r}^\perp \ellperp}\notag \\
 \leq\ & \underbrace{2r^2 \int \qua{\dot \rho_r(y-z) +\rho_r(y-z)} \ps{\nabla u}{\ellperp}\ps{\nabla u}{\pi_{z,r}^\perp \ellperp}}_{=A_1} \ \ \underbrace{- 2r^2 \int \rho_r(y-z) \ps{\nabla u}{\ellperp}\ps{\nabla u}{\pi_{z,r}^\perp \ellperp}}_{=A_2}\, .
\end{align}
By \eqref{e:rho_primitive_difference}, we can bound
\begin{gather}
 A_1\leq C(m)e^{-R/2}\vartheta(z,2r)\stackrel{\ref{rm:vartheta_comparison}}{\leq} C(m)e^{-R/2}\vartheta(z,r)\, .
\end{gather}
As for $A_2$ let $d_{\Gr}(L,\LL_{z,r})=\bar d$. This can be assumed to be small, since $d_{\Gr}(L,L_{\cA})\leq c(m)<<1$ and $d_{\Gr}(L_{\cA},\LL_{z,r})\leq C(m,\epsilon_0)\sqrt \delta$ by Remark \ref{r:d_LA_LL_small}. Thus, given that $\ellperp$ is assumed to have norm $1$ we have
\begin{align}
 \pi_{z,r}^\perp\ellperp-\ellperp\in \LL_{z,r} \ \ \ \ \text{with} \ \ \ \abs{\pi_{z,r}^\perp\ellperp-\ellperp}\leq C \bar d\, .
\end{align}
From this we conclude
\begin{align}
 A_2 \leq &-2r^2 \int \rho_r(y-z) \ps{\nabla u}{\pi_{z,r}^\perp\ellperp}\ps{\nabla u}{\pi_{z,r}^\perp \ellperp}+2r^2 \int \rho_r(y-z) \ps{\nabla u}{\pi_{z,r}^\perp\ellperp-\ellperp}\ps{\nabla u}{\pi_{z,r}^\perp \ellperp}\leq\notag \\
 \leq &-2r^2 \int \rho_r(y-z) \ps{\nabla u}{\pi_{z,r}^\perp\ellperp}\ps{\nabla u}{\pi_{z,r}^\perp \ellperp}+2\bar d \sqrt{\vartheta_\LL(z,r)}\sqrt{\vartheta(z,r)}\, .
\end{align}
The eigenvalue separation in \eqref{eq_lambda_sep}, and the fact that
\begin{gather}
 \abs{\pi_{z,r}^\perp \ellperp}\geq \sqrt {1-C(m)\ton{\bar d^2 + \delta^2}}\geq \frac 1 2
\end{gather}
imply the (negative) upper bound independent of $\ellperp$:
\begin{gather}
 -2r^2 \int \rho_r(y-z) \ps{\nabla u}{\pi_{z,r}^\perp\ellperp}\ps{\nabla u}{\pi_{z,r}^\perp \ellperp}\leq - \frac 1 4 \vartheta(z,r)\, .
\end{gather}
Summing all up, we obtain that
\begin{gather}
 r\pi_{L^\perp}\nabla \IF(z)\leq \qua{\ton{-\frac 1 4 +C(m) e^{-R/4}}\vartheta(z,r) + C(m) \sqrt{\vartheta(z,r)} \sqrt{\delta} +\vartheta(z,r) \frac{d(z,\T)}{r}  } \operatorname{id_{L^\perp}}\, .
\end{gather}
Assuming $R\geq R(m)$ is sufficiently large and $\delta\leq \delta(m)$ sufficiently small, we can conclude that
\begin{gather}\label{e:main_bound_implicit_function_theorem}
 r\pi_{L^\perp}\nabla \IF(z)\leq \frac 1 5 \vartheta(z,r)\qua{-1+C(m) \frac{d(z,\T)}{r}} \operatorname{id_{L^\perp}}\, .
\end{gather}

\vspace{.3cm}

In particular, we obtain that if $d(z,\T)\leq r/C(m)$ then
\begin{gather}
 r\pi_{L^\perp}\nabla \IF(z)\leq -\frac 1 {10} \vartheta(z,r)\operatorname{id_{L^\perp}}\, ,
\end{gather}
as desired. $\qed$

\subsubsection{Proof of the gradient and hessian estimates.}\label{ss:gradient_hessian_estimates} By the implicit function theorem, the gradient estimate follows from
\begin{gather}
 \abs{r\pi_L \nabla \IF(z)} \leq C(m) \sqrt{\vartheta(z,2r)}\sqrt {\vartheta(z,2r;L)}\, .
\end{gather}
The proof of this in analogous to the one just carried out, actually simpler. In particular, consider $\ell$ to be a unit tangent vector in $L$.
Then
\begin{align}
 r\abs{\nabla_\ell \IF(z)} \leq& r\abs{\nabla \pi_{z,r}} \abs{\nabla \vartheta(z,r)} + r\abs{\nabla_\ell \vartheta}\notag \\
 \leq& C(m) \frac{\sqrt{\vartheta_{\LL}(z,2r)
 }}{\sqrt{\vartheta(z,r)}} \vartheta(z,2r) + C(m)\sqrt{\vartheta(z,2r)} \sqrt{\vartheta(x,2r;L)}\, ,
\end{align}
where we used the bounds in \eqref{eq_LL_basic} and Theorem \ref{t:prelim:spacial_gradient}. By the definition of $\LL$, $\vartheta_{\LL}(x,2r)\leq \vartheta(x,2r;L)$ for all $L$. This and \eqref{e:main_bound_implicit_function_theorem} prove that
\begin{gather}
 \abs{\nabla \ft_{L,r}(z)} \leq C(m)\frac{\sqrt {\vartheta(z,2r;L)}}{\sqrt{\vartheta(z,r)}}\, .
\end{gather}

\vspace{5mm}
In a similar way, one obtains the bounds
\begin{gather}\label{e:T_r_II_estimates_sharp}
 \abs{r \nabla^2 \ft_{L,r}(z)}+\abs{r^2 \nabla^3 \ft_{L,r}(z)} \leq C(m)\frac{\sqrt {\vartheta(z,2r;L)}}{\sqrt{\vartheta(z,r)}}\, .
\end{gather}

\subsubsection{Proof of the radial derivative estimate.} The radial derivative estimate can be obtained in a fashion similar to the spacial gradient estimate.  The estimate is weaker however, so it is worth discussing. A morally equivalent but slightly different approach is to take the derivative $\frac d {dr}$ of the identity \eqref{e:implicit_function_identity}. As a result, one obtains for all $z\in \tilde \T_r$, $z=(p,\ft_{L,r}(p))\in L\times L^\perp$: 
\begin{gather}
 r\frac{d}{dr} \IF(p,\ft_{L,r}(p))=0 \, ,\\
 \implies 0=r\pi_{L^\perp}\qua{r \frac{d}{dr} \ton{\pi_{z,r}^\perp}\nabla \vartheta(z,r) + \pi_{z,r}^\perp \nabla r\dot \vartheta(z,r)  + \nabla \IF(z)\qua{r\frac{d}{dr} \ft_{L,r}(p) }} \\
 \implies \abs{r\frac{d}{dr} \ft_{L,r}(p) }\stackrel{\eqref{e:main_bound_implicit_function_theorem}}{\leq}\frac{C(m)}{\vartheta(z,r)}\abs{r^2 \pi_L \nabla \IF(z) \qua{\frac{d}{dr} \ft_{L,r}(p) }}\leq \frac{C(m)}{\vartheta(z,r)}\ton{\abs{r\frac{d}{dr} \pi_{z,r}}\abs{r\nabla \vartheta(z,r)} + \abs{r^2 \nabla \dot \vartheta(z,r)}}\, .
\end{gather}
As seen before, by Lemma \ref{t:prelim:spacial_gradient} we have the brutal bound $\abs{r\nabla \vartheta(z,r)}\leq C(m)\vartheta(z,2r)$, and
\begin{gather}
 \abs{r \frac{d}{dr}\pi_{z,r}}\stackrel{\eqref{eq_LL_basic}}{\leq} C(m)\frac{\sqrt{\vartheta_\LL(z,2r)}}{\sqrt{\vartheta(z,2r)}}\, .
\end{gather}
Moreover, a direct computation shows that for all vectors $e$:
\begin{gather}
 r^2\nabla_e \dot \vartheta(z,r)= 4r^2 \int -\dot \rho_r(y-z) \ps{\nabla u}{y-z}\ps{\nabla u}{e} + 2r^2\int \ddot \rho_r(y-z) \ps{\nabla u}{y-z}^2 \ps{\frac{y-z}{r}}{e}\, ,\\
 \abs{r^2\nabla \dot \vartheta(z,r)}\stackrel{\eqref{e:trho_doubleradius}}{\leq} C(m)\sqrt{\vartheta(z,2r)}\sqrt{r\dot \vartheta(z,r)}\, .
\end{gather}
Thus we obtain
\begin{gather}
 \abs{\frac{d}{dr} \ft_{L,r}(p) }\leq C(m)\frac{\sqrt{r \dot \vartheta(z,2r)}+\sqrt{\vartheta_{\LL}(z,2r)}}{\sqrt{\vartheta(z,r)}}
\end{gather}
In a similar way, one obtains also
\begin{gather}
 \abs{r\frac{d}{dr} \nabla \ft_{L,r}(p)}\leq C(m)\frac{\sqrt{r \dot \vartheta(z,2r)}+\sqrt{\vartheta_{\LL}(z,2r)}}{\sqrt{\vartheta(z,r)}}\, .
\end{gather}

\subsubsection{Estimating the distance between $\T$ and $\tilde \T_r$.} Let $\hat z $ be the only point $\hat z\in \T$ with $\pi_L(\hat z)=\pi_L(z)$. By \eqref{e:implicit_function_claim_1} and \eqref{e:main_bound_implicit_function_theorem} it is immediate to see for all $z\in \tilde \T_r$ that we have
\begin{gather}
 \abs{z-\hat z}^2\leq C(m) r^2 \frac{r\dot \vartheta\ton{\hat z,2r}}{\vartheta(z,r)}\, .
\end{gather}
This immediately implies the desired distance bound. As a corollary, we also obtain \eqref{e:pinching_comparison_T_Tr}. Indeed, let $z\in \tilde \T_r$ and $r'\in [r,10r]$. Then
\begin{align}\label{e:pinching_correspondence}
 r'\dot \vartheta(z,r')= &2\int -\dot \rho_{r'}(y-z) \ps{\nabla u}{y-z}^2\leq 4\int -\dot \rho_{r'}(y-z) \ps{\nabla u}{y-\hat z}^2+4\int -\dot \rho_{r'}(y-z) \ps{\nabla u}{z-\hat z}^2\\
 \leq & C(m) r'\dot \vartheta\ton{\hat z,2r'} + C(m)\frac{\abs{z-\hat z}^2}{r^2} \vartheta(z,2r')\leq  C(m) r'\dot \vartheta\ton{\hat z,2r'}\, .
\end{align}

\vspace{.3cm}

\subsection{Global Construction of \texorpdfstring{$\T_r$}{Tr}}

We now aim to paste together the local constructions of the previous subsection in order to construct $\T_r$ and prove Theorem \ref{t:approximating_submanifold}.\eqref{i:Tr_graph}, \ref{t:approximating_submanifold}.\eqref{i:T_r=T} and \ref{t:approximating_submanifold}.\eqref{i:EL_Tr} .\\

Let us begin by taking $\{x_j\}\in \T\cap B_{4}$ so that $\big\{B_{r/4}(x_j)\big\}$ is a maximal disjoint subset.  It follows that $B_{r/2}(\T\cap B_{4})\subseteq \bigcup B_{r}(x_j)$ .  For each $x_j$ define $x'_j\equiv \pi_{L_\cA}(x_j)$, and note that $B_{r/5}(x'_j)$ are disjoint with $L_\cA\cap B_{4}\subseteq B_{3r/4}(x'_j)$.  Let $\phi_j:L_\cA\cap B_{4}\to \dR$ be a partition of unity so that
\begin{align}
	\supp{\phi_j}\subseteq B_r(x'_j)\, ,\; r^k|\nabla^{(k)}\phi_j|\leq C(m,k)\, ,\; \sum\phi_j(x)=1 \text{ for all }x\in L_\cA\cap B_{4}\, .
\end{align}

Now for each $x_j$ let $\ft_j: B_r(x'_j)\to L_\cA^\perp$ be defined as 
\begin{align}
	\begin{cases}
		 \ft_j\equiv \ft_{L_\cA} \text{ from Lemma \ref{l:approximating_submanifold:local_construction}}&\text{ if }\rf_{x_j}< 9r/4\\
		 \ft_j\equiv \ft_{\T} &\text{ if }\rf_{x_j}\geq  9r/4\, .
	\end{cases}
\end{align}

We can then globally define 
\begin{align}
	\ft_r(x)\equiv \sum \phi_j(x)\ft_j(x)\, .
\end{align}

If we now define $\T_r\equiv \ft_r(L_\cA\cap B_{4})$ then it is clear that \ref{t:approximating_submanifold}.\ref{i:T_r=T} and Theorem \ref{t:approximating_submanifold}.\ref{i:EL_Tr} hold by construction.
Also points \eqref{i:approximating_submanifold_space_gradient_hessian}, \eqref{i:ddrT}, \eqref{i:ddr_piTr} of Theorem \ref{t:approximating_submanifold} follow from Lemma \ref{l:approximating_submanifold:local_construction}.

\vspace{.3cm}

\subsection{Comparing \texorpdfstring{$\cL_{x,r}$}{Lxr} and \texorpdfstring{$T_x\T_r$}{tangent spaces of Tr}}

It is not the case that the tangent space of $\T_r$ at a point $x\in \T_r$ is equal to the best plane $\cL_{x,r}$ at that point.  However, it does turn out that we have better estimates on the comparison of the two than one might expect.  The main result of this subsection is the following, which proves Theorem \ref{t:approximating_submanifold}.\ref{i:approximating_submanifold_better_LL_comparison}:

\begin{lemma}\label{l:improved_comparison_LL_T}
For each  $x\in \T_r$ with $2r\geq \rf_x$ we have that
\begin{align}
 |\pi_{\T_r}-\pi_{x,r}|\leq & C(m)\frac{\sqrt{r\dot\vartheta(x,2r)} \sqrt{\vartheta_\cL(x,2r)}}{\vartheta(z,r)}+ C(m) e^{-R/2}\frac{\sqrt{\vartheta_\cL(x,2r)}}{\sqrt{\vartheta(x,r)}}\notag \\
 \leq &C(m,\epsilon_0)\sqrt{r\dot\vartheta(x,2r)} \sqrt{\vartheta_\cL(x,2r)}+ C(m,\epsilon_0) e^{-R/2}\sqrt{\vartheta_\cL(x,2r)}\, .
\end{align}

\end{lemma}
\begin{remark}\label{r:pinching_improvement_epsilon}
	Observe that for each $\epsilon>0$ if $R\geq R(m,\epsilon_0,\epsilon)$ and $\delta<\delta(m,\epsilon_0,\epsilon)$ then $|\pi_{\T_r}-\pi_{x,r}|\leq \epsilon\sqrt{\vartheta_\cL}(x,r)$ .
\end{remark}

\begin{proof}
 Let $v\in T_x\T_r$ be any (norm one) tangent vector at $x\in \T_r$ with $2r\geq \rf_x$, and set for convenience
\begin{gather}
 v^\perp = \pi_{x,r}^\perp(v)\, , \qquad v^\parallel=\pi_{x,r}(v)\, .
\end{gather}
We want to show that for all such $v$:
\begin{gather}\label{e:claim:improved_comparison_LL_T}
 \abs{v^\perp}\leq C(m)\frac{\sqrt{r\dot\vartheta(x,2r)} \sqrt{\vartheta_\cL(x,2r)}}{\vartheta(z,r)}+ C(m) e^{-R/2}\frac{\sqrt{\vartheta_\cL(x,2r)}}{\sqrt{\vartheta(x,r)}}\, ,
\end{gather}
which is equivalent to the thesis of the Lemma.

By definition of $\T_r$, we know that
\begin{gather}
 0=\nabla_{v} \qua{\pi_{x,r}^\perp \nabla \vartheta(x,r)}= \nabla_v \ton{\pi_{x,r}^\perp} \nabla \vartheta(x,r) + \pi_{x,r}^\perp \nabla \nabla_v \vartheta(x,r)\, .
\end{gather}
By Theorem \ref{t:prelim:spacial_gradient} and \eqref{eq_LL_basic}, we get
\begin{gather}
 \abs{\nabla_v \ton{\pi_{x,r}^\perp} \nabla \vartheta(x,r)}\leq C(m) \frac{\sqrt{\vartheta_\LL(x,2r)}}{\sqrt{\vartheta(x,r)}} \sqrt{\vartheta(x,2r)} {\sqrt{r\dot \vartheta(x,2r)}}\leq C(m) {\sqrt{\vartheta_\LL(x,2r)}}{\sqrt{r\dot \vartheta(x,2r)}}\, .
\end{gather}
Moreover, for all $w\in \LL_{x,r}^\perp$:
\begin{gather}
 r^2\nabla_w \nabla_v \vartheta(x,r)= r\nabla_w \ton{r^2 \int \dot \rho_r(y-x) \ps{\nabla u}{\frac{y-x}{r}} \ps{\nabla u}{v}}=\\
 =r^2 \int \ddot \rho_r(y-x) \ps{\frac{y-x}{r}}{w}\ps{\nabla u}{\frac{y-x}{r}} \ps{\nabla u}{v}+r^2 \int \dot \rho_r(y-x) \ps{\nabla u}{w} \ps{\nabla u}{v}\, .
\end{gather}
As a consequence, taking $w=\frac{v^\perp}{\abs{v^\perp}}$ we can estimate
\begin{gather}
 c(m)\vartheta(x,r) \abs{v^\perp}\leq r^2 \int -\dot \rho_r(y-x) \ps{\nabla u}{v^\parallel} \ps{\nabla u}{w} + C(m) \sqrt{\vartheta_\LL(x,2r)}\sqrt{r\dot \vartheta(x,2r)}\, .
\end{gather}
By definition of $\LL_{x,r}$, we know that
\begin{gather}
 \int \rho_r(y-x) \ps{\nabla u}{v^\parallel} \ps{\nabla u}{w}=0\, ,
\end{gather}
and thus we can estimate
\begin{align}
 \abs{r^2 \int -\dot \rho_r(y-x) \ps{\nabla u}{v^\parallel} \ps{\nabla u}{w}}\leq &r^2 \int \abs{\rho_r(y-x)+\dot \rho_r(y-x)}\abs{\pi_{\LL_{x,r}}\nabla u}\abs{\nabla u}\leq \notag \\
 \stackrel{\eqref{e:rho_primitive_difference}}{\leq}& C(m) e^{-R/2} \sqrt{\vartheta_\LL(x,2r)}\sqrt{\vartheta(x,2r)}\, ,
\end{align}
and this, together with Remark \ref{rm:vartheta_comparison}, proves \eqref{e:claim:improved_comparison_LL_T}.
\end{proof}

As an immediate corollary, we obtain that
\begin{corollary}\label{c:estimate_LL_T}
If $\delta\leq \delta_0(m,\epsilon_0)$ is sufficiently small and $R\geq R_0(m)$ is sufficiently big, then for each $x\in \T_r$ with $3r\geq \rf_x$ we have that
\begin{align}\label{e:estimate_LL_T}
 r^2\int \rho_r(y-x) \abs{\pi_{\T_r(x)}\nabla u(y)}^2\leq C(m)r^2 \int \rho_{2r}(y-x) \abs{\pi_{x,{2r}}\nabla u(y)}^2=C(m)\vartheta_\LL(x,2r)\, .
\end{align}
\end{corollary}
\begin{proof}
 The proof follows immediately from Lemma \ref{l:improved_comparison_LL_T} and \eqref{r:d_LLr_LL2r_small}, using the identity
 \begin{gather}
  \pi_{\T_r(x)}\nabla u (y)=\pi_{x,r}\nabla u (y) + \ton{\pi_{\T_r(x)}-\pi_{x,r}}\nabla u (y)\,
 \end{gather}
 and standard estimates.
\end{proof}

\subsection{The Euler-Lagrange Equation}\label{ss:T_r_Euler_Lagange}

Here we study the Euler-Lagrange equation of Theorem \ref{t:approximating_submanifold}.\eqref{i:EL_Tr} .  In particular, we want to see that by combining with the correct stationary equation we can rewrite the Euler-Lagrange formula $\pi^\perp_{x,r}\nabla\vartheta(x,r)=0$ in a manner which will be more applicable:\\

\begin{lemma}
We have that
 \begin{gather}\label{eq_T_r_max_cond_explicit}
  \pi_{x,r}^\perp \nabla \vartheta(x,r)=0 \qquad \Longleftrightarrow \qquad \forall v\in \LL_{x,r}^\perp: \ \ \int -\dot\rho_r \ton{y-x} \ps{\nabla u(y)}{y-x} \ps{\nabla u(y)}{v}dy=0\, .
 \end{gather}
\end{lemma}
\begin{proof}
 The maximality condition immediately gives
 \begin{gather}
  \pi_{x,r}^\perp \nabla \hvt(x,r)=0 \qquad \Longleftrightarrow \qquad \forall v\in \LL_{x,r}^\perp: \ \ \int -\dot\rho_r(y-x)\abs{\nabla u(y)}^2\ps{y-x}{v}\,dy=0
 \end{gather}
The proof is then an application of the stationary equation \eqref{e:stationary_equation} with vector field $\xi= \rho_r (y-x)v$, which yields:
 \begin{gather}
  \int -\dot\rho_r(y-x)\abs{\nabla u(y)}^2\ps{y-x}{v}=2\int -\dot\rho_r(y-x)\ps{\nabla u(y)}{y-x} \ps{\nabla u(y)}{v}\, .
 \end{gather}
\end{proof}

\subsection{Comparison of pinching.} Theorem \ref{t:approximating_submanifold}.\eqref{i:integral_dot:vartheta_Tr_vs_T} follows from \eqref{e:pinching_comparison_T_Tr} of Lemma \ref{l:approximating_submanifold:local_construction}.\eqref{i:pinching_comparison_T_Tr} and the construction of $\T$.

\vspace{.3cm}

\vspace{.5cm}

\section{Energy Decomposition in \texorpdfstring{$\cA$}{the annular region}}\label{s:energy_decomposition}

The proof that annular regions have small energy in Theorem \ref{t:outline:annular_regions_energy} will depend on decomposing the energy of $u$ into three base components.  On a very rough level, we will break 
\begin{align}
	|\nabla u|^2 \approx\abs{\pi_L \nabla u(y)}^2+\abs{\pi_L^\perp \nabla u(y)}^2= \abs{\pi_L \nabla u(y)}^2+\langle\nabla u,\alpha^\perp\rangle^2+\langle\nabla u,n^\perp\rangle^2\, ,
\end{align}
the energy into its $L$-energy, angular energy and radial energy components as discussed in Section \ref{s:outline_toymodel} and Section \ref{s:outline_general}.  Our goal in this Section is to use the best planes $\cL_{x,r}$ and approximating submanifolds $\T_r$ to make this more precise and collect together some technical understanding of the behavior of our energy decomposition.\\

\vspace{.3cm}

\subsection{Heat Mollified Energy Decomposition in \texorpdfstring{$\cA$}{the annular region}}

Following Section \ref{s:outline_general} we will be interested in the following energy functionals:
\begin{align}
	\vartheta_\cL(x,r) &\equiv r^2\min_L \int\rho_r(y-x)\abs{\pi_L \nabla u(y)}^2 = r^2\int\rho_r(y-x)\abs{\pi_{x,r} \nabla u(y)}^2\, ,\notag\\
	\vartheta_\alpha(x,r) &\equiv \int\rho_r(y-x)|\pi_{x,r}^\perp(y-x)|^2\langle \nabla u, \alpha_{x,r}^\perp\rangle^2 \, ,\notag\\
	\hat\vartheta_\alpha(x,r) &\equiv \int\hat\rho_r(y-x;\cL_{x,r})|\pi_{x,r}^\perp(y-x)|^2\langle \nabla u, \alpha_{x,r}^\perp\rangle^2 \, ,\notag\\
	\vartheta_n(x,r) &\equiv \int\rho_r(y-x)\langle \nabla u, \pi_{x,r}^\perp(y-x)\rangle^2= \int\rho_r(y-x)|\pi_{x,r}^\perp(y-x)|^2\langle \nabla u, n_{x,r}^\perp\rangle^2 \, ,\notag \\
	\hat\vartheta_n(x,r) &\equiv \int\hat\rho_r(y-x;\cL_{x,r})\langle \nabla u, \pi_{x,r}^\perp(y-x)\rangle^2= \int\hat\rho_r(y-x;\cL_{x,r})|\pi_{x,r}^\perp(y-x)|^2\langle \nabla u, n_{x,r}^\perp\rangle^2 \, .
\end{align}
Note that our $L$-energy is measured with respect to the heat kernel $\rho_r$ from Definition \ref{d:heat_mollifier} defined in \ref{d:restricted_energy_functionals}, while the ``hatted'' angular and radial energies are measured with respect to the restricted heat kernel $\hat\rho_r$, which is cutoff near the best plane $\cL_{x,r}$.  Note also that the angular $\alpha^\perp_{x,r}$ and radial $n^\perp_{x,r}$ coordinates are measured with respect to the affine plane $x+\cL_{x,r}^\perp$.
Recall that by definition of $\hat \rho$ we clearly have
\begin{align}
	\hat\vartheta_\alpha(x,r) \leq \vartheta_\alpha(x,r)\, ,\qquad
	\hat\vartheta_n(x,r) \leq \vartheta_n(x,r)\, .
\end{align}
We need the hatted versions of the energies because, as pointed out in the outline, the angular energy enjoys a uniform subharmonicity property.  This subharmonicity holds only in the annular region, hence only if the point considered is not too close to $\T$. \\

The following Lemmas give us our basic control over our energy functionals:\\

\begin{lemma}\label{l:energy_decomposition:pointwise_properties_nohat}
	Let $u:B_{10R}(p)\to N$ be a stationary harmonic map with $R^2\fint_{B_{10R}}|\nabla u|^2\leq \Lambda$ and let $\cA=B_2\setminus \overline{B_{\rf_z}(\T)}$ be a $\delta$-annular region with $\delta\leq \delta(m,\Lambda,K_N,R)$.  Then the following hold:
\begin{enumerate}
	\item Let $B_s(z)\subseteq B_r(x)$, then $\vartheta_\cL(z,s)\leq C(m,r/s)\,\vartheta_\cL(x,r)$
\item For $x\in \T_r$ with $2r>\rf_x$  we have that
\begin{align}
	\vartheta_\cL(x,r)+\vartheta_\alpha(x,r)+\vartheta_n(x,r)=\vartheta_\cL(x,r)+\vartheta(x,r;\LL_{x,r}^\perp)&\leq C(m)\qua{r\dot \vartheta(x,2r)+\fint_{\T_r\cap B_r(x)} r\dot \vartheta(y,2r)}
\end{align}
\item For $x\in \T$ with $2r>\rf_x$  we have that
\begin{align}
	\vartheta_\cL(x,r)+\vartheta_\alpha(x,r)+\vartheta_n(x,r)=\vartheta_\cL(x,r)+\vartheta(x,r;\LL_{x,r}^\perp)&\leq C(m)\qua{r\dot \vartheta(x,2r)+\fint_{\T\cap B_r(x)} r\dot \vartheta(y,2r)}
\end{align}
\end{enumerate}
\end{lemma}
\begin{proof}
 The first inequality follows from
	\begin{gather}
	 \vartheta_\LL(z,s) = s^2 \int \rho_s(z-y) \abs{ \pi_{\LL_{z,s}}\nabla u}^2\leq s^2 \int \rho_s(z-y) \abs{\pi_{x,r} \nabla u}^2\leq 
	 C(m) \vartheta_\LL(x,r)\, .
	\end{gather}
    As for the second and third points, they are just a re-writing of Lemma \ref{t:cone_splitting_annular}.
\end{proof}

\vspace{.3cm}

The following Lemma follows easily from the definition of $\hat\rho$, however it seems worth stating explicitly that the full energy on the annular region can be recovered by our energy decomposition in the following sense:
\begin{lemma}\label{l:energy_bounded_by_restricted_vartheta}
Let $x\in B_{10^{-3}r}(\T)$ with $r\geq \rf_x/10$, then we have 
 \begin{align}\label{e:energy_bounded_by_restricted_vartheta}
  r^{2-m}\int_{B_{10r}(x)\setminus \B{10^{-2}r}{x+\LL_{x,r}}} \abs{\nabla u}^2 &\leq C(m) \qua{\hat\vartheta_\LL(x,r)+ \hat \vartheta_{\alpha}(x,r)+\hat \vartheta_n(x,r)}\notag \\
  &\leq C(m) \qua{\vartheta_\LL(x,r)+ \vartheta_{\alpha}(x,r)+\vartheta_n(x,r)} \, .
 \end{align}
\end{lemma}

\vspace{.3cm}

\subsection{Estimates on the Heat Kernel Mollifiers \texorpdfstring{$\rho_r(y)$, $\hat\rho_r(y;L)$}{} }\label{ss:energy_decomposition:kernel_estimates}

Let us collect together a handful of basic estimates on our choice of mollifiers. These estimates are a simple corollary of Lemma \ref{l:rho_basic_properties} and \ref{l:rho_tilde_estimates}. We will apply these in a variety of estimates in future Sections. Recall from Section \ref{ss:restricted_energy} any notation:\\
\begin{lemma}[Estimates for $\rho_r$]\label{l:energy_decomposition:rho_estimates}
	The following hold:
\begin{align}
 \abs{r\frac{d}{dr} \rho_r(y-x)}+\abs{\ton{r\frac{d}{dr}}^2 \rho_r(y-x)}\leq &C(m)\rho_{1.1r}(y-x)\, ,\notag \\
 r\abs{\nabla \rho_r(y-x)}+r^2\abs{\nabla^2 \rho_r(y-x)}+r^2\abs{\frac{d}{dr}\nabla \rho_r(y-x)}\leq &C(m)\rho_{1.1r}(y-x) \label{e:brutal_rho_r}
\end{align}
\end{lemma}
\begin{proof}
	The proof is a straightforward computation from the definitions, indeed
	\begin{align}
	 \rho_r(y-x)= &r^{2-m}\rho\ton{\frac{\abs{y-x}^2}{2r^2}}\, ,\\
	 r\frac{d}{dr}\rho_r(y-x)= & -m \rho_r (y-x) - \dot \rho_r (y-x)\frac{\abs{y-x}^2}{r^2}\, ,\\
	 r\nabla^{(y)} \rho_r(y-x)=&-r\nabla^{(x)} \rho_r(y-x)= \dot \rho_r (y-x) \frac{y-x}{r}\, .
	\end{align}
	Higher derivatives have similar formulas. The estimates follow from \eqref{e:trho_doubleradius}.
\end{proof}

Similar estimates can be obtained for the $\hat \rho_r$ mollifier, which we record here for future reference.

\begin{lemma}[Estimates for $\hat \rho_r$]\label{l:energy_decomposition:hat_rho_estimates}
	Under the assumptions of Theorem \ref{t:best_plane:best_plane}, if $\delta<<1$ the following hold: 
\begin{align}
 \abs{r\frac{d}{dr} \hat \rho_r(y-x;\LL_{x,r})}+\abs{\ton{r\frac{d}{dr}}^2 \hat \rho_r(y-x;\LL_{x,r})}\leq &C(m)\rho_{1.1r}(y-x)\, ,\notag \\
 r\abs{\nabla^{(x)} \hat \rho_r(y-x;\LL_{x,r})}+r^2\abs{\nabla^{(x)}\nabla^{(x)} \hat \rho_r(y-x;\LL_{x,r})}+r^2\abs{\frac{d}{dr}\nabla^{(x)} \hat \rho_r(y-x;\LL_{x,r})}\leq& C(m)\rho_{1.1r}(y-x) \notag \\
 r\abs{\nabla^{(y)} \hat \rho_r(y-x;\LL_{x,r})}+r^2\abs{\nabla^{(y)}\nabla^{(y)} \hat \rho_r(y-x;\LL_{x,r})}+r^2\abs{\frac{d}{dr}\nabla^{(y)} \hat \rho_r(y-x;\LL_{x,r})}\leq& C(m)\rho_{1.1r}(y-x) \label{e:brutal_hat_rho_r}\, .
\end{align}
\end{lemma}
\begin{proof}
	The proof is a simple computation, similar to the one above, using Definition \ref{d:restricted_energy_functionals} of $\hat \rho_r$, and the derivative properties of $\pi_{x,r}$ in Theorem Theorem \ref{t:best_plane:best_plane}.

	Notice here that $\nabla^{(x)}\hat \rho_r(y-x,\LL_{x,r})$ also hits $\LL_{x,r}$, and in particular the projection $\pi_{x,r}^\perp$ in the definition of $\hat \rho$. For this reason, $\nabla^{(x)}\hat \rho_r(y-x,\LL_{x,r})\neq -\nabla^{(y)}\hat \rho_r(y-x,\LL_{x,r})$, albeit the difference is small since $\nabla \pi_{x,r}$ is small.
\end{proof}

\vspace{.3cm}

\subsection{Bubble Center Cutoff \texorpdfstring{$\psi_\T(x,r)$}{}}\label{ss:bubble_cutoff}

We want to use the pointwise energies from the last subsection and integrate them over $\T_r$ in order to define our energies on each scale in $\cA$ .  In reality one only wants to integrate on the part of $\T_r$ for which $\rf_x<r$, as this is the range needed to capture the energy on $\cA$ .  The goal of this subsection is to define and prove some basic properties of our cutoff function:

\begin{definition}[$\T$-Cutoff]
	Let $u:B_{10R}(p)\to N$ be a stationary harmonic map with $R^2\fint_{B_{10R}}|\nabla u|^2\leq \Lambda$ and let $\cA=B_2\setminus \overline{B_{\rf_z}(\T)}$ be a $\delta$-annular region with $\delta\leq \delta(m,\Lambda,K_N,R)$. Let also $\phi:\R\to [0,1]$ be a smooth function with
	\begin{gather}
	 \phi(t)=\begin{cases}
	       1 & \text{for } \ t\leq 1.1\\
	       0 & \text{for } \ t\geq 1.9
	      \end{cases}
	\end{gather}
	and such that
	\begin{gather}
	 \dot \phi \leq 0\, , \qquad  \abs{\dot \phi} + \abs{\ddot \phi} \leq 10\, .
	\end{gather}
We define $\psi_{\T}:\B 2 p \to [0,1]$ by
	\begin{align}
	 \psi_{\T}(x,r)=& \phi \ton{\frac{\rf_x}{r}} \phi(r)\phi\ton{\abs{\pi_{L_\cA}(x-p)}^2 }\, .
	\end{align}
	where $\rf_x$ is defined as in \ref{d:annular_region} and Remark \ref{r:annular_region:extension}.
\end{definition}
\begin{remark}
	Though $\psi_\T(x,r)$ is defined for $x\in \R^m$ we will in practice restrict it to either $\T$ or $\T_r$.
\end{remark}

\begin{remark}
 Note that $\psi_\T(x,\cdot)$ has support contained in the set $r\in [\rf_x/2,2]$.
\end{remark}

The next lemma encodes the basic properties of $\psi_{\T}$.
\begin{lemma}\label{l:energy_decomposition:T_cutoff}
Let $u:B_{10R}(p)\to N$ be a stationary harmonic map with $R^2\fint_{B_{10R}}|\nabla u|^2\leq \Lambda$ and let $\cA=B_2\setminus \overline{B_{\rf_z}(\T)}$ be a $\delta$-annular region with $\delta\leq \delta(m,\Lambda,K_N,R)$.  Then the cutoff function $\psi_\T:\B 2 p \to \dR^+$ is a $C^\infty_c$ function that satisfies
\begin{enumerate}
	\item $\psi_\T(x,r)=0$ if $\abs{x-p}\geq 2$ and $\psi_\T(x,r)=0$ if $\rf_x\geq 2r$,
	\item $\psi_\T(x,r)=1$ if $\rf_x\leq r$ and $\abs{x-p}\leq 1$,
	\item\label{i:T_cutoff_global_bounds} $\abs{r\frac{d}{dr} \psi_{\T}(x,r)}+r|\nabla \psi_\T|+r^2|\nabla^2 \psi_\T|+r^2 \abs{\nabla \frac{d}{dr} \psi_\T} +r^2\abs{\frac{d^2}{dr^2} \psi_\T}< C(m)$ .
	\item\label{i:T_cutoff_global_integral} $\int\int_{\T_r}\ton{r|\nabla \psi_\T|+r^2|\nabla^2 \psi_\T|+r^2 \abs{\nabla \frac{d}{dr} \psi_\T} + r\frac{d}{dr} \psi_\T+r^2 \abs{\frac{d^2}{dr^2} \psi_\T}}\frac{dr}{r} < C(m)$
\end{enumerate}
\end{lemma}
\begin{proof}

	Properties (1), (2) and (3) follow immediately from the definitions.

	In order to prove \eqref{i:T_cutoff_global_integral}, we observe that
	\begin{gather}
     r\abs{\nabla \psi_\T(x,r)}\leq \abs{\nabla \rf_x} + r\abs{\nabla \phi}\leq \abs{\nabla \rf_x}+Cr\, , \qquad r^2\abs{\nabla^2 \psi_\T(x,r)}\leq C\ton{\abs{\nabla \rf_x}+r\abs{\nabla^2 \rf_x}}+Cr\, .
	\end{gather}
	Thus, given that $\T$ and $\T_r$ are all Lipschitz graphs over $L_{\cA}$ with small Lipschitz constant, we can estimate
	\begin{gather}
	 \int\int_{\T_r}r|\nabla \psi_\T|+r^2|\nabla^2 \psi_\T|\frac{dr}{r} \leq C + C\int\int_{\T_r}\abs{\nabla \rf_x}+r\abs{\nabla^2 \rf_x} \frac{dr}{r}\, .
	\end{gather}
	Since $\rf_x$ is invariant wrt translations in $L_{\cA}^\perp$, and since $\T_r$ and $\T$ are both Lipschitz graphs over $L_{\cA}$ with small Lipschitz constant, we can bound
	\begin{gather}
	 \int\int_{\T_r}\abs{\nabla \rf_x}+r\abs{\nabla^2 \rf_x} \frac{dr}{r}\leq 2\int\int_{\T}\abs{\nabla \rf_x}+r\abs{\nabla^2 \rf_x} \frac{dr}{r}=2\int_{\T}\int\abs{\nabla \rf_x}+r\abs{\nabla^2 \rf_x} \frac{dr}{r}
	\end{gather}

	Considering that $|\nabla \rf_x|+r\abs{\nabla^2 \rf_x}\leq C$ and $|\nabla \rf_x|\neq 0$ only for $r\in [\rf_x/2,\rf_x]$, we can conclude
	\begin{gather}
	 \int_{\T}\int\abs{\nabla \rf_x}+r\abs{\nabla^2 \rf_x} \frac{dr}{r}\leq C(m) \, ,
	\end{gather}
	where we used the fact $\T\cap \supp{\psi_{\T}(\cdot,r)}\subset B_2$ has finite $m-2$ measure.
	
	Estimating in a similar fashion proves the bounds on the radial derivatives of $\psi_\T$.
\end{proof}

\vspace{.3cm}

\subsection{Heat Mollified Energies along \texorpdfstring{$\T_r$}{Tr} }

Combining the tools of the last two subsections we can define the energies along $\T_r$ by

\begin{align}
	\vartheta_\cL(\T_r) &\equiv \int_{\T_r}\psi_\T(x,r)\,\vartheta_\cL(x,r)\,=r^2\int_{\T_r}\psi_\T(x,r)\int\rho_r(y-x)\abs{\pi_{x,r}\nabla u}^2\, ,\notag\\
	\hat\vartheta_\alpha(\T_r) &\equiv \int_{\T_r}\psi_\T(x,r)\,\hat\vartheta_\alpha(x,r)= \int_{\T_r}\psi_\T(x,r)\int\hat\rho_r(y-x;\cL_{x,r})|\pi_{x,r}^\perp(y-x)|^2\langle \nabla u, \alpha_{x,r}^\perp\rangle^2 \, ,\notag\\
	\hat\vartheta_n(\T_r) &\equiv \int_{\T_r}\psi_\T(x,r)\,\hat\vartheta_n(x,r)= \int_{\T_r}\psi_\T(x,r)\int\hat\rho_r(y-x;\cL_{x,r})|\pi_{x,r}^\perp(y-x)|^2\langle \nabla u, n_{x,r}^\perp\rangle^2 \, .
\end{align}

\begin{remark}
 In these definitions, we always used the best plane $\LL_{x,r}$ at $x\in \T_r$ to set the various quantities. Thanks to Lemma \ref{l:improved_comparison_LL_T}, this is almost equivalent to considering the tangent space $\T_r(x)$ instead of $\LL_{x,r}$. However, in the future we will need to use the Euler-Lagrange equation for the best plane $\LL_{x,r}$ to cancel some of the highest order error terms, and so it important to use $\LL_{x,r}$ in these definitions.\\
\end{remark}

Let us now collect together some basic technical tools which will come in handy.

\begin{lemma}\label{l:energy_decomposition:scale_r_properties}
Let $u:B_{10R}(p)\to N$ be a stationary harmonic map with $R^2\fint_{B_{10R}}|\nabla u|^2\leq \Lambda$ and let $\cA=B_2\setminus \overline{B_{\rf_z}(\T)}$ be a $\delta$-annular region with $\delta\leq \delta(m,\Lambda,K_N,R)$.  Then the following hold:
\begin{enumerate}
	\item\label{i_2BCHECKED} Let $s\leq r\leq 1$, then $\vartheta_\cL(\T_s)\leq C(m)\,\vartheta_\cL(\T_r)$ .
\item We can estimate 
\begin{align}
	\vartheta_\cL(\T_r)+\hat\vartheta_\alpha(\T_r)+\hat\vartheta_n(\T_r)&\leq \vartheta_\cL(\T_r)+\vartheta_\alpha(\T_r)+\vartheta_n(\T_r)\leq C(m)\int_{\T\cap \B 4 p }    r\dot \vartheta(x,4r)\, .
\end{align}
\end{enumerate}
\end{lemma}
\begin{remark}
	Note that condition $(1)$ holds for $s\leq r\leq 1$ with a constant independent of $r/s$.  This almost monotonicity is not equation based but due to the $m-2$ dimensional nature of $\T$.
\end{remark}
\begin{proof}

	We split the first inequality into two regions: $\frac{r}{100}\leq s\leq r$ and $s\leq \frac r {100}$.

	\textbf{Region 1:} if $\frac{r}{100}\leq s\leq r$, we will see that the estimate follows from easy considerations. In particular, we clearly have by definition of $\LL_{x,s}$ that
	\begin{gather}
	 \int\rho_s(y-x)\abs{\pi_{x,s}\nabla u}^2\leq \int\rho_s(y-x)\abs{\pi_{x,r}\nabla u}^2
	\end{gather}
	Now let $x\in \T_s$ with $\psi_\tau(x,s)>0$, so that $r\geq \rf_x$. Set $\bar x \in \T_r$ to be the only point with $\pi_{L_A}(x)=\pi_{L_A}(\bar x)$. Given Theorem \ref{t:approximating_submanifold}.\eqref{i:ddrT}, we know that $\abs{x-\bar x} <<r/100$, and so it is easy to see that
	\begin{gather}
	 \int\rho_s(y-x)\abs{\pi_{x,s}\nabla u}^2\leq\int\rho_s(y-x)\abs{\pi_{x,r}\nabla u}^2\leq C(m) \int\rho_r(y-\bar x)\abs{\pi_{x,r}\nabla u}^2\, ,
	\end{gather}
	where in the last inequality we exploited the condition $r\leq 100 s$.

	Since both $\T_r$ and $\T_s$ are Lipschitz graphs over $L_\cA$ with small Lipschitz constant, and since by construction the functions $\psi(\cdot,r)$ are all invariant wrt translations in $L_\cA^\perp$, we can conclude
	\begin{gather}
	 \vartheta_\cL(\T_s)=s^2\int_{\T_s}\psi_\T(x,s)\int\rho_s(y-x)\abs{\pi_{x,s}\nabla u}^2\leq C(m) r^2 \int_{\T_s}\psi_\T(x,s)\int\rho_r(y-x)\abs{\pi_{x,r}\nabla u}^2\, .
	\end{gather}

	\textbf{Region 2:} the same approach would not work if $s<<r$ because the constants would blow up as $s\to 0$, so we need a different argument in the case $s\leq \frac r {100}$. In particular here we rely on the $m-2$ dimensional nature of the integral we are considering.

	In detail, fix any $z\in \T_s$ with $\psi_\T(z,s)>0$, and let as before $\bar z\in \T_r$ be the only point with $\pi_{L_{\cA}}(z)=\pi_{L_{\cA}}(\bar z)$. Notice that $\abs{\bar z - z}\leq C\sqrt\delta r <<r$. Moreover, since $s\leq r/100$ we have by the definition of $\psi_\T$ that $\psi_\T(\bar z,r)=1$ if $\psi_\T(z,s)>0$. Given the Lipschitz bounds on $\psi_\T$, we also have that $\psi_\T(x,r)\geq \frac 1 2$ for all $x\in \T_r\cap \B{r/100}{\bar z}$.\\

	\textbf{Claim:} We claim that for $z\in \T_s$ with $\psi_\T(z,s)>0$ we have
	\begin{gather}
	 s^2\int_{\T_s\cap \B {r/200}{z}} \psi_\T(x,s) \int \rho_s(y-x)\abs{\pi_{x,s}\nabla u}^2\leq C(m) r^2\int_{\T_r\cap \B {r/100}{\bar z}} \psi_\T(x,r) \int \rho_r(y-x)\abs{\pi_{x,r}\nabla u}^2
	\end{gather}
	Once this claim is proved, we obtain the final estimate with a simple covering argument.

	\textit{Proof of claim:} By definition of the best plane $\LL_{x,r}$ and monotonicity of the integral, we get that for all $x\in \B {r/200}{z}$:
	\begin{gather}
	 s^2\int_{\T_s\cap \B {r/200}{z}} \psi_\T(x,s) \int \rho_s(y-x)\abs{\pi_{x,s}\nabla u}^2\leq s^2\int_{\T_s\cap \B {r/200}{z}} \int \rho_s(y-x)\abs{\pi_{\bar z,r/2}\nabla u}^2\, .
	\end{gather}
	Now we can switch the integrals and get
	\begin{gather}
	 s^2\int_{\T_s\cap \B {r/200}{z}} \int \rho_s(y-x)\abs{\pi_{\bar z,r/2}\nabla u(y)}^2=\int \abs{\pi_{\bar z,r/2}\nabla u(y)}^2\int_{\T_s\cap \B {r/200}{ z}} s^2\rho_s(y-x)\, .
	\end{gather}
	By definition of $\rho_s(y-x)=s^{-m} \rho\ton{\frac{\abs{y-x}^2}{2s^2}}$, and since $\T_s$ is a Lipschitz graph over the $m-2$ dimensional plane $L_{\cA}$, we get
	\begin{gather}
	 \int_{\T_s\cap \B {r/200}{z}} s^2\rho_s(y-x)\leq C(m)r^m \rho_{r/2}(y-\bar z)\, .
	\end{gather}
	So far we have proved that
	\begin{gather}
	 s^2\int_{\T_s\cap \B {r/200}{z}} \psi_\T(x,s) \int \rho_s(y-x)\abs{\pi_{x,s}\nabla u}^2\leq C(m) r^m \int \rho_{r/2}(y-\bar z) \abs{\pi_{\bar z,r/2} \nabla u}^2=C(m) r^{m-2}\vartheta_\LL(\bar z,r/2)
	\end{gather}

	Considering that for all $x\in \T_r\cap \B {r/100}{\bar z}$: 
	\begin{enumerate}
	 \item $\T_r$ is an $m-2$ dimensional Lipschitz graph with estimates
	 \item $\vartheta_\LL(\bar z,r/2)\leq C(m)\vartheta_\LL(x,r)$
	 \item $\psi_\T(x,r)\geq \frac 1 2 $
	\end{enumerate}
    we conclude the proof of $(1)$.

	\vspace{5mm}

	\textbf{Proof of Estimate (2) .} This estimate follows from Theorem \ref{t:cone_splitting_annular} and some manipulations.

	The inequality
	\begin{gather}
	 \vartheta_\cL(\T_r)+\hat\vartheta_\alpha(\T_r)+\hat\vartheta_n(\T_r)\leq \vartheta_\cL(\T_r)+\vartheta_\alpha(\T_r)+\vartheta_n(\T_r)
	\end{gather}
	follows trivially from the definition of $\hat \vartheta_\alpha$ and $\hat \vartheta_n$.

	Moreover, we have
	\begin{align}
	 \vartheta_\cL(\T_r)+\vartheta_\alpha(\T_r)+\vartheta_n(\T_r) = &\int_{\T_r} \psi_\T(x,r) \qua{\vartheta_\LL(x,r)+\vartheta(x,r;\LL_{x,r}^\perp)}\notag \\\stackrel{\eqref{e:cone_splitting_annular_integrated}}{\leq}& C(m) \int_{\T_r} \psi_\T(x,r) \ton{r\dot\vartheta(x,3r)+\fint_{\T_r\cap B_r(x)} r\dot\vartheta(z,3r)}\notag \\
	 \leq & C(m)\int_{\B 3 p}  r\dot\vartheta(x,3r)\, ,
	\end{align}
	where the last inequality follows from the fact that $\T_r$ is a Lipschitz graph over $L_{\cA}$ with controlled Lipschitz constant. Using \eqref{e:integral_dot_vartheta_Tr_vs_T} we can the turn the last integral over $\T_r$ into an integral over $\T$, up to a enlarging the radius $3r$ by a small factor.
\end{proof}

\vspace{.3cm}

\subsection{Heat Mollified Energies on \texorpdfstring{$\cA$}{the annular region}}

Let us finally integrate up the energies along $\T_r$ in order to define the mollified energies along $\cA$:

\begin{align}
	\vartheta_\cL(\cA) &\equiv \int\vartheta_\cL(\T_r)\frac{dr}{r}\, ,\;\;\;\hat\vartheta_\alpha(\cA)\equiv \int\hat\vartheta_\alpha(\T_r)\frac{dr}{r}\, ,\;\;\;\hat\vartheta_n(\cA) \equiv \int\hat\vartheta_n(\T_r)\frac{dr}{r}\, .\\\notag
\end{align}

Our main result in this subsection is to check that the above energies do in fact control the full energy on $\cA$, precisely:

\begin{lemma}\label{l:energy_decomposition:full_energy_on_annulus}
Let $u:B_{10R}(p)\to N$ be a stationary harmonic map with $R^2\fint_{B_{10R}}|\nabla u|^2\leq \Lambda$ and let $\cA=B_2\setminus \overline{B_{\rf_z}(\T)}$ be a $\delta$-annular region with $\delta\leq \delta(m,\Lambda,K_N,R)$.  Then we have
\begin{align}
	\int_{\cA\cap B_1}|\nabla u|^2\leq C(m)\Big(\vartheta_\cL(\cA)+\hat \vartheta_\alpha(\cA)+\hat \vartheta_n(\cA)\Big)
\end{align}
\end{lemma}
\begin{proof}

	Set
	\begin{gather}
	 \cA_k = \cA \cap B_1 \cap \cur{y\in \R^m \ \ s.t. \ \ 2^{-k-1}< d(y,\T)\leq 2^{-k}} \, .
	\end{gather}
    We will show that there exists a constant $C(m)$ such that for all $r\in \qua{2^{-k+1},2^{-k+2}}$:
    \begin{gather}\label{e:annular_k_estimate}
     \int_{\cA_k} \abs{\nabla u}^2 \leq C(m) \qua{\vartheta_\LL (\T_r)+\hat \vartheta_\alpha (\T_r) + \hat \vartheta_n (\T_r)}\, .
    \end{gather}
	This proves the final estimate after summing over all $k\in \N$.

	In order to prove \eqref{e:annular_k_estimate}, we rely on Lemma \ref{l:energy_bounded_by_restricted_vartheta}. Fix $k$ and $r\in \qua{2^{-k+1};2^{-k+2}}$ and notice that
	for all $y\in \cA_k$ we have $\rf_y \leq 2^{-k}$. Given that $\abs{\nabla \rf}\leq C\delta$, we know for all $x\in \B {2^{-k}}{y}\cap B_{1.1}$ we have that $\rf_x\leq 2^{-k+1}$.  This in turn implies that for such $x$ we have $\psi_{\T}(x,r)= 1$. Thus we can estimate
	\begin{gather}
	 \int_{\cA_k} \abs{\nabla u}^2 \leq C(m)\int_{\T_r} \psi_{\T}(x,r) r^{2-m}\int_{B_{10r}(x)\setminus \B{10^{-2}r}{x+\LL_{x,r}}}\abs{\nabla u}^2\stackrel{\ref{l:energy_bounded_by_restricted_vartheta}}{\leq} C(m)\ton{\vartheta_\LL (\T_r)+\hat \vartheta_\alpha(\T_r)+\hat \vartheta_n(\T_r)}\,
	\end{gather}
	as claimed.
\end{proof}

\vspace{.5cm}

\section{Superconvexity of Angular Energy on Annular Regions}\label{s:angular_energy}

The goal of this Section is to make rigorous the methodology outlined in Section \ref{ss:outline:general:angular_energy} and prove the angular energy bounds in an annular region.  Our main result will be Theorem \ref{t:outline_general:angular_energy}, which we restate below:

\begin{theorem}[Angular Energy on Annular Regions]\label{t:angular_energy:angular_energy}
Let $u:B_{10R}(p)\to N$ be a stationary harmonic map with $R^2\fint_{B_{10R}}|\nabla u|^2\leq \Lambda$ , and let $\cA=B_2\setminus \overline{B_{\rf_z}(\T)}$ be a $\delta$-annular region.	 Then
\begin{align}
	\hat\vartheta_\alpha(\cA) \leq C(m,R,K_N,\Lambda)\,\sqrt\delta\, .
\end{align} 	
\end{theorem}
\vspace{.3cm}

The remainder of this Section is dedicated to proving the above by the outline of Section \ref{ss:outline:general:angular_energy}.

\vspace{.3cm}

\subsection{The \texorpdfstring{$S^1$}{S1} Averaged Angular Energy \texorpdfstring{$\cE_\alpha$}{}}

Let $x\in \T_r$ be fixed, and let us consider the affine best plane $x+\cL_{x,r}$.  Given $y\in \dR^m$ we can use this affine plane to write
\begin{align}
	y\to (y_{x,r},y^\perp_{x,r})\equiv (\pi_{x,r}(y-x),\pi_{x,r}^\perp(y-x))\in \cL_{x,r}\oplus \cL_{x,r}^\perp\, .
\end{align}
With respect to this decomposition we can write $y^\perp_{x,r}=(s_{x,r}^\perp,\alpha_{x,r}^\perp)=(|y^\perp_{x,r}|,\alpha^\perp_{x,r})$ in polar coordinates and define the averaged scale-invariant angular energy
\begin{align}
	e_\alpha(x,r,y) =&\cEn(y)=|y^\perp_{x,r}|^2\langle \nabla u(y), \alpha^\perp_{x,r}\rangle^2\, ,\\
	\cE_\alpha(x,r,y) =&\cEE(y)=|y^\perp_{x,r}|^2\fint_{\alpha_{x,r}^\perp \in S^1} \langle \nabla u(y), \alpha^\perp_{x,r}\rangle^2d\alpha_{x,r}^\perp\, ,\label{e:definition_of_cEE}
\end{align}
as in Section \ref{ss:outline:general:angular_energy}.  By construction, $\cEE$ is rotationally invariant around the affine plane $x+\LL_{x,r}$.
Our main computation in this subsection is to show that in the $y$ variable there is a form of uniform subharmonicity.  To state it recall as in Section \ref{ss:outline:general:angular_energy} that we can define the conformal Laplacian
\begin{align}
	\bar\Delta_{x,r} \equiv \ton{\frac{\partial}{\partial\ln s_{x,r}^\perp}}
	^2+ (s_{x,r}^\perp)^2\Delta_{\cL_{x,r}}\, ,
\end{align}
where $\Delta_{\cL_{x,r}}$ is the usual Laplacian in the $\cL_{x,r}$ directions.  Our first result is the following:\\

\begin{theorem}[Uniform Subharmonicity]\label{t:angular_energy:uniform_subharmonic}
Let $u:B_{10R}(p)\to N$ be a stationary harmonic map with $R^2\fint_{B_{10R}}|\nabla u|^2\leq \Lambda$ , and let $\cA=B_2\setminus \overline{B_{\rf_z}(\T)}$ be a $\delta$-annular region.	 For each $x\in \T_r$ with $\rf_x\leq r$ and each $y\in B_{Rr}(x)$ such that $|y^\perp_{x,r}|\in [e^{-R},R]r$ we have that
\begin{align}
	\bar\Delta_{x,r}\cEE(y) \geq \big(2-C(m)\sqrt \delta\,\big)\,\cEE(y)\geq \cEE(y)\, .
\end{align}
\end{theorem}
\begin{proof}
The proof is exactly as Lemma \ref{l:outline:toy_model:angular_energy:superconvexity} as we are treating $\cEE$ as a family of the toy models in the $y$-variable parametrized by $x\in \T_r$ . 
\end{proof}

\vspace{.3cm}

\subsection{Technical Estimates}

Throughout the course of this Section we will need a handful of technical estimates that appear so frequently that it seems worth proving them separately so that we may quote them at will.  The technical estimates of this subsection may be returned to as needed throughout the remainder of the Section.\\

\vspace{5mm}

Let us open with the following technical result:\\

\begin{lemma}\label{l:error_estimates_psi_with_pinching}
Let $u:B_{10R}(p)\to N$ be a stationary harmonic map with $R^2\fint_{B_{10R}}|\nabla u|^2\leq \Lambda$ , and let $\cA=B_2\setminus \overline{B_{\rf_z}(\T)}$ be a $\delta$-annular region.
We have
\begin{align}
	&\int_{\T_r} \ton{\psi_\T(x,r) + r\abs{\nabla \psi_{\T}(x,r)}+ r\abs{\dot \psi_\T(x,r)} + r^2 \abs{\ddot \psi_\T(x,r)}} \ton{\vartheta(x,2r;\LL_{x,r})+ \vartheta(x,2r;\LL_{x,r}^\perp)}\notag \\
	&\leq C(m)\int_{\T\cap \B 4 p} r\dot \vartheta(x,8r)\, .
\end{align}
\end{lemma}
\begin{remark}
 Notice that this estimate is very rough. Indeed, although one cannot replace the integral on the rhs with $\int_{\T}\psi_{\T}(x,r) r\dot \vartheta(x,8r)$, it is possible to integrate only on the set $\cur{x\in \T \ \ s.t. \ \ 3r\geq \rf_x}$ and still get a valid estimate. However a rough estimate is enough for the purposes of this Lemma.
\end{remark}

\begin{proof}

 On the one hand we have by Lemma \ref{l:energy_decomposition:T_cutoff}.\ref{i:T_cutoff_global_bounds}:
 \begin{gather}
  \psi_\T(x,r) + r\abs{\nabla \psi_{\T}(x,r)}+ r\abs{\dot \psi_\T(x,r)}+r^2\abs{\ddot \psi_\T(x,r)}\leq C(m)\, ,
 \end{gather}
and all these functions are supported on $\B {2} p$.
Notice also that for all $x\in \T_r\cap \supp{\psi_{\T}(x,r)}$, we have $3r\geq \rf_x$.  We can apply Remark \ref{r:d_LLr_LL2r_small} and Remark  \ref{rm:cone_splitting_annular_V} 
to obtain
\begin{align}
 &\int_{\T_r} \ton{\psi_\T(x,r) + r\abs{\nabla \psi_{\T}(x,r)}+ r\abs{\dot \psi_\T(x,r)}} \qua{\vartheta(x,2r;\LL_{x,r})+ \vartheta(x,2r;\LL_{x,r}^\perp)}\notag \\
 &\leq C(m) \int_{\T_r \cap \cur{3r\geq \rf_x} } \vartheta(x,2r;\LL_{x,r})+ \vartheta(x,2r;\LL_{x,r}^\perp)\notag \\
 &\leq C(m) \int_{\T_r \cap \cur{3r\geq \rf_x} } \qua{r \dot \vartheta(x,4r) + \fint_{\B r x \cap \T_r} r\dot \vartheta(y,4r) }\leq C(m) \int_{\T_r \cap \B {3} p} r\dot \vartheta(x,4r)
\end{align}
 Now we can conclude using \eqref{e:integral_dot_vartheta_Tr_vs_T} to exchange the integral over $\T_r$ with an integral over $\T$.
\end{proof}

We also record some rough and ugly looking bounds on $\cEn$ and related quantities. \\

\begin{lemma}[Estimates on $\cEn$]\label{l:angular_energy:Ena_estimates}
Let $u:B_{10R}(p)\to N$ be a stationary harmonic map with $R^2\fint_{B_{10R}}|\nabla u|^2\leq \Lambda$ , and let $\cA=B_2\setminus \overline{B_{\rf_z}(\T)}$ be a $\delta$-annular region.	For each $x\in \T_r$ with $\rf_x\leq 2r$, we have that
\begin{align}
 &r^{-1}\int \abs{\nabla^{(x)}\ton{|s_{x,r}^\perp|^2\hat\rho_r(y-x)}}\,\cEn(y)\leq C(m,R)  \vartheta(x,2r;\LL_{x,r}^\perp)\, ,\\
 &\int \qua{\hat \rho_{r}(y-x;\LL_{x,r})+\abs{r\nabla^{(y)} \hat \rho_r(y-x;\LL_{x,r})}} \qua{\abs{r\frac{d}{dr} \cEn(y)} + \abs{r \nabla^{(x)} \cEn(y)}} \notag \\
 &\leq C(m,R)  \sqrt \delta \qua{\vartheta(x,2r;\LL_{x,r})+ \vartheta(x,2r;\LL_{x,r}^\perp)}\, ,\\
 &\int \abs{\nabla^{(x)}\ton{|s_{x,r}^\perp|^2\hat\rho_r(y-x)}}\,\abs{\nabla^{(x)}\cEn(y)}\leq C(m,R) \sqrt \delta \qua{\vartheta(x,2r;\LL_{x,r})+ \vartheta(x,2r;\LL_{x,r}^\perp)}
\end{align}
\end{lemma}
\begin{proof}
 Although a more refined estimate is possible, we can also estimate less carefully by considering the following. All spacial gradients in $x$ and $r$ derivatives of $\cEn(y)$ act only on the vectors $\pi_{x,r}^\perp (y-x)$ and $\alpha_{x,r}$, and not on $\nabla u(y)$. Since $\hat \rho_{r}(y-x;\LL_{x,r})$ is supported in the set
 \begin{gather}
  \supp{\hat \rho_{r}(y-x;\LL_{x,r})} \subseteq \B {(R+3)r}{x}\setminus \B{e^{-5R} r} {x+\LL_{x,r}}\, ,
 \end{gather}
 by Theorem \ref{t:best_plane:best_plane}.\ref{i:best_plane_bounds_on_projections} we have that for all $y$ in this support:
 \begin{gather}
  \abs{r\frac{d}{dr} \cEn(y)} + \abs{r \nabla^{(x)} \cEn(y)}\leq C(m,R) r^2\sqrt \delta \abs{\nabla u(y)}^2\notag \\
  \leq C(m,R) \sqrt \delta \qua{r^2\abs{\pi_{x,r} \nabla u(y)}^2 + C(R) \abs{\pi_{x,r}^\perp(y-x)}^2\abs{\pi_{x,r}^\perp \nabla u(y)}^2}
 \end{gather}
With a similar reasoning, we can bound the support of $\hat \rho$:
\begin{gather}
 r^{-1}\abs{\nabla^{(x)} \abs{s_{x,r}^\perp}^2 }=r^{-1}\abs{\nabla^{(x)} \abs{\pi_{x,r}^\perp (y-x)}^2 }\leq C(m,R)\, .
\end{gather}
Since we have
\begin{gather}
 \hat \rho_{r}(y-x;\LL_{x,r})+\abs{r\nabla^{(y)} \hat \rho_r(y-x;\LL_{x,r})}+\abs{r\nabla^{(x)} \hat \rho_r(y-x;\LL_{x,r})}\leq C(m) \rho_{2r}(y-x)\, ,
\end{gather}
we conclude the result.
\end{proof}

\vspace{.3cm}

\subsection{The Superconvexity of \texorpdfstring{$\hat\vartheta_\alpha(\T_r)$}{the restricted angular energy}}

Our main application of the subharmonic estimate of Theorem \ref{t:angular_energy:uniform_subharmonic} will be to prove a superconvexity estimate on $\hat\vartheta_\alpha(\T_r)$:

\begin{proposition}[Superconvexity for $\hat\vartheta_\alpha(\T_r)$]\label{p:angular_energy:superconvexity} Let $u:B_{10R}(p)\to N$ be a stationary harmonic map with $R^2\fint_{B_{10R}}|\nabla u|^2\leq \Lambda$ , and let $\cA=B_2\setminus \overline{B_{\rf_z}(\T)}$ be a $\delta$-annular region with $ \rf_x\geq r_0>0$ .  Then we can estimate
\begin{align}
	\ton{r\frac{\partial}{\partial r}}^2\hat\vartheta_\alpha(\T_r) \geq \,\hat\vartheta_\alpha(\T_r)-\err(r)-r\frac{d}{dr} \epsilon_1(r)\, ,
\end{align}
where $\err(r), \epsilon_1(r)$ are $C^1_c((0,\infty))$ function,
and $\err(r)$ satisfies
\begin{gather}
 \int \abs{\err(r)}\frac{dr}{r}\leq C(m,R,K_N,\Lambda)\sqrt\delta\, .
\end{gather}
\end{proposition}
\begin{remark}
We do not prove estimates on $r\frac{d}{dr}\epsilon_1(r)$ as we do not need them, however one can prove estimates in the verbatim method one proves estimates on $\err(r)$. 	
\end{remark}
\begin{remark}
	The assumption $\rf_x\geq r_0>0$ is not necessarily and saves some technical footwork on avoiding proving direct estimates on $r\frac{d}{dr} \epsilon_1(r)$.  If one proves estimates on $r\frac{d}{dr} \epsilon_1(r)$ then its fully possible to drop this assumption.
\end{remark}

As we saw at the end of the outline in Section \ref{ss:outline:general:radial_energy}, the above implies Theorem \ref{t:angular_energy:angular_energy} via the integration $\int \frac{dr}{r}$. In particular, as a corollary we obtain
\begin{corollary}[Uniform integrability for $\hat \vartheta_\alpha(\T_r)$]\label{c:angular_energy:final_estimate} Let $u:B_{10R}(p)\to N$ be a stationary harmonic map with $R^2\fint_{B_{10R}}|\nabla u|^2\leq \Lambda$ , and let $\cA=B_2\setminus \overline{B_{\rf_z}(\T)}$ be a $\delta$-annular region. Then
\begin{gather}
 \hat \vartheta_\alpha= \int \hat \vartheta_\alpha(\T_r)\frac{dr}{r} \leq C(m,R,K_N,\Lambda)\sqrt \delta\, .
\end{gather}
\end{corollary}
\begin{proof}
Notice that the assumption $\rf_x\geq r_0>0$ in Proposition \ref{p:angular_energy:superconvexity} is not restrictive, since all the estimates are independent of $r_0$.  More precisely, if $\cA$ is a $\delta$ annular region wrt $\T$ and $\rf_x$, then we can define the $C(m)\delta$-annular region $\cA'$ if we take $\T'=\T$ and
\begin{gather}
 \rf_x'=\begin{cases}
         \rf_x & \text{if} \ \ \rf_x> 2r_0\, ,\\
         g(x) & \text{if} \ \ r_0\leq  \rf_x\leq 2r_0\, ,\\
         r_0 & \text{if} \ \ \rf_x<r_0\, .
        \end{cases}
\end{gather}
where $g(x)\in [r_0,2r_0]$ is chosen so that $\rf_x'\in C^2$ with 
\begin{gather}
 \abs{\nabla \rf_x'}+\rf_x'\abs{\nabla^2 \rf_x'}\leq C(m)\delta\, .
\end{gather}

By applying our results to the annular regions $\cA'$ and letting $r_0\to 0$, we obtain the integrability results for the original $\cA$.

 To prove this Corollary, we note then that by Proposition \ref{p:angular_energy:superconvexity} that we have
 \begin{gather}
  \int \hat \vartheta_\alpha(\T_r)\frac{dr}{r}\leq \underbrace{\int \ton{r\frac{d}{dr}}^2\hat \vartheta_\alpha(\T_r)\frac{dr}{r}}_{=0}
  +\int \err(r) \frac {dr}{r} + \underbrace{\int r\frac{d}{dr} \epsilon_1(r) \frac{dr}{r}}_{=0}\leq C(m,R,K_N,\Lambda)\sqrt \delta\, .
 \end{gather}
where both $\hat \vartheta_{\alpha}(\T_r)$ and $\epsilon_1(r)$ are $C^\infty_c((0,\infty))$ functions, so the related integrals vanish trivially.
\end{proof}

The remainder of this Section will be focused on proving Proposition \ref{p:angular_energy:superconvexity} itself.\\

\textbf{Notation:} In the following estimates, as it is standard, $C(...)$ denotes a constant that can vary from line to line. In a similar fashion $\err(r)$ denotes a function of $r$ that can vary from line to line, but always satisfies
\begin{gather}
 \int \abs{\err(r)} \frac{dr}{r} \leq C(m,R,K_N,\Lambda) \sqrt \delta\, ,
\end{gather}
possibly with a different constant $C$.  For example one such function is $  \sqrt \delta \int_{\T\cap \B 2 p} r\dot \vartheta(x,8r)$, because it is immediate to see that
 \begin{gather}
  \sqrt \delta \int_{\T\cap \B 4 p} \qua{\int_0^2 r\dot \vartheta(x,r)\frac{dr}{r}}\leq C(m)\Lambda \sqrt \delta\, .
 \end{gather}\\

\textbf{Regularity for $u$ on the support of $\hat \rho_r(y-x)$}.  Note by Remark \ref{rm:annular:regularity} that if $\delta\leq \delta_0(m,R)$ and $\psi_\T(x,r)>0$, then on $\supp{\hat \rho_r(y-x;\LL_{x,r}}$ we have the pointwise estimate $r^2|\nabla u|^2\leq C(m,R)\delta$ .  Notice that this regularity fails to be true on the support of $\rho_r(y-x)$ .

\vspace{.3cm}
\subsection{Computing \texorpdfstring{$r\frac{d}{dr}\hat\vartheta_\alpha(\T_r)$}{the first radial derivative}}

For the first derivative of $\hat \vartheta_\alpha(\T_r)$ we compute
\begin{lemma}\label{l:angular_energy:ddrtheta}
	Under the assumptions of Proposition \ref{p:angular_energy:superconvexity}, we have that
	\begin{align}
	 r\frac{d}{dr}\hat\vartheta_\alpha(\T_r) &= \int_{\T_r}\psi_\T(x,r)\int \hat\rho_r(y-x;\cL_{x,r})\Big(s_{x,r}^\perp\,\nabla_{s_{x,r}^\perp}\Big)\cEn(y)+\epsilon_1(r)\, ,
	\end{align}
	where $s_{x,r}^\perp\,\nabla_{s_{x,r}^\perp}\cEn = \nabla^{(y)}_{\pi^\perp_{x,r}(y-x)}\cEn$ is the $\LL_{x,r}^\perp$-radial derivative and $\epsilon_1(r)$ is a $C^1$ function compactly supported in $(0,\infty)$
\end{lemma}
\vspace{.3cm}


\begin{remark}
 Since we do not need them, we do not prove any estimates on $\epsilon_1(r)$. However, we point out that with the same techniques that will be employed for the computation of $\ton{r\frac {d}{dr}}^2 \hat \vartheta_\alpha(\T_r)$, one can prove that $\epsilon_1(r)$ has the same estimates as $\err(r)$.
\end{remark}

\vspace{.3cm}
\subsubsection{Term by Term Computation of $r\frac{d}{dr}\hat\vartheta_\alpha(\T_r)$ }

Let us begin with the straightforward computation of $r\frac{d}{dr}\hat\vartheta_\alpha(\T_r)$ given by

\begin{align}
	r\frac{d}{dr}\hat\vartheta_\alpha(\T_r)=&\int_{\T_r}\ton{r\frac{d}{dr}\psi_\T}\int\hat\rho_r \cEn+\int_{\T_r}\psi_\T\int\ton{r\frac{d}{dr}\hat\rho_r} +\int_{\T_r}\psi_\T\int\hat\rho_r \ton{r\frac{d}{dr}\cEn}\, .
\end{align}
For convenience, we split the total $r$ derivative of $\hat\rho_r$ as
\begin{gather}
 r\frac{d}{dr}\hat\rho_r(y-x;\cL_{x,r})=r\frac{d}{dr}\ton{r^{-m} \rho\ton{\frac{\abs{y-x}^2}{2r^2}}\hat \psi_R\ton{\frac{\abs{\pi_{x,r}(y-x)}^2}{2r^2}}  }\notag \\
 =\left.r\frac{\partial}{\partial r}\ton{r^{-m} \rho\ton{\frac{\abs{y-x}^2}{2r^2}}\hat \psi_R\ton{\frac{\abs{\pi_{\cL}(y-x)}^2}{2r^2}}  }\right\vert_{\cL=\cL_{x,r}} + \left.r\frac{\partial}{\partial \cL}\ton{r^{-m} \rho\ton{\frac{\abs{y-x}^2}{2r^2}}\hat \psi_R\ton{\frac{\abs{\pi_{\cL}(y-x)}^2}{2r^2}}  }\right\vert_{\cL=\cL_{x,r}}\cdot \frac{d\cL_{x,r}}{dr}\notag \\[5pt]
 \equiv r\partial_r\hat\rho_r + \langle \partial_\cL\hat\rho_r,r\dot\cL_{x,r}\rangle\, .
\end{gather}

Thus we have
\begin{align}
	r\frac{d}{dr}\hat\vartheta_\alpha(\T_r)=&\int_{\T_r}\ton{r\frac{d}{dr}\psi_\T}\int\hat\rho_r \cEn+\int_{\T_r}\psi_\T\int\big(r\partial_r\hat\rho_r\big) \cEn\notag\\
	&+\int_{\T_r}\psi_\T\int\langle \partial_\cL\hat\rho_r,r\dot\cL_{x,r}\rangle \cEn+\int_{\T_r}\psi_\T\int\hat\rho_r \ton{r\frac{d}{dr}\cEn}
\end{align}

Let us now regroup these terms in a leading fashion: 

\begin{align}
	r\frac{d}{dr}\hat\vartheta_\alpha(\T_r)&=\int_{\T_r}\psi_\T(x,r)\int\hat\rho_r(y-x;\cL_{x,r})\,\Bigg(-2-\Big\langle\nabla^{(y)}\ln\hat\rho_r,\pi_{x,r}^{\perp}(y-x)\Big\rangle\,\Bigg)\,\cEn\\
	&\;\;+\int_{\T_r}\psi_\T(x,r)\int\hat\rho_r(y-x;\cL_{x,r})\Big((r\partial_r\ln\hat\rho_r)+m+\Big\langle\nabla^{(y)}\ln\hat\rho_r,y-x\Big\rangle\Big)\,\cEn\\
	&-\int_{\T_r}\psi_\T(x,r)\int\hat\rho_r(y-x;\cL_{x,r})\Big(m-2+\Big\langle\nabla^{(y)}\ln\hat\rho_r,\pi_{x,r}(y-x)\Big\rangle\,\Big)\\
	&+\int_{\T_r}\ton{r\frac{d}{dr}\psi_\T}\int\hat\rho_r\,\cEn+\int_{\T_r}\psi_\T\int \langle \partial_{\cL}\hat\rho_r,r\dot\cL_{x,r}\rangle\cEn+\int_{\T_r}\psi_\T\int\hat\rho_r \ton{r\frac{d}{dr}\cEn}\, ,\notag\\
	&\equiv\int_{\T_r}\psi_\T(x,r)\int\hat\rho_r(y-x;\cL_{x,r})\,\Bigg(-2-\Big\langle\nabla^{(y)}\ln\hat\rho_r,\pi_{x,r}^{\perp}(y-x)\Big\rangle\,\Bigg)+\epsilon_1(r)\, .
\end{align}

Notice that by the definition of $\psi_\T(x,r)$ and smoothness properties of all the functions involved in the computation, all of these terms are $C^1_c((0,\infty))$. In particular, the smoothness of $\cEn$ comes from the $\epsilon$-regularity theorem. The only term we need to manipulate is the first one.\\

{\bf Claim: } We have that
\begin{gather}\label{e:claim_cEn_first_estimate}
\int_{\T_r}\psi_\T\int\hat\rho_r\,\Bigg(-2-\Big\langle\nabla^{(y)}\ln\hat\rho_r,\pi_{x,r}^{\perp}(y-x)\Big\rangle\,\Bigg)\cEn = \int_{\T_r}\psi_\T\int \hat\rho_r\,\Big(s_{x,r}^\perp\,\nabla_{s_{x,r}^\perp}\Big)\cEn\, .
\end{gather}

\begin{proof}[Proof of claim]
 The proof is just an integration by parts wrt $y$. In particular, we can estimate
 \begin{gather}
  \int_{\T_r}\psi_\T\int\hat\rho_r\,\Big\langle\nabla^{(y)}\ln\hat\rho_r,\pi_{x,r}^{\perp}(y-x)\Big\rangle\,\cEn =\int_{\T_r}\psi_\T\int\,\Big\langle\nabla^{(y)}\hat\rho_r,\pi_{x,r}^{\perp}(y-x)\Big\rangle\,\cEn = \\
  =-2 \int_{\T_r}\psi_\T\int\,\hat \rho_r \cEn -\int_{\T_r}\psi_\T\int \hat\rho_r\,\Big(s_{x,r}^\perp\,\nabla_{s_{x,r}^\perp}\Big)\cEn\, .
 \end{gather}
\end{proof}

This concludes the proof of Lemma \ref{l:angular_energy:ddrtheta}.

\vspace{.3cm}
\subsection{Computing \texorpdfstring{$r\frac{d^2}{dr^2}\hat\vartheta_\alpha(\T_r)$}{the second radial derivative}}
Computing and estimating the second derivative of $\hat \vartheta_\alpha(\T_r)$ is very similar to the previous lemma, but this time we do need to have careful estimates on the error terms, which is the involved part.

\begin{lemma}\label{l:angular_energy:ddrtheta_2}
	We have that
\begin{align}
	\ton{r\frac{d}{dr}}^2\hat\vartheta_\alpha(\T_r) &= \int_{\T_r}\psi_\T(x,r)\int \hat\rho_r(y-x;\cL_{x,r})\Big(s_{x,r}^\perp\,\nabla_{s_{x,r}^\perp}\Big)^2\cEn(y)+\err(r)+r\frac{d}{dr} \epsilon_1(r)\, ,
\end{align}
where $s_{x,r}^\perp\,\nabla_{s_{x,r}^\perp}\cEn=\nabla^{(y)}_{\pi^\perp_{x,r}(y-x)}\cEn$ is the $\LL_{x,r}^\perp$-radial derivative and $\epsilon_1(r),\err(r)$ are smooth functions compactly supported in $(0,\infty)$
satisfying
\begin{gather}
\int \abs{\err(r)}\frac{dr}{r}\leq C(m,\epsilon_0,R,\Lambda)\sqrt\delta\, .
\end{gather}
\end{lemma}
\vspace{.3cm}

We will break the proof down into several steps. For convenience we introduce the notation
\begin{gather}
 \cFn(y)=s_{x,r}^\perp\,\nabla_{s_{x,r}^\perp}\cEn=\nabla^{(y)}_{\pi^\perp_{x,r}(y-x)}\cEn(y)=\ps{\nabla^{(y)}\cEn(y)}{\pi^\perp_{x,r}(y-x)}\, .
\end{gather}

\vspace{.3cm}
\subsubsection{Term by Term Computation of $\ton{r\frac{d}{dr}}^2\hat\vartheta_\alpha(\T_r)$ }

Let us begin with the straightforward computation of $\ton{r\frac{d}{dr}}^2\hat\vartheta_\alpha(\T_r)$. Using Lemma \ref{l:angular_energy:ddrtheta} we get

\begin{align}
	\ton{r\frac{d}{dr}}^2\hat\vartheta_\alpha(\T_r)&=\int_{\T_r}\ton{r\frac{d}{dr}\psi_\T}\int\hat\rho_r \cFn+\int_{\T_r}\psi_\T\int\big(r\partial_r\hat\rho_r\big) \cFn\notag \\
	&+\int_{\T_r}\psi_\T\int\langle \partial_\cL\hat\rho_r,r\dot\cL_{x,r}\rangle \cFn+\int_{\T_r}\psi_\T\int\hat\rho_r \ton{r\frac{d}{dr}\cFn} + r\frac{d}{dr} \epsilon_1(r)\, .
\end{align}

Let us now regroup these terms in a leading fashion, as done for the first derivative: 

\begin{align}
	\ton{r\frac{d}{dr}}^2\hat\vartheta_\alpha(\T_r)&=\int_{\T_r}\psi_\T(x,r)\int\hat\rho_r(y-x;\cL_{x,r})\,\Bigg(-2-\Big\langle\nabla^{(y)}\ln\hat\rho_r,\pi_{x,r}^{\perp}(y-x)\Big\rangle\,\Bigg)\,\cFn(y)\label{e:angular_energy:drtheta:1}\\
	&\;\;+\int_{\T_r}\psi_\T(x,r)\int\hat\rho_r(y-x;\cL_{x,r})\Big((r\partial_r\ln\hat\rho_r)+m+\Big\langle\nabla^{(y)}\ln\hat\rho_r,y-x\Big\rangle\Big)\,\cFn(y)\label{e:angular_energy:drtheta:2}\\
	&-\int_{\T_r}\psi_\T(x,r)\int\hat\rho_r(y-x;\cL_{x,r})\Big(m-2+\Big\langle\nabla^{(y)}\ln\hat\rho_r,\pi_{x,r}(y-x)\Big\rangle\,\Big)\cFn(y)\label{e:angular_energy:drtheta:3}\\
	&+\int_{\T_r}\ton{r\frac{d}{dr}\psi_\T}\int\hat\rho_r\,\cFn(y)+\int_{\T_r}\psi_\T\int \langle \partial_{\cL}\hat\rho_r,r\dot\cL_{x,r}\rangle\cFn(y)+\int_{\T_r}\psi_\T\int\hat\rho_r \ton{r\frac{d}{dr}\cFn(y)}\,\label{e:angular_energy:drtheta:4}\, \\
	&+r\frac{d}{dr}\epsilon_1(r)\, .
\end{align}

 We will split the remainder of the subsection into pieces where we deal with each of these terms individually, but before we do that let us briefly remark that all terms except \eqref{e:angular_energy:drtheta:1} are going to be small errors or boundary terms.


\vspace{.3cm}

\subsubsection{Computing \eqref{e:angular_energy:drtheta:1}} 
The claim we wish to address here is the equivalent of Claim \eqref{e:claim_cEn_first_estimate}. In particular:\\

{\bf Claim: } We have that
\begin{gather}
\int_{\T_r}\psi_\T\int\hat\rho_r\,\Bigg(-2-\Big\langle\nabla^{(y)}\ln\hat\rho_r,\pi_{x,r}^{\perp}(y-x)\Big\rangle\,\Bigg)\cFn(y) = \int_{\T_r}\psi_\T\int \hat\rho_r\,\Big(s_{x,r}^\perp\,\nabla_{s_{x,r}^\perp}\Big)\cFn(y)\, .
\end{gather}

\begin{proof}[Proof of claim]
 As with \eqref{e:claim_cEn_first_estimate}, the proof is just an integration by parts wrt $y$. In particular, we can write
 \begin{gather}
  \int_{\T_r}\psi_\T\int\hat\rho_r\,\Big\langle\nabla^{(y)}\ln\hat\rho_r,\pi_{x,r}^{\perp}(y-x)\Big\rangle\,\cFn(y) =\int_{\T_r}\psi_\T\int\,\Big\langle\nabla^{(y)}\hat\rho_r,\pi_{x,r}^{\perp}(y-x)\Big\rangle\,\cFn(y) = \\
  =-2 \int_{\T_r}\psi_\T\int\,\hat \rho_r \cFn(y) -\int_{\T_r}\psi_\T\int \hat\rho_r\,\Big(s_{x,r}^\perp\,\nabla_{s_{x,r}^\perp}\Big)\cFn(y)\, .
 \end{gather}
 \end{proof}

\vspace{.3cm}
\subsubsection{Computing \eqref{e:angular_energy:drtheta:2}}

This part is flat out zero, indeed we claim that \\

{\bf Claim: } We have
\begin{gather}
 (r\partial_r\ln\hat\rho_r)+m+\Big\langle\nabla^{(y)}\ln\hat\rho_r,y-x\Big\rangle=0\, .
\end{gather}

\begin{proof}[Proof of claim.]
Recall that $\hat \rho_r = \hat\rho_r(x;L)$, and so our partial $r$-derivative is only hitting the $r$-component of this.  From the definition of $\rho$ and $\hat \rho$, in particular Definitions \ref{d:heat_mollifier}, \ref{d:restricted_energy_functionals}, we have
\begin{gather}
 \hat\rho_r (y-x;L)= r^{-m} \rho \ton{\frac{\abs{y-x}^2} {2r^2}} \psi \ton{e^{2R} \frac{\abs{\pi_{L}^\perp (y-x)}^2 }{2r^2}}\, .
\end{gather}
Checking the claim is then a straightforward computation.

\end{proof}

\vspace{.3cm}
\subsubsection{Computing \eqref{e:angular_energy:drtheta:3}}

The main claim we wish to address here is the following:\\

{\bf Claim: } We have that 
\begin{gather}\label{e:claim_superconvexity_3}
\int_{\T_r}\psi_\T(x,r)\int\hat\rho_r(y-x;\cL_{x,r})\Big(m-2+\Big\langle\nabla^{(y)}\ln\hat\rho_r,\pi_{x,r}(y-x)\Big\rangle\,\Big)\cFn(y)=\err(r) \, .
\end{gather}\

\begin{proof}[Proof of claim.] The idea is to perform an integration by parts, but with respect to the $x$ variable being integrated over $\T_r$ . First, we expand
\begin{gather}
 \int_{\T_r}\psi_\T(x,r)\int\hat\rho_r(y-x;\cL_{x,r})\Big(m-2+\Big\langle\nabla^{(y)}\ln\hat\rho_r,\pi_{x,r}(y-x)\Big\rangle\,\Big)\cFn(y)=\\
 =(m-2)\int_{\T_r}\psi_\T(x,r)\int\hat\rho_r(y-x;\cL_{x,r})\cFn(y) + \int_{\T_r}\psi_\T(x,r)\int \ps{\nabla^{(y)}\hat\rho_r(y-x;\cL_{x,r})}{\pi_{x,r}(y-x)}\cFn(y)
\end{gather}

\paragraph{Part 1: exchanging $x$ and $y$ variables.} We claim that
\begin{align}
 &\int_{\T_r}\psi_\T(x,r)\int \ps{\nabla^{(y)}\hat\rho_r(y-x;\cL_{x,r})}{\pi_{x,r}(y-x)}\cFn(y)\notag\\
 =-&\int_{\T_r}\psi_\T(x,r)\int \ps{\nabla^{(x)}\hat\rho_r(y-x;\cL_{x,r})}{\pi_{x,r}(y-x)}\cFn(y)+\err(r)\label{e:superconvexity_error_1}\, .
\end{align}
This is, we can exchange $y$ and $x$ variables in the gradient up to a change in sign and an error $\err$.

\begin{proof}[Proof of subclaim]

Notice that by symmetry in the $x$ and $y$ variables
\begin{align}
\ps{\nabla^{(y)}\hat\rho_r}{\pi_{x,r}(y-x)} =& r^{-m} \ps{\nabla^{(y)} \rho\ton{\frac{\abs{y-x}^2}{2r^2}} }{\pi_{x,r}(y-x)}\psi \ton{e^{2R} \frac{\abs{\pi_{x,r}^\perp (y-x)}^2 }{2r^2}}\\
=&-r^{-m} \ps{\nabla^{(x)} \rho\ton{\frac{\abs{y-x}^2}{2r^2}} }{\pi_{x,r}(y-x)}\psi \ton{e^{2R} \frac{\abs{\pi_{x,r}^\perp (y-x)}^2 }{2r^2}}=\\
=&-r^{-m} \ps{\nabla^{(x)} \hat \rho_r(y-x,\LL_{x,r})}{\pi_{x,r}(y-x)}\\
+&\rho_r(y-x)\ps{\pi_{x,r}(y-x)}{\nabla^{(x)}\psi \ton{e^{2R} \frac{\abs{\pi_{x,r}^\perp (y-x)}^2 }{2r^2}}}\, .\label{eq_nablahatrho}
\end{align}
Here we exploited the fact that
\begin{gather}
 \ps{\nabla^{(y)}\psi \ton{e^{2R} \frac{\abs{\pi_{x,r}^\perp (y-x)}^2 }{2r^2}}}{\pi_{x,r}(y-x)}=0\neq \ps{\nabla^{(x)}\psi \ton{e^{2R} \frac{\abs{\pi_{x,r}^\perp (y-x)}^2 }{2r^2}}}{\pi_{x,r}(y-x)}\, ,
\end{gather}
since $\nabla^{(x)}$ also hits $\pi_{x,r}^\perp $, as opposed to $\nabla^{(y)}$.

Let us see that the last term in \eqref{eq_nablahatrho} gives rise to an $\err$ term.
In order to prove it, set for notational convenience
\begin{gather}
 t=\frac{\abs{\pi_{x,r}^\perp (y-x)}^2 }{2r^2}\, , \qquad v=\pi_{x,r}(y-x)\, , \qquad w=\pi_{x,r}^\perp(y-x)\, .
\end{gather}
Recall that ${{s_{x,r}^\perp\, \nabla_{s_{x,r}^\perp}^{(y)} f }}=\nabla^{(y)}_w f=\ps{\nabla^{(y)} f}{w}$. We have
\begin{gather}
\ps{\pi_{x,r}(y-x)}{\nabla^{(x)}\psi \ton{e^{2R} \frac{\abs{\pi_{x,r}^\perp (y-x)}^2 }{2r^2}}} = \nabla_v^{(x)}\psi \ton{e^{2R} t} =\dot \psi \ton{e^{2R} t}\frac{e^{2R}}{r^2}\ps{w}{\nabla_v(\pi_{x,r}^\perp) (y-x)}\, .
\end{gather}
By plugging in this error term in \eqref{e:superconvexity_error_1}, we obtain
\begin{align}
 &\int_{\T_r}\psi_\T(x,r)\int \rho_r(y-x)\dot \psi \ton{e^{2R} t}\frac{e^{2R}}{r^2}\ps{w}{\nabla_v(\pi_{x,r}^\perp) (y-x)} \cFn(y)\notag\\
 =-&\int_{\T_r}\psi_\T(x,r)\int {{s_{x,r}^\perp \nabla_{s_{x,r}^\perp}^{(y)}  }}\qua{\rho_r(y-x)\dot \psi \ton{e^{2R} t}\frac{e^{2R}}{r^2}\ps{w}{\nabla_v(\pi_{x,r}^\perp) (y-x)}} \cEn\\
-2&\int_{\T_r}\psi_\T(x,r)\int \rho_r(y-x)\dot \psi \ton{e^{2R} t}\frac{e^{2R}}{r^2}\ps{w}{\nabla_v(\pi_{x,r}^\perp) (y-x)} \cEn
\end{align}
Given the definitions of $\rho_r$ and $\psi$, and the estimates on $\abs{\nabla \pi_{x,r}}=\abs{\nabla \pi_{x,r}^\perp}$ of Theorem \ref{t:best_plane:best_plane}.\ref{i:best_plane_bounds_on_projections}, we have
\begin{gather}
 \abs{{{s_{x,r}^\perp \nabla_{s_{x,r}^\perp}^{(y)}  }}\qua{\rho_r(y-x)\dot \psi \ton{e^{2R} t}\frac{e^{2R}}{r^2}\ps{w}{\nabla_v(\pi_{x,r}^\perp) (y-x)}}}\leq C(m,\epsilon_0,R)\sqrt \delta\, \rho_{2r}(y-x)\, .
\end{gather}
This allows us to estimate that the error term is bounded in absolute value by
\begin{gather}
 C(m,\epsilon_0,R)\sqrt \delta \int_{\T_r}\psi_\T(x,r) \vartheta(x,2r;\LL_{x,r}^\perp)\stackrel{\text{Lemma \ref{l:error_estimates_psi_with_pinching}}}{\leq }C(m,\epsilon_0,R)\sqrt \delta \int_{\T\cap \B 2 p }r \dot \vartheta(x,8r)=\err(r)\, .
\end{gather}
This concludes Part 1 of the estimate.

\end{proof}

\textbf{Part 2: exchanging $\pi_{x,r}$ with $\pi_{\T_r(x)}$.} We claim that
\begin{align}
 &-\int_{\T_r}\psi_\T(x,r)\int \ps{\nabla^{(x)}\hat\rho_r(y-x;\cL_{x,r})}{\pi_{x,r}(y-x)}\cFn(y)\notag \\
 =&-\int_{\T_r}\psi_\T(x,r)\int \ps{\nabla^{(x)}\hat\rho_r(y-x;\cL_{x,r})}{\pi_{\T_r(x)}(y-x)}\cFn(y)+\err\, ,
\end{align}
in other words we can exchange $\pi_{x,r}(y-x)$ with $\pi_{\T_r}(y-x)$, up to an error $\err$.

\begin{proof}[Proof of subclaim]

The key estimate is that, by Lemma \ref{l:improved_comparison_LL_T}, on the support of $\rho_r$ we have
\begin{gather}\label{e:superconvexity_key_2}
 \abs{\pi_{x,r}(y-x)-\pi_{\T_r(x)}(y-x)}\leq C(m,\epsilon_0) Rr \sqrt \delta\, .
\end{gather}
As a consequence
\begin{align}
 &\int_{\T_r}\psi_\T(x,r)\int \ps{\nabla^{(x)}\hat\rho_r(y-x;\cL_{x,r})}{\pi_{\T_r(x)}(y-x)-\pi_{x,r}(y-x)}\cFn(y)\notag\\
 =& \int_{\T_r}\psi_\T(x,r)\int \ps{\nabla^{(x)}\hat\rho_r(y-x;\cL_{x,r})}{\pi_{\T_r(x)}(y-x)-\pi_{x,r}(y-x)}{{s_{x,r}^\perp \nabla_{s_{x,r}^\perp}^{(y)}  }}\cEn\notag\\
 =& - 2\int_{\T_r}\psi_\T(x,r)\int \ps{\nabla^{(x)}\hat\rho_r(y-x;\cL_{x,r})}{\pi_{\T_r(x)}(y-x)-\pi_{x,r}(y-x)}\cEn\, \\
 & - \int_{\T_r}\psi_\T(x,r)\int {{s_{x,r}^\perp \nabla_{s_{x,r}^\perp}^{(y)}  }}\ps{\nabla^{(x)}\hat\rho_r(y-x;\cL_{x,r})}{\pi_{\T_r(x)}(y-x)-\pi_{x,r}(y-x)}\cEn\, .\label{e:superconvexity_error_2}
\end{align}
Following the same computations as in Part 1, we obtain
\begin{gather}
 r\abs{\nabla^{(x)}\hat\rho_r(y-x;\cL_{x,r})}+\abs{{{s_{x,r}^\perp \nabla_{s_{x,r}^\perp}^{(y)}  }} \nabla^{(x)}\hat\rho_r(y-x;\cL_{x,r})}\leq C(m)\rho_{2r}(y-x)\, .
\end{gather}
On the other hand by \eqref{e:superconvexity_key_2} we have on the support of $\rho_r(y-x)$, where $\abs{y-x}\leq C R r$, that :
\begin{gather}
 \abs{\pi_{x,r}(y-x)-\pi_{\T_r(x)}(y-x)} + \abs{{{s_{x,r}^\perp \nabla_{s_{x,r}^\perp}^{(y)}  }} \qua{\pi_{x,r}(y-x)-\pi_{\T_r(x)}(y-x)}}=\abs{\pi_{\T_r(x)}\ton{\pi_{x,r}^\perp (y-x)}}\leq C(m,K_N) R r \sqrt \delta
\end{gather}
This and \eqref{e:superconvexity_key_2} ensure that \eqref{e:superconvexity_error_2} is bounded in absolute value by
\begin{gather}
 C(m,\epsilon_0,R)\sqrt \delta \, \int_{\T_r} \psi_\T(x,r) \int \rho_{2r}(y-x) \cEn\leq C(m,\epsilon_0,R)\sqrt \delta \, \int_{\T_r} \psi_\T(x,r) \vartheta(x,2r;\LL_{x,r}^\perp)\, .
\end{gather}
Invoking Lemma \ref{l:error_estimates_psi_with_pinching} we conclude that this error term is of the form $\err(r)$.
\end{proof}

\textbf{Part 3: integration by parts and error control.}
 We claim that
\begin{align}
  &(m-2)\int_{\T_r}\psi_\T(x,r)\int\hat\rho_r(y-x;\cL_{x,r})\cFn(y) - \int_{\T_r}\psi_\T(x,r)\int \ps{\nabla^{(x)}\hat\rho_r(y-x;\cL_{x,r})}{\pi_{\T_r(x)}(y-x)}\cFn(y)\notag \\
  &=\err(r)\, .
\end{align}
The above will be based on an integration by parts on the $x\in \T_r$ variable, and with it the proof of claim \eqref{e:claim_superconvexity_3} will be finished.

\begin{proof}[Proof of subclaim]
We have
\begin{align}
 &\int_{\T_r}\psi_\T(x,r)\int \ps{\nabla^{(x)}\hat\rho_r(y-x;\cL_{x,r})}{\pi_{\T_r(x)}(y-x)}\cFn(y) dydv_\T(x)\notag\\
 =&\int \int_{\T_r}\psi_\T(x,r)\ps{\nabla^{(x)}\hat\rho_r(y-x;\cL_{x,r})}{\pi_{\T_r(x)}(y-x)}\cFn(y) \notag\\
 =&-\int \int_{\T_r}\hat\rho_r(y-x;\cL_{x,r})\dive^{(x)}\ton{\psi_\T(x,r)\cFn(y) \ \pi_{\T_r(x)}(y-x)} \notag\\
 =&-\int \int_{\T_r}\hat\rho_r(y-x;\cL_{x,r})\qua{\ps{\nabla^{(x)}\psi_\T(x,r)}{\pi_{\T_r(x)}(y-x)}\cFn(y) \ } \label{e:superconvexity_part_3_1}\\
 &-\int \int_{\T_r}\hat\rho_r(y-x;\cL_{x,r})\qua{\psi_\T(x,r)\ps{\nabla^{(x)}\cFn(y)}{\pi_{\T_r(x)}(y-x)}} \label{e:superconvexity_part_3_2}\\
 &-\int \int_{\T_r}\hat\rho_r(y-x;\cL_{x,r})\qua{\psi_\T(x,r)\cFn(y)\dive^{(x)}\ton{\pi_{\T_r(x)}(y-x)}} \label{e:superconvexity_part_3_3}\, .
\end{align}
We analyze these terms one by one.

\vspace{5mm}
The first term \eqref{e:superconvexity_part_3_1} can be controlled by 
\begin{align}
 &-\int \int_{\T_r}\hat\rho_r(y-x;\cL_{x,r})\qua{\ps{\nabla^{(x)}\psi_\T(x,r)}{\pi_{\T_r(x)}(y-x)}{{s_{x,r}^\perp \nabla_{s_{x,r}^\perp}^{(y)}  }}\cEn(y) \ } dv_\T(x)dV(y)\notag\\
 =&2\int \int_{\T_r}\hat\rho_r(y-x;\cL_{x,r})\qua{\ps{\nabla^{(x)}\psi_\T(x,r)}{\pi_{\T_r(x)}(y-x)}\cEn(y) \ } \notag\\
 +&\int \int_{\T_r}{{s_{x,r}^\perp \nabla_{s_{x,r}^\perp}^{(y)}  }}\hat\rho_r(y-x;\cL_{x,r})\qua{\ps{\nabla^{(x)}\psi_\T(x,r)}{\pi_{\T_r(x)}(y-x)}\cEn(y) \ } \notag\\
 +&\int \int_{\T_r}\hat\rho_r(y-x;\cL_{x,r})\qua{\ps{\nabla^{(x)}\psi_\T(x,r)}{\pi_{\T_r(x)}(\pi_{x,r}^\perp(y-x))}\cEn(y) \ } 
\end{align}

As seen before, we can bound
\begin{gather}\label{e:nabla_w_hat_rho}
\abs{\pi_{\T_r(x)}(\pi_{x,r}^\perp(y-x))}\leq C(m,K_N) \abs{y-x} \sqrt \delta\, ,\\
 \abs{{{s_{x,r}^\perp \nabla_{s_{x,r}^\perp}^{(y)}  }}\hat\rho_r(y-x;\cL_{x,r})}\leq C(m) \rho_{2r}(y-x)\, .
\end{gather}
Thus we obtain that the absolute value of \eqref{e:superconvexity_part_3_1} is bounded by
\begin{align}
 &C(m,\epsilon_0,R)\int_{\T_r}r \abs{\nabla \psi_{\T}(x,r)} \int \rho_{2r}(y-x)\cEn(y) \leq C(m,\epsilon_0,R)\int_{\T_r}r \abs{\nabla \psi_{\T}(x,r)} \vartheta(x,2r;\LL_{x,r})\notag \\
 \leq &C(m,\epsilon_0,R) \sqrt \delta \int_{\T_r}r \abs{\nabla \psi_{\T}(x,r)} \stackrel{\text{Lemma \ref{l:energy_decomposition:T_cutoff} }}{=} \err(r)
\end{align}

\vspace{5mm}
The second term \eqref{e:superconvexity_part_3_2} can similarly be controlled by 
\begin{align}
 &-\int \int_{\T_r}\hat\rho_r(y-x;\cL_{x,r})\qua{\psi_\T(x,r)\ps{\nabla^{(x)}\cFn(y)}{\pi_{\T_r(x)}(y-x)}} \notag \\
 =&-\int \int_{\T_r}\hat\rho_r(y-x;\cL_{x,r})\qua{\psi_\T(x,r)\ps{\nabla^{(x)}s_{x,r}^\perp \nabla_{s_{x,r}^\perp}^{(y)}\cEn(y)}{\pi_{\T_r(x)}(y-x)}} \notag \\
 =&2\int \int_{\T_r}\hat\rho_r(y-x;\cL_{x,r})\qua{\psi_\T(x,r)\ps{\nabla^{(x)}\cEn(y)}{\pi_{\T_r(x)}(y-x)}}\notag \\
 &+\int \int_{\T_r}{{s_{x,r}^\perp \nabla_{s_{x,r}^\perp}^{(y)}  }}\hat\rho_r(y-x;\cL_{x,r})\qua{\psi_\T(x,r)\ps{\nabla^{(x)}\cEn(y)}{\pi_{\T_r(x)}(y-x)}}\notag \\
 &+ \int \int_{\T_r}\hat\rho_r(y-x;\cL_{x,r})\qua{\psi_\T(x,r)\ps{\nabla^{(x)}\cEn(y)}{\pi_{\T_r(x)}(\pi_{x,r}^\perp(y-x))}}
\end{align}
By Lemma \ref{l:angular_energy:Ena_estimates} and Lemma \ref{l:error_estimates_psi_with_pinching}, we can argue as with the first term to see its of the form $\err(r)$.

\vspace{5mm} We are left with \eqref{e:superconvexity_part_3_3}. Notice that since
\begin{gather}
 \abs{\dive^{(x)}\ton{\pi_{\T_r(x)}(y-x)}+m-2}\leq \abs{\nabla^{(x)} \pi_{\T_r(x)}}\abs{y-x}\, ,
\end{gather}
we have
\begin{align}
 &\abs{-\int \int_{\T_r}\hat\rho_r(y-x;\cL_{x,r})\qua{\psi_\T(x,r)\cFn(y)\dive^{(x)}\ton{\pi_{\T_r(x)}(y-x)}}-(m-2)\int_{\T_r}\psi_\T(x,r) \int \hat\rho_r(y-x;\cL_{x,r})\cFn(y)}\notag \\
 &= \abs{\int \int_{\T_r}\hat\rho_r(y-x;\cL_{x,r})\qua{\psi_\T(x,r){s_{x,r}^\perp \nabla_{s_{x,r}^\perp}^{(y)}  }\ton{\cEn(y)} \ \ton{\dive^{(x)}\ton{\pi_{\T_r(x)}(y-x)} +m-2}   }}\notag \\
 &\leq \int \int_{\T_r}\abs{\ton{2+{s_{x,r}^\perp \nabla_{s_{x,r}^\perp}^{(y)}  } }\hat\rho_r(y-x;\cL_{x,r})}\psi_\T(x,r)\cEn(y) \ \frac{\abs{y-x}}{r} \abs{r\nabla^{(x)} \pi_{\T_r(x)}} \notag \\
 &+\abs{\int \int_{\T_r}\hat\rho_r(y-x;\cL_{x,r})\ \psi_\T(x,r)\cEn(y) \ s_{x,r}^\perp \nabla_{s_{x,r}^\perp}^{(y)}\dive^{(x)}\ton{\pi_{\T_r(x)}(y-x)} }\notag\\
 \leq & C\int \int_{\T_r}\qua{\hat\rho_r(y-x;\cL_{x,r})+ \abs{{{s_{x,r}^\perp \nabla_{s_{x,r}^\perp}^{(y)}  }} \hat \rho_r(y-x;\LL_{x,r})}}\frac{\abs{y-x}}{r}\psi_\T(x,r)\cEn(y)\abs{r\nabla^{(x)} \pi_{\T_r(x)}}\, .
\end{align}
By Lemma \ref{t:approximating_submanifold}.\eqref{i:approximating_submanifold_space_gradient_hessian}, $\abs{r\nabla^{(x)} \pi_{\T_r(x)}}\leq C(m,\epsilon_0)\sqrt \delta$. Moreover,
\begin{gather}
 \frac{\abs{y-x}}{r}\ton{\hat\rho_r(y-x;\cL_{x,r})+ \abs{{{s_{x,r}^\perp \nabla_{s_{x,r}^\perp}^{(y)}  }} \hat \rho_r(y-x;\LL_{x,r})} } \leq C(m) \rho_{2r}(y-x)\, ,
\end{gather}
and so we can bound this term by
\begin{gather}
 C(m,\epsilon_0,R)\sqrt \delta \int_{\T_r}\psi_\T(x,r)\int \rho_{2r}(y-x)\cEn(y)\, ,
\end{gather}
which is an $\err(r)$ error term by Lemma \ref{l:error_estimates_psi_with_pinching}.

\end{proof}
\end{proof}

\vspace{.3cm}
\subsubsection{Computing \eqref{e:angular_energy:drtheta:4}}

The main claim we wish to address is the following:\\

{\bf Claim: } We have that
\begin{gather}
 \int_{\T_r}\ton{r\frac{d}{dr}\psi_\T}\int\hat\rho_r\,\cFn+\int_{\T_r}\psi_\T\int \langle \partial_{\cL}\hat\rho_r,r\dot\cL_{x,r}\rangle\cFn+\int_{\T_r}\psi_\T\int\hat\rho_r \ton{r\frac{d}{dr}\cFn}=\err(r)
\end{gather}

\begin{proof}[Proof of claim]
 The proof of this claim follows from the same estimates as the ones above. In particular, the first term can be controlled considering that for all $x\in \T_r\cap \supp{\psi_\T}$:
 \begin{gather}
  \abs{\int \hat \rho_r\cFn}=\abs{\int \hat \rho_r{{s_{x,r}^\perp \nabla_{s_{x,r}^\perp}^{(y)}  }}\cEn} \leq C(m) \int \rho_{2r}(y-x)\cEn(y)\leq C(m) \delta\, ,
 \end{gather}
and thus the first term is of the form $\err$ by Lemma \ref{l:energy_decomposition:T_cutoff}.\ref{i:T_cutoff_global_integral}.

Moving to the second term, we observe that
\begin{align}
 &\int_{\T_r}\psi_\T\int \langle \partial_{\cL}\hat\rho_r,r\dot\cL_{x,r}\rangle{{s_{x,r}^\perp \nabla_{s_{x,r}^\perp}^{(y)}  }}\cEn \notag \\
 =&-2\int_{\T_r}\psi_\T\int \langle \partial_{\cL}\hat\rho_r,r\dot\cL_{x,r}\rangle \cEn -\int_{\T_r}\psi_\T\int \langle \partial_{\cL}{{s_{x,r}^\perp \nabla_{s_{x,r}^\perp}^{(y)}  }}\hat\rho_r,r\dot\cL_{x,r}\rangle\cEn\, .
\end{align}
By Theorem \ref{t:best_plane:best_plane}.\eqref{i:best_plane_bounds_on_projections}, $\abs{r\dot \LL_{x,r}}\leq C(m,K_N)\sqrt \delta$, and so
\begin{gather}
 \abs{\int_{\T_r}\psi_\T\int \langle \partial_{\cL}\hat\rho_r,r\dot\cL_{x,r}\rangle\cFn }\leq C(m,\epsilon_0,R)\sqrt \delta \int_{\T_r} \psi_\T(x,r)\vartheta(x,2r;\LL_{x,r}) \, ,
\end{gather}
which, thanks to Lemma \ref{l:error_estimates_psi_with_pinching}, is of the form $\err$.

The last term can be controlled by
\begin{gather}
\abs{\int_{\T_r}\psi_\T\int\hat\rho_r \ton{r\frac{d}{dr}\cFn}}\leq C(m,R)\int_{\T_r} \psi_\T\int\qua{\hat \rho_r + \abs{ r\nabla^{(y)} \hat\rho_r}} \abs{r \frac{d}{dr}\cEn}\, ,
\end{gather}
which is of the form $\err$ by Lemma \ref{l:angular_energy:Ena_estimates} and Lemma \ref{l:error_estimates_psi_with_pinching}.

\end{proof}

Having taken care of all the terms, this concludes the proof of Lemma \ref{l:angular_energy:ddrtheta_2}.\\

\subsection{Computing \texorpdfstring{$L$}{L}-Laplacian}.
The next important ingredient is the following Lemma. Recall that $\Delta_{\LL_{x,r}}$ is the usual Laplacian along the $\LL_{x,r}$ directions. Sometimes we will denote it by $\Delta_{\LL_{x,r}}^{(y)}$ to underline the fact that it acts on the $y$ variable.  

We will also use the Laplacian on the submanifold $\T_r$, and we will denote it by $\Delta^{(x)}_{\T_r}=\Delta_{\T_r}$.

\begin{lemma}\label{l:angular_energy_L_part}
We have that 
 \begin{gather}
  \int_{\T_r}\psi_\T\int\hat\rho_r(y-x)\,|s_{x,r}^\perp|^2\Delta_{\LL_{x,r}}\cEn(y) = \err(r)\, .
 \end{gather}
\end{lemma}

The proof is based on turning $\Delta_{\cL_{x,r}}$ into $\Delta_{\T_r}$, up to errors, and integrating by parts.  We split it into a (technical) claim and its consequences.\\

\textbf{Claim:} we have
\begin{gather}\label{e:claim_L_laplacian_computation}
 \int_{\T_r}\psi_\T\int\hat\rho_r(y-x)\,|s_{x,r}^\perp|^2\Delta^{(y)}_{\LL_{x,r}}\cEn(y)=\int_{\T_r}\psi_\T\int \Delta^{(x)}_{\T_r}\ton{|s_{x,r}^\perp|^2\hat\rho_r(y-x)}\,\cEn(y) + \err(r)
\end{gather}
Before proving this claim, let us use it to prove the main Lemma.
Integrating by parts in $x$ we get
\begin{align}
 &-\int_{\T_r}\psi_\T\int \Delta^{(x)}_{\T}\ton{|s_{x,r}^\perp|^2\hat\rho_r(y-x)}\,\cEn(y) \notag \\
 =&\int_{\T_r}\nabla_{i}^{(x)}\psi_\T\int \nabla^{(x),i}\ton{|s_{x,r}^\perp|^2\hat\rho_r(y-x)}\,\cEn(y) + \int_{\T_r}\psi_\T\int \nabla_{i}^{(x)}\ton{|s_{x,r}^\perp|^2\hat\rho_r(y-x)}\,\nabla^{(x),i}\cEn(y)
\end{align}
By Lemma \ref{l:angular_energy:Ena_estimates} the absolute value of the first term can be bounded by
\begin{gather}
 C(m,R) \int_{\T_r} r\abs{\nabla \psi_\T}\vartheta(x,2r;\LL_{x,r})
\end{gather}
By Remark \ref{rm:cone_splitting_annular_V}, on the support of $\psi_\T$, $\vartheta(x,2r;\LL_{x,r}^\perp)\leq C \delta$, so the first term is of the form $\err$ thanks to Lemma \ref{l:energy_decomposition:T_cutoff}.\ref{i:T_cutoff_global_integral}.

The second term can be bounded by Lemma \ref{l:angular_energy:Ena_estimates} and Lemma \ref{l:error_estimates_psi_with_pinching}, and this concludes the proof of the main Lemma, up to checking the Claim.
%
%

\vspace{5mm}
\begin{proof}[Proof of Claim \eqref{e:claim_L_laplacian_computation}]
 We start by observing that by integrating by parts in the $y$ variables:
 \begin{gather}
  \int\hat\rho_r(y-x)\,|s_{x,r}^\perp|^2\Delta^{(y)}_{\LL_{x,r}}\cEn(y)=\int \Delta^{(y)}_{\LL_{x,r}}\ton{|s_{x,r}^\perp|^2\hat\rho_r(y-x)}\,\cEn(y)\, .
 \end{gather}
 By definition and direct computation:
 \begin{gather}
  \frac 1 2 \Delta^{(y)}_{\LL_{x,r}}\ton{|s_{x,r}^\perp|^2\hat\rho_r(y-x)}
  =\Delta^{(y)}_{\LL_{x,r}}\ton{\rho\ton{\frac{\abs{y-x}^2}{2r^2} }r^2\frac{\abs{\pi_{x,r}^\perp(y-x)}^2}{2 r^2} \psi \ton{e^{2R} \frac{\abs{\pi_{x,r}^\perp(y-x)}^2}{2r^2} }}\notag \\
  = \ddot \rho \ton{\frac{\abs{y-x}^2}{2r^2}}\frac{\abs{\pi_{x,r}(y-x)}^2}{r^2}g\ton{\frac{\abs{\pi_{x,r}^\perp(y-x)}^2}{2r^2}} + (m-2) r^2\dot \rho \ton{\frac{\abs{y-x}^2}{2r^2}}g\ton{\frac{\abs{\pi_{x,r}^\perp(y-x)}^2}{2r^2}}\, ,
 \end{gather}
where we have set for convenience $g(t)=t \psi\ton{e^{2R}t}$.  On the other hand, using
 \begin{itemize}
  \item  the bounds on the second fundamental form of $\T_r$ proved in \ref{t:approximating_submanifold}.\ref{i:approximating_submanifold_space_gradient_hessian},
  \item the estimate $\abs{\pi_{x,r}-\pi_{\T_r(x)}}\leq C\sqrt \delta$ in Lemma \ref{l:improved_comparison_LL_T}
  \item the estimates on $r\abs{\nabla ^{(x)}\pi_{x,r}}\leq C\sqrt \delta$ in Theorem \ref{t:best_plane:best_plane}.\ref{i:best_plane_bounds_on_projections}
 \end{itemize}
we can conclude via an easy albeit tedious computation that is very similar to the ones carried out in the previous subsections that
\begin{align}
 &\frac 1 2 \Delta_{\T}^{(x)}\ton{|s_{x,r}^\perp|^2\hat\rho_r(y-x)}  \notag\\
 =&\ddot \rho \ton{\frac{\abs{y-x}^2}{2r^2}}\frac{\abs{\pi_{\T_r(x)}(y-x)}^2}{r^2}g\ton{\frac{\abs{\pi_{x,r}^\perp(y-x)}^2}{2r^2}} + (m-2) r^2\dot \rho \ton{\frac{\abs{y-x}^2}{2r^2}}g\ton{\frac{\abs{\pi_{x,r}^\perp(y-x)}^2}{2r^2}}+\err(x,y,r)\notag \\
 =&\frac 1 2 \Delta^{(y)}_{\LL_{x,r}}\ton{|s_{x,r}^\perp|^2\hat\rho_r(y-x)} + \err(x,y,r)\, ,
\end{align}
 where
\begin{gather}
 \abs{\err(x,y,r)}\leq C(m,\epsilon_0,R)\sqrt \delta \rho_{2r} (y-x)\, .
\end{gather}
By Lemma \ref{l:error_estimates_psi_with_pinching}, we can estimate:
\begin{align}
 &\int_{\T_r}\psi_\T\int\abs{\Delta^{(y)}_{\LL_{x,r}}\ton{\hat\rho_r(y-x)\,|s_{x,r}^\perp|^2} - \Delta^{(x)}_{\T}\ton{|s_{x,r}^\perp|^2\hat\rho_r(y-x)}}\cEn(y) \notag\\
 \leq & C(m,\epsilon_0,R) \sqrt \delta \int_{\T_r}\psi_\T\int \rho_{2r}(y-x)e_{\alpha,x,r}(y)\leq C(m,\epsilon_0,R) \sqrt \delta \int_{\T_r}\vartheta(x,2r;\LL_{x,r}^\perp)\notag\\
 \leq &C(m,\epsilon_0,R) \sqrt \delta \int_{\T\cap\B 2 p}r\dot\vartheta(x,8r)=\err(r)\, .
\end{align}

\

\end{proof}

\vspace{.3cm}
\subsection{Finishing the Proof of Proposition \ref{p:angular_energy:superconvexity}}

Let us now combine Theorem \ref{t:angular_energy:uniform_subharmonic} and Lemmas \ref{l:angular_energy:ddrtheta_2} and \ref{l:angular_energy_L_part} in order to prove Proposition \ref{p:angular_energy:superconvexity}:
\begin{align}
	\ton{r\frac{d}{dr}}^2\hat\vartheta_\alpha(\T_r) &= \int_{\T_r}\psi_\T(x,r)\int \hat\rho_r(y-x;\cL_{x,r})\Bigg(\,\Big(s_{x,r}^\perp\,\nabla_{s_{x,r}^\perp}\Big)^2+|s_{x,r}^\perp|^2\Delta_{\LL_{x,r}}\Bigg)\,\cEn(y)-\err(r)-r\frac d {dr} \epsilon_1(r)\, \notag\\
	&=\int_{\T_r}\psi_\T(x,r)\int \hat\rho_r(y-x;\cL_{x,r})\bar\Delta_{x,r}\cEE(y)-\err(r)-r\frac d {dr} \epsilon_1(r)\, ,\notag\\
	&\geq (2-C(m,R)\delta)\int_{\T_r}\psi_\T(x,r)\int \hat\rho_r(y-x;\cL_{x,r})\cEE-\err(r)-r\frac d {dr} \epsilon_1(r)\, ,\notag\\
	&=(2-C(m,R)\delta)\hat\vartheta_\alpha(\T_r)-\err(r)-r\frac d {dr} \epsilon_1(r)\, ,
\end{align}
as claimed. \hfill $\qed$

\vspace{.5cm}

\section{Radial Energy on Annular Regions}\label{s:radial_energy}

We turn our attention to the last component needed in the proof of Theorem \ref{t:outline:annular_regions_energy}, namely we produce bounds on the radial and $L$ energy of the annular region.  The estimates of this Section are the most subtle of the paper, as some error terms are apriori of strictly larger order than is allowed and the Euler-Lagrange equations for $\T_r$ and $\cL_{x,r}$ play a crucial role.  We refer the reader to Section \ref{ss:outline:general:radial_energy} for a more in depth discussion, but we recall here the relevant definitions of the radial and $L$ energy
\begin{align}
	\hat\vartheta_n(x,r)&\equiv \int \hat\rho_r(y-x;\LL_{x,r})\big\langle\nabla u,\pi^\perp_{x,r}(y-x)\big\rangle^2\, ,\\
	\vartheta_\LL(x,r)&\equiv \int \rho_r(y-x)\abs{\pi_{x,r}\nabla u}^2\, ,
\end{align}
where $\hat\rho_r$ is the heat kernel mollifier cutoff outside $B_R$ and inside the $e^{-R}$ neighborhood of $L$ as in \eqref{e:prelim:L_mollifier}.  We can integrate this over $\T_r$ with our bubble weight $\psi_\T$ from Lemma \ref{l:outline_general:T_cutoff} to define the radial and $L$ energy on $\T_r$ :
\begin{align}
	\hat\vartheta_n(\T_r)&\equiv \int_{\T_r}\psi_\T(x,r)\, \hat\vartheta_n(x,r) = \int_{\T_r}\psi_\T(x,r)\int\hat\rho_r(x-y;\LL_{x,r})\langle\nabla u,\pi^\perp_{x,r}(y-x)\rangle^2\, ,\\
	\vartheta_\LL(\T_r)&\equiv \int_{\T_r}\psi_\T(x,r)\, \hat\vartheta_\LL(x,r) = \int_{\T_r}\psi_\T(x,r)\int \rho_r(x-y;\LL_{x,r})\abs{\pi_{x,r}\nabla u}^2\, ,
\end{align}
and we can integrate them in $\frac{dr}{r}$ to obtain the energy for the full annular region:
\begin{align}
	\hat\vartheta_n(\cA)&\equiv \int \hat \vartheta_n(\T_r) \frac{dr}{r}\, , \qquad \vartheta_\LL(\cA)\equiv \int \vartheta_\LL(\T_r) \frac{dr}{r}\, .
\end{align}

Recall that by Lemma \ref{l:energy_decomposition:full_energy_on_annulus}, we have
\begin{gather}
 \int_{\cA\cap B_1}|\nabla u|^2\leq C(m)\Big(\hat\vartheta_\cL(\cA)+\hat \vartheta_\alpha(\cA)+\hat \vartheta_n(\cA)\Big)
\end{gather}
Given our previous estimates on the $L$-energy and the angular energy, the crucial last ingredient needed to prove Theorem \ref{t:outline:annular_regions_energy} is that
\begin{gather}
 \hat \vartheta_n(\cA)\leq \epsilon\, .
\end{gather}


Our primary goal in this Section will be to control the radial energy, and along the way we will give an improved $\LL$ energy bound over the annular region.  Precisely we will want to prove:\\

\begin{theorem}[Radial Energy Bound on Annular Regions]\label{t:radial_energy:radial_energy}
Let $u:B_{10R}(p)\to N$ be a stationary harmonic map with $R^2\fint_{B_{10R}}|\nabla u|^2\leq \Lambda$ , and let $\cA=B_2\setminus \overline{B_{\rf_z}(\T)}$ be a $\delta$-annular region.  For each $\epsilon>0$ if $R\geq R(m,\Lambda,\epsilon)$ and $\delta\leq \delta(m,K_N,R,\Lambda,\epsilon)$ then we have
\begin{align}
\hat\vartheta_n(\cA)+\vartheta_\LL(\cA)\leq  \epsilon\, .
\end{align}
\end{theorem}
\begin{remark}
	Using the techniques of Section \ref{ss:outline:general:L_energy} and Lemma \ref{l:energy_decomposition:scale_r_properties} the above additionally implies the log-improved $L$-energy bound $\vartheta_\cL(\T_r)\leq C(m)\min\Big\{\delta, \frac{\epsilon}{|\ln r|}\Big\}$ .
\end{remark}

\vspace{.3cm}
\subsection{Estimates along \texorpdfstring{$\T_r$}{Tr}}

Let us begin by stating our main technical Proposition toward our control of the radial energy on an annular region.  In this subsection we will use the following to quickly prove Theorem \ref{t:radial_energy:radial_energy}, and then the remainder of this Section will be dedicated to proving the Proposition itself:

\begin{proposition}[Radial Energy Bound on $\T_r$]\label{p:radial_energy:radial_energy}
Let $u:B_{10R}(p)\to N$ be a stationary harmonic map with $R^2\fint_{B_{10R}}|\nabla u|^2\leq \Lambda$ , and let $\cA=B_2\setminus \overline{B_{\rf_z}(\T)}$ be a $\delta$-annular region.  Then we can estimate
\begin{align}
\hat\vartheta_n(\T_r) +2\vartheta_\cL(\T_r)	\leq (m-1)r\,\frac{d}{dr}\vartheta_\cL(\T_r)+\hat\vartheta_\alpha(\T_r)+\ers(r)\, ,
\end{align}
where
\begin{align}\label{e:radial_energy:p_radial_energy_error}
	\abs{\ers(r)}&\leq C(m,K_N,\Lambda)\Bigg[\Big(\sqrt{\delta}+e^{-R/2}\Big)\int_{\T_r}\psi_\T(x,r)\Big(\vartheta_\cL(x,2r)+r\dot\vartheta(x,2r)\Big)+\sqrt\delta\int_{\T_r}r\Big|\frac{\partial}{\partial r}\psi_\T\Big|+r\Big|\nabla\psi_\T\Big|\;\Bigg]\, .
\end{align}
\end{proposition}

\begin{remark}
 As is done for $\err(r)$, we will use the notation $\ers(r)$ to represent any error that can be estimated by \eqref{e:radial_energy:p_radial_energy_error}, and the exact value of $\ers(r)$ may vary from line to line.
\end{remark}

\begin{proof}[Proof of Theorem \ref{t:radial_energy:radial_energy} given Proposition \ref{p:radial_energy:radial_energy}]

Let us first observe the estimate
\begin{align}
	\int \ers(r)\frac{dr}{r} &\stackrel{\ref{t:approximating_submanifold}.\ref{i:integral_dot:vartheta_Tr_vs_T}}{\leq} C(m,\Lambda)\ton{\sqrt\delta+e^{-R/2}}\Bigg(\int \vartheta_\cL(\T_r)\frac{dr}{r}+\int_\T\psi_\T(x,r)|\vartheta(x,4)-\vartheta(x,\rf_x)|\Bigg)\notag\\
	&+C(m,\Lambda)\delta\int \qua{\int_{\T_r}r\Big|\frac{\partial}{\partial r}\psi_\T\Big|+r\Big|\nabla\psi_\T\Big|\;} \frac{dr}{r} + C(m,\Lambda)\delta\, \notag\\
	&\leq \epsilon'\int \vartheta_\cL(\T_r)\frac{dr}{r}+\epsilon'\Lambda+C(m,\Lambda)\delta\, ,
\end{align}
where in the last line we have chosen $\delta\leq \delta(m,\Lambda,\epsilon')$ and $R\geq R(m,\Lambda,\epsilon')$ while applying Lemma \ref{l:energy_decomposition:T_cutoff}.  If we integrate Proposition \ref{p:radial_energy:radial_energy} with the Dini measure $\frac{dr}{r}$ this then gives
\begin{align}
	\hat\vartheta_n(\cA)+2\hat\vartheta_\cL(\cA)\leq 0+\hat\vartheta_\alpha(\cA)+\epsilon'\hat\vartheta_\cL(\cA)+\epsilon'\Lambda+C(m,\Lambda)\delta\, .
\end{align} 
Rearranging this gives
\begin{align}
	\hat\vartheta_n(\cA)+(2-\epsilon')\hat\vartheta_\cL(\cA)\leq \hat\vartheta_\alpha(\cA)+\epsilon'\Lambda+C(m,\Lambda)\delta\, .
\end{align} 
By Theorem \ref{t:angular_energy:angular_energy}, if we choose  $\epsilon'<\epsilon'(\epsilon,\Lambda)$ and $\delta\leq \delta(m,\Lambda,\epsilon)$, we obtain the desired result.
\end{proof}

\vspace{.3cm}
\subsection{Outline Proof of Proposition \ref{p:radial_energy:radial_energy}}\label{ss:radial_energy:outline}

The remainder of this section will now be dedicated toward the proof of Proposition \ref{p:radial_energy:radial_energy}.  Arriving at Proposition \ref{p:radial_energy:radial_energy} begins with the introduction of a stationary equation induced by a specially chosen vector field and involves several subtle points.
There will of course be many technical errors to control, but several of these errors will apriori be strictly larger in nature than those that appear for our end formula for $\ers(r)$ in Proposition \ref{p:radial_energy:radial_energy}.
These errors do not appear in the toy model of Section \ref{s:outline_toymodel}, and more to the point they are larger than allowed for the energy identity to hold.
Controlling them will rely on the observation that each time they occur they can be written in terms of the Euler Lagrange equations of either our best approximating submanifold $\T_r$ or our best planes $\cL_{x,r}$.
Thus for the right precise choices of energy the errors will reduce to lower order terms.  Let us now break down each subsection and discuss the various technical subtleties which will appear.\\

{\bf Section \ref{ss:radial_energy:stationary_vector_field} Outline: The Stationary Vector Field. } The first ingredient of the proof is to study the stationary equation induced by the vector field
\begin{align}
	  \xi(y) = r^2\int_{\T_r }  \psi_{\T}(x,r) \tilde\rho_{r}(y-x;\LL_{x,r})\pi^\perp_{x,r}(y-x)\, ,
\end{align}
where recall from Section \ref{ss:restricted_energy} that $\tilde\rho_r$ is well defined by the property that its perpendicular gradient $\nabla_{\LL_{x,r}^\perp}\tilde\rho_r \equiv \hat\rho_r$ vanishes near the bubble region.  One should interpret the above as the appropriate approximation of the radial vector field emanating from $\T_r$.  The main result of this subsection is Lemma \ref{l:radial_energy:stationary_vector_field}, where we arrive at the estimate
\begin{align}
\hat\vartheta_n(\T_r) +2\vartheta_\cL(\T_r)	= \hat\vartheta_\alpha(\T_r)+\int_{\T_r}\int\rho_r(y-x)|\pi^\perp_{x,r}(y-x)|^2\langle\nabla u,\cL_{x,r}\rangle^2+\ff_1(r)+\ers(r)\, ,
\end{align}
where
\begin{align}\label{e:radial_energy:f_1}
	\ff_1(r)\equiv  -2\int_{\T_r} \psi_{\T}(x,r)\int \rho_r(y-x)\langle \nabla u, \pi_{x,r}(y-x)\rangle\langle \nabla u, \pi^\perp_{x,r}(y-x)\rangle \, .
\end{align}

Observe that the $\LL_{x,r}^\perp$ component in the energy stops $\ff_1(r)$ from being easily controllable.  That is, with a little bit of work one can estimate $\int |\ff_1(r)|\frac{dr}{r}\leq C \Lambda$, however getting a small estimate on the right hand side is a primary challenge of this Section.  The next several subsections will be dedicated to understanding this term.\\

{\bf Section \ref{ss:radial_energy:f1} Outline: Rewriting $\ff_1(r)$ using the Euler Lagrange Equation for $\T_r$. } Controlling the term $\ff_1(r)$ will require several steps.  The first will be to observe that the Gaussian nature of $\rho_r$ gives that $r^2\nabla^{(x)}_{\T_r}\rho_r\approx \rho_r\cdot \pi_{x,r}(y-x)$, and thus we will be able to rewrite
\begin{align}
	\ff_1(r) = -2r^2\int_{\T_r} \psi_{\T}(x,r)\int\big\langle \nabla^{(x)}_{\T_r}\rho_r,\nabla u(y)\big\rangle\big\langle \nabla u(y), \pi^\perp_{x,r}(y-x)\big\rangle +\ers(r)\, ,
\end{align}
where as is our practice the terms $\ers(r)$ represent small error or boundary terms which are more clearly controllable.  If we then integrate by parts we will arrive at a variety of controllable errors together with our two main terms
\begin{align}
	\ff_1(r) =& 
	2r^2\int_{\T_r} \psi_{\T}(x,r)\int \rho_r(y-x)\big\langle \nabla u(y),H_{\T_r}(x)\big\rangle\big\langle \nabla u(y), \pi^\perp_{\T_r(x)}(y-x)\big\rangle \, \notag\\
	&-2r^2\int_{\T_r} \psi_{\T}(x,r)\int \rho_r(y-x)\Big\langle \nabla u(y)\otimes \nabla u(y), \langle A_{\T_r}(x),\pi_{\T_r(x)}(y-x)\rangle\,\Big\rangle+\ers(r)\, ,
\end{align}
where $H_{\T_r}\approx \text{div}\,\pi_{x,r}$ is the mean curvature of $\T_r$ and $A_{\T_r}\approx\nabla\pi_{x,r}$ is the second fundamental form.  The last key observation of this subsection will be that since $H_{\T_r}$ is almost an element of $\cL_{x,r}^\perp$, $\ff_2(r)$ is in the right form to now apply the Euler-Lagrange equation of $\T_r$.  This, up to small errors, will allow us to turn the $L^\perp$-gradient into a $L$-gradient and write:
\begin{align}\label{e:radial_energy:f_2}
\ff_1(r) &= \ff_2(r)+\ff_3(r)+\ers(r)\\
	\ff_2(r)&=-2r^2\int_{\T_r} \psi_{\T}(x,r)\int \rho_r(y-x)\big\langle \nabla u(y),H_{\T_r}(x)\big\rangle\big\langle \nabla u(y), \pi_{x,r}(y-x)\big\rangle \, ,\notag\\
	\ff_3(r)&=-2r^2\int_{\T_r} \psi_{\T}(x,r)\int \rho_r(y-x)\Big\langle \nabla u(y)\otimes \nabla u(y), \langle A_{\T_r}(x),\pi_{x,r}(y-x)\rangle\,\Big\rangle\, .
\end{align} 

We have therefore written our stationary equation as 
\begin{align}
\hat\vartheta_n(\T_r) +2\vartheta_\cL(\T_r)	= \hat\vartheta_\alpha(\T_r)+\int_{\T_r}\psi_\T\int\rho_r(y-x)|\pi^\perp_{x,r}(y-x)|^2|\pi_{x,r}\nabla u|^2+\ff_2(r)+\ff_3(r)+\ers(r)\, .\\\notag
\end{align}

{\bf Section \ref{ss:radial_energy:f2f3} Outline: Estimating $\ff_2+\ff_3\leq 2(m-1)r^4\int_{\T_r}\psi_{\T}\int\rho_r \langle A_{\T_r},\nabla u\rangle^2+\epsilon_5(r)$ using the Euler-Lagrange for the best subspace $\cL_{x,r}$. }  The main estimate of Section \ref{ss:radial_energy:f2f3} is Lemma \ref{l:radial_energy:f2f3}, where we will prove the estimates
\begin{align}
	\ff_2(r)&\leq (m-2)\ff_3(r)+\ers(r)\, ,\notag\\
	\ff_3(r)&\leq \int_{\T_r}\psi_{\T}(x,r)\int\rho_r(y-x) \langle A_{\T_r}(x),\nabla u(y)\rangle^2+\ers(r)\, ,
\end{align}
where $A_{\T_r}$ is the second fundamental form of $\T_r$.  In one of the errors for these estimates is a term which is of strictly larger order than the allowed $\ers(r)$.  We will see that the key ingredient to remove this higher order error term is once again the precise equations being solved by our approximations.  In this case, the Euler Lagrange for the best subspace $\cL_{x,r}$ will remove the most challenging of these errors and reduce us to something estimable.  This will allow us to rewrite our stationary equation estimate as

\begin{align}
\hat\vartheta_n(\T_r) +2\vartheta_\cL(\T_r)	= \hat\vartheta_\alpha(\T_r)+\int_{\T_r}\psi_\T\int\rho_r|\pi^\perp_{x,r}(y-x)|^2|\pi_{x,r}\nabla u|^2+2(m-1)r^4\int_{\T_r}\psi_{\T}\int\rho_r\langle A_{\T_r},\nabla u\rangle^2+\ers(r)\, .\notag\\\notag
\end{align}

{\bf Section \ref{ss:radial_energy:radial} Outline: The Radial Derivative $r\frac{d}{dr}\vartheta_\cL(\T_r)$ .}  In our stationary equation estimate we still have two nonnegative terms to contend with:
\begin{align}
	\int_{\T_r}\psi_\T\int\rho_r|\pi^\perp_{x,r}(y-x)|^2|\pi_{x,r}\nabla u|^2+(m-1)\int_{\T_r}\psi_{\T}\int\rho_r\langle A_{\T_r},\nabla u\rangle^2\, .
\end{align}

Neither of these terms have clear apriori smallness bounds.  The main goal of Section \ref{ss:radial_energy:radial} is Lemma \ref{l:radial_energy:radial_derivative}, which will show the estimate
\begin{align}\label{e:radial_energy:radial_derivative}
	\int_{\T_r}\psi_\T\int\rho_r|\pi^\perp_{x,r}(y-x)|^2|\pi_{x,r}\nabla u|^2+(m-1)\int_{\T_r}\psi_{\T}\int\rho_r\langle A_{\T_r},\nabla u\rangle^2\leq (m-1)r\frac{d}{dr}\vartheta_\cL(\T_r)+\ers(r)\, .\\\notag
\end{align}

There are some subtle points about the above estimate.  One is that, as in previous estimates, there will be a higher order error term which is not clearly controllable.  The Euler-Lagrange equation for the best plane $\cL_{x,r}$ will be needed to reduce this error.
Another subtle observation is that the radial derivative $r\frac{d}{dr}\vartheta_\cL(\T_r)$ is not a signed term.  More importantly, though we clearly have the bound $\int r\frac{d}{dr}\vartheta_\cL(\T_r)\frac{dr}{r}=0$, it is apriori not clear that we have a bound on $\int |r\frac{d}{dr}\vartheta_\cL(\T_r)|\frac{dr}{r}$ .
This is all to emphasize that we will control \eqref{e:radial_energy:radial_derivative} by a term which is Dini integrable, but not apriori absolutely so.  Notice also that the sign of the term with the second fundamental form is essential. Indeed, we do not claim that
\begin{gather}
 (m-1)\int_{\T_r}\psi_{\T}\int\rho_r\langle A_{\T_r},\nabla u\rangle^2
\end{gather}
is an apriori negligible term, but just that it has the right sign for our estimates to work.  
\\

Let us now put this outline into action:\\

\subsection{The Stationary Vector Field}\label{ss:radial_energy:stationary_vector_field}

Let us consider the vector field
\begin{align}\label{e:radial_energy:stationary_vector_field}
	  \xi_r(y) = r^2\int_{\T_r }  \psi_{\T}(x,r) \tilde\rho_{r}(y-x;\LL_{x,r})\pi^\perp_{x,r}(y-x)\, ,
\end{align}
where recall from Definition \ref{d:restricted_energy_functionals} that $\tilde\rho_r$ is well defined by the property that its perpendicular gradient $\pi_{x,r}^\perp\nabla\tilde\rho_r = \hat\rho_r$ vanishes near the bubble region.  We will plug this into the stationary equation \eqref{e:stationary_equation}, and spend the remainder of this section understanding the terms which appear.  The main result of this subsection is the following, which takes the first few steps in this direction:

\begin{lemma}\label{l:radial_energy:stationary_vector_field}
Let $u:B_{10R}(p)\to N$ be a stationary harmonic map with $R^2\fint_{B_{10R}}|\nabla u|^2\leq \Lambda$ , and let $\cA=B_2\setminus \overline{B_{\rf_z}(\T)}$ be a $\delta$-annular region.	Then the following hold:
\begin{align}\label{e:radial_energy:stationary_equation1}
\hat\vartheta_n(\T_r) +2\vartheta_\cL(\T_r)	= \hat\vartheta_\alpha(\T_r)+\int_{\T_r}\int\rho_r(y-x)|\pi^\perp_{x,r}(y-x)|^2\,|\pi_{x,r}\nabla u|^2+\ff_1(r)+\ers(r)\, ,
\end{align}
where $\ers(r)$ is defined by \eqref{e:radial_energy:p_radial_energy_error} and
\begin{align}\label{e:radial_energy:stationary_vector:f_1}
	\ff_1(r)\equiv  -2\int_{\T_r} \psi_{\T}(x,r)\int \rho_r(y-x)\big\langle \nabla u, \pi_{x,r}(y-x)\big\rangle\big\langle \nabla u, \pi^\perp_{x,r}(y-x)\big\rangle \, .
\end{align}
\end{lemma}
\begin{proof}
	Let us use the vector field \eqref{e:radial_energy:stationary_vector_field} in the stationary equation \eqref{e:stationary_equation}. We have
	\begin{align}
	 \nabla \xi_r(y) &= r^2\int_{\T_r }  \psi_{\T}(x,r) \qua{\nabla \tilde\rho_{r}(y-x;\LL_{x,r})\otimes \pi^\perp_{x,r}(y-x) + \tilde \rho_r(y-x;\LL_{x,r}) \pi_{x,r}^\perp}\notag \\
	 &= r^2\int_{\T_r }  \psi_{\T}(x,r) \qua{\pi_{x,r}\nabla \tilde\rho_{r}(y-x;\LL_{x,r})\otimes\pi^\perp_{x,r}(y-x) +\pi_{x,r}^\perp\nabla \tilde\rho_{r}(y-x;\LL_{x,r})\otimes\pi^\perp_{x,r}(y-x) }\notag\\
	 &+r^2\int_{\T_r} \psi_\T(x,r) \tilde \rho_r(y-x;\LL_{x,r}) \pi_{x,r}^\perp\, .
	\end{align}
	By Definition \ref{d:restricted_energy_functionals} of $\tilde \rho$, we have
	\begin{gather}
	 \pi_{x,r}^\perp\nabla \tilde\rho_{r}(y-x;\LL_{x,r})=- \frac{\pi_{x,r}^\perp (y-x)}{r^2} \hat \rho_r(y-x;\LL_{x,r})
	\end{gather}
    Moreover, recall that by Lemma \ref{l:rho_tilde_estimates}, if we set
    \begin{gather}
     \tilde e_r(y-x;\LL_{x,r}) = \rho_r(y-x) - \tilde \rho_r(y-x;\LL_{x,r})\, ,
    \end{gather}
    we get 
    \begin{gather}
     \abs{\tilde e_r(y-x;\LL_{x,r}) }+\abs{r\nabla^{(y)} \tilde e_r(y-x;\LL_{x,r}) }\leq C(m) e^{-R/2}\rho_{2r}(y)\, .
    \end{gather}
    Thus, plugging \eqref{e:radial_energy:stationary_vector_field} into the stationary equation \eqref{e:stationary_equation} we obtain
    \begin{align}
	2 r^2 \int_{\T_r} &\psi_{\T}(x,r) \int \tilde\rho_r(y-x;\cL_{x,r}) \abs{\pi_{x,r} \nabla u(y)}^2 \label{e:radial_energy:stationary1}\\
	&= \int_{\T_r} \psi_{\T}(x,r) \int \hat\rho_r(y-x;\cL_{x,r})\qua{|\nabla u|^2(y)|\pi^\perp_{x,r}(y-x)|^2 - 2\Big\langle \nabla u(y), \pi^\perp_{x,r}(y-x)\Big\rangle^2}\label{e:radial_energy:stationary2}\\
	&\;\;-2\int_{\T_r} \psi_{\T}(x,r) \int (-\dot\rho_r(y-x))\Big\langle \nabla u, \pi_{x,r}(y-x)\Big\rangle\Big\langle \nabla u, \pi^\perp_{x,r}(y-x)\Big\rangle \label{e:radial_energy:stationary3}\\
	&\;\;-2r^2\int_{\T_r} \psi_{\T}(x,r) \int \Big\langle \nabla u(y), \pi_{x,r} \nabla\tilde e_r(y-x;\LL_{x,r})\Big\rangle \Big\langle \nabla u(y), \pi^\perp_{x,r} (y-x)\Big\rangle\label{e:radial_energy:stationary4}
\end{align}

Now we have to deal with some tedious but easy estimates in order to obtain the desired result. Before moving to the specific details, let us just point out that philosophically:
\begin{enumerate}
 \item $\tilde\rho_r(y-x;\cL_{x,r})\approx \rho(y-x)$, so the first term is easy enough
 \item the second term will be split naturally into $\hat \vartheta_\alpha-\hat \vartheta_n$ plus the $\LL$ energy component.
 \item since $-\dot \rho \approx \rho$, the third term gives rise to $\ff_1(r)$ plus a small error
 \item the last term is controllably small
\end{enumerate}

As done previously, we split the proof into smaller subclaims, hopefully easier to analyze for the reader.\\

\textbf{Claim:} for \eqref{e:radial_energy:stationary1}, we have
\begin{align}
 2 r^2 \int_{\T_r} \psi_{\T}(x,r) \int \tilde\rho_r(y-x;\cL_{x,r}) \abs{\pi_{x,r} \nabla u(y)}^2  =& 2 r^2 \int_{\T_r} \psi_{\T}(x,r) \int \rho_r(y-x) \abs{\pi_{x,r} \nabla u(y)}^2+\ers(r)\notag\\
 =&2 \vartheta_\LL(\T_r)+\ers(r)\, .
\end{align}
The estimates on $\ers$ come directly from Lemma \ref{l:rho_tilde_estimates} and \eqref{e:vartheta_L_r}. In particular we have
\begin{align}
 &2 r^2 \int_{\T_r} \psi_{\T}(x,r) \int \tilde e_r(y-x;\cL_{x,r}) \abs{\pi_{x,r} \nabla u(y)}^2\leq C e^{-R/2} r^2 \int_{\T_r} \psi_{\T}(x,r) \int \rho_{2r}(y-x) \abs{\pi_{x,r} \nabla u(y)}^2\notag \\
 =& C e^{-R/2}  \int_{\T_r} \psi_{\T}(x,r) \vartheta(x,2r;\LL_{x,r})\leq C(m)e^{-R/2} \int_{\T_r} \psi_{\T}(x,r) \vartheta_\LL(x,2r)\, .
\end{align}
This proves that the error is of the form \eqref{e:radial_energy:p_radial_energy_error}.\\

\textbf{Claim:} for \eqref{e:radial_energy:stationary2}, we have
\begin{align}
 &\int_{\T_r} \psi_{\T}(x,r) \int \hat\rho_r(y-x;\cL_{x,r})\qua{|\nabla u|^2(y)|\pi^\perp_{x,r}(y-x)|^2 - 2\Big\langle \nabla u(y), \pi^\perp_{x,r}(y-x)\Big\rangle^2}\\
 =& \hat \vartheta_\alpha(\T_r)- \hat \vartheta_n(\T_r) +\int_{\T_r} \psi_{\T}(x,r) \int \rho_r(y-x)\abs{\pi_{x,r} \nabla u(y)}^2 \abs{\pi^\perp_{x,r}(y-x)}^2  + \ers(r)\, .
\end{align}
This follows immediately from the definitions, Lemma \ref{l:rho_tilde_estimates} and \eqref{e:vartheta_L_r}. In particular, we have
\begin{align}
 \int_{\T_r} \psi_{\T}(x,r) \int \abs{\rho_r(y-x)-\hat \rho_r(y-x;\LL_{x,r})}\abs{\pi_{x,r} \nabla u(y)}^2 \abs{\pi^\perp_{x,r}(y-x)}^2\leq C(m) e^{-R}\vartheta_\LL(\T_{2r})\, .
\end{align}

\vspace{.3cm}
\textbf{Claim:} for \eqref{e:radial_energy:stationary3} we have
\begin{align}
 &-2\int_{\T_r} \psi_{\T}(x,r) \int (-\dot\rho_r(y-x))\Big\langle \nabla u, \pi_{x,r}(y-x)\Big\rangle\Big\langle \nabla u, \pi^\perp_{x,r}(y-x)\Big\rangle\notag \\
 =& -2\int_{\T_r} \psi_{\T}(x,r) \int \rho_r(y-x)\Big\langle \nabla u, \pi_{x,r}(y-x)\Big\rangle\Big\langle \nabla u, \pi^\perp_{x,r}(y-x)\Big\rangle+\ers(r)=\ff_1(r)+\ers(r)
\end{align}
This follows immediately from \eqref{e:rho_primitive_difference} and \eqref{e:trho_doubleradius}, which yield
\begin{align}
 &2\int_{\T_r} \psi_{\T}(x,r) \int \abs{\rho_r(y-x)-\dot\rho_r(y-x)}\Big\langle \nabla u, \pi_{x,r}(y-x)\Big\rangle\Big\langle \nabla u, \pi^\perp_{x,r}(y-x)\Big\rangle\notag \\
 \leq& C(m) e^{-R/2}\int_{\T_r} \psi_\T(x,r) \qua{ r^2\int \rho_{2r}(y-x) \abs{\pi_{x,r}\nabla u(y)}^2 + \int \rho_{2r}(y-x) \ps{\nabla u(y)}{y-x}^2}\notag\\
 \leq& C(m)e^{-R/2}\int_{\T_r} \psi_\T(x,r) \qua{\vartheta_\LL(x,2r)+r\dot \vartheta(x,2r)}\, ,
\end{align}
where we used \eqref{e:vartheta_L_r} in the last line.\\

\textbf{Claim:} for \eqref{e:radial_energy:stationary4} we have
\begin{align}
 \abs{2r^2\int_{\T_r} \psi_{\T}(x,r) \int \Big\langle \nabla u(y), \pi_{x,r} \nabla\tilde e_r(y-x;\LL_{x,r})\Big\rangle \Big\langle \nabla u(y), \pi^\perp_{x,r} (y-x5)\Big\rangle }= \ers(r)\, .
\end{align}

In order to prove this, observe that by Lemma \ref{l:rho_tilde_estimates} and Cauchy-Schwartz:
\begin{align}
 &\abs{\Big\langle \nabla u(y), \pi_{x,r} \nabla\tilde e_r(y-x;\LL_{x,r})\Big\rangle \Big\langle \nabla u(y), \pi^\perp_{x,r} (y-x)\Big\rangle }\notag \\
 \leq & C(m) e^{-R/2}\rho_{2r}(y-x) \qua{\abs{\pi_{x,r}\nabla u(y)}^2 + \ps{\nabla u(y)}{y-x}^2}\, .
\end{align}
This implies that 
\begin{align}
 &\abs{2r^2\int_{\T_r} \psi_{\T}(x,r) \int \Big\langle \nabla u(y), \pi_{x,r} \nabla\tilde e_r(y-x;\LL_{x,r})\Big\rangle \Big\langle \nabla u(y), \pi^\perp_{x,r} (y-x)\Big\rangle }\notag \\
 \leq &C(m)e^{-R/2}\int_{\T_r} \psi_{\T}(x,r) \int \rho_{2r}(y-x)\qua{r^2  \abs{\pi_{x,r} \nabla u(y)}^2 + \ps{\nabla u(y)}{y-x}^2}\notag\\
 =&C(m)e^{-R/2}\int_{\T_r} \psi_{\T}(x,r) \qua{\vartheta(x,2r;\LL_{x,r}) + r\dot \vartheta(x,2r)}\, .
\end{align}
By \eqref{e:vartheta_L_r}, we conclude the estimate.

\end{proof}

\vspace{.3cm}
\subsection{Rewriting \texorpdfstring{$\ff_1(r)$}{f1(r)} using the Euler Lagrange Equation for \texorpdfstring{$\T_r$}{Tr}}\label{ss:radial_energy:f1}

Our goal in this subsection is to begin rewriting the uncontrollable error $\ff_1(r)$ in a more manageable fashion.  There will be two key points in this subsection.  First is to view the $\rho_r(y-x) \pi_{x,r}(y-x)$ term as a gradient and use an integration by parts to remove the $\pi_{x,r}(y-x)$ from the equation.  This is important as this reduces $\ff_1$ to a form where we can apply the Euler Lagrange equation for $\T_r$ and turn $L^\perp$-gradients into $L$-gradients.  The main result of this subsection is the following:

\begin{lemma}\label{l:radial_energy:f1}
	Let $\ff_1(r)$ be as in \eqref{e:radial_energy:stationary_vector:f_1}.  Then we can write $\ff_1(r)= \ff_2(r)+\ff_3(r)+\ers(r)$ where
\begin{align}
	\ff_2(r)&=-2r^2\int_{\T_r} \psi_{\T}(x,r)\int \rho_r(y-x)\big\langle \nabla u(y),H_{\T_r}(x)\big\rangle\big\langle \nabla u(y), \pi_{\T_r(x)}(y-x)\big\rangle \, ,\notag\\
	\ff_3(r)&=-2r^2\int_{\T_r} \psi_{\T}(x,r)\int \rho_r(y-x)\ps{\II_{\T_r(x)}(\pi_{\T_r(x)}\nabla u(y),\pi_{\T_r(x)}(y-x))}{\nabla u(y) }\, ,
\end{align}
with $H_{\T_r}$ the mean curvature of $\T_r$, $\II_{\T_r(x)}$ the second fundamental form of $\T_r$ at $x$, and $\ers(r)$ as in \eqref{e:radial_energy:p_radial_energy_error} .
\end{lemma}

We will prove this in several steps throughout this subsection:

\vspace{.3cm}
\subsubsection{Exchanging $\rho$ with $\dot \rho$.}\label{sss:integrating_rho_vs_dot_rho} Using Lemma \ref{l:rho_basic_properties}, we show that
\begin{align}
 \ff_1(r)=& -2\int_{\T_r} \psi_{\T}(x,r)\int \rho_r(y-x)\big\langle \nabla u, \pi_{x,r}(y-x)\big\rangle\big\langle \nabla u, \pi^\perp_{x,r}(y-x)\big\rangle\\
 =&2\int_{\T_r} \psi_{\T}(x,r)\int \dot \rho_r(y-x)\big\langle \nabla u, \pi_{x,r}(y-x)\big\rangle\big\langle \nabla u, \pi^\perp_{x,r}(y-x)\big\rangle+\ers
\end{align}
In order to do so, we use \eqref{e:rho_primitive_difference} to see that
\begin{align}
 &2\int_{\T_r} \psi_{\T}(x,r)\int \abs{\rho_r(y-x)-\dot \rho_r(y-x)}\abs{\big\langle \nabla u, \pi_{x,r}(y-x)\big\rangle} \abs{\big\langle \nabla u, \pi^\perp_{x,r}(y-x)\big\rangle}\notag\\
 \leq & C e^{-R/2}\int_{\T_r} \psi_{\T}(x,r)\int \rho_{2r}(y-x)\abs{\big\langle \nabla u, \pi_{x,r}(y-x)\big\rangle} \abs{\big\langle \nabla u, \pi^\perp_{x,r}(y-x)\big\rangle}\notag\\
 \leq & C e^{-R/2}\int_{\T_r} \psi_{\T}(x,r)\ton{\sqrt{\vartheta_\LL(x,2r)}\sqrt{r\dot \vartheta(x,2r)}}\, ,
\end{align}
where we used \eqref{e:vartheta_L_r} in the last estimate. By Cauchy-Schwarts, the last term is of the $\ers(r)$ form.

\vspace{.3cm}
\subsubsection{Rewriting $\ff_1$ in terms of $\pi_{\T_r}$ .  }\label{sss:exchange_pi} Here we use an integration by parts trick similar to the one used to prove \eqref{e:claim_superconvexity_3}. First of all, we notice that by Lemma \ref{l:improved_comparison_LL_T} we can exchange $\pi_{x,r}$ with $\pi_{\T_r(x)}$ up to paying a controllable error. In particular, we claim that
\begin{align}
 \ff_1(r)=& 2\int_{\T_r} \psi_{\T}(x,r)\int \dot \rho_r(y-x)\big\langle \nabla u, \pi_{x,r}(y-x)\big\rangle\big\langle \nabla u, \pi^\perp_{x,r}(y-x)\big\rangle+\ers(r)\\
 =&2\int_{\T_r} \psi_{\T}(x,r)\int \dot \rho_r(y-x)\big\langle \nabla u, \pi_{\T_r(x)}(y-x)\big\rangle\big\langle \nabla u, \pi^\perp_{\T_r(x)}(y-x)\big\rangle+\ers(r)\, .
\end{align}
This comes from the fact that by Lemma \ref{l:improved_comparison_LL_T}, \eqref{e:trho_doubleradius} and standard estimates:
\begin{align}
&\abs{\int_{\T_r} \psi_{\T}(x,r)\int \dot \rho_r(y-x)\big\langle \nabla u, \pi_{\T_r(x)}(y-x)-\pi_{x,r}(y-x)\big\rangle\big\langle \nabla u, \pi^\perp_{x,r}(y-x)\big\rangle}\notag \\
\leq &\int_{\T_r} \psi_{\T}(x,r)\int \dot \rho_r(y-x)\abs{\nabla u}\abs{y-x}\abs{\pi_{\T_r(x)}-\pi_{x,r}}\abs{\big\langle \nabla u, \pi^\perp_{x,r}(y-x)\big\rangle}\notag \\
\leq & \int_{\T_r} \psi_{\T}(x,r)\abs{\pi_{\T_r(x)}-\pi_{x,r}}\sqrt{\vartheta(x,2r)}\sqrt{r\dot \vartheta(x,2r)}\notag \\
\leq & C(m,\epsilon_0) \int_{\T_r} \psi_{\T}(x,r)\ton{\sqrt{\vartheta_\LL(x,2r)} r\dot \vartheta(x,2r) + e^{-R/2} \sqrt{\vartheta_\LL(x,2r)} \sqrt{r\dot \vartheta(x,2r)}}\notag \\
\leq & C(m,\epsilon_0) \int_{\T_r} \psi_{\T}(x,r)\ton{\sqrt \delta + e^{-R/2}}\sqrt{\vartheta_\LL(x,2r)} \sqrt{r\dot \vartheta(x,2r)}=\ers(r)\, .
\end{align}
In a similar way, we can exchange $\pi_{x,r}^\perp$ with $\pi_{\T_r(x)}^\perp$ and obtain the desired result.

\vspace{.3cm}
\subsubsection{Integration by Parts. }\label{sss:integrating_by_parts}
We now have
\begin{align}
 \ff_1(r)=& 2\int_{\T_r} \psi_{\T}(x,r)\int \dot \rho_r(y-x)\big\langle \nabla u, \pi_{\T_r(x)}(y-x)\big\rangle\big\langle \nabla u, \pi^\perp_{\T_r(x)}(y-x)\big\rangle+\ers(r)\notag\\
 =& -2\int_{\T_r} \psi_{\T}(x,r)\int \ps{\pi_{\T_r(x)}\nabla^{(x)}\rho_r(y-x)}{\nabla u(y)}\big\langle \nabla u, \pi_{\T_r(x)}(y-x)\big\rangle+\ers(r)\notag\\
 =&2\int_{\T_r} \psi_{\T}(x,r)\int \rho_r(y-x) \ r^2 \ps{\operatorname{div}^{(x)}_{\T_r(x)}(\pi_{\T_r(x)})}{\nabla u(y)}\ps{\nabla u(y)}{\pi^\perp_{\T_r(x)}(y-x)} \notag\\
 &-2\int_{\T_r} \psi_{\T}(x,r) \int \rho_r(y-x) r^2 \ps{\nabla u(y)\otimes \nabla u(y)}{\nabla^{(x)}_{\T_r}(\pi_{\T_r(x)})(y-x)}+\ers(r)\, .
\end{align}
In order to prove this, observe that
\begin{gather}
 \nabla^{(x)} \rho_r(y-x)=r^{-m}\nabla^{(x)} \rho\ton{\frac{\abs{y-x}^2}{2r^2}} = r^{-m}\dot \rho \ton{\frac{\abs{y-x}^2}{2r^2}} \frac{x-y}{r^2}\, ,
\end{gather}
Notice that
\begin{gather}
\nabla^{(y)}\rho_r(y-x)=\dot \rho_r(y-x) \frac{y-x}{r} = -\dot \rho_r(y-x)\frac{x-y}{r} = -\nabla^{(x)}\rho_r(y-x) \, .
\end{gather}
We underline this because signs are important in this computation, and it is easy to get confused.

From this we get
\begin{align}
 & 2\int_{\T_r} \psi_{\T}(x,r)\qua{\int \dot \rho_r(y-x)\big\langle \nabla u(y), \pi_{\T_r(x)}(y-x)\big\rangle\big\langle \nabla u(y), \pi^\perp_{\T_r(x)}(y-x)\big\rangle dy }\,dv_{\T_r}(x)\notag \\
 =&-2\int \nabla u(y)\otimes \nabla u(y) \qua{\int_{\T_r} \psi_{\T}(x,r)  \ \ton{r^2 \pi_{\T_r(x)}\nabla^{(x)} \rho_r(y-x) }\otimes \pi^\perp_{\T_r(x)}(y-x) \,dv_{\T_r}(x) }dy
\end{align}
As before, we integrate by parts in the variable $x\in \T_r$, and we get
\begin{align}
 &-\int_{\T_r} \psi_{\T}(x,r)  \ \ton{r^2 \pi_{\T_r(x)}\nabla^{(x)} \rho_r(y-x) }\otimes \pi^\perp_{\T_r(x)}(y-x) \notag\\
 =&  \int_{\T_r} \rho_r(y-x)  \ \ton{r^2 \pi_{\T_r(x)}\nabla^{(x)} \psi_{\T}(x,r)}\otimes \pi^\perp_{\T_r(x)}(y-x) \notag\\
 &+ \int_{\T_r} \psi_{\T}(x,r) \rho_r(y-x)  \ r^2 \operatorname{div}^{(x)}_{\T_r(x)}(\pi_{\T_r(x)}) \otimes \pi^\perp_{\T_r(x)}(y-x) \notag\\
 &-  \int_{\T_r} \psi_{\T}(x,r) \rho_r(y-x) r^2 \nabla^{(x)}_{\T_r}(\pi_{\T_r(x)})(y-x)\, ,
\end{align}
where we used the fact that $\pi_{\T_r(x)}^\perp$ is null on any tangent vector on $\T_r$ and $\nabla \pi_{\T_r(x)}^\perp = - \nabla \pi_{\T_r(x)}$.

Now we notice that the first term is of the form $\ers$, indeed:
\begin{align}\label{e:nabla_psi_is_ers}
 &\abs{\int \nabla u(y)\otimes \nabla u(y) \int_{\T_r} \rho_r(y-x)  \ \ton{r^2 \pi_{\T_r(x)}\nabla^{(x)} \psi_{\T}(x,r)}\otimes \pi^\perp_{\T_r(x)}(y-x) \dvx \dVy}\notag\\
 \leq &\int_{\T_r} r\abs{\nabla \psi_\T(x,r)} \int \rho_r(y-x) r\abs{\nabla u(y)}\abs{\ps{\nabla u(y)}{y-x}} \leq  C(m)\sqrt{\Lambda}\sqrt{\delta} \int_{\T_r} r\abs{\nabla \psi_\T(x,r)}\, .
\end{align}
This concludes the proof.

\subsubsection{Mean curvature and second fundamental form.}\label{sss:mean_and_second} Up to $\ers$ errors, we can write $\ff_1(r)$ in terms of the mean curvature and second fundamental form of $\T_r$. Recall that the second fundamental form $\II$ of $\T_r$ applied to the tangent vectors $V,W\in T_x\T_r$ gives $\II_{T_x\T_r}(V,W)=\pi_{\T_r(x)}^\perp\ton{\nabla_V W}=\II_{T_x\T_r}(W,V)$, and the mean curvature vector $H_{\T_r(x)}$ is just the trace over $T_x\T_r$ of $\II_{T_x\T_r}$.

We claim that
\begin{align}\label{e:ff1-2-3_first}
 \ff_1(r)=&2\int_{\T_r} \psi_{\T}(x,r)\int \rho_r(y-x) \ r^2 \ps{H_{\T_r(x)}}{\nabla u(y)}\ps{\nabla u(y)}{\pi^\perp_{\T_r(x)}(y-x)} \notag\\
 &+\ff_3(r)+\ers(r)
\end{align}

Indeed, we can write
\begin{align}\label{e:divergence_to_mean_curvature}
 &2\int_{\T_r} \psi_{\T}(x,r)\int \rho_r(y-x) \ r^2 \ps{\operatorname{div}^{(x)}_{\T_r(x)}(\pi_{\T_r(x)})}{\nabla u(y)}\ps{\nabla u(y)}{\pi^\perp_{\T_r(x)}(y-x)} \notag\\
 = & 2\int_{\T_r} \psi_{\T}(x,r)\int \rho_r(y-x) \ r^2 \ps{H_{\T_r}(x)} {\nabla u(y)}\ps{\nabla u(y)}{\pi^\perp_{\T_r(x)}(y-x)}+\ers(r)\, .
\end{align}
This comes from splitting
\begin{gather}
 \operatorname{div}^{(x)}_{\T_r(x)}(\pi_{\T_r(x)})= \pi_{\T_r(x)} \qua{\operatorname{div}^{(x)}_{\T_r(x)}(\pi_{\T_r(x)})} +\pi_{\T_r(x)}^\perp \qua{\operatorname{div}^{(x)}_{\T_r(x)}(\pi_{\T_r(x)})}\, ,
\end{gather}
and observing that the parallel component give rise to an $\ers$ term since $r\abs{\nabla \pi_{\T_r(x)}}\leq C\sqrt \delta$ and
\begin{align}
 &\abs{\int \rho_r(y-x) \ r^2 \ps{\pi_{\T_r(x)}\operatorname{div}^{(x)}_{\T_r(x)}(\pi_{\T_r(x)})}{\nabla u(y)}\ps{\nabla u(y)}{\pi^\perp_{\T_r(x)}(y-x)}} \notag\\
 \leq & r\abs{\nabla \pi_{\T_r(x)}} \int \rho_r(y-x) r\abs{\pi_{\T_r(x)}\nabla u(y)}\abs{\ps{\nabla u}{y-x}}\notag\\
 \stackrel{\eqref{e:estimate_LL_T}}{\leq} & C(m)\sqrt \delta  \int \rho_{2r}(y-x) \qua{r^2\abs{\pi_{x,2r} \nabla u(y)}^2+ \abs{\ps{\nabla u}{y-x}}^2}
\end{align}

As for the second piece, proceeding in a similar way we obtain that
\begin{align}
 &-2r^2\int_{\T_r} \psi_{\T}(x,r) \int \rho_r(y-x) \ps{\nabla u(y)\otimes \nabla u(y)}{\nabla^{(x)}_{\T_r}(\pi_{\T_r(x)})(y-x)}\notag\\
 =&-2r^2\int_{\T_r} \psi_{\T}(x,r)\int \rho_r(y-x)\ps{\II_{\T_r(x)}(\pi_{\T_r(x)}\nabla u(y),\pi_{\T_r(x)}(y-x))}{\nabla u(y) }+\ers(r)\notag\\
 =&\ff_3(r)+\ers(r)\, .
\end{align}

\vspace{.3cm}
\subsubsection{Euler Lagrange for $\T_r$. }

Here we use the definition of $\T_r$, and in particular Theorem \ref{t:approximating_submanifold}.\eqref{i:EL_Tr}, to prove that we can turn \eqref{e:ff1-2-3_first} into
\begin{gather}
 \ff_1(r)= \ff_2(r)+\ff_3(r)+\ers(r)\, .
\end{gather}
This is equivalent to showing that
\begin{gather}
 2\int_{\T_r} \psi_{\T}(x,r)\int \rho_r(y-x) \ r^2 \ps{H_{\T_r(x)}}{\nabla u(y)}\ps{\nabla u(y)}{\pi^\perp_{\T_r(x)}(y-x)}= \ff_2(r)+\ers(r)\, ,
\end{gather}
where recall
\begin{align}
	\ff_2(r)&=-2r^2\int_{\T_r} \psi_{\T}(x,r)\int \rho_r(y-x)\big\langle \nabla u(y),H_{\T_r}(x)\big\rangle\big\langle \nabla u(y), \pi_{\T_r(x)}(y-x)\big\rangle \, .
\end{align}

First of all, as done previously we can exchange $\rho_r(y-x)$ with $-\dot \rho_r(y-x)$, up to an $\ers(r)$ error.

Moreover, by the same estimates as in subsection \ref{sss:exchange_pi}, we can exchange $\pi_{\T_r(x)}$ with $\pi_{x,r}$ and get 
\begin{align}
 &2\int_{\T_r} \psi_{\T}(x,r)\int (-\dot \rho_r(y-x)) \ r^2 \ps{H_{\T_r(x)}}{\nabla u(y)}\ps{\nabla u(y)}{\pi^\perp_{\T_r(x)}(y-x)}\notag \\
 =&2\int_{\T_r} \psi_{\T}(x,r)\int (-\dot \rho_r(y-x)) \ r^2 \ps{H_{\T_r(x)}}{\nabla u(y)}\ps{\nabla u(y)}{\pi^\perp_{x,r}(y-x)}+\ers(r)\notag \\
 =&2\int_{\T_r} \psi_{\T}(x,r)\int (-\dot \rho_r(y-x)) \ r^2 \ps{ \pi_{x,r}^\perp H_{\T_r(x)}}{\nabla u(y)}\ps{\nabla u(y)}{\pi^\perp_{x,r}(y-x)}+\ers(r)\, .
\end{align}
In the last equality we used the fact that, since $H_{\T_r(x)}$ is orthogonal to $T_x\T_r$, and given Lemma \ref{l:improved_comparison_LL_T}, for all $x\in \T_r$ with $r\geq \rf_x$:
\begin{align}
 &\abs{\int (-\dot \rho_r(y-x)) \ r^2 \ps{ \pi_{x,r} H_{\T_r(x)}}{\nabla u(y)}\ps{\nabla u(y)}{\pi^\perp_{x,r}(y-x)}}\notag\\
 \leq & C(m,\epsilon_0) \vartheta_\LL(x,2r)\sqrt{r\dot \vartheta(x,2r)} \leq C(m,\epsilon_0) \sqrt \delta \vartheta_{\LL}(x,2r)\, .
\end{align}

Now we can apply the Euler-Lagrange of Theorem \ref{t:approximating_submanifold}.\eqref{i:EL_Tr} in order to exchange the $\cL_{x,r}^\perp$ derivative for a $\cL_{x,r}$ derivative and obtain that
\begin{align}
 &2\int_{\T_r} \psi_{\T}(x,r)\int (-\dot \rho_r(y-x)) \ r^2 \ps{ \pi_{x,r}^\perp H_{\T_r(x)}}{\nabla u(y)}\ps{\nabla u(y)}{\pi^\perp_{x,r}(y-x)}\notag\\
 =&-2\int_{\T_r} \psi_{\T}(x,r)\int (-\dot \rho_r(y-x)) \ r^2 \ps{ \pi_{x,r}^\perp H_{\T_r(x)}}{\nabla u(y)}\ps{\nabla u(y)}{\pi_{x,r}(y-x)}\, ,
\end{align}
and proceeding with the usual estimates turn everything back to
\begin{align}
&-2\int_{\T_r} \psi_{\T}(x,r)\int (-\dot \rho_r(y-x)) \ r^2 \ps{ \pi_{x,r}^\perp H_{\T_r(x)}}{\nabla u(y)}\ps{\nabla u(y)}{\pi_{x,r}(y-x)}\notag \\
= & -2\int_{\T_r} \psi_{\T}(x,r)\int  \rho_r(y-x) \ r^2 \ps{ H_{\T_r(x)}}{\nabla u(y)}\ps{\nabla u(y)}{\pi_{\T_r(x)}(y-x)}+\ers(r)= \ff_2(r)+\ers(r)\, .
\end{align}
This concludes this subsection.

\vspace{.3cm}
\subsection{Estimating \texorpdfstring{$\ff_2+\ff_3\leq 2(m-1)r^4\int_{\T_r}\psi_{\T}\int\rho_r \langle A_{\T_r},\nabla u\rangle^2+\ers(r)$}{errors} using the Euler-Lagrange for the best subspace \texorpdfstring{$\cL_{x,r}$}{Lxr}. }\label{ss:radial_energy:f2f3}

We have now rewritten our stationary equation \eqref{e:radial_energy:stationary_equation1} as
\begin{align}\label{e:radial_energy:stationary_equation2}
\hat\vartheta_n(\T_r) +2\vartheta_\cL(\T_r)	= \hat\vartheta_\alpha(\T_r)+\int_{\T_r}\int\rho_r\,|\pi^\perp_{x,r}(y-x)|^2|\pi_{x,r}\nabla u|^2+\ff_2(r)+\ff_3(r)+\ers(r)\, .
\end{align}

Our main result for this subsection is the following

%
%
\begin{lemma}\label{l:radial_energy:f2f3}
	Let $\ff_2(r)$ and $\ff_3(r)$ be as in Lemma \ref{l:radial_energy:f1}, then the following hold:
\begin{enumerate}
	\item We can estimate $\ff_2(r)$:
\begin{align}
	\ff_2(r)&= 2r^4\int_{\T_r}\psi_{\T}(x,r)\int\rho_r(y-x) \langle H_{\T_r}(x),\nabla u(y)\rangle^2+\ers(r)\notag\\
	&\leq 2(m-2)r^4 \int_{\T_r}\psi_{\T}(x,r)\int\rho_r(y-x) \langle \II_{\T_r}(x),\nabla u(y)\rangle^2+\ers(r)
\end{align}
\item  We can estimate $\ff_3(r)$:
\begin{align}
	\ff_3(r)&= 2r^4\int_{\T_r}\psi_{\T}(x,r)\int\rho_r(y-x) \langle \II_{\T_r}(x),\nabla u(y)\rangle^2+\ers(r)\, ,
\end{align}
\end{enumerate}
with $H_{\T_r}$ the mean curvature of $\T_r$, $A_{\T_r}$ the second fundamental form of $\T_r$. Here
\begin{gather}
 \ps{\II_{\T_r}(x)}{\nabla u(y)} = \II_{ij}^k(x) \nabla_k u^s(y)
\end{gather}
is a symmetric $2$-tensor on the tangent space $T_x\T_r$, with values in $T_{u(y)} N \subset \R^N$, and by $ \ps{\II_{\T_r}(x)}{\nabla u(y)}^2$ we mean the Hilbert-Schmidt square norm of this vector valued $2$-form.
\\
\end{lemma}

There are three main ingredients in the above estimate.  The first inequality in both Lemma \ref{l:radial_energy:f2f3}.1 and \ref{l:radial_energy:f2f3}.2 is a rewriting of $\ff_2(r)$ and $\ff_3(r)$ based on a trick not dissimilar from our estimate on $\ff_1(r)$.  Namely we will use an integration by parts to remove the linear $\pi_{\T_r(x)}(y-x)$ term, and study what remains.  We will in the process run into another error which is of strictly larger order than our allowed $\ers(r)$ errors.  This time we will see how to rewrite this error, up to lower order errors, in terms of the Euler Lagrange equation for the best plane $\cL_{x,r}$.  The second inequality Lemma \ref{l:radial_energy:f2f3}.1 will also be based on the (standard) matrix inequality $(\operatorname{tr}(M))^2\leq \operatorname{dim}(M)\,|M|^2$, where $\abs{M}$ is the Hilbert-Schmidt norm of $M$.

\vspace{.3cm}
\subsubsection{Integrating by parts in $\ff_2$.}
We claim that
\begin{align}
	\ff_2(r)&=2r^4\int_{\T_r}\psi_{\T}(x,r)\int\rho_r(y-x) \langle H_{\T_r}(x),\nabla u(y)\rangle^2+\ers(r)\, .
\end{align}

In order to prove this, we proceed as done previously in Subsections \ref{sss:integrating_rho_vs_dot_rho} and \ref{sss:integrating_by_parts}. In particular, we can write
\begin{align}
	\ff_2(r)&=-2r^2\int_{\T_r} \psi_{\T}(x,r)\int (-\dot \rho_r(y-x))\big\langle \nabla u(y),H_{\T_r}(x)\big\rangle\big\langle \nabla u(y), \pi_{\T_r(x)}(y-x)\big\rangle +\ers(r)\notag\\
	&=-2r^4\int_{\T_r} \psi_{\T}(x,r)\int \ps{\pi_{\T_r(x)}\nabla^{(x)} \rho_r(y-x)}{\nabla u (y)}\big\langle \nabla u(y),H_{\T_r}(x)\big\rangle+\ers(r)\, .
\end{align}
Now we can integrate by parts wrt $x\in \T_r$ and get
\begin{align}
 -&2r^4\int_{\T_r} \psi_{\T}(x,r)\int \ps{\pi_{\T_r(x)}\nabla^{(x)} \rho_r(y-x)}{\nabla u (y)}\big\langle \nabla u(y),H_{\T_r}(x)\big\rangle\notag\\
 =&2r^4\int_{\T_r} \int \rho_r(y-x)\ps{\pi_{\T_r(x)}\nabla^{(x)} \psi_{\T}(x,r)}{\nabla u (y)}\big\langle \nabla u(y),H_{\T_r}(x)\big\rangle\label{e:f2_est_1}\\
 + & 2r^4\int_{\T_r} \psi_{\T}(x,r)\int \rho_r(y-x)\ps{\operatorname{div}^{(x)}_{\T_r(x)}(\pi_{\T_r(x)}) }{\nabla u (y)}\big\langle \nabla u(y),H_{\T_r}(x)\big\rangle\label{e:f2_est_2}\\
 + &2r^4\int_{\T_r} \psi_{\T}(x,r)\int \rho_r(y-x)\big\langle \nabla u(y),\nabla^{(x)}_{\pi_{\T_r(x)} \nabla u(y)}H_{\T_r}(x)\big\rangle\label{e:f2_est_3}
\end{align}
Arguing as in the proof of \eqref{e:nabla_psi_is_ers}, the \eqref{e:f2_est_1} term is an $\ers(r)$ term. We can handle the \eqref{e:f2_est_2} term as in Subsection \ref{sss:mean_and_second}, and obtain that
\begin{align}
 &2r^4\int_{\T_r} \psi_{\T}(x,r)\int \rho_r(y-x)\ps{\operatorname{div}^{(x)}_{\T_r(x)}(\pi_{\T_r(x)}) }{\nabla u (y)}\big\langle \nabla u(y),H_{\T_r}(x)\big\rangle\notag\\
 =&2r^4\int_{\T_r} \psi_{\T}(x,r)\int \rho_r(y-x)\big\langle \nabla u(y),H_{\T_r}(x)\big\rangle^2 + \ers(r)
\end{align}
We are left to deal with \eqref{e:f2_est_3}.

\textbf{We claim that}:
\begin{gather}
 2r^4\int_{\T_r} \psi_{\T}(x,r)\int \rho_r(y-x)\big\langle \nabla u(y),\nabla^{(x)}_{\pi_{\T_r(x)} \nabla u(y)}H_{\T_r}(x)\big\rangle=\ers(r)\, .
\end{gather}
In order to prove the claim, we split it into:
\begin{align}
 &2r^4\int_{\T_r} \psi_{\T}(x,r)\int \rho_r(y-x)\big\langle \nabla u(y),\nabla^{(x)}_{\pi_{\T_r(x)} \nabla u(y)}H_{\T_r}(x)\big\rangle\notag\\
 =&2r^4\int_{\T_r} \psi_{\T}(x,r)\int \rho_r(y-x)\big\langle \pi_{x,r}\nabla u(y),\nabla^{(x)}_{\pi_{\T_r(x)} \nabla u(y)}H_{\T_r}(x)\big\rangle\notag \\
 +&2r^4\int_{\T_r} \psi_{\T}(x,r)\int \rho_r(y-x)\big\langle \pi_{x,r}^\perp\nabla u(y),\nabla^{(x)}_{\pi_{\T_r(x)} \nabla u(y)}H_{\T_r}(x)\big\rangle
\end{align}
By Remark \ref{r:T_r_II_estimates} and Corollary \ref{c:estimate_LL_T}, we get
\begin{align}
 &\abs{2r^4\int_{\T_r} \psi_{\T}(x,r)\int \rho_r(y-x)\big\langle \pi_{x,r}\nabla u(y),\nabla^{(x)}_{\pi_{\T_r(x)} \nabla u(y)}H_{\T_r}(x)\big\rangle}\notag \\
 \leq & 2 C(m,\epsilon_0) \sqrt \delta \int_{\T_r} \psi_{\T}(x,r) \vartheta_\LL(x,2r) = \ers(r)\, .
\end{align}
As for the second term with the perpendicular projection, by Remark \ref{r:T_r_II_estimates} and Lemma \ref{l:improved_comparison_LL_T} we can exchange $\pi_{\T_r(x)}$ with $\pi_{x,r}$ paying just an $\ers$ price, as done previously. In other words:
\begin{align}
 &2r^4\int_{\T_r} \psi_{\T}(x,r)\int \rho_r(y-x)\big\langle \pi_{x,r}^\perp\nabla u(y),\nabla^{(x)}_{\pi_{\T_r(x)} \nabla u(y)}H_{\T_r}(x)\big\rangle\notag \\
 =&2r^4\int_{\T_r} \psi_{\T}(x,r)\int \rho_r(y-x)\big\langle \pi_{x,r}^\perp\nabla u(y),\nabla^{(x)}_{\pi_{x,r} \nabla u(y)}H_{\T_r}(x)\big\rangle+\ers(r)\, .
\end{align}
But now we notice that for each $x$ fixed, the Euler-Lagrange equation coming from the definition of best plane $\LL_{x,r}$ yields
\begin{align}
 &\int \rho_r(y-x)\big\langle \pi_{x,r}^\perp\nabla u(y),\nabla^{(x)}_{\pi_{x,r} \nabla u(y)}H_{\T_r}(x)\big\rangle\notag \\
 =&\sum_k\int \rho_r(y-x)\big\langle \nabla u(y),\pi_{x,r}^\perp\nabla^{(x)}_{\pi_{x,r}(e_k)}H_{\T_r}(x)\big\rangle\ps{\nabla u(y)}{\pi_{x,r}(e_k)}=0\, .
\end{align}

\subsubsection{Integrating by parts in $\ff_3$.}
We claim that
\begin{align}
	\ff_3(r)&=2r^4\int_{\T_r}\psi_{\T}(x,r)\int\rho_r(y-x) \langle \II_{\T_r}(x),\nabla u(y)\rangle^2 + \ers(r)\, .
\end{align}

In order to prove this, we proceed as done previously and write
\begin{align}
	\ff_3(r)&=-2r^2\int_{\T_r} \psi_{\T}(x,r)\int (-\dot \rho_r(y-x))\ps{\II_{\T_r(x)}(\pi_{\T_r(x)}\nabla u(y),\pi_{\T_r(x)}(y-x))}{\nabla u(y) } + \ers(r)\notag \\
	&=-2r^4\int_{\T_r} \psi_{\T}(x,r)\int \ps{\II_{\T_r(x)}(\pi_{\T_r(x)}\nabla u(y),\pi_{\T_r(x)}\nabla^{(x)} \rho_r(y-x))}{\nabla u(y) } + \ers(r)\, ,
\end{align}
Integrating by parts we get
\begin{align}
 \ff_3(r)=& 2r^4\int_{\T_r} \int \rho_r(y-x)\ps{\II_{\T_r(x)}(\pi_{\T_r(x)}\nabla u(y),\pi_{\T_r(x)}\nabla^{(x)} \psi_\T(x,r)}{\nabla u(y) }\notag\\
 + & 2r^4\int_{\T_r}\psi_{\T}(x,r)\int\rho_r(y-x) \langle \II_{\T_r}(x),\nabla u(y)\rangle^2\notag \\
 +&2r^4\int_{\T_r} \psi_{\T}(x,r)\int \rho_r(y-x)\big\langle \nabla u(y),\nabla^{(x)}_{\pi_{\T_r(x)} \nabla u(y)}H_{\T_r}(x)\big\rangle +\ers(r)\, .
\end{align}
We can deal with the first and last terms as in the previous subsection and conclude that
\begin{align}
 \ff_3(r)=&2r^4\int_{\T_r}\psi_{\T}(x,r)\int\rho_r(y-x) \langle \II_{\T_r}(x),\nabla u(y)\rangle^2+\ers(r)\,
\end{align}
as desired.
\vspace{.3cm}
\subsubsection{Estimating $\ff_2\leq (m-2)\ff_3+\ers(r)$.  } The last piece to prove is the inequality
\begin{align}
 \langle H_{\T_r}(x),\nabla u(y)\rangle^2\leq (m-2)\langle \II_{\T_r}(x),\nabla u(y)\rangle^2\, ,
\end{align}
which simply follows from the fact that $\T_r$ is $m-2$ dimensional and $H_{\T_r(x)} = \operatorname{tr}(\II_{\T_r(x)})$.

	\vspace{.3cm}
\subsection{The Radial Derivative \texorpdfstring{$r\frac{d}{dr}\vartheta_\cL(\T_r)$}{} .}\label{ss:radial_energy:radial}

By applying Lemma \ref{l:radial_energy:f2f3} we can now rewrite the stationary equation \eqref{e:radial_energy:stationary_equation1} as
\begin{align}
\hat\vartheta_n(\T_r) +2\vartheta_\cL(\T_r)	\leq \hat\vartheta_\alpha(\T_r)+\int_{\T_r}\psi_\T\int\rho_r|\pi^\perp_{x,r}(y-x)|^2|\pi_{x,r}\nabla u|^2+2(m-1)r^4\int_{\T_r}\psi_{\T}\int\rho_r\langle A_{\T_r},\nabla u\rangle^2+\ers(r)\, .
\end{align}

We are left with two terms above which are not apriori controllable by $\ers(r)$ .
Our goal will be to bound these terms by a Dini integrable term $r\frac{d}{dr}\vartheta_\cL(\T_r)$ plus lower order errors.
As explained in the outline of Section \ref{ss:radial_energy:outline}, there are two subtle points to these estimates.
The first is that while $r\frac{d}{dr}\vartheta_\cL(\T_r)$ has vanishing Dini integral $\int r\frac{d}{dr}\vartheta_\cL(\T_r)\frac{dr}{r}=0$, it is not apriori clear its absolute Dini integral $\int r|\frac{d}{dr}\vartheta_\cL(\T_r)|\frac{dr}{r}$ is small.
The second is that the estimating of $r\frac{d}{dr}\vartheta_\cL(\T_r)$ does itself involve some higher order errors which are not controllable by $\ers(r)$.
The Euler Lagrange for the best plane $\cL_{x,r}$ will again be our savior in these estimates.
Our main result for this subsection is the following:
\begin{lemma}\label{l:radial_energy:radial_derivative}
We have that
\begin{align}
r\frac{d}{dr}\vartheta_\cL(\T_r)=\int_{\T_r}\psi_\T\int\rho_r|\pi^\perp_{x,r}(y-x)|^2|\pi_{x,r}\nabla u|^2+2r^4\int_{\T_r}\psi_{\T}\int\rho_r\langle A_{\T_r},\nabla u\rangle^2+\ers(r)\, ,
\end{align}
where $\ers(r)$ satisfies \eqref{e:radial_energy:p_radial_energy_error} .
\end{lemma}

\vspace{.3cm}
\subsubsection{Computing $r\frac{d}{dr}\vartheta_\cL(\T_r)$.  }

Recall that
\begin{align}
	\vartheta_\cL(\T_r) = r^2\int_{\T_r}\psi_\T(x,r)\int \rho_r(y-x)\abs{\pi_{x,r}\nabla u(y)}^2\, ,
\end{align}
so that we can make the first computation
\begin{align}
\ r\frac{d}{dr}\vartheta_\cL(\T_r) &= r^2\int_{\T_r}\psi_\T(x,r)\ps{rH_{\T_r}}{\frac{d}{dr}\T_r}\int \rho_r(y-x)\abs{\pi_{x,r}\nabla u(y)}^2\label{e:radial_energy:ddrtheta1}\\
&+r^2\int_{\T_r}r\frac{d}{dr}\psi_\T(x,r)\int \rho_r(y-x)\abs{\pi_{x,r}\nabla u(y)}^2\label{e:radial_energy:ddrtheta2}\\
&+r^2\int_{\T_r}\psi_\T(x,r)\int \rho_r(y-x)\ton{r\partial_r \ln\rho_r +m - \frac{\abs{y-x}^2}{r^2}}\abs{\pi_{x,r}\nabla u(y)}^2\label{e:radial_energy:ddrtheta3}\\
&+r^2\int_{\T_r}\psi_\T(x,r)\int \rho_r(y-x)\frac{d}{dr}\abs{\pi_{x,r} \nabla u(y)}^2\label{e:radial_energy:ddrtheta5}\\
&+(2-m)\vartheta_\cL(\T_r)+r^2\int_{\T_r}\psi_\T(x,r)\int \rho_r(y-x)\frac{\abs{y-x}^2}{r^2}\,\abs{\pi_{x,r}\nabla u(y)}^2\label{e:radial_energy:ddrtheta4}
\end{align}

Let us now deal with these terms one at a time.

\vspace{.3cm}
\subsubsection{Estimating \eqref{e:radial_energy:ddrtheta1} }

We claim that
\begin{align}r^2\int_{\T_r}\psi_\T(x,r)\ps{rH_{\T_r}}{\frac{d}{dr}\T_r}\int \rho_r(y-x)|\pi_{x,r}\nabla u(y)|^2=\ers(r)\, .
\end{align}
This follows immediately from Theorem \ref{t:approximating_submanifold} and its Remark \ref{r:T_r_II_estimates}.

\vspace{.3cm}
\subsubsection{Estimating \eqref{e:radial_energy:ddrtheta2} } It is clear by definition of $\ers(r)$ (see \eqref{e:radial_energy:p_radial_energy_error}) that
\begin{gather}
 r^2\int_{\T_r}r\frac{d}{dr}\psi_\T(x,r)\int \rho_r(y-x)|\pi_{x,r}\nabla u(y)|^2 = \ers(r)\, .
\end{gather}

\vspace{.3cm}
\subsubsection{Estimating the error terms of \eqref{e:radial_energy:ddrtheta3} }
We claim that
\begin{align}
 r^2\int_{\T_r}\psi_\T(x,r)\int \rho_r(y-x)\ton{r\partial_r \ln\rho_r +m - \frac{\abs{y-x}^2}{r^2}}\abs{\pi_{x,r}\nabla u(y)}^2=\ers(r)\, .
\end{align}
In order to prove it, we observe that
\begin{gather}
 r\partial_r \ln \rho_r(y-x)= -m + \frac{-\dot \rho_r(y-x)}{\rho_r(y-x)}\frac{\abs{y-x}^2}{r^2}\, ,
\end{gather}
thus we obtain that
\begin{align}
 &r^2\int_{\T_r}\psi_\T(x,r)\int \rho_r(y-x)\ton{r\partial_r \ln\rho_r +m - \frac{\abs{y-x}^2}{r^2}}\abs{\pi_{x,r}\nabla u(y)}^2\notag \\
 =  - &r^2\int_{\T_r}\psi_\T(x,r)\int \qua{\rho_r(y-x) +\dot \rho_r(y-x)}\frac{\abs{y-x}^2}{r^2}\abs{\pi_{x,r}\nabla u(y)}^2
\end{align}
By Lemma \ref{l:rho_basic_properties}, we have
\begin{gather}
 \frac{\abs{y-x}^2}{r^2}\abs{\rho_r(y-x) +\dot \rho_r(y-x)}\leq C(m) \frac{\abs{y-x}^2}{r^2} e^{-R/2}\rho_{\sqrt 2 r} (y-x) \leq C(m) e^{-R/2} \rho_{2r}(y-x)\, ,
\end{gather}
and Remark \ref{r:d_LLr_LL2r_small} allows us to conclude.

\vspace{.3cm}
\subsubsection{Estimating \eqref{e:radial_energy:ddrtheta5} } This term may be apriori problematic, but is actually flat out zero.  This is because of the definition of $\LL_{x,r}$, which is an eigenspace for the quadratic form
\begin{gather}
 Q_{x,r}(v,w)=\int \rho_r(y-x) \ps{\nabla u(y)}{v}\ps{\nabla u(y)}{w}dy\, .
\end{gather}
In particular, this implies that $\LL_{x,r}$ is preserved by $Q_{x,r}$. That is, $\int \rho_r(y-x) \ps{\nabla_i u(y)}{v}\nabla^i u(y)dy\in \LL_{x,r}$ for all $v\in \LL_{x,r}$, and as a consequence
\begin{gather}
 \int \rho_r(y-x)\frac{d}{dr}\abs{\pi_{x,r} \nabla u(y)}^2=2\int \rho_r(y-x)\pi_{x,r}(\nabla u) \cdot \frac{d}{dr} \pi_{x,r}(\nabla u)=0\, .
\end{gather}

\vspace{.3cm}
\subsubsection{Estimating the $L$-energy in \eqref{e:radial_energy:ddrtheta4} }
We claim that
\begin{align}
 &(2-m)\vartheta_\cL(\T_r)+r^2\int_{\T_r}\psi_\T(x,r)\int \rho_r(y-x)\frac{\abs{y-x}^2}{r^2}\,\abs{\pi_{x,r}\nabla u(y)}^2\notag\\
 =&r^2\int_{\T_r}\psi_\T(x,r)\int \rho_r(y-x)\frac{\abs{\pi_{x,r}^\perp(y-x)}^2}{r^2}\,\abs{\pi_{x,r}\nabla u(y)}^2 + \ff_3(r)+\ers(r)\, .
\end{align}
In order to prove this, let us split $\abs{y-x}^2= \abs{\pi_{x,r}(y-x)}^2+\abs{\pi_{x,r}^\perp(y-x)}^2$. With this, the claim is equivalent to
\begin{align}
&(2-m)\vartheta_\cL(\T_r)+r^2\int_{\T_r}\psi_\T(x,r)\int \rho_r(y-x)\frac{\abs{\pi_{x,r}(y-x)}^2}{r^2}\,\abs{\pi_{x,r}\nabla u(y)}^2=\ff_3(r)+\ers(r)\, .
\end{align}

Proceeding as in Subsection \ref{sss:integrating_rho_vs_dot_rho}, we can exchange $\rho_r$ with $-\dot \rho_r$ up to an $\ers(r)$ error. In particular:
\begin{align}
 &r^2\int_{\T_r}\psi_\T(x,r)\int \rho_r(y-x)\frac{\abs{\pi_{x,r}(y-x)}^2}{r^2}\,\abs{\pi_{x,r}\nabla u(y)}^2\notag \\
 =&r^2\int_{\T_r}\psi_\T(x,r)\int (-\dot \rho_r(y-x))\frac{\abs{\pi_{x,r}(y-x)}^2}{r^2}\,\abs{\pi_{x,r}\nabla u(y)}^2+\ers(r)\, .
\end{align}

As in Subsection \ref{sss:exchange_pi}, by Lemma \ref{l:improved_comparison_LL_T} we can exchange $\pi_{x,r}$ with $\pi_{\T_r(x)}$ up to an $\ers(r)$ error, in particular:
\begin{align}
 &r^2\int_{\T_r}\psi_\T(x,r)\int (-\dot \rho_r(y-x))\frac{\abs{\pi_{x,r}(y-x)}^2}{r^2}\,\abs{\pi_{x,r}\nabla u(y)}^2\notag \\
 =&r^2\int_{\T_r}\psi_\T(x,r)\int (-\dot \rho_r(y-x))\frac{\abs{\pi_{\T_r(x)}(y-x)}^2}{r^2}\,\abs{\pi_{\T_r(x)}\nabla u(y)}^2+\ers(r)\notag\\
 =&r^2\int_{\T_r}\psi_\T(x,r)\int \ps{\pi_{\T_r(x)} \nabla^{(x)} \rho_r(y-x)}{\pi_{\T_r(x)}(y-x)}\,\abs{\pi_{\T_r(x)}\nabla u(y)}^2+\ers(r)\, .
\end{align}
Integrating by parts in $x\in \T_r$, and estimating in a similar way than in Subsection \ref{sss:mean_and_second}, we obtain
\begin{align}
 &r^2\int_{\T_r}\psi_\T(x,r)\int \ps{\pi_{\T_r(x)} \nabla^{(x)} \rho_r(y-x)}{\pi_{\T_r(x)}(y-x)}\,\abs{\pi_{\T_r(x)}\nabla u(y)}^2\notag\\
 = (m-2) \vartheta_\LL(\T_r)-2&r^2\int_{\T_r}\psi_\T(x,r)\int \rho_r(y-x)\ps{\nabla u(y)\otimes \nabla u(y)}{\nabla^{(x)}_{\T_r}(\pi_{\T_r(x)})(y-x)}  +\ers(r)\notag\\
 =(m-2) \vartheta_\LL(\T_r)+\ff_3&(r) +\ers(r)\, .
\end{align}

\vspace{.3cm}
\subsubsection{Finishing the Proof of Lemma \ref{l:radial_energy:radial_derivative}. }
Summing up, we have obtained that
\begin{gather}
 r\frac{d}{dr}\vartheta_\cL(\T_r)=\int_{\T_r}\psi_\T\int\rho_r|\pi^\perp_{x,r}(y-x)|^2|\pi_{x,r}\nabla u|^2+\ff_3(r)+\ers(r)
\end{gather}

Lemma \ref{l:radial_energy:radial_derivative} follows immediately from Lemma \ref{l:radial_energy:f2f3}.(2), and this concludes the proof.

	\vspace{.5cm}
\subsection{Proof of Proposition \ref{p:radial_energy:radial_energy}}\label{ss:radial_energy:proof_proposition}

Let us combine the various Lemmas of this Section to prove Proposition \ref{p:radial_energy:radial_energy}.  Beginning with Lemmas \ref{l:radial_energy:stationary_vector_field}, \ref{l:radial_energy:f1} and \ref{l:radial_energy:f2f3} we have that
\begin{align}
\hat\vartheta_n(\T_r) +2\vartheta_\cL(\T_r)	\leq \hat\vartheta_\alpha(\T_r)+\int_{\T_r}\psi_\T\int\rho_r|\pi^\perp_{x,r}(y-x)|^2|\pi_{x,r}\nabla u|^2+2(m-1)\int_{\T_r}\psi_{\T}\int\rho_r\langle A_{\T_r},\nabla u\rangle^2+\ers(r)\notag\\
\leq\hat\vartheta_\alpha(\T_r)+(m-1)\int_{\T_r}\psi_\T\int\rho_r|\pi^\perp_{x,r}(y-x)|^2|\pi_{x,r}\nabla u|^2+2(m-1)\int_{\T_r}\psi_{\T}\int\rho_r\langle A_{\T_r},\nabla u\rangle^2+\ers(r)\, .
\end{align}

If we now combine this with Lemma \ref{l:radial_energy:radial_derivative} we have that
\begin{align}
\hat\vartheta_n(\T_r) +2\vartheta_\cL(\T_r)	\leq  \hat\vartheta_\alpha(\T_r)+(m-1)r\frac{d}{dr}\vartheta_\cL(\T_r)+\ers(r)\, ,
\end{align}
which precisely proves Proposition \ref{p:radial_energy:radial_energy}. $\qed$\\

\bibliographystyle{aomalpha}
\bibliography{qstrat}

\end{document}